\documentclass[reqno,english]{amsart}

\usepackage{amsmath,amsfonts,amssymb,graphicx,amsthm,url}
\usepackage[noadjust]{cite}
\usepackage{stmaryrd}
\usepackage{mathrsfs,booktabs,tabularx}
\usepackage{xifthen,xcolor,tikz,setspace}
\usetikzlibrary{decorations.pathmorphing,patterns,shapes,calc,decorations}
\usetikzlibrary{decorations.pathreplacing}
\usepackage{mathtools,upgreek,paralist}
\usepackage[final]{showkeys} %add in 'final' into parameter to remove showkeys

\definecolor{refkey}{gray}{.75}
\definecolor{labelkey}{gray}{.5}
\usepackage[algo2e,boxed,vlined,algoruled]{algorithm2e}
\usepackage{comment}
\usepackage[shortlabels]{enumitem}

\usepackage[letterpaper]{geometry}
\usepackage[colorinlistoftodos]{todonotes}
\presetkeys{todonotes}{inline, color=green}{}

\usepackage{accents}
\usepackage[algo2e,boxed,vlined,algoruled]{algorithm2e}

\usepackage[small]{caption}
\usepackage[colorlinks=true]{hyperref}
\colorlet{DarkGreen}{green!50!black}
\colorlet{DarkGray}{gray!60!black}
\numberwithin{equation}{section}

\DeclarePairedDelimiter{\ceil}{\lceil}{\rceil}

\renewcommand{\restriction}{\mathord{\upharpoonright}}
\renewcommand{\epsilon}{\varepsilon}

\newcommand{\one}{\mathbf{1}}
\newcommand\norm[1]{\lVert#1\rVert}

\usepackage{color}
 \definecolor{refkey}{gray}{.5}
 \definecolor{labelkey}{gray}{.5}
\definecolor{light}{gray}{.9}

\usepackage{soul}

\usepackage[capitalise]{cleveref}

\newtheorem{theorem}{Theorem}[section]
\newtheorem*{theorem*}{Theorem}
\newtheorem{lemma}[theorem]{Lemma}

\newtheorem{claim}[theorem]{Claim}
\crefname{claim}{Claim}{Claims}
\newtheorem{proposition}[theorem]{Proposition}

\newtheorem{observation}[theorem]{Observation}

\newtheorem{corollary}[theorem]{Corollary}

\theoremstyle{definition}{
\newtheorem{example}[theorem]{Example}
\newtheorem{definition}[theorem]{Definition}

\newtheorem*{definition*}{Definition}

\newtheorem{remark}[theorem]{Remark}
\newtheorem*{remark*}{Remark}

}

\newcommand{\E}{\mathbb E}

\newcommand{\R}{\mathbb R}
\newcommand{\Z}{\mathbb Z}

\newcommand{\cA}{\ensuremath{\mathcal A}}

\newcommand{\cC}{\ensuremath{\mathcal C}}

\newcommand{\cF}{\ensuremath{\mathcal F}}

\newcommand{\cH}{\ensuremath{\mathcal H}}
\newcommand{\cI}{\ensuremath{\mathcal I}}
\newcommand{\cJ}{\ensuremath{\mathcal J}}
\newcommand{\cK}{\ensuremath{\mathcal K}}
\newcommand{\cL}{\ensuremath{\mathcal L}}

\newcommand{\cP}{\ensuremath{\mathcal P}}

\newcommand{\cS}{\ensuremath{\mathcal S}}

\newcommand{\cV}{\ensuremath{\mathcal V}}
\newcommand{\cW}{\ensuremath{\mathcal W}}
\newcommand{\cX}{\ensuremath{\mathcal X}}
\newcommand{\cY}{\ensuremath{\mathcal Y}}

\newcommand{\llb }{\llbracket}
\newcommand{\rrb }{\rrbracket}

\newcommand{\sH}{{\ensuremath{\mathscr H}}}
\newcommand{\sT}{{\ensuremath{\mathscr T}}}

\newcommand{\fB}{\mathfrak{B}}

\newcommand{\fm}{\mathfrak{m}}
\newcommand{\fF}{\mathfrak{F}}

\newcommand{\fW}{\mathfrak{W}}

\newcommand{\sB}{{\ensuremath{\mathscr B}}}

\newcommand{\sE}{{\ensuremath{\mathscr E}}}
\newcommand{\sF}{{\ensuremath{\mathscr F}}}
\newcommand{\sG}{{\ensuremath{\mathscr G}}}

\newcommand{\sX}{{\ensuremath{\mathscr X}}}
\newcommand{\sY}{{\ensuremath{\mathscr Y}}}
\newcommand{\sZ}{{\ensuremath{\mathscr Z}}}

\newcommand{\fD}{\mathfrak{D}}

\newcommand{\scE}{{\ensuremath{\textsc{e}}}}

\newcommand{\scV}{{\ensuremath{\textsc{v}}}}

\newcommand{\g}{{\ensuremath{\mathbf g}}}

\newcommand{\bB}{{\ensuremath{\mathbf B}}}

\newcommand{\bD}{{\ensuremath{\mathbf D}}}

\newcommand{\bF}{{\ensuremath{\mathbf F}}}

\newcommand{\bh}{{\ensuremath{\mathbf h}}}

\newcommand{\bR}{{\ensuremath{\mathbf R}}}

\newcommand{\bW}{{\ensuremath{\mathbf W}}}
\newcommand{\bX}{{\ensuremath{\mathbf X}}}

\newcommand{\fs}{\mathfrak{s}}

 \renewcommand{\epsilon}{\varepsilon}

\newcommand{\Clust}{{\mathsf{Clust}}}
\newcommand{\Cone}{{\mathsf{Cone}}}
\newcommand{\Iso}{{\mathsf{Iso}}}
\newcommand{\Incr}{{\mathsf{Incr}}}

\newcommand{\Cyl}{{\mathsf{Cyl}}}

\newcommand{\ext}{{\mathsf{ex}}}
\newcommand{\Itop}{\cI_{\mathsf{top}}}
\newcommand{\Ibot}{\cI_{\mathsf{bot}}}
\newcommand{\Vtop}{\cV_\mathsf{top}}
\newcommand{\Vbot}{\cV_\mathsf{bot}}
\newcommand{\Atop}{\widehat{\cV}_{\mathsf{top}}}
\newcommand{\Abot}{\widehat{\cV}_{\mathsf{bot}}}
\newcommand{\Top}{\mathsf{top}}
\newcommand{\Bot}{\mathsf{bot}}
\newcommand{\Vred}{\cV_\mathsf{red}}
\newcommand{\Vblue}{\cV_\mathsf{blue}}
\newcommand{\Ared}{\widehat{\cV}_{\mathsf{red}}}

\newcommand{\Ablue}{\widehat{\cV}_{\mathsf{blue}}}

\newcommand{\Red}{{\mathsf{red}}}
\newcommand{\Blue}{{\mathsf{blue}}}
\newcommand{\noRed}{{\mathsf{nred}}}
\newcommand{\out}{{\mathsf{out}}}
\newcommand{\ins}{{\mathsf{ins}}}
\newcommand{\sep}{{\fD}}
\newcommand{\clfaces}[1][\omega]{{\fF_#1^{\texttt{c}}}}
\newcommand{\opfaces}[1][\omega]{{\fF_#1}}
\newcommand{\ex}{{\mathfrak{e}_1}}
\newcommand{\ey}{{\mathfrak{e}_2}}
\newcommand{\ez}{{\mathfrak{e}_3}}

\DeclareMathOperator{\hgt}{ht}

\newcommand{\tv}{{\textsc{tv}}}
\newcommand{\trivincr}{\varnothing}

\makeatletter
\crefname{step}{Step}{Steps}
\crefname{case}{Case}{Cases}
\newcommand{\superimpose}[2]{%
  {\ooalign{$#1\@firstoftwo#2$\cr\hfil$#1\@secondoftwo#2$\hfil\cr}}}
  
\newcommand{\sbullet}{%
  \hbox{\fontfamily{lmr}\fontsize{.4\dimexpr(\f@size pt)}{0}\selectfont\textbullet}}

\makeatother

\begin{document}

\title{Extrema of 3D Potts interfaces}

\author{Joseph Chen}
\address{J.\ Chen\hfill\break
Courant Institute\\ New York University\\
251 Mercer Street\\ New York, NY 10012, USA.}
\email{jlc871@courant.nyu.edu}

\author{Eyal Lubetzky}
\address{E.\ Lubetzky\hfill\break
Courant Institute\\ New York University\\
251 Mercer Street\\ New York, NY 10012, USA.}
\email{eyal@courant.nyu.edu}

\vspace{-1cm}

\begin{abstract}
    The interface between the plus and minus phases in the low temperature 3D Ising model has been intensely studied since Dobrushin’s pioneering works in the early 1970’s established its rigidity. Advances in the last decade yielded the tightness of the maximum of the interface of this Ising model on the cylinder of side length $n$, around a mean that is asymptotically $c\log n$ for an explicit  $c$ (temperature dependent). In~this work, we establish   analogous results for the 3D Potts and random cluster (FK) models. Compared to 3D Ising, the Potts model and its lack of monotonicity form obstacles for existing methods, calling for new proof ideas, while its interfaces (and associated extrema) exhibit richer behavior. 
    We show that the maxima and minima of the interface bounding the blue component in the 3D Potts interface, and those of the interface bounding the bottom component in the 3D FK model, are governed by 4 different large deviation rates, whence the corresponding global extrema feature 4 distinct constants $c$ as above.
    Due to the above obstacles, our methods are initially only applicable to 1 of these 4 interface extrema, and additional ideas are needed to 
    recover the other 3 rates given the behavior of the first one.
\end{abstract}

%\maketitle
%\vspace{-0.8cm}
{\mbox{}
\vspace{-1.5cm}
\maketitle
}
\vspace{-.8cm}

\section{Introduction}
The \emph{Potts model} on a finite graph $\Lambda = (V, E)$ is a random assignment of colors to vertices of $V$ that penalizes adjacent vertices assigned with different colors. The number of possible colors is given by the integer parameter $q \geq 2$, and the aforementioned penalization is governed by the parameter $\beta > 0$, the inverse-temperature of the system: the probability of a vertex coloring $\sigma:V\to\{1,\ldots,q\}$ is given by
\[ \phi_\Lambda(\sigma) \propto e^{-\beta \cH(\sigma)}\,,\qquad\mbox{where}\qquad\cH(\sigma)= \#\{[u,v]\in E\,:\; \sigma_u\neq\sigma_v\}\,.
\]
Consider the half integer lattice with vertices $(\Z + \frac{1}{2})^3$. We will mainly consider the Potts model on the subgraph~$\Lambda_n$ of this lattice with vertices $\llb -\frac n2,\frac n 2\rrb^2 \times (\Z+\frac12)$. (Although $\Lambda_n$ is an infinite graph, one can e.g.\ consider the model on the finite truncation of $\Lambda_n$ to heights in $\llb -m,m\rrb$, then take the weak limit $m \to \infty$.) Define $\partial \Lambda_n^+$ as the vertices $x=(x_1,x_2,x_3) \in \Lambda_n$ such that $x_3 >0$ and $x$ is adjacent to some vertex of $\Lambda_n^c$, and define $\partial \Lambda_n^-$ analogously. We refer to the model with a \emph{boundary condition} $\eta$ as the conditional distribution of the model on some larger graph containing $\Lambda_n$ where we fix $\sigma_v = \eta_v$ for all vertices not in $\Lambda_n$. Our focus is on the Potts model on $\Lambda_n$ with \emph{Dobrushin boundary conditions}, which correspond to $\eta$ that is $\Red$ for all vertices with height $> 0$ and $\Blue$ for all vertices with height $< 0$. Denote this distribution by $\phi_n$ for brevity.

We will consider the low temperature regime, where $\beta > \beta_0$ for a fixed large enough $\beta_0$. It is easy to see (via a standard Peierls argument) that $\phi_n$-almost surely there is a unique infinite connected component of $\Red$ vertices in $\sigma$ --- the one containing $\partial \Lambda_n^+$ --- and a unique infinite $\Blue$ component, the one containing~$\partial\Lambda_n^-$. Thus, there naturally arise two \emph{interfaces}, one separating the infinite $\Red$ component from everything below it, and one separating the infinite $\Blue$ component from everything above it. Formally,
to every edge $e = [x, y]$, consider the dual face $f_{[x, y]}$ that is the closed unit square centered at $\frac{x+y}{2}$ and perpendicular to~$e$. An \emph{interface} is a collection of faces such that every $\Lambda_n$-path of vertices from $\partial \Lambda_n^-$ to $\partial \Lambda_n^+$ must cross the interface. 

\begin{definition}[Potts interfaces] Let $\Vred$ denote the vertices of $\Lambda_n$ in the (a.s.\ unique) infinite $\Red$  cluster, i.e., every $v\in\Lambda_n$ from which there is a $\Lambda_n$-path of $\Red$ vertices in $\sigma$ from $v$ to $\partial \Lambda_n^+$. Let the augmented $\Red$ component, $\Ared$, be the union of $\Vred$ with all finite components of $\Vred^c$. Define the $\Red$ interface $\cI_\Red$ as the set of 
faces separating $\Ared$ and $\Ared^c$; that is,
every face $f_{[x,y]}$ between $x\in\Ared$ and $y\in\Ared^c$.
Analogously, the $\Blue$ interface $\cI_\Blue$ is defined via the infinite $\Blue$ cluster $\Vblue$, which is augmented into $\Ablue$.
\label{def:potts-interfaces}
\end{definition}

The interfaces are illustrated in \cref{fig:potts-3d,fig:potts_interfaces} in dimensions $d=3$ and $d=2$,  resp. Note that our results can be extended to dimensions $d\geq3$, yet the 2D behavior is starkly different (see \cref{sec:related-work} on the famous works of Dobrushin \cite{Dobrushin72,Dobrushin73} on the rigidity of Ising interface --- the case $q=2$ of the Potts model --- for $d\geq3$). Further, if we were to use $*$-connectivity (whereby $x,y$ are $*$-adjacent if $\|x-y\|_\infty \leq 1$, instead of graph adjacency in $\Z^3$ which corresponds to $\|x-y\|_1 \leq 1$) for defining the component $\Vred$ and the augmented~$\Ared$, then the interface $\cI_\Red$ would exactly coincide with the one defined in~\cite{Dobrushin72,GL_max}. As usual, different notions of connectivity would only affect the inclusion (or lack thereof) of finite bubbles
in $\cI_\Red$ (hence giving slightly different constants in the large deviation rates of local extrema); our choice here of standard adjacency maintains consistency with the classical random cluster interfaces (see \cref{subsec:rc-interfaces}).

Closely related to the Potts model is the \emph{random-cluster} or Fortuin--Kasteleyn (FK) model, which is a random edge configuration on the edges $E$ of $\Lambda$ with parameters $0<p<1$ and $q>0$. In every configuration $\omega$, edges are either \emph{open} (present, $\omega_e = 1)$ or \emph{closed} (missing, $\omega_e = 0)$. The probability of $\omega$ is given by
\[\mu_{\Lambda}(\omega) \propto p^{\#\{e\in E\,:\;\omega_e = 1\}}(1-p)^{\#\{e\in E\,:\;\omega_e = 0\}}q^{\kappa(\omega)}\,,\]
where the term $\kappa(\omega)$ denotes the number of connected components of the graph $(V, \{e\,:\;\omega_e=1\})$. We will refer to connected components of said graph as \emph{open clusters}. 

Let $\mu_n$ denote the random-cluster measure on $\Lambda_n$ with \emph{Dobrushin boundary conditions}, given by $\eta_e = 0$ if $e$ separates the upper and lower half-spaces ---  $e = [x, y]$ for some $x = (x_1, x_2, \frac12)$ and $ y = (y_1, y_2, -\frac12)$ --- and $\eta_e = 1$ otherwise. The relation between the Potts and random-cluster model, which we describe next, will necessitate a further conditioning on the (exponentially unlikely) event that $\partial \Lambda_n^+$ and $\partial \Lambda_n^-$ are not part of the same open cluster in $\omega$: denote this event by $\sep_n$, and let \[ \bar{\mu}_n(\cdot) = \mu_n(\cdot \mid \sep_n)\,.\]
When the Potts and random-cluster models on the same graph have the same (integer) value of $q$ and parameters $p = 1 - e^{-\beta}$, the two models can be coupled via the Edwards--Sokal coupling. We will assume this relation throughout this paper, with the exception that the results for the random-cluster model will be established for all real $q \geq 1$, not just integer valued $q$. Explicitly, for any finite graph $G = (V, E)$, the coupled FK--Potts model is given by the following joint measure on vertex spins $\sigma$ and edge spins $\omega$:
\begin{equation*}
    \phi(\sigma, \omega) \propto p^{\#\{e \in E\,:\;\omega_{e} = 1\}}(1-p)^{\#\{e \in E\,:\;\omega_e = 0\}}
    \prod_{e=[u,v]\,:\;
    \omega_e=1}\one_{\{\sigma_u = \sigma_v\}}\,.
\end{equation*}
It is easy to verify that the marginals on the spin and edge configurations give the Potts and random-cluster models respectively; furthermore, the conditional probabilities are such that if one samples a random-cluster model and colors each cluster uniformly at random, then the resulting coloring has the law of a Potts model. Consequently, (by considering the finite truncation of $\Lambda_n$ between heights $-m$ and $m$ and taking the weak limit as $m \to \infty$,) if we sample a random-cluster model on $\Lambda_n$ with Dobrushin boundary conditions conditional on $\sep_n$, fix the colors of clusters incident to $\partial \Lambda_n^+$ and $\partial \Lambda_n^-$ to be $\Red$ and $\Blue$ respectively, and color the remaining open clusters of vertices uniformly at random via $q$ colors, we get a Potts model with Dobrushin boundary conditions (e.g., 
\cite[\S2.2]{GheissariLubetzky18},\cite[Fact 3.4 and Cor.~3.5]{LubetzkySly12}.) As we always consider the Potts model in this context, by an abuse of notation we also let $\phi_n$ denote the coupled FK--Potts measure on~$\Lambda_n$. 

As was the case for the Potts model, there are two natural interfaces arising in the conditional FK distribution $\bar\mu_n$: one separating the 
``top'' open cluster containing $\partial \Lambda_n^+$ from everything below it, and one separating the 
``bottom'' open cluster containing $\partial \Lambda_n^-$ from everything above it.

\begin{figure}
\vspace{-0.15in}
\hspace{-.55in}
    \begin{tikzpicture}
   \node (fig1) at (0.5,0) {
    	\includegraphics[width=0.65\textwidth]{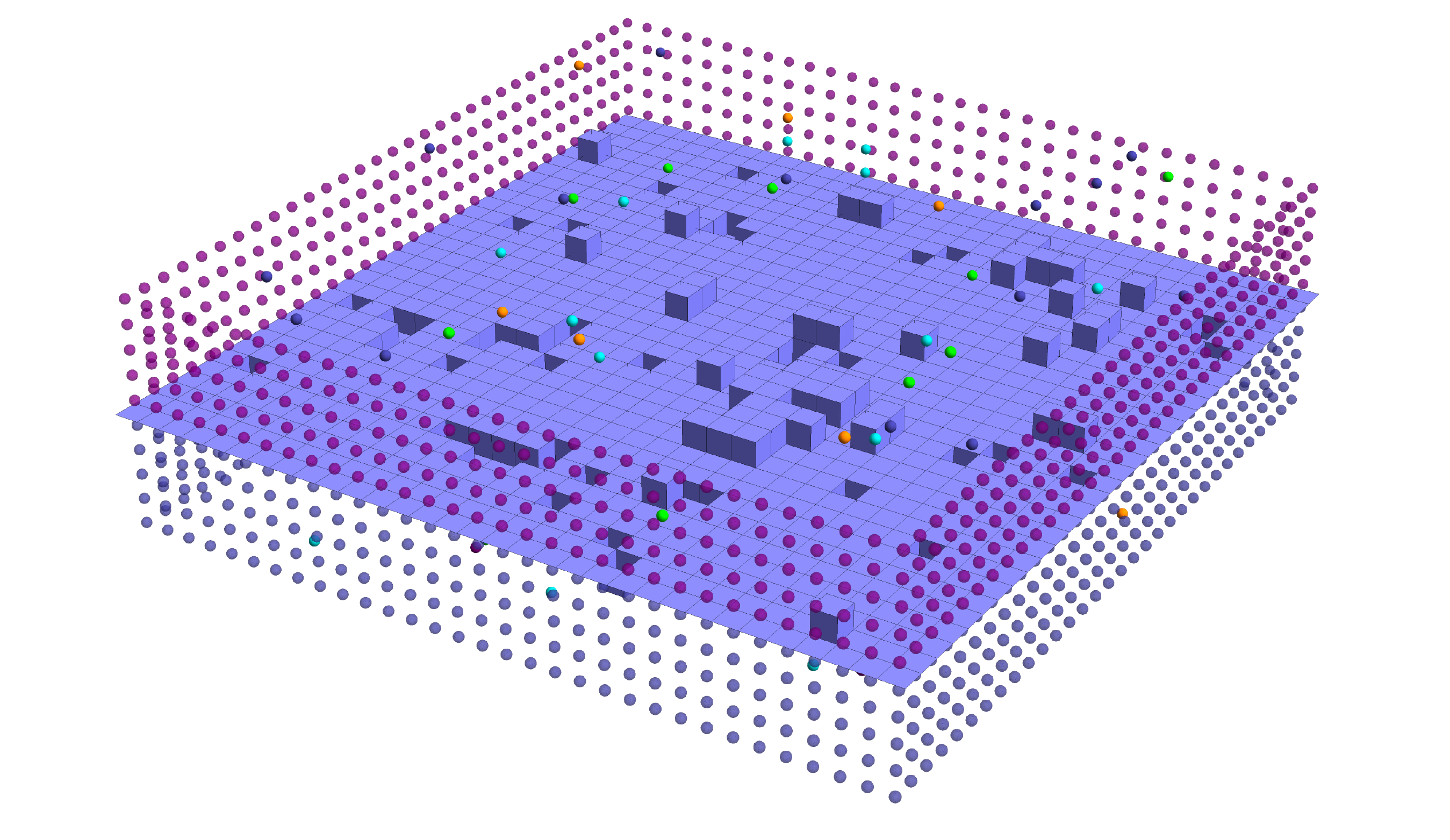}};
   \node (fig2) at (8.6,-1.1) {
   	\includegraphics[width=0.4\textwidth]{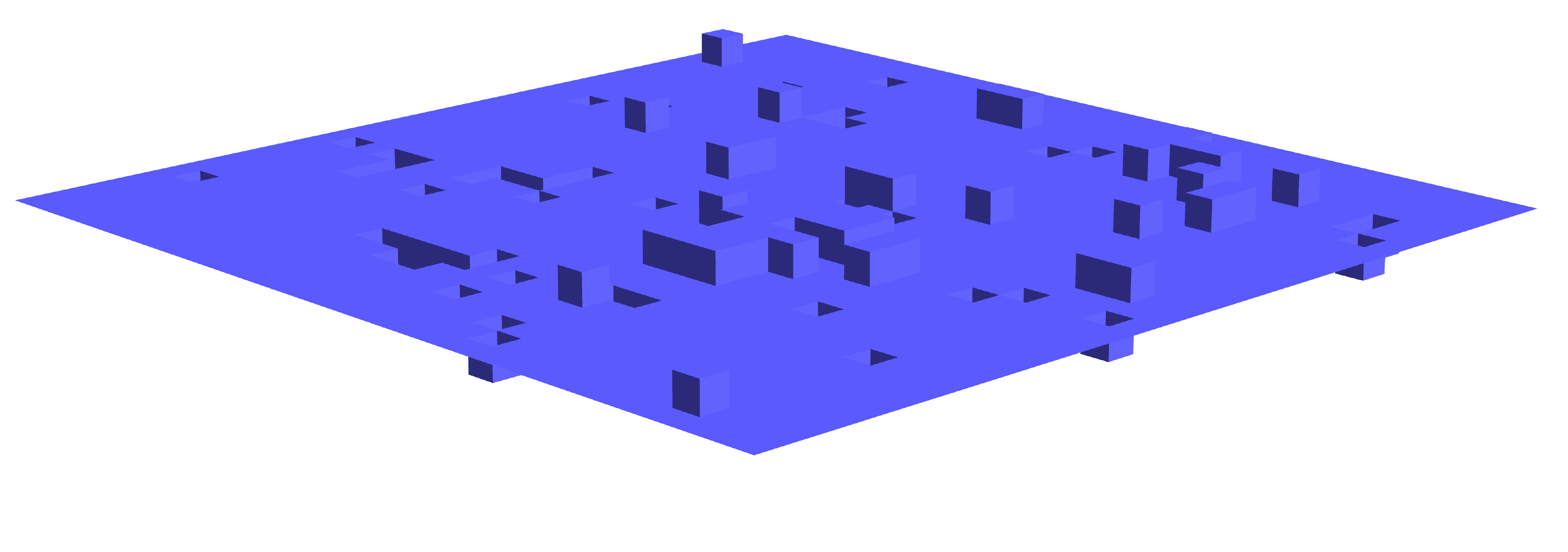}};
    \node (fig3) at (8.6,1.1) {
   	\includegraphics[width=0.4\textwidth]{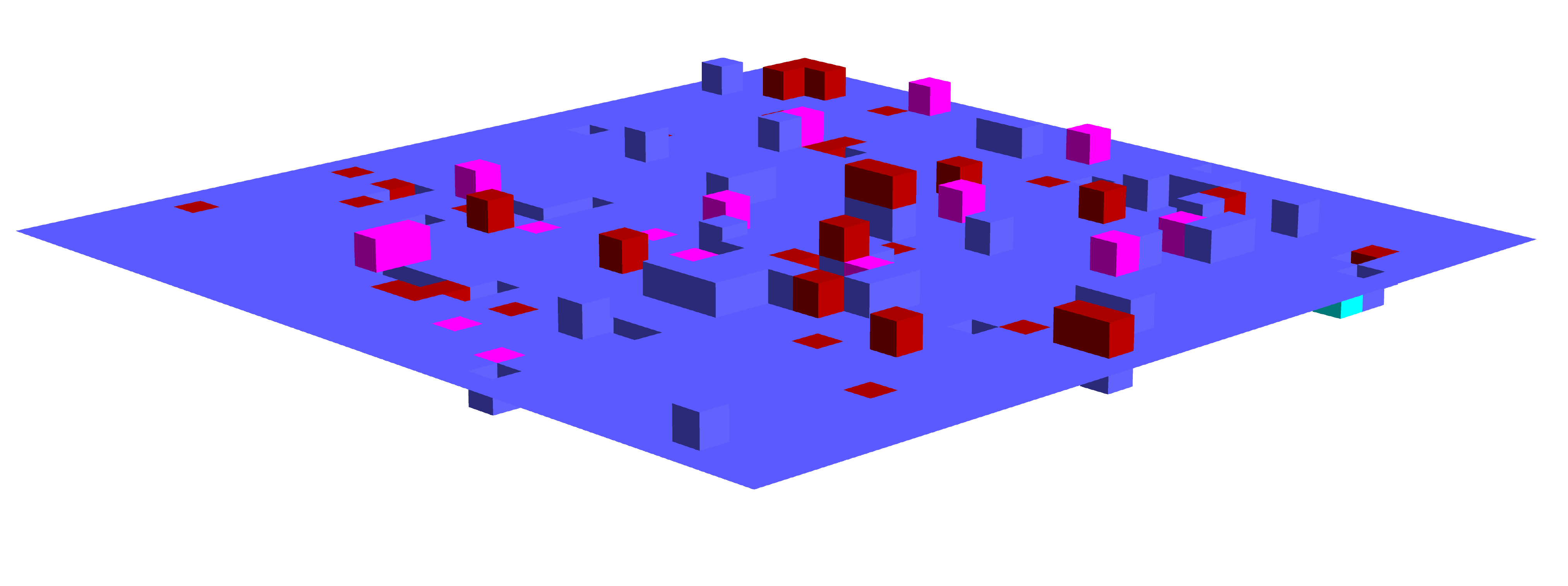}};
    \end{tikzpicture}
    \vspace{-0.7cm}
    \caption{The $\Blue$ interface $\cI_\Blue$ in the $5$-color 3D Potts model (not showing the $\Red$ vertices above $\cI_\Blue$ nor the $\Blue$ vertices below it). \emph{Right bottom}: different view of the same  $\Blue$ interface.
    \emph{Right top}: the faces of $\cI_\Blue$ and the other Potts and random-cluster interfaces $
    \cI_\Red$, $\cI_\Top$, $\cI_\Bot$.}
    \label{fig:potts-3d}
\vspace{-0.15in}
\end{figure}

\begin{definition}[Random-cluster interfaces]\label{def:top-interface} Let $\Vtop$ denote the vertices of $\Lambda_n$ in the $\Top$ open cluster of~$\omega$, i.e., every $v$ connected via an $\omega$-path  to $\partial \Lambda_n^+$. Let the augmented top component, $\Atop$, be $\Vtop$ along with all finite components of $\Vtop^c$ (w.r.t.\ to the full graph $\Lambda_n$). 
Define the $\Top$ interface $\Itop$ to be the set of faces separating vertices from $\Atop$ and $\Atop^c$. Analogously, define the $\mathsf{bottom}$ interface $\cI_\Bot$, and the augmented set $\Abot$ by starting with the vertices of the bottom component, i.e., the infinite open cluster containing $\partial \Lambda_n^-$.
\end{definition}

\begin{remark*}
When the Potts and FK configurations $\sigma,\omega$ are coupled through the Edwards--Sokal coupling $\phi$, as $\Atop \subseteq \Ared\subseteq\Ablue^c $ and $\Abot\subseteq\Ablue$, the 4 corresponding interfaces are \emph{ordered}: $\cI_\Top$, $\cI_\Red$, $\cI_\Blue$, $\cI_\Bot$.
\end{remark*}
\subsection{Results}
For the Ising model ($q=2$), the asymptotics of the maximum of the 3D interface, and its tightness around its mean, were recently established in \cite{GL_max,GL_tightness}.
Our main results are the analogous statements for the 4 interfaces (3D Potts $\cI_\Blue$ and $\cI_\Red$; 3D FK $\cI_\Bot$ and $\cI_\Top$) defined above. As we explain in \cref{sec:subsec-proof-ideas}, 
significant work is required compared to the Ising case, mainly due to the lack of monotonicity (both in the Potts model and in the conditional FK model $\bar\mu_n$), as well as the more delicate interactions in the FK model. Notably, a large portion of the proof is dedicated to an argument that is applicable for the maximum of 1 of these 4 interfaces, $\cI_\Top$, yet fails for the other 3 interfaces. We then recover the remaining maxima by analyzing the conditional behavior of the respective interface conditional on the behavior of the $\Top$ interface.

\begin{figure}
\vspace{-0.1in}
    \begin{tikzpicture}
   \node (fig1) at (0.5,0) {
    	\includegraphics[width=0.65\textwidth]{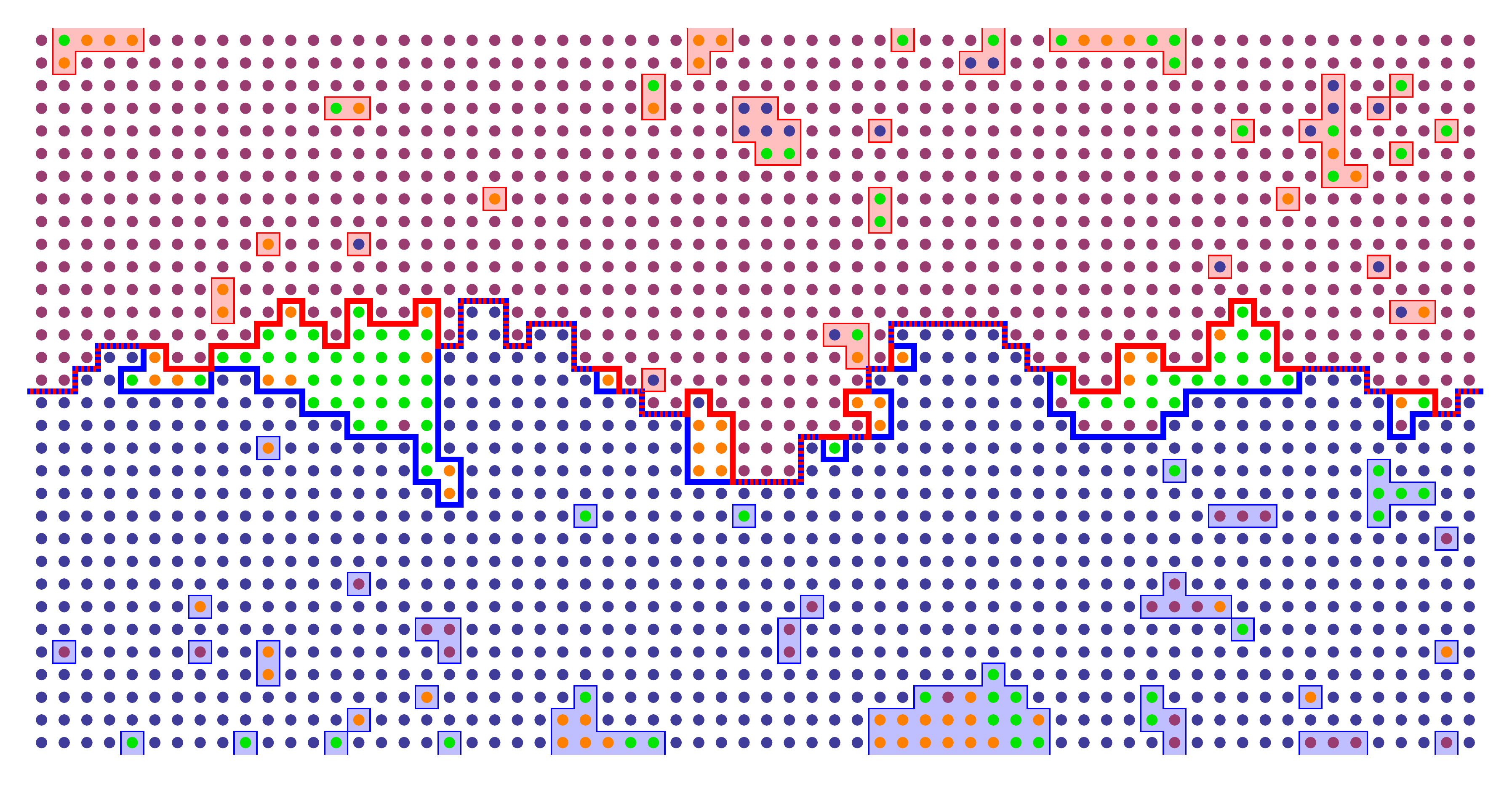}};
   \node (fig2) at (8.6,-1.35) {
   	\includegraphics[width=0.3\textwidth]{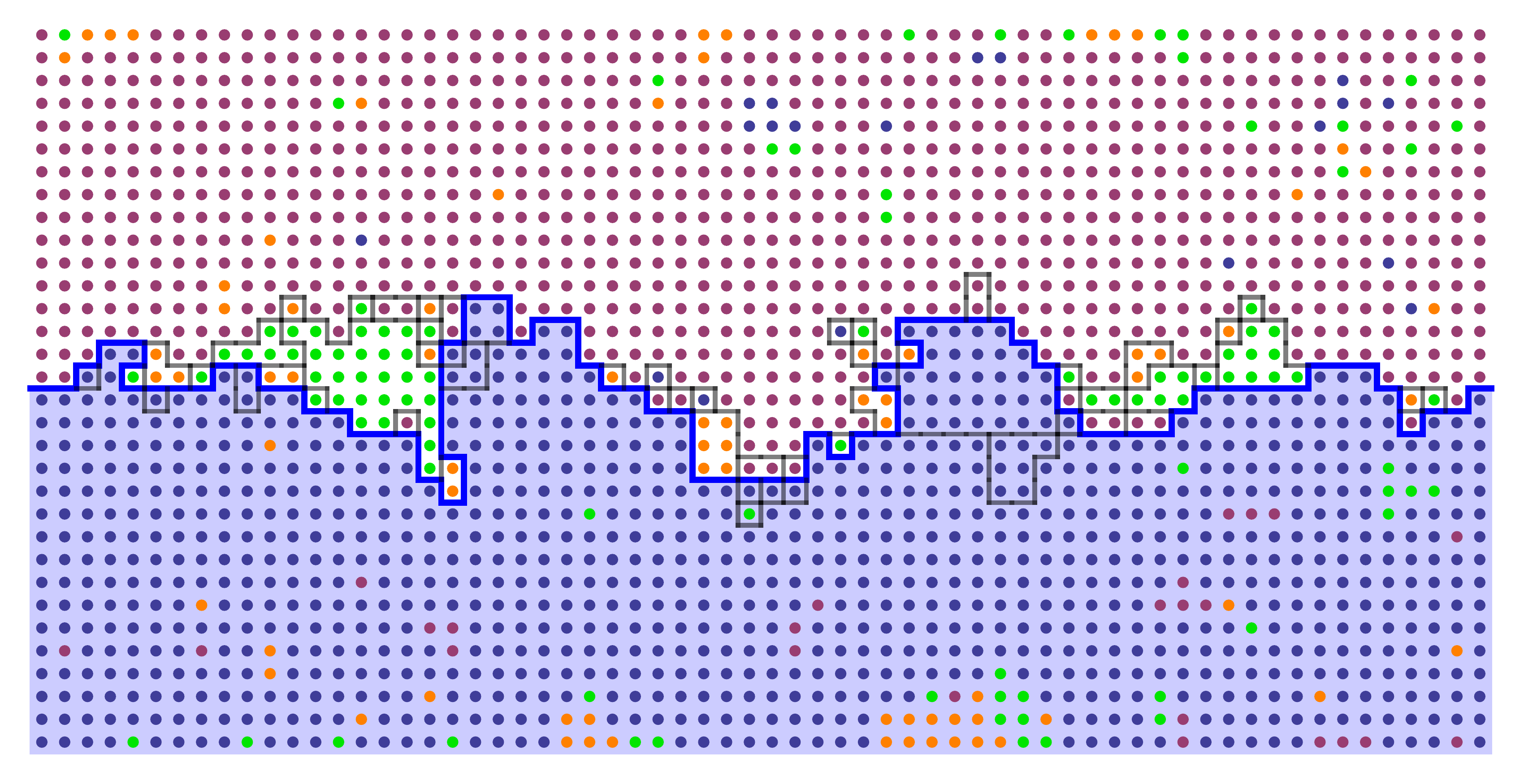}};
    \node (fig3) at (8.6,1.35) {
   	\includegraphics[width=0.3\textwidth]{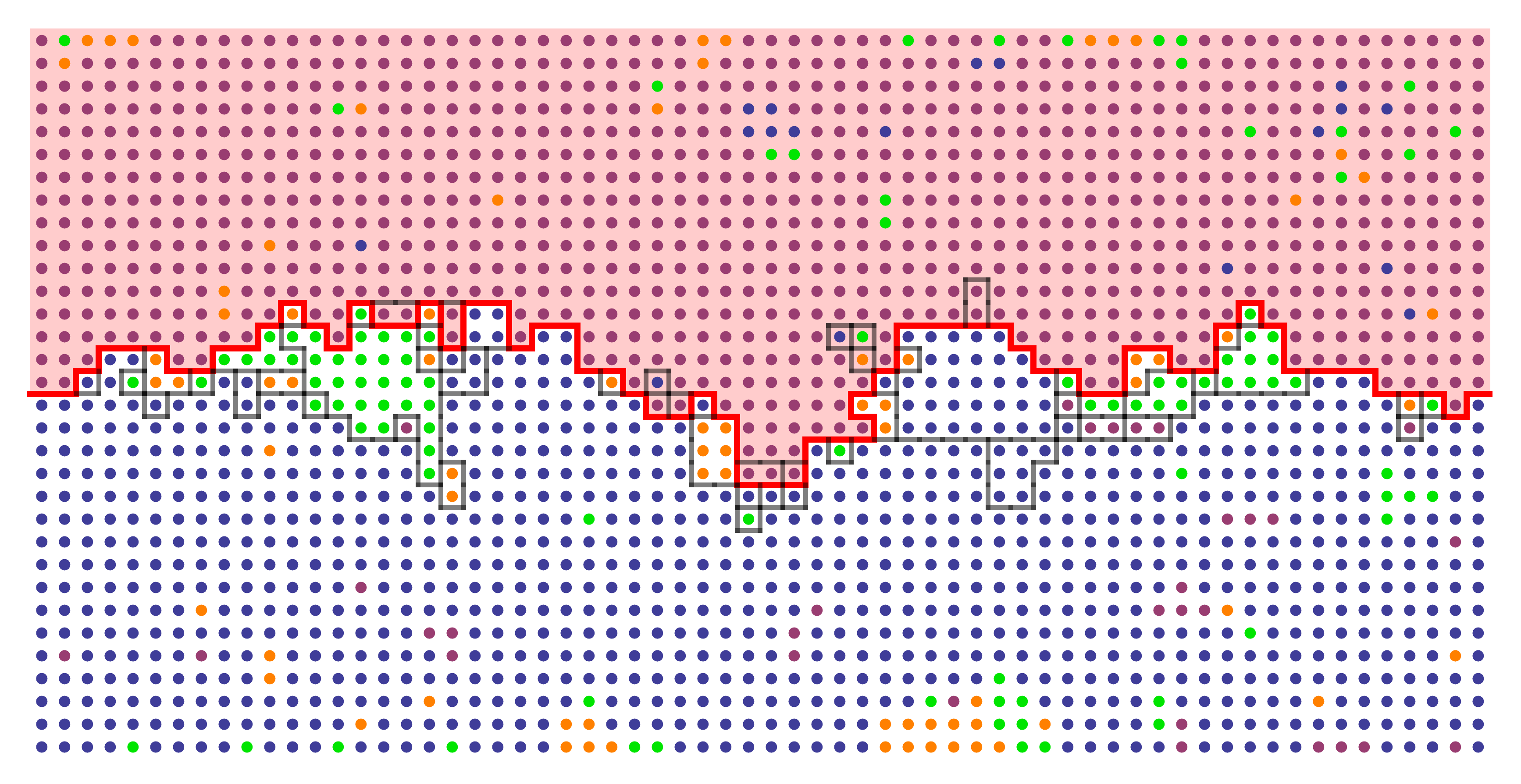}};
    \end{tikzpicture}
    \vspace{-18pt}
    \caption{The $\Red$ and $\Blue$ interfaces of a $4$-color 2D Potts model. \emph{Right bottom}: the interface~$\cI_\Blue$ and augmented $\Blue$ component $\Ablue$.  \emph{Right top}:  $\cI_\Red$ and the augmented $\Red$ component $\Ared$.}
    \label{fig:potts_interfaces}
\end{figure}

\begin{figure}
\vspace{-0.1in}
    \begin{tikzpicture}
   \node (fig1) at (0.5,0) {
    	\includegraphics[width=0.65\textwidth]{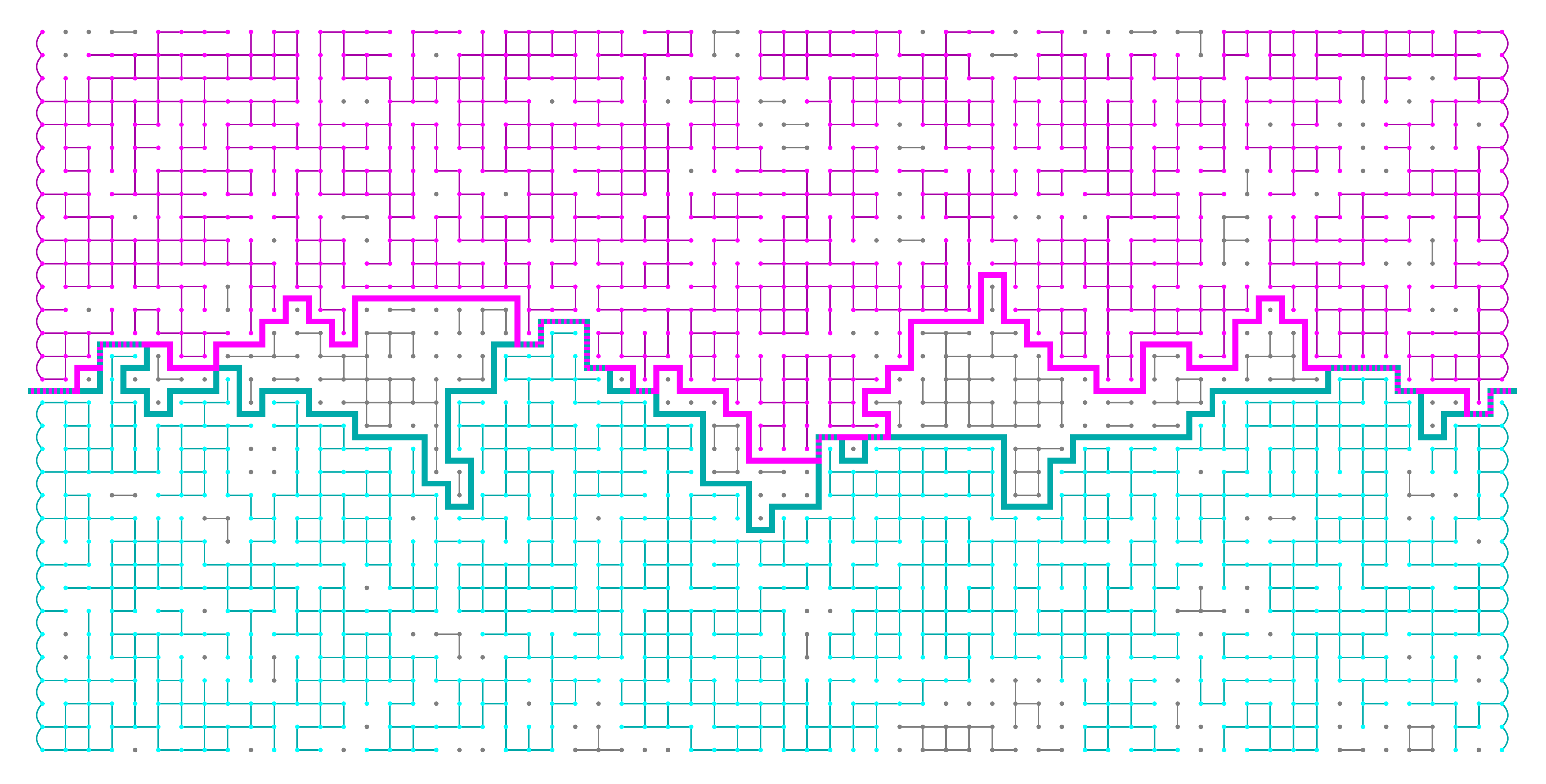}};
   \node (fig2) at (8.6,-1.35) {
   	\includegraphics[width=0.3\textwidth]{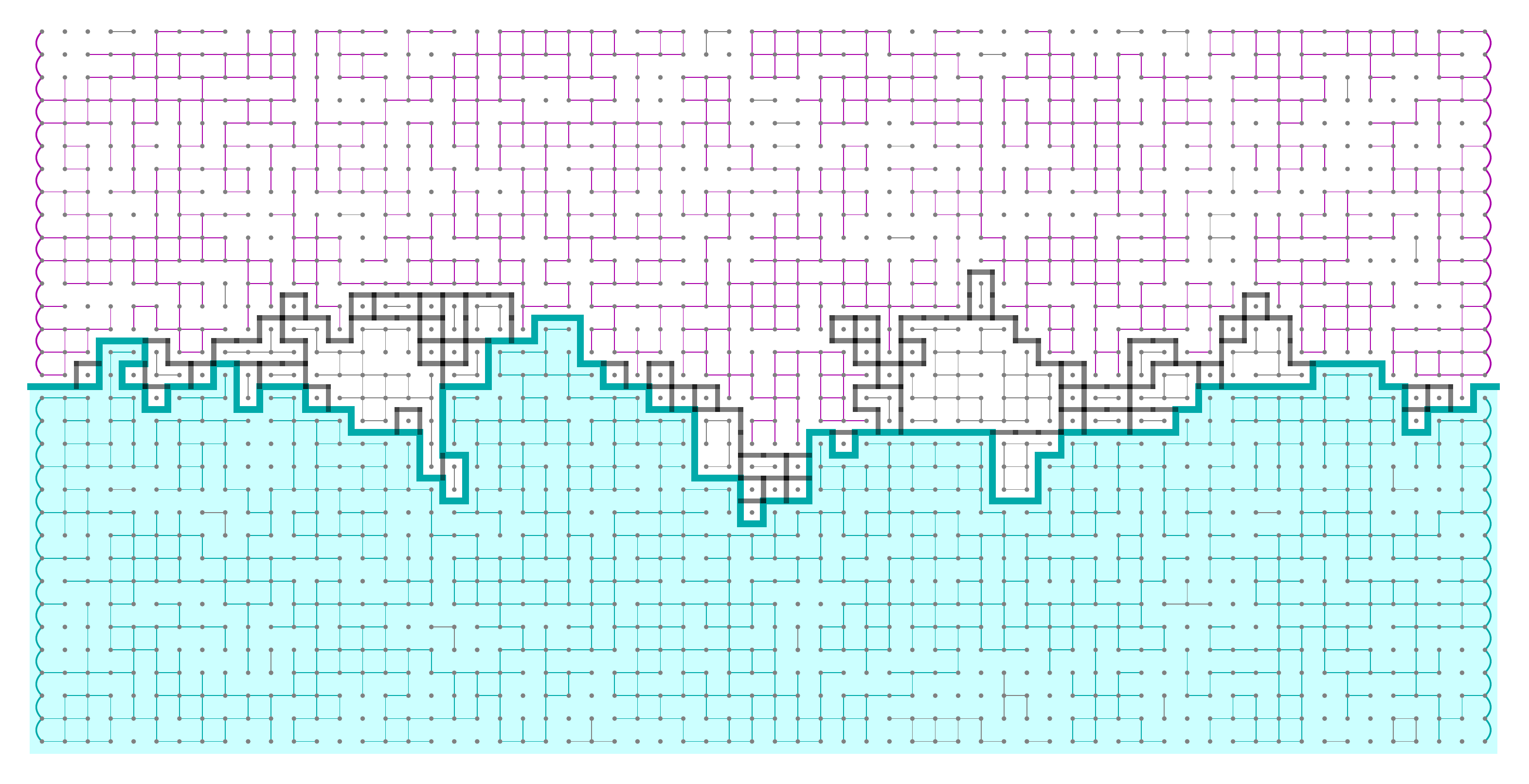}};
    \node (fig3) at (8.6,1.35) {
   	\includegraphics[width=0.3\textwidth]{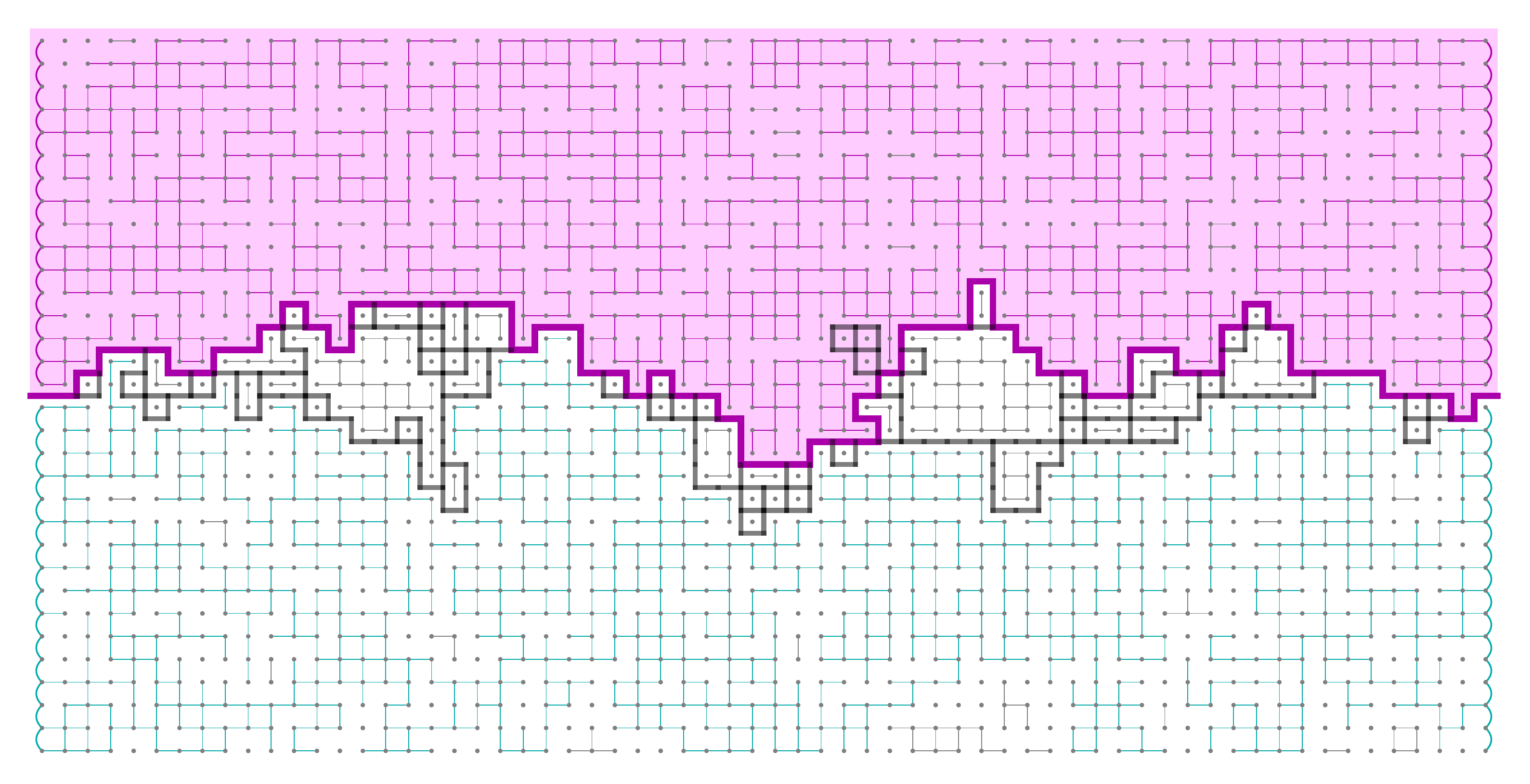}};
    \end{tikzpicture}
    \vspace{-18pt}
    \caption{The $\Top$ interface $\cI_\Top$ and $\mathsf{bottom}$ interface $\cI_\Bot$ of the random-cluster model coupled via the Edwards--Sokal coupling to the Potts model from \cref{fig:potts_interfaces}.
    \emph{Right bottom}: The interface $\cI_\Bot$ and augmented bottom component $\Abot$. \emph{Right top}: $\cI_\Top$ and the augmented top component $\Atop$.}
    \label{fig:rc_interfaces}
\end{figure}

\begin{theorem}[Potts]\label{thm:potts} Fix an integer $q\geq 2$. For $\beta$ large enough, the maximum height $M_n$ and absolute value of the minimum height $M'_n$ of the $\Red$ interface $\cI_\Red$ are tight once centered around their means, i.e.,
\[ M_n-\E[M_n] = O_{\textsc{p}}(1) \qquad\mbox{and}\qquad M'_n-\E[M'_n] = O_{\textsc{p}}(1) \,.\]
Furthermore, there exist $\gamma,\gamma'>0$ such that $\E[M_n]\sim (2/\gamma) \log n$, $\E[M'_n] \sim (2/\gamma') \log n$,  and $\gamma' > \gamma$ for  $q\neq 2$. The same holds for the $\Blue$ interface $\cI_{\Blue}$ when swapping the roles of $M_n$ and $M'_n$.
\end{theorem}

\begin{theorem}[Random cluster]\label{thm:rc} Fix $q\geq 1$. For $\beta$ large enough, the maximum height $M_n$ and the absolute value of the minimum height $M'_n$ of the $\Top$ interface $\cI_\Top$ are tight once centered around their means, i.e.,
\[ M_n-\E[M_n] = O_{\textsc{p}}(1) \qquad\mbox{and}\qquad M'_n-\E[M'_n] = O_{\textsc{p}}(1) \,.\]
Furthermore, there exist $\alpha,\alpha'>0$ such that $\E[M_n] \sim (2/\alpha) \log n$, $\E[M'_n]\sim (2/\alpha') \log n$ and $\alpha'> \alpha$. The same holds for the $\mathsf{bottom}$ interface $\cI_{\Bot}$ when swapping the roles of $M_n$ and $M'_n$.
\end{theorem}

\begin{figure}
    \begin{tikzpicture}
   \node (fig1) at (8.1,0) {
    	\includegraphics[width=0.33\textwidth]{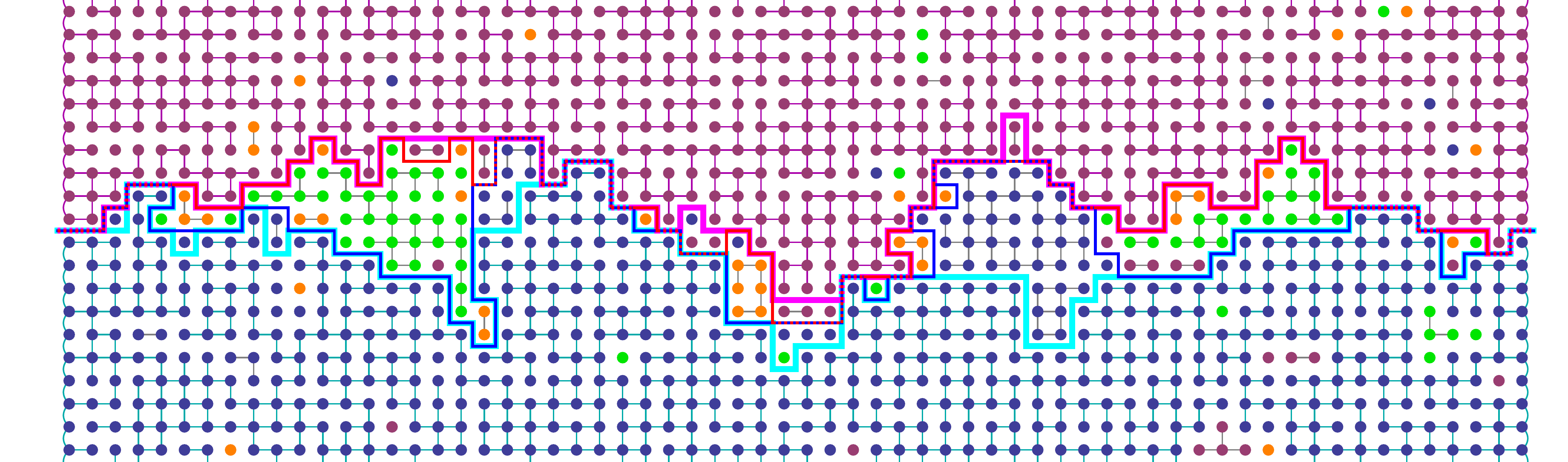}};

     \begin{scope}[shift={(0,0)},font=\small]
   \node (fig2) at (0,0) {
   	\includegraphics[width=0.66\textwidth]{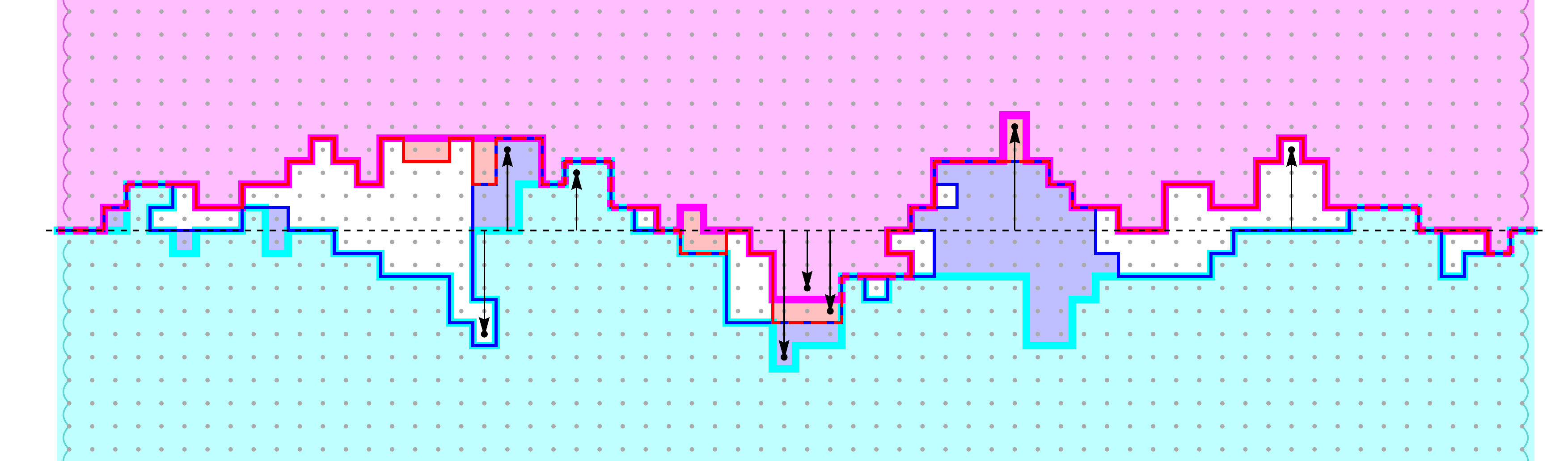}};
    \end{scope}
    \end{tikzpicture}
    \vspace{-0.75cm}
    \caption{\emph{Left}: The four interfaces from \cref{fig:potts_interfaces,fig:rc_interfaces}, pinpointing the minima and maxima of~each. As a result of the Edwards--Sokal coupling, the interfaces are layered in the following order: $\Top$, $\Red$, $\Blue$, $\Bot$. \emph{Right}: The same picture with all the colors and edges of the joint configuration.}
    \label{fig:4-rates}
\end{figure}

\begin{remark*}
The 3D Potts model has up-down asymmetry at a macroscopic level (even though at a microscopic level, such colors only appear in clusters with exponential tails on their~size). In particular, it is easier for the $\Red$ component to ``recede'' via upward deviations (where the global extremum has a prefactor of $1/\gamma$)
than it is to ``advance'' via downward deviations (the global extremum has a prefactor of $1/\gamma'$), as the finite clusters with colors other than $\Blue$ and $\Red$ also invade its territory, resulting in the strict inequality $\gamma < \gamma'$. \end{remark*}
The constants $\alpha,\alpha',\gamma,\gamma'$ in the above theorems are given explicitly in terms of large deviation events of the different interfaces (see \cref{prop:RC-rate,prop:RC-Potts-diff-rate} and \cref{eq:non-red-rate,eq:blue-rate,eq:bot-rate}). The following proposition shows that all 4 rates are distinct, and provides estimates for their differences, sharp up to a factor of $1\pm\epsilon_\beta$.
\begin{proposition}\label{prop:comparison-of-rates}[Comparison of means]
The constants $\alpha,\alpha',\gamma,\gamma'$ governing the asymptotic means of the maxima and minima of 3D Potts and 3D FK interfaces, as per \cref{thm:potts,thm:rc}, satisfy
\begin{align*}
4\beta-C &\leq \alpha \leq 4\beta \,,\\
\gamma - \alpha &=(1\pm\epsilon_\beta) e^{-\beta}\,,\\
\gamma' - \alpha &=(1\pm\epsilon_\beta) (q-1) e^{-\beta}\,,\\ \alpha' - \alpha &=(1\pm\epsilon_\beta) q e^{-\beta}\,,
\end{align*}
where $C$ depends only on $q$, the notation $a=(1\pm\epsilon)b$ denotes $a\in[(1-\epsilon)b,(1+\epsilon)b]$, and $\epsilon_\beta\to0$ as $\beta\to\infty$.
\end{proposition}

\subsection{Related works}\label{sec:related-work}
In what follows, and due to the extensive list of related literature, we will provide only a brief and non-exhaustive overview of these studies, focusing on those that were instrumental to the proofs. (The reader is referred to referred to \cite{GL_max,GL_tightness} for a more comprehensive account of the related work.)
An important milestone in the study of low temperature 3D Ising interfaces was the breakthrough works of Dobrushin \cite{Dobrushin72,Dobrushin73}. There, the rigidity of the interface was proven (valid also in higher dimensions), leading to the existence of non-translation-invariant infinite volume Gibbs measures in $\Z^3$. These results were extended to a variety of other models (e.g., \cite{BLOP79a,DMN00,GielisGrimmett02,CeKo03,HKZ88,MMRS91}, to name a few). In our context, it is particularly important to highlight the following works. First, the work of Gielis and Grimmett \cite{GielisGrimmett02}, establishing the rigidity of the 3D FK interface under $\bar\mu_n$ for $p$ sufficiently close to $1$ (related results for the FK interface at $p=p_c$ and large  $q$ were obtained in~\cite{CeKo03}). The machinery built in ~\cite{GielisGrimmett02} and \cite[\S7]{Grimmett_RC} is a prerequisite for our analysis.
Second, decorrelation estimates for 3D Ising interfaces have been extended to a more general setting by Bricmont, Lebowitz, and Pfister \cite{BLP79b}, which  will allow us to control global (in terms of local) extrema.
Third, and most relevant, a series of recent papers by Gheissari and the second author   \cite{GL_max,GL_tightness,GL_DMP} established detailed results on local and global maxima of the 3D Ising interface. While it readily follows from Dobrushin's  work that the maximum of the Ising interface in a cylinder of side length $n$ should be of order $\log n$, the authors in the above papers prove that the maximum is in fact tight around its mean which is $(c+o(1))\log n$ for an explicit $c=c(\beta)$ (governed by the large deviation rate of the interface height above the origin in infinite volume). Furthermore, those works provide a description of the shape of the Ising interface around a location at which a tall peak is reached, using Dobrushin's argument as a starting point for an analysis of operations on 3D ``pillars''  (as the 2D analysis within Dobrushin's rigidity argument is too crude to recover the correct~$c(\beta)$). 
The ideas in \cite{GL_max,GL_tightness,GL_DMP}, along with the work of \cite{GielisGrimmett02} extending Dobrushin's work to FK interfaces, form the foundation of our analysis of Potts and FK models. We next describe some of the key issues arising there.

\subsection{Proof ideas}\label{sec:subsec-proof-ideas}
Here, we discuss the proof ideas in the context of the main obstacles we encountered. Before detailing the additional challenges that the Potts and FK model present us with, let us recap the approach used in \cite{GL_max} to analyze 3D Ising interface (the case $q=2$). The proof in that case can be roughly summarized in three steps. 
 \begin{enumerate}[label=(\roman*),wide=1pt,topsep=1pt]
    \item \label[step]{it:ising-pf-pillar} 
    \emph{Pillar shape:} Cluster expansion is used to show that if the interface reaches a large height $h$ above a given location~$x$, then with probability $1 - \epsilon_\beta$, it does so in a very controlled manner: define the \emph{pillar} $\cP_x$ to be the local portion of the interface above $x$ (see \cite[Def.~2.16]{GL_max}, or \cref{def:pillar} in our random-cluster setting)---roughly put, this is the cluster of plus spins containing $x$ in the positive half-space; the bulk of the proof in \cite{GL_max,GL_tightness} aims to show that this cluster, conditional on reaching height $h$, behaves as a directed random walk (RW), visiting $1-\epsilon_\beta$ of the slabs at exactly one location as it climbs to height~$h$.
\item \label[step]{it:ising-pf-sub} 
\emph{Large deviation rate:} 
Submultiplicativity of the probability that the pillar $\cP_x$ reaches height $h$ is then argued by comparing the conditional probability of reaching height $h_1+h_2$ given that the pillar already reached height $h_1$, to the unconditional probability it reaches height $h_2$ above $x$. This submultiplicativity implies the existence of the sought large deviation rate, which can also be phrased in terms of a certain spin-connectivity event (some care is required as $|\Lambda_n|$ needs to grow with $h$; see \cref{prop:RC-rate}, for instance). 
\item \label[step]{it:ising-pf-conc} \emph{Mean and tightness for the maximum:} Combining the large deviation rate with decorrelation estimates and a second moment argument gives the desired results concerning the maximum of the interface.
\end{enumerate}
\Cref{it:ising-pf-conc} in this program can be readily adapted to the random-cluster setting via the mentioned decorrelation estimates of~\cite{BLP79b}. 
To carry out \cref{it:ising-pf-pillar} in the FK model, we employ the cluster expansion machinery of~\cite{GielisGrimmett02}, which adds technical difficulties to what had been a fairly delicate argument already for Ising---for instance, the random-cluster pillars must now be decorated by ``hairs''---certain $1$-connected sets of extra faces---that can penetrate their interior and connect them to one another (see \cref{sec:subsec-CE-difficulties} for more on this).  Finally, as we next elaborate, the Ising argument for the critical \cref{it:ising-pf-sub} collapses in the absence of monotonicity, and we resort to establishing the large deviation rates in two stages: first, we obtain the rate for upward deviations of 
the $\Top$ interface $\cI_\Top$ in $\bar\mu_n$ (see
 \cref{sec:subsec-monotonicity}), which is the ``highest'' among the four coupled interfaces; then---building on that result---we derive the rate for upward deviations of the other three interfaces $\cI_\Blue,\cI_\Red$ and $\cI_\Bot$ (see \cref{sec:subsec-LD-rate-Potts}).

\subsubsection{Large deviation rate for the FK top interface}\label{sec:subsec-monotonicity}
The submultiplicativity argument in the Ising proof (\cref{it:ising-pf-sub} above) crucially relied on FKG---a property missing from the Potts model. Without this argument, while one could still establish that the pillars in the Potts interface resemble directed RWs (via \cref{it:ising-pf-pillar}), one would not be able to derive the large deviation rate of them reaching height $h$. A well-known remedy to the lack of monotonicity in the Potts model is to turn to the random-cluster model---which does enjoy FKG---via the Edwards--Sokal coupling (and then attempt to go back to Potts to recover the counterpart~behavior). However,  the Dobrushin boundary conditions for our Potts model correspond (via this coupling) to the conditional FK model $\bar{\mu}_n = \mu_n(\cdot\mid\sep_n)$ (where we aim to analyze the interface and prove submultiplicativity) rather than $\mu_n$, and unfortunately $\bar{\mu}_n$ does not have FKG either. Our workaround leverages the fact that the separation event $\sep_n$ is decreasing. We will define an event $A_h$ comparable to the event that the (suitably defined) \emph{pillar} $\cP_x$ of the $\Top$ interface at a given point $x$ reaches height $h$ (see \cref{def:A_h^x}). Instead of proving a bound of the form $\bar\mu_n(A_{h_1+h_2}) \leq \bar\mu_n(A_{h_1})\bar\mu_n(A_{h_2})$, we resort to proving (after additional technical modifications, as we briefly describe below)
a bound of the form $\mu_n(A_{h_1+h_2}\mid\sep_n) \leq \mu_n(A_{h_1}\mid\sep_n)\mu_n(A_{h_2})$, towards which monotonicity is still available, and then use the fact that $\mu_n(A_{h_2})\leq \mu_n(A_{h_2}\mid\sep_n)$ as long as the event $A_{h_2}$ is decreasing (by FKG in $\mu_n$).  Consequently, this approach, while valid for the upward deviations of $\cI_\Top$, fails for its downward deviations (equivalently, upward deviations of $\cI_\Bot$ ---
addressing the increasing event that there is an open path connecting $\partial \Lambda_n^-$ to height $h$), let alone for understanding the two Potts interfaces. Understanding the maximum of $\cI_\Bot$ requires additional ingredients, and is handled together with the analysis of the Potts interfaces $\cI_\Blue$ and $\cI_\Red$.

An extra complication that is associated with the move to the random-cluster model is that, when studying its interfaces, one  needs to be far more careful when applying a Domain Markov argument, which was also a crucial part of the submultiplicativity argument. More precisely, in the Ising case, revealing the interface up to height $h_1$ exposes a boundary of minus spins, upon which one can apply the Domain Markov property to ignore all of the information ``below'' these minus spins when bounding the probability that the interface further climbs from height $h_1$ to $h_1 + h_2$. In the random-cluster case however, revealing the interface exposes a boundary of open edges, rather than vertices. Making sure that the revealed set forms a boundary condition in the FK model (disconnecting it from the edges that lie ``below''), while the event of climbing to height $h_2$ in the yet-unrevealed subgraph can still be related to the unconditional probability of climbing to height $h_2$ (see also \cref{sec:subsec-CE-difficulties} accounting for some of these difficulties) becomes a delicate part of the analysis.

\subsubsection{Large deviation rate for the Potts interfaces and FK bottom interface}\label{sec:subsec-LD-rate-Potts}
Our approach to establishing the rate of upwards deviations in the remaining three interfaces (Potts $\Blue$ and $\Red$ and FK $\mathsf{bottom}$) modulo the analysis of the $\Top$ interface, is as follows. Consider $\cI_\Blue$  (the other two interfaces are handled similarly). As noted above, in the coupled FK--Potts model $\phi$, the $\Top$ interface always lies above the $\Blue$ interface. Thus, to estimate the probability that the $\Blue$ interface reaches height $h$ above a given point $x$, we may instead look at the conditional probability that it does so given the $\Top$ interface reaches height $h$ above~$x$ (see, e.g., \cref{prop:RC-Potts-diff-rate}), which we had already studied.
Heuristically, this can be thought of as computing the probability that underneath the $\Top$ interface there is a path of $\Blue$ vertices connecting $x$ to height $h$. 

The following heuristic, albeit flawed, gives insight into this problem.  As mentioned above when discussing the shape of the pillar $\cP_x$, one could show that conditional on the top interface reaching a large height $h$ above $x$, the pillar resembles a stack of i.i.d.\ increments---more precisely, its increments are asymptotically stationary and $\alpha$-mixing (for Ising this was shown in \cite[Props.~7.1 and~7.2]{GL_max}). One could then expect that the probability of having a path of $\Blue$ vertices passing through all of these increments would be comparable to the conditional probability of having a path of $\Blue$ vertices passing through a single increment, raised to the power of the number of increments (via the LLN for the i.i.d.\ increments). As the number of increments is comparable to $h$, this would then give the desired rate explicitly in terms of this conditional probability. Unfortunately, this approach fails since we are trying to estimate probabilities on the order of $e^{-ch}$, and the interface may likely achieve a large upward $\Blue$ deviation via an \emph{atypical} $\Top$ pillar occurring with such a probability---whereas the asymptotic mixing and stationarity only apply to a \emph{typical} $\cP_x$ achieving height $h$...

Instead, we employ another submultiplicativity argument to show the \emph{existence} of a $\Blue$ upward deviation rate (similarly for the other interfaces, postponing the problem of comparing these rates per \cref{prop:comparison-of-rates}). The basic idea is to show that (a) sampling a ``nice'' $\Top$ pillar with height $h_1 + h_2$ is comparable to sampling a $\Top$ pillar with height $h_1$ and another ``nice'' $\Top$ pillar with height $h_2$ independently, then stacking them on top of each other; and (b) this comparison further extends when considering the Potts coloring of the interior vertices (which is nontrivial since, e.g., information does leak through our interface via hairs).
In \cite[Section 7]{GL_max}, the key to showing $\alpha$-mixing and asymptotic stationarity of a (typical) pillar $\cP_x$ was elevating the standard map modifying a single interface into a ``2-to-2'' map, acting on a pair of interfaces: to evaluate the effect of having two different extensions of a bottom part of a pillar, one compares the effect of swapping the two possible top pillar parts through the cluster expansion.
Here, we further elevate it to a ``3-to-3'' map, acting on a triple of interfaces as follows.
Suppose that $P_B, Q_B$ are two pillars with height $h_1$, and that $P^T, Q^T$ are two pillars with height $h_2$. Let $P_B \times P^T$ be the result of stacking $P^T$ on top of $P_B$, and similarly for $Q_B \times Q^T$. Our 3-to-3 map sends $(P_B \times P^T, Q_B, Q^T)\mapsto (Q_B \times Q^T, P_B, P^T)$, and its analysis via the cluster-expansion allows us to show that
\[\bar\mu_n(P_B \times P^T)\bar\mu_n(Q_B)\bar\mu_n(Q^T) \approx \bar\mu_n(Q_B\times Q^T)\bar\mu_n(P_B)\bar\mu_n(P^T),\,\]
where the error is multiplicative and not additive. (Recall that all errors must be multiplicative for this approach to stand a chance, as we are estimating events that are exponentially unlikely in the height $h$.) See \cref{lem:3-to-3-map} for a precise statement of this result, and \cref{fig:3-to-3} for an illustration. 

With this estimate in hand, we can sum over all possible $Q_B$ and $Q^T$ to prove the desired claim on the law of pillars. To conclude the submultiplicativity with respect to the probability of having a $\Blue$ path within the pillar, we prove and employ an appropriate Domain Markov property in the coupled FK--Potts model, saying that if we fix an increment, then regardless of the environment outside of the increment, the joint configuration inside has the law of another coupled FK--Potts model with appropriate boundary conditions. This strategy allows us to establish the sought limiting rates, yet without any comparison between them (e.g., they could potentially all coincide with the rate of the $\Top$ interface). 
To estimate the rates of $\Blue$, $\Red$ and $\textsf{bottom}$, we need to bound from below and above the probability of coloring the interiors of the pillars---which are comprised mostly of trivial increments (cubes stacked one on top of the other). To leverage this structure, we must fend off the effect of the environment, since revealing the pillar in the FK model will include interior information. To this end, we introduce a notion of a \emph{pillar shell}, which excludes the latter faces, thus its analysis supports the comparison of the rates.

\subsubsection{Difficulties arising from cluster expansion}\label{sec:subsec-CE-difficulties}
We conclude this section with a discussion of some of the difficulties surrounding cluster expansion for the random-cluster model. In \cite[Lem.~9]{GielisGrimmett02}, Grimmett and Gielis proved the following for the law of the random-cluster interface $\cI$:
\[\bar{\mu}_n(\cI) \propto (1-e^{-\beta})^{|\partial \cI|}q^{\kappa_\cI}\exp\Big[-\beta|\cI|+\sum_{f\in \cI} \g(f, \cI)\Big]\]
for a suitably ``nice'' function $g$ (see \cref{prop:grimmett-cluster-exp} for details). Compared to the Ising cluster expansion, which only contained the last exponent $\exp[-\beta|\cI|+\sum_{f\in\cI} \g(f,\cI)]$, we see here that the number of components $\kappa_\cI$ and the size of the boundary of $\cI$ plays a role; moreover,  the interface $\cI$ appearing in that work was what we refer to as the \emph{full interface}: the $1$-connected component of faces that are dual to closed edges in $\omega$ and are incident to a boundary face at height $0$ (see \cref{def:full-interface}). This larger collection of faces contains all of our 4 interfaces $\cI_\Top$, $\cI_\Red$, $\cI_\Blue$ and $\cI_\Bot$, as well as additional connected components of faces ``protruding'' from them, which we will call hairs. In the absence of cluster expansion for our $\Top$ interface, for instance, we have to apply the cluster expansion arguments on objects in the full interface instead. Namely, the pillar now must include these additional hairs in the full interface, even though our focus is on pillars in $\cI_\Top$ (it is much easier for the full interface to exhibit upward deviations via said hairs, but those will not represent a boundary between connected components in the FK nor Potts model and hence are irrelevant for us). 
What further complicates matters is that these hairs can potentially reattach the pillar to other parts of the interface, leading to unwanted correlations. 
The Ising results in \cite{GL_max,GL_tightness} did not need to face such issues, however the follow-up work \cite{GL_DMP} did treat a situation where, conditional on the existence of level-lines, one would like to establish that the local law of the pillar can be coupled to the standard one in infinite-volume. That was achieved in that work via restricting the analysis to pillars that are confined to appropriate \emph{cones}. Adapting this concept to the FK model allows us to separate the pillars from affecting each other via the long range interactions of the FK model (see \cref{thm:iso-map}).
Then, when establishing the rate of upward deviations of the $\Top$ interface, extra care must be taken to ensure that despite including the extra hairs, no information leaks ``from below" when we reveal the interface up to height $h_1$ (otherwise the Domain Markov argument mentioned in \cref{sec:subsec-monotonicity} would fail). And finally, when studying the rates of the $\Blue,\Red$ and $\textsf{bottom}$ interfaces, we must ensure that no information leaks ``inside the pillar" when conditioning on the $\Top$ interface (otherwise, e.g., we would not be able to address the Potts rates using the Edwards--Sokal coupling as described in \cref{sec:subsec-LD-rate-Potts}). 

\subsection{Organization}
This paper is organized as follows. \Cref{sec:prelim} summarizes the preliminary results we will need on the low temperature FK model, and sets up the notion of pillars. \Cref{sec:pillar-maps} establishes the basic results needed on typical pillars---notably, being confined to appropriate cones and consisting of mostly trivial increments. \Cref{sec:ld-RC} derives the FK model large deviation rate for upward deviations of the $\Top$ interface. \Cref{sec:ld-potts-bot} establishes the corresponding rates for the remaining 3 interfaces ($\cI_\Red$, $\cI_\Blue$, $\cI_\Bot$) modulo the behavior of $\cI_\Top$, and further estimates these rates. \Cref{sec:max} derives the tightness of the minima and maxima of the different interfaces from the above results and certain decorrelation estimates, whose proof is relegated to \cref{sec:decorrelation}.

\section{Preliminaries}\label{sec:prelim}
We begin by introducing various notation that will be used throughout the paper, and recalling the setup work done in \cite{GielisGrimmett02} for the random-cluster model. Then we will define and prove basic properties about pillars, the geometrical objects used to study the upward deviations of the $\Top$ interface.

Let $\ex, \ey, \ez$ denote $(1, 0, 0), (0, 1, 0), (0, 0, 1)$ respectively. For every configuration $\omega$, let $\opfaces$ (resp., $\clfaces$) denote the set of faces dual to open (resp., closed) edges:
\[\opfaces = \left\{ f_e\,:  \omega_e = 1\right\}\quad,\quad\clfaces =  \left\{ f_e\,:  \omega_e = 0\right\}\,.\]
We will identify edges and faces by their midpoints when referring to their location and height, so that horizontal faces have integer heights and vertical faces have half-integer heights. We denote by $\cL_h$ the set of vertices, faces, and edges with height equal to $h$. 

\begin{definition}[Connectivity and Boundaries]
We define two faces to be 0-connected if their intersection contains at least one point, and 1-connected if their intersection contains an edge. For any set of faces $H$, we define $\overline{H}$ to be the union of $H$ with the set of faces that are 1-connected to $H$. We define $\partial H := \overline{H} \setminus H$. Note that this usage of $\partial$ is different from when we write $\partial \Lambda_n$ in the sense that $\partial \Lambda_n \subseteq \Lambda_n$ while $\partial H \cap H = \emptyset$. That is, $\partial \Lambda_n$ refers to an interior boundary of vertices, while $\partial H$ refers to an exterior boundary of faces. Despite this overload in notation, we will keep this convention for the sake of clarity of certain proofs, and this distinction should be noted whenever $\partial$ is used in front of a set of faces.
\end{definition}

\subsection{Cluster Expansion and random-cluster rigidity}\label{subsec:rc-interfaces}
To prove finer details about the random-cluster interfaces, we recall the setup used in \cite{GielisGrimmett02,Grimmett_RC}. We begin with the classical definition due to Gielis and Grimmett, referred in what follows as the \emph{full interface}. (Note that this definition uses 1-connectivity for faces, hence our choice of graph adjacency for $\Vred$ in \cref{def:potts-interfaces}, as discussed below that definition.)

\begin{definition}[Full interface]  The full interface $\cI$ is the 1-connected component of faces in $\clfaces$ containing the boundary faces at height $0$ which separate $\partial \Lambda_n^+$ and $\partial \Lambda_n^-$. See \cref{fig:gr-interface} for a visualization. Note that as a set of faces, this interface includes the previous four interfaces. Denote by $\kappa_\cI$ the number of open clusters in a configuration where the only closed edges are $e$ such that $f_e \in \cI$.
\label{def:full-interface}
\end{definition} 

\begin{figure}
    \begin{tikzpicture}
   \node (fig1) at (8.05,0) {
    	\includegraphics[width=0.55\textwidth]{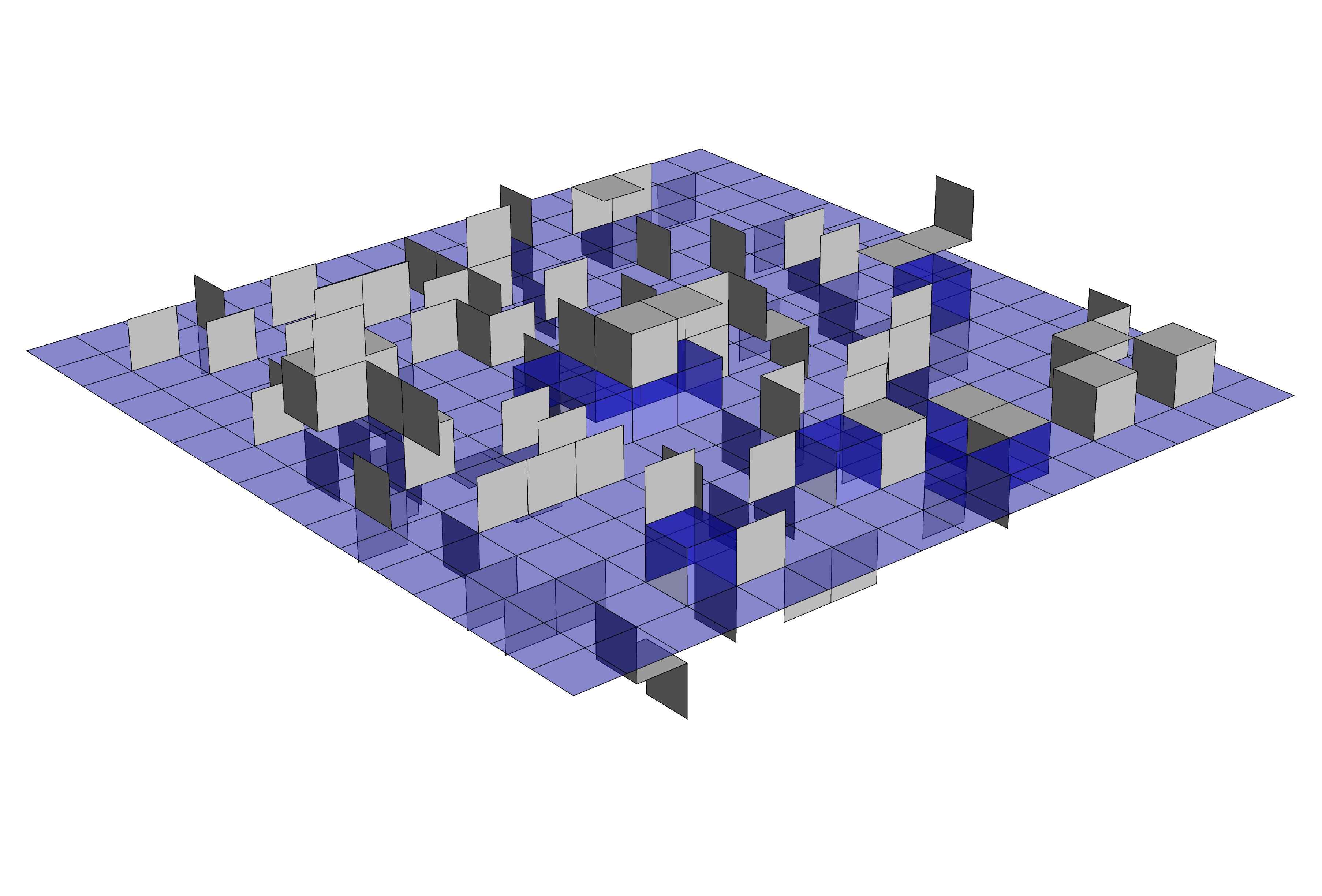}};

     \begin{scope}[shift={(0,0)},font=\small]
   \node (fig2) at (0,0) {
   	\includegraphics[width=0.45\textwidth]{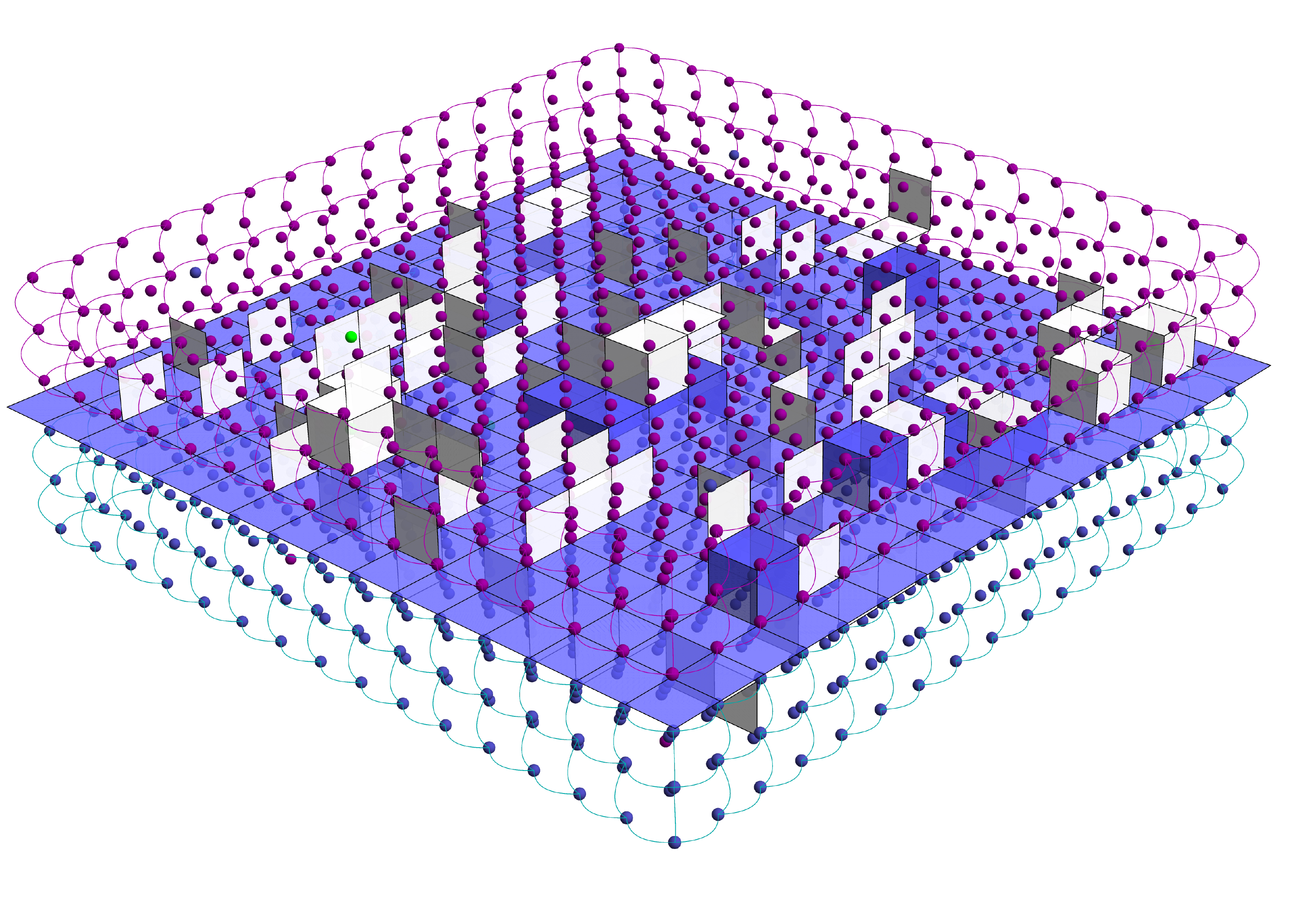}};
    \end{scope}
    \end{tikzpicture}
    \vspace{-1.5cm}
    \caption{\emph{Left}: A 3D joint configuration of edges and vertex colors under the Edwards--Sokal coupling, including the full interface, $\cI$. \emph{Right}: The same model, with just the full interface displayed. Note that the full interface should not be thought of as a surface --- there are many sheets of faces sticking out and creating overhangs.}
    \label{fig:gr-interface}
\end{figure}

\begin{definition}[Semi-extended interface] Let $\cI^\star$  be the union of $\cI$  with all horizontal faces that are 1-connected to $\cI$. 
\end{definition}

\begin{definition}[Ceilings/Walls, Indexing, Nesting] For a face $f$ or vertex $v$, let $\rho(f)$ and $\rho(v)$ denote the face (if $f$ is horizontal) or edge (if $f$ is vertical) that $f$ projects onto at height 0, or the point that $v$ projects to. For a face $f \in \cI^\star$, we call $f$ a ceiling face if it is horizontal and there are no other faces of $\cI^\star$ with projection equal to $\rho(f)$. We call all other faces of $\cI^\star$ wall faces. Ceilings and walls are 0-connected components of ceiling and wall faces respectively. 

For a wall $W$, we can decompose it as $W = (A, B)$ where $A = W \cap \cI$ and $B = W \cap (\cI^\star \setminus \cI)$.
We can index walls by assigning $x$ a wall $W$ if $x$ is in $\rho(W)$. By \cref{lem: wall-ceiling-geometry} below, each vertex is only assigned to one wall, so the notation $W_x$ is well defined (though each wall can be assigned to multiple vertices). Let the empty set of walls be denoted $\sE$, so if there is no wall at $x$, we assign it $\sE_x$.
For a wall $W$, we can consider the complement of its projection $\rho(W)^c$ to be the collection of faces and edges at height 0 that are not in $\rho(W)$. There is an infinite component of $\rho(W)^c$, and possibly some finite ones. We say that a vertex, edge or face is interior to, or nested in a wall $W$ if its projection is not in the infinite component of $\rho(W)^c$. A wall $W'$ is interior to, or nested in a wall $W$ if $\rho(W')$ is disjoint from the infinite component of $\rho(W)^c$, and similarly for ceilings interior to $W$.
For a vertex $x$, we can consider the set of all walls $W$ that nest $x$. The collection of all such walls is denoted $\fW_x = (W_1, \ldots, W_s)$.
\end{definition}

\begin{lemma}[{\cite[Lem.~10]{GielisGrimmett02},\cite[Lem.~7.125]{Grimmett_RC}}]
\label{lem: wall-ceiling-geometry}
The following geometric properties of walls and ceilings hold:
\begin{enumerate}[(i)]

\item The projections $\rho(W_1), \rho(W_2)$ of two different walls $W_1$ and $W_2$ are not 0-connected.

\item All faces of the semi-extended $\cI^\star$ which are 1-connected to a ceiling face are horizontal faces in~$\cI$.
\end{enumerate}
\end{lemma}

\begin{definition}[Standard wall]
For sets of faces $A, B$, we call $S = (A, B)$ a standard wall if there exists an interface $\cI$ such that $A \subseteq \cI$ and $B \subseteq \cI^\star \setminus \cI$ and $A \cup B$ is the unique wall of $\cI$. The interface $\cI$ in the above definition is unique (see \cite[Lem.~11]{GielisGrimmett02},\cite[Lem.~7.126]{Grimmett_RC}). A collection of standard walls is admissible if no two walls have 0-connected projections.
\end{definition}  
\begin{lemma}[{\cite[Lem.~12]{GielisGrimmett02},\cite[Lem.~7.127]{Grimmett_RC}}] There is a 1-1 correspondence between interfaces and admissible families of standard walls.
\end{lemma} 

As a result of the above lemma, we can view interfaces as (admissible) collections of standard walls, and we use this to define groups of walls and the excess area of walls. 

\begin{definition}[Groups of walls] Two standard walls $W_1, W_2$ are close if there exist faces $f_1 \in \rho(W_1)$, and $f_2 \in \rho(W_2)$ such that $d(f_1, f_2) < \sqrt{N(f_1, W_1)} + \sqrt{N(f_2, W_2)}$, where $N(f, W)$ is the number of faces of $W$ whose projection is a subset of $f$. (Recall that $f$ is a closed unit square, and so this definition also counts vertical faces whose projection is a single bounding edge of $f$.) A group of standard walls $F$ is a maximal connected component of standard walls via the closeness adjacency relation. That is, if $W_1, W_2 \in F$, then there exists a sequence walls $W_1 = S_1, \ldots, S_k = W_2 \in F$ such that $S_i$ and $S_{i+1}$ are close, and any wall not in $F$ is not close to any wall in $F$.
\end{definition}

\begin{definition}[Excess area of interfaces and walls] 
For two interfaces $\cI$ and $\cJ$, we will define the excess area of $\cI$ with respect to $\cJ$ to be 
\begin{equation*}
\fm(\cI;\cJ) = |\cI| - |\cJ|\,,
\end{equation*}
where $|\cI|$ denotes the number of faces in the interface $\cI$. 
For a standard wall $W = (A, B)$, let $N(W) = |A|$, and $|W| = |A \cup B|$. Then, we define its excess area to be
\begin{equation*}
\fm(W) = N(W) - |\rho(W)|\,.
\end{equation*}
\end{definition}

\begin{lemma}[{\cite[Lem.~13]{GielisGrimmett02},\cite[Lem.~7.128]{Grimmett_RC}}]
 We have the following inequalities:
 \begin{enumerate}[(i)]
\item $N(W) \geq \frac{14}{13}|\rho(W)|$, which implies $\fm(W) \geq \frac{1}{13}|\rho(W)|$ and $\fm(W) \geq \frac{1}{14} N(W)$;
\item $N(W) \geq \frac{1}{5}|W|$;
\end{enumerate}
\end{lemma} 

In order to prove that a typical interface has certain ``nice" geometrical properties, our proof strategy will be to construct a map that sends every interface to a ``nice" one, and control the energy gain and entropy loss of the map. To do this, we use the powerful tool of cluster expansion, which allows us to compare the measure of two interfaces.

\begin{proposition}[Cluster Expansion; {\cite[Lem.~9]{GielisGrimmett02},\cite[Lem~7.118]{Grimmett_RC}}]
\label{prop:grimmett-cluster-exp}
There exists $\beta_0$ and a function $g$ such that for every $\beta \geq \beta_0$, the induced law on interfaces is given by
\begin{equation}\label{eq:CE}
    \bar{\mu}_n(\cI) = \frac{1}{Z_n}(1-e^{-\beta})^{|\partial \cI|}e^{-\beta|\cI|}q^{\kappa_\cI}\exp\Big(\sum_{f\in \cI} \g(f, \cI)\Big)\,,
\end{equation}
where the function g has the following properties: there exists universal constants $c, K > 0$ independent of $\beta$ such that for all $f, \cI, f', \cI'$, 
\begin{align}\label{eq:g-bound-1-face}
    |\g(f, \cI)| &\leq K\,, \\
\label{eq:g-bound-2-faces}
    |\g(f, \cI) - \g(f', \cI')| &\leq Ke^{-cr(f, \cI; f', \cI')}\,,
\end{align}
where $r(f, \cI; f', \cI') = \sup \{r: \cI \cap B_r(f) \equiv \cI' \cap B_r(f')\}$, i.e., $r(f, \cI; f', \cI')$ is the largest radius such that the interfaces $\cI, \cI'$ agree on the balls of this radius around the faces $f, f'$ (the intersections with a ball of this radius are translates of one another). 
\end{proposition}

The following geometrical lemma will be useful for controlling the entropy of maps.

\begin{lemma}[{\cite[Lem.~14]{GielisGrimmett02},\cite[Lem.~7.131]
{Grimmett_RC}}] 
\label{lem:num-of-0-connected-sets}
The number of 1-connected sets of faces of size $k$ containing a given face $x$ is bounded above by $s^k$ for some universal constant $s$.
\end{lemma} 

The above tools were used by Gielis and Grimmett to prove the following rigidity results:

\begin{proposition}[Exponential tails on groups of walls, {
\cite[Lem.~15]{GielisGrimmett02},\cite[Lem.~7.132]{Grimmett_RC}}]
\label{prop:grimmett-exp-tail-grp-of-walls}
Let $\sF_x$ be the random variable denoting the group of walls at $x$, and recall that $\sE_x$ denotes the empty set of walls, i.e. that there are no walls indexed by $x$. There exists $\beta_0$ and a constant $C > 0$ such that for every $\beta \geq \beta_0$, for any admissible collection of groups of walls $\{(F_y)_{y \neq x},\ F_x\}$,
\begin{equation*}
    \frac{\bar{\mu}_n(\sF_x = F_x,\, (\sF_y)_{y \neq x} \equiv (F_y)_{y \neq x})}{\bar{\mu}_n(\sF_x = \sE_x,\, (\sF_y)_{y \neq x} \equiv (F_y)_{y \neq x})} \leq \exp\big(-(\beta - C)\fm(F_x)\big)\,.
\end{equation*}
\end{proposition}

\begin{proposition}[Rigidity, {\cite[Thm.~2]{GielisGrimmett02},
\cite[Thm.~7.142]{Grimmett_RC}}] \label{prop:grimmett-localization}
Let $f$ be any horizontal face at height 0. Denote by $\{f \leftrightarrow \infty\}$ the event that there is a 1-connected sequence of faces $\{f_i\}$ from $f$ to the boundary $\partial \Lambda_n$ such that all the faces $f_i$ are ceiling faces of $\cI$ at height 0. Then, there exists $\beta_0$ such that for any $\beta \geq \beta_0$, there is a constant $\epsilon_\beta$ such that for all $n$ and all starting $f$,
\begin{equation*}
    \bar{\mu}_n(f \leftrightarrow \infty) \geq 1 - \epsilon_\beta\,.
\end{equation*}
\end{proposition}

\subsection{Pillars}
The general strategy will be to use cluster expansion arguments to prove results about the full interface, and then transfer these results to the Potts and random-cluster interfaces of interest. For technical reasons that will become apparent later, we need to begin with the $\Top$ interface. To measure the ``height of the $\Top$ interface above a location $x$", we will start at $x$ and follow the upward intrusion of $\Atop^c$ vertices into the $\Atop$ phase of the model. Although the actual $\Top$ interface may reach a higher point above $x$ via an intrusion beginning from another vertex $y$, we choose to measure this more ``local" height of the interface, and this suffices since the maximum of $\Itop$ will still be equal to the maximum height over all such intrusions. We begin this section by first proving some basic properties of the $\Top$ interface, and then making the above idea rigorous through the introduction of pillars. The section then concludes with some preliminary results on the height of a pillar.

\begin{remark}[Properties of $\Itop$]\label{rem:properties-of-Itop}
    We begin by proving a few properties of $\Itop$ that will be useful throughout the paper. Note that $\Itop$ really is an interface in the sense that every path from $\partial \Lambda_n^-$ to $\partial \Lambda_n^+$ must at some step go from a vertex in $\Atop^c$ to a vertex in $\Atop$, which then must cross a face of the $\Top$ interface. Note also that for every edge $e = [v, w]$ such that $f_e \in \Itop$, one of $v$ or $w$ is in the $\Top$ component $\Vtop$, and the other is not. Indeed, say that $v \in \Atop$ and $w \in \Atop^c$. Then, $w$ is not in the $\Top$ component by definition, and $v$ is either in $\Vtop$ or in a finite component of $\Vtop^c$. But, the latter case is impossible since $w$ is in the infinite component of $\Vtop^c$ and is adjacent to $v$.

    Finally, we claim that $\Itop$ determines $\Atop$, and both $\Atop$ and $\Atop^c$ are connected. Moreover, we show that if $V$ is the set of vertices that are not separated from $\partial \Lambda_n^+$ by $\Itop$, then $V = \Atop$. First we show that $V, V^c$ are both infinite connected components. Suppose for contradiction that there is a finite component $A \subseteq V^c$ surrounded by vertices $B \subseteq V$. Then, for every edge $e$ incident to both a vertex of $A$ and $B$, we have $f_e \in \Itop$. The vertices of $A$ are all in $\Vtop^c$, so that as noted above, all the vertices in $B$ must be in $\Vtop$. But $B$ surrounds $A$, and so $A$ is a finite component of $\Vtop^c$ and must be in $\Atop$, which contradicts the fact that the faces separating $A$ from $B$ are in $\Itop$. Similarly, if $A \subseteq V$ is a finite component surrounded by vertices of $V$, then $A$ must be surrounded by faces of $\Itop$ and thus be separated from $\partial \Lambda_n^+$ by $\Itop$, contradicting the definition of $V$. Now we show that $V = \Atop$. Since $V^c \subseteq \Vtop^c$ and is an infinite connected component, then $V^c \subseteq \Atop^c$. On the other hand, if $v \in \Atop^c \cap V$, then there must be a $\Lambda_n$-path $P$ from $v$ to $V^c$ consisting only of vertices in $\Vtop^c$ by the fact that $\Atop^c$ is the infinite component of $\Vtop^c$. But since $v \in V$, there must be an edge $e = [u, w]$ in $P$ that crosses from $u \in V$ to $w \in V^c$. The face $f = f_e$ must then be in $\Itop$, but then at least one of $u, w$ is in $\Vtop$, which contradicts the construction of the path~$P$. Thus, $\Atop^c \cap V = \emptyset$.
\end{remark}

\begin{definition} [Pillar]\label{def:pillar}  
Given an interface $\cI$, we can read from it the corresponding $\Top$ interface $\Itop$. As above, this defines a set of vertices $\Atop$ and $\Atop^c$. Let $x$ be a vertex at height 1/2. Let $V$ be the connected component of vertices in $\Atop^c$ with height $\geq 1/2$ that contains $x$, which we call the vertices of the pillar. Denote by $F$ the set of faces bounding $V$ with height $\geq 1/2$. Note that $F$ is a subset of $\Itop$. Possibly attached to $F$ are some hairs, which we define to be 1-connected components of $\cI \setminus \Itop$. We define the pillar at $x$, $\cP_x$, as the union of $F$ with all hairs that are 1-connected to $F$ at an edge with height $\geq 1/2$. (Note that if a hair connects to $F$ at height $\geq 1/2$ and then descends below that height, it is still included in $\cP_x$.) We analogously define a pillar in the $\Bot$ interface.
\end{definition}

Note that a priori, it is possible that the hairs of the pillar reconnect to other walls of $\Itop$. However, this will not happen for pillars which are in an isolated cone (see \cref{def:isolated-pillar}), and whenever this may be problematic, we will first restrict to such a space of pillars.

By abuse of notation, we will sometimes also use $\cP_x$ to refer to the set of vertices in the pillar. We also define the height of $\cP_x$, denoted $\hgt(\cP_x)$, as the height of the face set $F$, so that the max height of the $\Top$ interface is equal to the maximum height over all pillars. Denote the event 
\begin{equation}\label{eq:def-Eh}
E_h^x := \{\hgt(\cP_x) \geq h\}\,.
\end{equation} 
Recall that $V(\cP_x)$ is connected by definition; the following observation notes that it is also co-connected (i.e., its complement is connected).

\begin{observation}\label{obs:pillar-vert-simp-conn}
    The vertices of a pillar $V(\cP_x)$ form a co-connected set. Indeed, any finite component $A$ of $V(\cP_x)^c$ is by definition surrounded by vertices of $\Atop^c$, and hence a part of $\Atop^c$. All the vertices of $A$ also have height $\geq 1/2$, and thus by definition should actually be included as a part of $V(\cP_x)$.
\end{observation}

\begin{remark}\label{rem:height-0-pillar}
We will distinguish between the events $\{\hgt(\cP_x) = 0\}$ and $\{\hgt(\cP_x) < 0\}$ even though both correspond to the event that $V$ in \cref{def:pillar} is empty (and henceforth, the event $E_0^x$ will not include the event $\{\hgt(\cP_x) < 0\})$. We say that the pillar height is 0 in this case only when the face corresponding to the edge $[x - \ez, x]$ is in the $\Top$ interface, otherwise we say that the pillar has negative height. Note that if the pillar height is 0, the fact that the face below $x$ is in the $\Top$ interface implies that exactly one of $x$ or $x - \ez$ is in the $\Top$ component (i.e., has a wired path to the upper half boundary). But, it must be that $x$ is in the $\Top$ component since the other case implies $x \in \Atop^c$, which contradicts $\cP_x = \emptyset$.
\end{remark}

When we eventually move to the Potts model, it will be helpful at times to reveal only the outer shell $\cP_x^\mathrm{o}$ of the pillar without revealing any edges inside the pillar. This motivates the following definition:
\begin{definition}[Pillar shell]\label{def:pillar-shell}
    We define $\cP_x^\mathrm{o}$ as above, except when adding hairs to the face set $F$, we do not include any faces dual to edges with endpoints in $\Atop^c$.
\end{definition}

\begin{observation}\label{obs:pillar-given-by-walls}
The faces of a pillar $\cP_x$ is a subset of the faces of the walls nesting $x$, $\fW_x$, together with any walls interior $\fW_x$, together with all interior ceilings of such walls.
\end{observation}

\begin{comment}
\begin{proof}
Denote the collection of faces in the lemma by $H$, with wall faces $W$ and ceiling faces $C$. Suppose that the only faces in the interface were the top interface. This gives a possibly different set of wall and ceiling faces denoted $\tilde{W}$ and $\tilde{C}$, and thus a different $\tilde{H}$ (i.e., $\tilde{H}$ is all walls that nest $x$ plus all interior walls and plus all interior ceilings if $\cI = \Itop$). In this case, the geometry is as in the Ising case, and so the faces of the pillar which lie in the top interface belong to $\tilde{H}$ (see \cite[Observation 2.17]{GL_max}). But, for any wall that nests $x$, it will still nest $x$ upon adding faces to the wall, so $\tilde{H} \subseteq H$. It just remains to show that adding the hairs on the pillar won't cause any problems. If a hair attaches to a face of $\cP_x$ in $\tilde{W}$, then it is a part of the same wall and is thus in $H$. If it attaches to a face in $\tilde{C}$, then its projection either leaves the projection of that ceiling or stays within it. If it leaves, then it must cross into the projection of a wall in $W$, in which case it must be the same wall by \cref{lem: wall-ceiling-geometry}. If it stays within the ceiling, then it is interior to a wall in $W$. In either case, these faces are still in $H$.
\end{proof}
\end{comment}

\begin{observation}\label{obs:wall-nest-face}
For all faces $f \in \cP_x$, there exists a wall $W$ that nests both $f$ and $x$. Similarly, for any vertex $y \in \cP_x$, there exists a wall that nests both $y$ and $x$.
\end{observation}

The decomposition of the full interface into walls and ceilings, though powerful in proving rigidity, is not sufficient in studying the pillar. We instead decompose the pillar itself into increments.

\begin{definition}[Spine, base, increments, cut-height/point] We call a half integer $h$ a cut height of $\cP_x$ if there is only one vertex $v$ of $\cP_x$ with height $h$, and the only faces of $\cP_x$ at height $h$ are the four faces bounding the sides of $v$. We call $v$ a cut-point of $\cP_x$, and we enumerate the cut-points by increasing height.
The spine of $\cP_x$, denoted $\cS_x$, is the set of faces of $\cP_x$ with height $\geq \hgt(v_1)$. The base $\sB_x$ will be the remaining faces of $\cP_x$. 
Suppose that the spine has $\sT + 1$ cut-points. For $i \leq \sT$, the $i$-th increment $\sX_i$ is the set of faces of $\cP_x$ in the slab $\cL_{[\hgt(v_i), \hgt(v_{i+1})]}$. The vertices of $\sX_i$ are the vertices of $\cP_x$ in the same slab. Sometimes we will write $\sF(\sX_i)$ to reference specifically the face set of the increment. Note that the spine does not necessarily end at a cut-point, and so there may also be a remainder increment which is the set of faces of $\cP_x$ with height in $[\hgt(v_{\sT+1}), \infty)$. We denote this by $\sX_{>\sT}$ or $\sX_{\sT + 1}$. A trivial increment consists of just two vertices $v, v + \ez$, where the faces of the increment are just the 8 faces which bound the sides of the these vertices. We denote such an increment by $X_\trivincr$.
Finally, we can define the spine, base, and increments also with respect to the pillar shell, and denote these by $\cS_x^\mathrm{o}, \sB_x^\mathrm{o}, \sX_i^\mathrm{o}$ respectively. Note however that the cut-points of $\cP_x$ and $\cP_x^\mathrm{o}$ are the same.
\end{definition}

\begin{definition}[Excess area of increments]
For an increment $\sX_i$, we define the excess area $\fm(\sX_i) = |\sF(\sX_i)| - 4(\hgt(v_{i+1}) - \hgt(v_i)+1)$, i.e., the number of extra faces compared to a stack of trivial increments of the same height. This definition applies to the remainder increment if we set $\hgt(v_{\sT + 2}) = \hgt(\cP_x) - 1/2$. Note that if $\sX_i$ is not a trivial increment, then the fact that each height in between $\hgt(v_{i+1})$ and $\hgt(v_i)$ is not a cut-point implies that
$\fm(\sX_i) \geq (\hgt(v_{i+1}) - \hgt(v_i) - 1) \vee 1$,
which implies that \begin{equation}\label{eq:incr-excess-area}
|\sF(\sX_i)| \leq 5\fm(\sX_i) + 8\,.
\end{equation}
\end{definition}

\begin{proposition} [Exponential tail on height of pillar]\label{prop:exp-tail-pillar-height}
There exists $\beta_0$ and a constant $C>0$ such that for every $\beta \geq \beta_0$, for all $x$, and for all $h \geq 1$, 
\begin{equation*}
    \bar{\mu}_n(\hgt(\cP_x) \geq h) \leq \exp\big(-4(\beta - C)h\big)\,.
\end{equation*}
\begin{proof}
This is a direct consequence of the exponential tail on the size of groups of walls. We direct the reader to the proof of \cite[Theorem 2.26]{GL_max} to see how it follows, and just provide a sketch here. The idea is that in order for the pillar at $x$ of the $\Top$ interface to reach height $h$, there needs to be a sequence of nested walls $(W_{x_s})$ nesting $x$ such that such that $\sum_s \fm(W_{x_s}) = h_1$, and a different sequence of nested walls $(W_{y_t})$ interior to a ceiling of some $W_{x_s}$ such that $\sum_t \fm(W_{y_t}) \geq 4h - h_1$, for some $h_1$. 
The crucial bound to prove is that for $\fF_x$ denoting the group of walls of the nested sequence $\fW_x$, we still have an exponential tail:
\begin{equation}\label{eq:tail-on-nested-wall-group}
\bar{\mu}_n(\fm(\fF_x) \geq r) \leq Ce^{-(\beta - C)r}
\end{equation}
for some $C>0$, and one can prove this using the exponential tails on groups of walls established in \cref{prop:grimmett-exp-tail-grp-of-walls}. Then, the proof concludes by summing over possible values of $h_1$.
\end{proof}
\end{proposition}

\begin{observation}\label{obs:close-edge}
We have $\bar{\mu}_n(\omega_e = 0 \mid
\omega\restriction_{\Lambda_n\setminus \{e\}} \equiv \eta\restriction_{\Lambda_n\setminus \{e\}}) \geq 1 - p$
for every fixed configuration $\eta$ and any edge $e$. The exact probability is either $1 - p$ or $\frac{q(1-p)}{q(1-p) + p}$ depending on whether closing $e$ creates a new open cluster or not. However, the latter term is increasing in $q$, and thus minimized at $q = 1$ where it is equal to $1 - p$. As a consequence, if $A$ is any event such that for every configuration $\omega \in A$, closing the edge $e$ will not take $\omega$ out of $A$, then we can sum over $\omega$ to get 
\[\mu_n(A) \leq (1-p)^{-1} \mu_n(A, \omega_e = 0)\,.\] In fact, $e$ can even be a random edge depending on $\omega \in A$. Finally, if $e$ depends on $\omega$ in such a way that closing $e$ always creates an additional open cluster, then the above inequality can be strengthened to \[\mu_n(A) \leq \frac{q(1-p)+p}{q(1-p)}\mu_n(A, \omega_e = 0) = \frac{e^\beta+q-1}{q}\mu_n(A, \omega_e = 0)\,.\]
\end{observation}

\begin{proposition}\label{prop:bound-RC-rate}
There exist $\beta_0$ and $C > 0$ such that for every $\beta \geq \beta_0$, for all $x$, and for all $h \geq 1$, 
\begin{equation*}
 -4\beta \leq \frac 1h \log\bar{\mu}_n(\hgt(\cP_x) \geq h) \leq -4(\beta-C)\,.
\end{equation*}
\end{proposition}
\begin{proof}
The upper bound follows from \cref{prop:grimmett-localization}
 above. For the lower bound,  let $F$ be the $4h+1$ faces that surround the sides and top of the column of $h$ vertices, $\{x, x + \ez, \ldots, x + (0, 0, h - 1)\}$. Let $E$ be the set of edges $e$ such that $f_e \in F$. Finally, define $A$ as the set of configurations $\omega$ such that $\omega \restriction_{E^c} = \eta \restriction_{E^c}$ for some $\eta \in \{f_{[x, x - \ez]} \leftrightarrow \infty\}$ (defined as in \cref{prop:grimmett-localization}). Note that $A \cap \{\omega_e = 0\,:\; e \in E\} \subseteq \{\hgt(\cP_x) \geq h\}$. With this definition of $A$, we can apply \cref{obs:close-edge} to close the edges of $E$ one by one, so that
\begin{equation*}
\frac{\mu_n(\hgt(\cP_x) \geq h, \sep_n)}{\mu_n(\sep_n)} \geq \frac{\mu_n(A, \{\omega_e = 0\,:\; e \in E\}, \sep_n)}{\mu_n(\sep_n)} \geq \frac{\mu_n(A, \sep_n)}{\mu_n(\sep_n)}e^{-\beta(4h+1)} \geq (1 - \epsilon_\beta)e^{-\beta(4h+1)}\,.\qedhere
\end{equation*}
\end{proof}

\section{Finer properties of tall pillars}\label{sec:pillar-maps}

This section focuses on proving analogues for the results of \cite[Section 4]{GL_DMP} in the random-cluster setting. There, it was shown that (in the Ising model) a typical tall pillar has a trivial base, and is isolated from the rest of the interface. This is crucial for us because on this isolated space of pillars, we no longer run into the issue that the faces of $\cP_x$ might be 1-connected to other walls of $\cI$. Furthermore, many times we will want to study the effects of straightening or deleting parts of the pillar using the cluster expansion expression established in \cref{prop:grimmett-cluster-exp}. This is in general a complicated endeavor because the ``$g$"-terms will see the interactions between a shifted or deleted increment and nearby walls. Moving to this isolated space first automatically controls such interactions, and thereby greatly simplifies all the cluster expansion arguments which follow. 
Several results in this section follow verbatim from the work in \cite{GL_DMP}, and we will omit those proofs. Our primary contribution here is in showing that the new terms related to $|\partial \cI|$ and $\kappa_\cI$ in the cluster expansion do not pose any problems to the argument provided in the Ising model, which we show in \cref{lem:control-1-connected-faces,lem:control-interface-clusters}.

\begin{definition}[Truncated interface]\label{def:truncated-interface} We can define a truncated interface $\cI \setminus \cP_x$ by removing from $\cI$ every face that is in $\cP_x$, and adding in a face below every vertex $v \in \cP_x$ with $\hgt(v) = 1/2$. Note the abuse of notation in that $\cI \setminus \cP_x$ as a set of faces includes more than $\cI$ set-minus $\cP_x$, because we need to fill in the gaps left by removing $\cP_x$ to ensure that $\cI \setminus \cP_x$ is still an interface. We can similarly define $\cI \setminus \cS_x$ by removing every face that is in $\cS_x$ and adding in the face below $v_1$.
\end{definition}

\begin{definition}[Isolated pillar]\label{def:isolated-pillar}
Let $\Iso_{x, L, h}$ be the set of interfaces $\cI$ satisfying the following:
\begin{enumerate}
\item \label{it:iso-spine} The pillar $\cP_x$ has an empty base (equivalently, $x$ itself is the first cut-point), its increment sequence satisfies
\[\fm(\sX_t) \leq \begin{cases}
0 & \text{if } t \leq L^3\\
t & \text{if } t > L^3
\end{cases}\,,\]
and the number of faces in the spine $\cS_x$ is at most $10h$.
\item \label{it:iso-walls} The walls $(\tilde{W}_y)$ of $\cI \setminus \cP_x$  satisfy
\[\fm(\tilde{W}_y) \leq \begin{cases}
0 & \text{if } d(y, x) \leq L\\
 \log(d(y, x)) & \text{if } L < d(y, x) < L^3h
\end{cases}\,,\]
and $f_{[x, x - \ez]} \notin \cI$.
\end{enumerate}
\end{definition}

Whereas the notion of an isolated pillar is the primary object of interest in our proofs, as mentioned in \cref{sec:subsec-LD-rate-Potts}, we will also need its analog for the pillar shell $\cP_x^{\mathrm o}$ (see \cref{def:pillar-shell}), so as to alleviate information leaking to the FK--Potts model on the interior of the pillar.

\begin{definition}[Isolated pillar shell]
Analogously, we can define $\Iso_{x, L, h}^\mathrm{o}$ as the set of interfaces such that
\begin{enumerate}
\item The pillar $\cP_x^\mathrm{o}$ has an empty base (equivalently, $x$ itself is the first cut-point), and increment sequence satisfying:
\[\fm(\sX_t^\mathrm{o}) \leq \begin{cases}
0 & \text{if } t \leq L^3\\
t & \text{if } t > L^3
\end{cases}\,,\]
and the number of faces in the spine $\cS_x^\mathrm{o}$ is at most $ 10h$.
\item The walls $(\tilde{W}_y)$ of $\cI \setminus \cP_x$  satisfy
\[\fm(\tilde{W}_y) \leq \begin{cases}
0 & \text{if } d(y, x) \leq L\\
 \log(d(y, x)) & \text{if } L < d(y, x) < L^3h
\end{cases}\,,\]
and $f_{[x, x - \ez]} \notin \cI$.
\end{enumerate}
\end{definition}

Note that $\Iso_{x, L, h} \subseteq \Iso_{x, L, h}^\mathrm{o}$. One nice property of these spaces is that the pillar is well separated from the rest of the interface, in the sense of \cref{prop:cone-seperation}
and \cref{lem:cone-separation} below. 

For any $L, h$, we can define the following cones:
\begin{align*}
\Cone_x^1 &= \{f: \hgt(f) > L^3, d(\rho(f), x) \leq \hgt(f)^2 \wedge 10h\}\,,\\
\Cone_x^2 &= \{f: d(\rho(f), x) \geq L, \hgt(f) \leq (\log d(\rho(f), x))^2\}\,.
\end{align*}
 Let $\bF_{\parallel}$ be the $4L^3$ vertical and $L^3 + 1$ horizontal bounding faces of the vertex column $\{x, \ldots x + (0, 0, L^3 - 1)\}$. Define the cylinder $\Cyl_{x,r} := \{f\in \mathscr F(\Z^3): d(\rho(f),x) \le r\}$

\begin{proposition}[{\cite[Claim 4.4]{GL_DMP}}] \label{prop:cone-seperation}

Fix any $L$ large and any $h$. Any interface $\cI \in \Iso_{x,L,h}^\mathrm{o}$ satisfies
\begin{align}\label{eq:interface-containment-definitions}
    \cI \subseteq \underbrace{(\Cone_{x}^1\cap\cL_{<10 h})}_{\bF_{\triangledown}} \;\cup\; \bF_{\parallel} \;\cup\;  \underbrace{(\cL_{0}\cap \Cyl_{x,L})}_{\bF_{-}} \cup \underbrace{\Cone_x^2}_{\bF_{\curlyvee}} \;\cup\; \underbrace{\Cyl^c_{x,L^3 h}}_{\bF_{\ext}}\,.
\end{align}
For $\bF_{\triangledown},\bF_{\parallel}, \bF_{-}, \bF_{\curlyvee},\bF_{\ext}$ defined as above, the right-hand side is a disjoint union, 
\begin{align*}
    (\bF_{\triangledown} \cup \bF_{\parallel})\cap (\bF_{-} \cup \bF_{\curlyvee}\cup \bF_{\ext})= \emptyset
\end{align*}
and the pillar $\cP_x$ is a subset of the first two sets above, while $\cI \setminus \cP_x$ is a subset of the latter three sets. 
\end{proposition}

\begin{proof}
The proof of \cite[Claim 4.4]{GL_DMP} applies in this setting verbatim. 
See \cref{fig:isolated-pillar} for a visualization of $\Cone_x^1$ and $\Cone_x^2$.
\end{proof}

\begin{figure}
    \vspace{-.5cm}
    \includegraphics[width=0.9\textwidth]{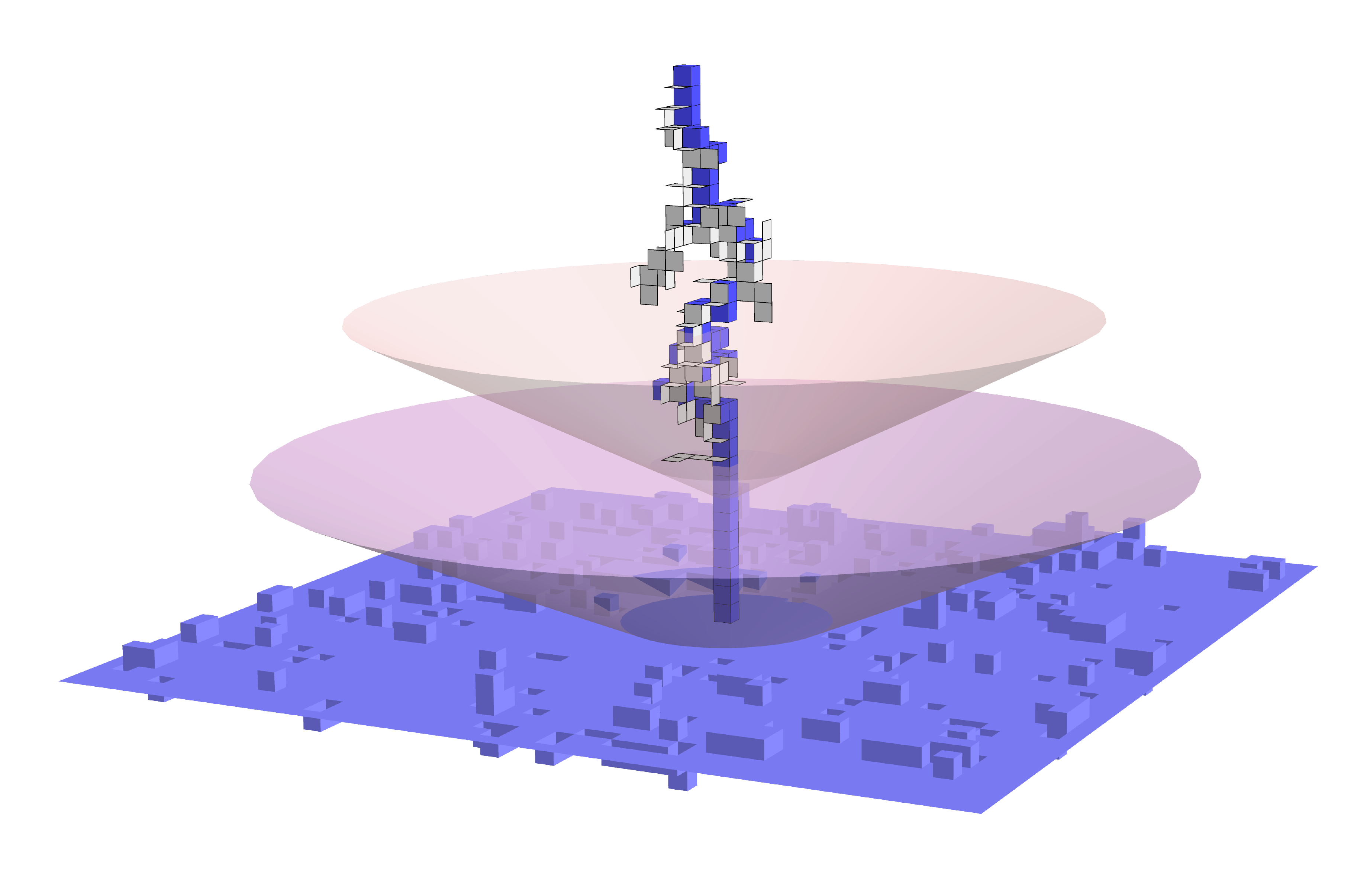}
    \vspace{-1cm}
    \caption{A typical isolated, tall pillar. The region above the tan cone is $\Cone_x^1$ and the region below the pink cone is $\Cone_x^2$.}
    \label{fig:isolated-pillar}
\end{figure}

\begin{corollary}\label{cor:pillar-doesnt-touch-other-walls}
For any $x, L, h$, on the event $\Iso_{x, L, h}^{\mathrm{o}}$ (and thus also on $\Iso_{x, L, h}$), the only faces of $\cI \setminus \cP_x$ which are 1-connected to $\cP_x$ are the four faces at height 0 which connect to the first cut-point of the pillar. (Explicitly, these are the faces $f_{[x+(\pm 1, 0, 0), x + (\pm 1 , 0, -1)]}$ and $f_{[x+(0, \pm 1, 0), x + (0, \pm 1, -1)]}$.)
\end{corollary}

\begin{lemma}[{\cite[Lemma 4.5]{GL_DMP}}]\label{lem:cone-separation}
There exists $C > 0$ and $c>0$ such that for all $L$ sufficiently large, and all $h \geq 1$, and any $\cI \in \Iso_{x, L, h}^\mathrm{o}$,
\begin{equation}\label{eq:pillar-in-cone-1}
\sum_{f \in F_\nabla \cup F_{||}}\;\sum_{g\in F_\Upsilon \cup F_\ext}e^{-cd(f, g)} \leq Ce^{-cL}\,,
\end{equation}
and 
\begin{equation}\label{eq:pillar-in-cone-2}
\sum_{f\in F_\nabla}\;\sum_{g \in F_-\cup F_\Upsilon \cup F_\ext}e^{-cd(f, g)} \leq Ce^{-cL}\,.
\end{equation}
\end{lemma}
\begin{proof}
See the proof of \cite[Lemma 4.5]{GL_DMP} with the following observation: In $\Iso_{x, L, h}^\mathrm{o}$, we only know that the spine of the pillar shell $\cS_x^\mathrm{o}$ has less than $10h$ faces, so it is possible that $\cS_x$ has more. However, all the additional faces must be between vertices in $\cP_x$, and so the spine $\cS_x$ still cannot have more than say $20h$ faces.
\end{proof}

Finally, we prove the following claim stating that in an isolated pillar, there is an $\omega$-path from $x$ to $\Lambda_n^-$, which will simplify certain proofs in \cref{sec:ld-RC,sec:ld-potts-bot}.

\begin{claim}\label{clm:x-in-Vbot}
     For the event $E_1^x = \{\hgt(\cP_x) \geq 1\}$ as per \cref{eq:def-Eh} and any $\omega \in \Iso_{x, L, h}^\mathrm{o} \cap E_1^x$, we have $x \in \Vbot$.
\end{claim}
\begin{proof}
    By definition, on $E_1^x$, we know that $x \in \Abot$. Suppose for contradiction that $x$ is actually part of some finite component of $\Vbot^c$, call it $A$. Let $F$ be the set of faces separating $A$ from $A^c$ (i.e., if $u \in A$ and $v \notin A$, then $f_{[u, v]} \in F$). Since $A$ must be connected and co-connected, $F$ is a 1-connected set of faces (for the justification that $F$ is 1-connected, see \cite[Prop.~5]{GielisGrimmett02},\cite[Thm.~7.3]{Grimmett_RC}). Moreover, $F \subseteq \clfaces$ . Since $x$ is a cut-point of $\cP_x$, $F$ must include the four faces to the sides of $x$ at height $1/2$, and hence $F \subseteq \cI$. Thus the condition $f_{[x, x - \ez]} \notin \cI$ implies that $x - \ez \in A$. However, if $F$ is to separate $x - \ez$ from $\partial \Lambda_n^-$, there must be some horizontal face of $F$ below $x - \ez$, which is necessarily a face of $\cI \setminus \cP_x$. Yet, no such faces can be in $\cI \setminus \cP_x$ by \cref{prop:cone-seperation}.
\end{proof}

We now prove that except on a set of probability $\epsilon_\beta$, a randomly sampled pillar of height $h$ is going to be an isolated pillar. The idea, as done in \cite[Theorem 4.2]{GL_DMP}, is to use cluster expansion to show that the energy gain in mapping an arbitrary interface to one in $\Iso_{x, L, h}$ beats the entropy of the map. A significant portion of that paper is spent on controlling the $g$-terms which appear in the cluster expansion, and controlling the entropy of the map $\Phi_\Iso$. We will omit those parts of the proof here as they apply exactly. One way to see why those proofs should still hold is to note that problems can only arise in the random-cluster model due to the more complicated geometry in including the hairs of the pillar. The entropy arguments of the cited paper are unaffected by this because they are based on counting the number of arbitrary 1-connected sets of size $k$, and are not limited to the Ising-type pillar structures to begin with. 

\begin{theorem}[{\cite[Theorem 4.2]{GL_DMP}}]\label{thm:iso-map} 
For $\beta > \beta_0$, there exist constants $L_\beta$, $\epsilon_\beta$ (going to $\infty$ and 0 respectively as $\beta \to \infty$) such that for every sequence $h = h_n \geq 1$, and $x = x_n$ with $h = o(d(x_n, \partial \Lambda_n))$, we have for all $0 \leq h' \leq h$, $0 \leq L \leq L_\beta$,
\begin{equation*}
    \bar{\mu}_n(\Iso_{x, L, h}|\hgt(\cP_x) \geq h') \geq 1 - \epsilon_\beta\,,
\end{equation*}
which also implies
\begin{equation*}
    \bar{\mu}_n(\Iso_{x, L, h}^\mathrm{o}|\hgt(\cP_x) \geq h') \geq 1 - \epsilon_\beta\,.
\end{equation*}
\end{theorem}

\begin{figure}
\begin{algorithm2e}[H]
\DontPrintSemicolon
If $\cI \in \Iso_{x, L, h}$, then set $\Phi_\Iso(\cI) = \cI$. Otherwise, proceed as follows:\\
\nl
Let $\{\tilde W_y : y\in \cL_{0}\}$ be the walls of $\cI\setminus \cS_{x}$. Let $(\sX_i)_{i\geq 1}$ be the increments of $\cS_{x}$.
\;

\BlankLine
\tcp{Base modification}
\nl
Mark $\bar x:= \{y\in \cL_{1/2}: y\sim^* x\}\cup \{x\}$ and $\rho(v_1)$ for deletion.\; 

\nl 
\If{the interface with standard wall representation $\Theta_{\textsc{st}}\tilde \fW_{v_1}$ has a cut-height}
{Let $h^\dagger$ be the height of the highest such cut-height.\;
Let $y^\dagger$ be the index of a wall that intersects $(\cP_{x} \setminus \tilde \fW_{v_1}) \cap \cL_{h^\dagger}$ and mark $y^\dagger$ for deletion.\;
}
\BlankLine
\tcp{Spine modification}

\nl
\For{$j=1$ \KwTo $\sT+1$}{     
     \If(\hfill\tcp*[h]{(A1)}){
\quad$ \fm(\sX_j) \geq \begin{cases} 0 & \mbox{if }j\le L^3 \\ j-1 & \mbox{if }j>L^3\end{cases}$ \quad\mbox{}}{Let $\fs \gets j$.} 
      \If(\hfill\tcp*[h]{(A2)}){\quad $d( \tilde W_y\cup \lceil \tilde W_y\rceil,\sX_j)\le (j-1)/2$ \qquad\quad for some $y$\quad\mbox{}}{
      Let $\fs\gets j$ and let $y^*$ be the minimal index $y$ for which~({\tt A2}) holds.}
}
 Let $j^* \gets \fs$ and mark $y^*$ for deletion.\;

 \nl 
 \If(\hfill\tcp*[h]{(A3)}){$|\sF(\cS_{x})| >5h$}{let $\fs \gets \sT+1$ and $j^* \gets \fs$.}
 
 \BlankLine
\tcp{Environment modification}
 \nl
 \For {$y\in \cL_{1/2}\cap \Cyl_{L^3 h}(x)$}{
    \If{\quad$\fm(\tilde W_y) \ge \begin{cases}0 & \mbox{if }d(y,x)\le L \\  \log [d(y,x)] & \mbox{else}\end{cases}$}{Mark $y$ for deletion}
    }
 
\BlankLine
\tcp{Reconstructing the interface}
\BlankLine

\nl 
\lForEach{$y$ marked for deletion}
{remove $\Theta_{\textsc{st}}\Clust(\tilde \fW_{y})$ from $(\Theta_{\textsc{st}}\tilde W_y)_{y\in \cL_{1/2}}$.}

 \nl Add the standard wall $\Theta_{\textsc{st}}W_{x,\parallel}^{\mathbf h}$ consisting of the bounding vertical faces of $(x+ (0,0, i))_{i=1}^{\mathbf h-1}$ where $\mathbf h := (\hgt(v_1)- \frac 12)$.

\nl Let $\cK$ be the interface with the resulting standard wall representation.

\nl Let
\[ \cS \gets 
\begin{cases}\big(\underbrace{X_\trivincr,\ldots,X_\trivincr}_{\hgt(v_{j^*+1})-\hgt(v_1)},\sX_{j^*+1},\ldots,\sX_{\sT},\sX_{>\sT}\big) &\mbox{ if \mbox{({\tt{A3}}) is not violated}},\\
\noalign{\medskip}
 \big( \underbrace{X_\trivincr,\ldots,X_\trivincr}_{h- \mathbf h}\big) &\mbox{ if \mbox{({\tt{A3}}) is violated}}\,.\end{cases}\,.
 \]
\nl Obtain $\Phi_{\Iso}(\cI)$ by appending the spine with increments $\cS$ to $\cK$ at $x+ (0,0,\hgt(\cC_\bW)+ \mathbf h)$.\;

\caption{The map $\Phi_{\Iso}= \Phi_{\Iso}(x,L,h)$}
\label{alg:Iso}
\end{algorithm2e}
\end{figure}

Let $\Phi_\Iso$ be defined as in \cref{alg:Iso}. In the algorithm and in what follows, we denote by $\lceil \tilde W_y\rceil$ the set of interior ceilings of the wall $W_y$. We first show that the map is well-defined and yields an interface in $\Iso_{x, L, h}$, which follows verbatim from the proof of \cite[Lemma~4.8]{GL_DMP}; we include the short proof here for completeness.

\begin{lemma}[{\cite[Lemma 4.8]{GL_DMP}}]\label{lem:well-defined-iso}
For every $L$ large, $0 \leq h' \leq h$, and $\cI \in E_{h'}^x$, the image $\cJ:= \Phi_\Iso(\cI)$ is a well-defined interface in $E_{h'}^x \cap \Iso_{x, L, h}$.
\end{lemma}
\begin{proof}
     Observe that $\cK$ in Step 9 of \cref{alg:Iso} is a valid interface since prior to Step 8, all we did was delete walls from the interface $\cI \setminus \cS_x$ (recall \cref{def:truncated-interface}), and adding a column at $x$ doesn't cause any problems because any walls that would intersect the column would have been deleted due to Step 2. Hence, it remains to show that the pillar generated in Steps 10 and 11 will not intersect $\cK$ except at the initial column added in Step 8. This follows easily as a result of the separation established in \cref{prop:cone-seperation}. Indeed, Steps 4 and 5 ensure that the pillar generated satisfies \cref{it:iso-spine} in the definition of an isolated pillar, and thus is a subset of $\bF_{\triangledown} \cup \bF_{\parallel}$ by \cref{prop:cone-seperation}. Similarly, the deletion of walls in Step 6 ensure that that $\cK$ satisfies the criterion of \cref{it:iso-walls} in an isolated pillar, whence \cref{prop:cone-seperation} implies that other than the initial column built at $x$, $\cK$ is a subset of  $\bF_{-} \cup \bF_{\curlyvee} \cup \bF_{\ext}$, and the disjointness follows by the same proposition.
\end{proof}
\begin{lemma}[{\cite[Corollary 4.11]{GL_DMP}}]\label{lem:wall-intersect-col-height}
In \cref{alg:Iso}, the walls $\tilde{\fW}_{v_1} \cup \tilde{\fW}_{y^\dagger}$ intersect heights $1/2, \ldots , \hgt(v_1) - 1$ in at least five faces.
\end{lemma}
\begin{proof}
By \cref{alg:Iso}, an interface consisting of just the walls $\tilde{\fW}_{v_1}$ has no cut-heights between $h^\dagger + 1$ and $\hgt(v_1) -1$. That means each of those heights must be intersected by $\tilde{\fW}_{v_1}$ in at least five faces.

By \cref{obs:wall-nest-face}, there exists a wall $W$ that nests both $y^\dagger$ and $v_1$. By the algorithm, the walls $\tilde{W}_{v_1}$ and $\tilde{W}_{y^\dagger}$ are distinct, so let their innermost nesting ceilings within $\ceil{W}$ be $\cC_{v_1}$ and $\cC_{y^\dagger}$. $W$ must surround the sides of every vertex below faces of these ceilings, and each ceiling must have at least two faces if it is to nest a wall. Since it takes at least six faces to surround the sides of two vertices, then $W$ must intersect every height below $\hgt(\cC_{v_1}) \vee \hgt(\cC_{y^\dagger})$ in at least 6 faces.

Finally, $\tilde{W}_{v_1}$ must surround at least one vertex at every height between $\hgt(\cC_{v_1})$ and $\hgt(v_1)$. Since $\tilde{W}_{y^\dagger}$ also reaches height $h^\dagger$, together they must contribute at least five faces to each height between $\hgt(\cC_{v_1}) \vee \hgt(\cC_{y^\dagger})$ and $h^\dagger$.
\end{proof}

Note that by definition, we have for $\cJ = \Phi_\Iso(\cI)$,
\[\fm(\cI;\cJ) = \begin{cases}
\sum_{z \in \bD} \fm(\tilde{W}_z) + \sum_{i = 1}^{j^*} \fm(\sX_i) - |W_{x, ||}^{\bh}| & \text{(A3) is not violated}\\
\sum_{z \in \bD} \fm(\tilde{W}_z) + \sum_{i = 1}^{\sT + 1} \fm(\sX_i) +4(\hgt(v_{\sT+1})-h) - |W_{x, ||}^{\bh}| & \text{(A3) is violated}
\end{cases}\,.\]

In the following claim, we provide an upper bound for $|W_{x,||}^{\bh}|$ and $j^*$ in terms of $\fm(\cI;\cJ)$. 

\begin{claim}[{\cite[Claim 4.9]{GL_DMP}}]\label{clm:m-lower-bnd-in-column}
For every $L$ large, $\cJ = \Phi_\Iso(\cI)$, we have
\begin{equation}\label{eq:m-lower-bnd-in-column}
|W_{x,||}^{\bh}| \leq \frac{4}{5}\fm(\tilde{\fW}_{v_1} \cup \tilde{\fW}_{y^\dagger}),\ \text{and thus } \fm(I;J) \geq \frac{1}{5}\fm(\bigcup_{y \in \bD} \tilde{W}_y) + \sum_{i = 1}^{j^*} \fm(\sX_i)\,.
\end{equation}
In particular, 
\begin{equation}\label{eq:m-lower-bnd-in-column-2}
|W_{x,||}^{\bh}| \leq 4\fm(\cI;\cJ),\ and\ \fm(\bigcup_{y \in \bD}\tilde{W}_y) \leq 5\fm(\cI;\cJ)\,,
\end{equation}
and
\[\begin{cases}
j^* - 1 \leq (2 \vee L^3)\fm(\cI;\cJ)\ & \text{if (A3) is not violated}\\
h - \bh \leq \fm(\cI;\cJ)\ & \text{if (A3) is violated}
\end{cases}\,.\]
\end{claim}
\begin{proof}               
By \cref{lem:wall-intersect-col-height}, we have
\begin{equation*}
|W_{x, ||}^\bh| = 4(\hgt(v_1) - \frac{1}{2}) \leq \frac{4}{5}(\fm(\tilde{\fW}_{v_1} \cup \tilde{\fW}_{y^\dagger}))\,,
\end{equation*}
which proves \cref{eq:m-lower-bnd-in-column,eq:m-lower-bnd-in-column-2}.

If (A3) is violated, then the spine replacement generates an excess area of $5h - 4(h - \bh) \geq h - \bh$.
If (A3) is not violated, if $j^* = 1$ then the bound is trivial. If $j^* > 1$, then $j^*$ is set for the last time either because of (A1) or (A2) being violated. If it was due to (A1) being violated, then either $j^* \leq L^3$ or $j^* \leq \fm(\sX_i)+1$. If it was due to (A2) being violated, then $d( \tilde W_y\cup \lceil \tilde W_y\rceil,\sX_j) \leq (j-1)/2$ for $y = y^*$. Now we note that in general for any $j, y$, we have
\begin{equation*}
j - 1 \leq d( \tilde W_y\cup \lceil \tilde W_y\rceil,\, \sX_j) + \fm(\tilde{\fW}_y)\,.
\end{equation*}
Indeed, the lowest part of $\sX_j$ has height $\geq j-1$, whereas the highest point reached by a face of $\tilde{W}_y$ is at most $\fm(\tilde{\fW}_y)$, and the remaining distance is made up by the term $d( \tilde W_y\cup \lceil \tilde W_y\rceil,\, \sX_j)$.  Applying this to $j^*, y^*$ gets \begin{equation*}
j^* - 1 \leq d( \tilde W_y\cup \lceil \tilde W_y\rceil,\sX_j) + \fm(\tilde{\fW}_y) \leq (j^*-1)/2 + \fm(\tilde{\fW}_y)\,,
\end{equation*}
so that $j^* - 1 \leq 2m(\cI;\cJ)$. 
\end{proof}

The following two lemmas control the terms related to $|\partial \cI|$ and $\kappa_\cI$ in the cluster expansion.

\begin{lemma}\label{lem:control-1-connected-faces}
Let $\cI,\cJ$ be two interfaces with $\cI \notin \Iso_{x, L, h}$ and $\cJ = \Phi_{\Iso}(\cI)$. Then $ |\partial \cJ|-|\partial \cI| \leq C\fm(\cI; \cJ)$ for some constant $C$ which can depend on $L$.
\end{lemma}

\begin{proof}
The goal is to construct an injective map $T$ from a subset of $\partial \cJ$ into $\partial \cI$, and show that the remaining set of faces that $T$ is not defined on has size smaller than $C\fm(\cI;\cJ)$, which would prove the lemma. Throughout this proof, let $C_0$ be the number of faces that can be 1-connected to a particular face (namely, $C_0 = 12$).

\begin{enumerate}[label=\textbf{Step~\arabic*}:, ref=\arabic*, wide=0pt, itemsep=1ex]
    \item \label[step]{st-1-control-1-conn-faces}
    % {\bf Step 1}: 
    Consider first the faces of $\partial \cJ$ which are 1-connected to the column of faces $W_{x,||}^{\bh}$. There are at most $C_0|W_{x,||}^{\bh}|$ faces to account for here, but we already know that $|W_{x,||}^\bh| \leq 4\fm(\cI;\cJ)$ by \cref{eq:m-lower-bnd-in-column-2}, so we do not need to define $T$ on these faces.
    
    \item \label[step]{st-2-control-1-conn-faces}
    % {\bf Step 2}: 
    If (A3) was not violated, then $\cP_x^\cJ$ consists of a stack of trivial increments and then a horizontally shifted copy of the increments of $\cP_x^\cI$ with index starting from $j^* + 1$. Each face in $\partial \cJ$ which is 1-connected to one of these latter increments therefore also has a copy in $\partial \cI$, and we associate them under the map $T$. Note that a priori, it is possible that a hair on an increment is actually 1-connected to $\cI \setminus \cP_x^\cI$ by connecting to another part of $\Itop$. However, this cannot happen for increments with index larger than $j^*$ by condition (A2) of the algorithm. We remark that because this portion of $\cP_x^I$ begins with the cut-point $v_{j^* + 1}$, the image of $T$ in this step consists only of faces with height $\geq \hgt(v_{j^* + 1}) - 1$ that are 1-connected to $\cS_x^\cI$.

\item \label[step]{st-3-control-1-conn-faces}
% {\bf Step 3}: 
The rest of $\cP_x^\cJ$ consists of trivial increments that replace the spine $\cS_x^\cI$ up to increment $j^*$, so it is a straight vertical column of vertices from $\hgt(v_1)$ to $\hgt(v_{j^*+1}) -1$ (or to $h$ if (A3) was violated). Let $\sY_i$ correspond to the stack of trivial increments that have the same height as the increment $\sX_i$ from $\cS_x^\cI$. Let $B$ be an empty set of faces, and begin the following iterative process: Start with $i = 1$. If $\sX_i$ (from $\cS_x^\cI$) is trivial, then $\sY_i$ is a single trivial increment. For every face $g \in \overline{\sY_i} \setminus \cJ$, there is a corresponding face $h \in \overline{\sX_i} \setminus \cI$ in the same orientation. If $g$ has height $\leq \hgt(v_{i+1})$ (where $v_{i+1}$ is the cut-point in $\cS_x^\cI$), is not in $B$, and has not yet been assigned a face under $T$, then let $T(g) = h$. (It is possible that some faces may have already been added to $B$ or been assigned a face under $T$ since two consecutive increments overlap at a common cut-point.) Otherwise, if $\sX_i$ is not a trivial increment, then we must have $\fm(\sX_i) \geq \hgt(v_{i+1}) - \hgt(v_i)$. So, we add to $B$ all the faces in $\overline{\sY_i} \setminus \cJ$ that have height $\leq \hgt(v_{i+1})$ and have not been assigned a face under $T$. Note that the number of faces added is at most  $4C_0(\hgt(v_{i+1}) - \hgt(v_i) + 1) \leq 8C_0\fm(\sX_i)$. Then, increase $i$ by 1 and repeat until $i = j^*$. Since $|B| \leq \sum_{i = 1}^{j^*} 8C_0\fm(\sX_i) \leq 8C_0\fm(I;J)$, we do not need to define $T$ on the faces of $|B|$. Now we show $T$ is still injective. There are no problems within this step since the image of $T$ in each iteration is either empty or contains faces with height in $[\hgt(v_i) + 1/2, \hgt(v_{i+1})]$ (except for the case $i = 1$, whence the image can contain faces with height in $[\hgt(v_1), \hgt(v_2)]$). This is because every assignment $T(g) = h$ here has $\hgt(g) = \hgt(h)$. Thus, by the comment at the end of \cref{st-2-control-1-conn-faces}, we only need to worry about the injectivity of $T$ in the last iteration $i = j^*$, and only if $\sX_{j^*}$ is trivial. But actually, in this iteration no faces would have been added to the domain of $T$ since any faces with height $\leq \hgt(v_{j^*})$ would have been handled in when $i = j^* - 1$, and any faces with height $\geq \hgt(v_{j^*}) + 1/2$ would have been handled in \cref{st-2-control-1-conn-faces}.

\item \label[step]{st-4-control-1-conn-faces}
%{\bf Step 4}: 
Reset $B$ to be an empty set. We will be adding pairs of faces $(g, h)$ to $B$, where $g$ is some face in $\partial \cJ$ that we choose not to define $T$ on, and $h$ will be used to keep track of the size of $B$. In the previous steps, we have already handled faces of $\partial J$ which are 1-connected to $\cP_x^\cJ$. We can divide the remaining faces of $\cJ$ into the following sets:
\begin{align*}
A_1 &= \mbox{Ceiling faces of $\cJ \setminus \cP_x^\cJ$ that are in the projection of a ceiling face of $\cI$}\,;\\
A_2 &= \mbox{Ceiling faces of $\cJ \setminus \cP_x^\cJ$ that are in the projection of a deleted wall of $\cI \setminus \cS_x^\cI$}\,;\\
A_3 &= \mbox{Ceiling faces of $\cJ \setminus \cP_x^\cJ$ that are in the projection of $\overline{\cS_x^\cI}$, and not in $A_2$}\,;\\
A_4 &= \mbox{Wall faces of $\cJ \setminus \cP_x^\cJ$}\,.
\end{align*}
We fix some ordering of the faces of $\cJ$ (say, lexicographical), and visit them one by one. Whenever we visit a face $f \in \cJ$, we consider all the faces $g$ which are 1-connected faces to $f$, in $\partial \cJ$, not yet in the domain of $T$, and have not yet been added into $B$ as the first face of a pair $(g, h)$:
\begin{enumerate}[1.]
\item If $f \in A_1$, then call the corresponding ceiling face in $\cI$ by $f'$. $f'$ is a vertical shift of $f$, so define $T(g) = h$ where $h$ is the same vertical shift applied to $g$. Necessarily, $h \in \partial \cI$. Note that $h$ also cannot yet have been in the image of $T$, since that would require the spine $\cS_x^\cI$ to be 1-connected to $f'$ or above $f'$, both of which are impossible if $f'$ is a ceiling face of $\cI$.

\item If $f \in A_2$, then we can use the vertical translation method from \cite[Lem.~15]{GielisGrimmett02}. There must exist some face $f' \in \cI \setminus \cS_x^\cI$ that is a vertical shift of $f$ (i.e., that $\rho(f') = \rho(f)$); pick one arbitrarily. By \cref{lem: wall-ceiling-geometry}, $g$ must be a vertical face either above or below $f$. If it is above $f$, define $h_0$ to be the face above $f'$ such that $\rho(h_0) = \rho(g)$. If $h_0 \not\in \cI \setminus \cS_x^\cI$, then set $h = h_0$. Otherwise, shift $h_0$ up by 1 to get $h_1$, and repeat until we have $h_n \notin \cI \setminus \cS_x^\cI$. Set $h = h_n$. (If $g$ was below $f$, we can instead shift $h_i$ down by 1 to get $h_{i+1}$.) If $h \in \partial \cI$ and $h$ is not yet in the image of $T$, set $T(g) = h$. Otherwise, add the pair of faces $(g, h)$ to $B$. (It is possible that $h$ is actually a hair of $\cS_x^\cI$, so that $h \in \cI$ even though $h \notin \cI \setminus \cS_x^\cI$).

\item First note that in $A_3$, the choice to take $\overline{\cS_x^\cI}$ as opposed to just $\cS_x^\cI$ is a technicality, because we defined walls on the semi-extended interface whereas the spine was just defined as a subset of the interface. For $f \in A_3$, note that there is a ceiling face $f'$ of $\cI \setminus \cS_x^\cI$ that has the same projection as $f$. (If there were instead a wall face of $\cI \setminus \cS_x^\cI$ with the same projection, then either the wall is deleted in $\Phi_\Iso$ or not, in which case $f$ should actually be in $A_2$ or $A_4$ respectively). As $f'$ is a vertical shift of $f$, let $h$ denote the face that is the same vertical shift applied to $g$. Suppose that $\rho(v_1) \in \rho(f)$. This is the one special case where because of how $\cI \setminus \cS_x^\cI$ was defined, then $f'$ might not be in $\cI$. There are however at most $C_0$ faces of $\partial \cJ$ that are 1-connected to this $f$, and so henceforth we will ignore them. Otherwise, $f' \in \cI$, and $h \notin \cI \setminus \cS_x^\cI$. If $h \in \partial \cI$ and $h$ is not yet in the image of $T$, set $T(g) = h$. Otherwise, add the pair of faces $(g, h)$ to $B$. (It is possible that $h$ is actually a hair of $\cS_x^\cI$, so that $h \in \cI$ even though $h \notin \cI \setminus \cS_x^\cI$).

\item Finally, for $f \in A_4$, every wall in $\cJ \setminus \cP_x^\cJ$ has a vertically shifted copy in $\cI \setminus \cS_x^\cI$ that is part of an undeleted wall. Let $h$ denote the face that is the same vertical shift applied to $g$.  If $h \in \partial \cI$ and $h$ is not yet in the image of $T$, set $T(g) = h$. Otherwise, add the pair of faces $(g, h)$ to $B$. We comment here for what follows that if $f\in A_4$, it cannot be that $\rho(f)$ is 0-connected with the projection of a deleted wall $\rho(\tilde{W})$ from $\cI \setminus \cS_x^\cI$, as otherwise by \cref{lem: wall-ceiling-geometry}, the vertically shifted copy of $f$ in $\cI \setminus \cS_x^\cI$ must actually be part of $\tilde{W}$, and therefore cannot be part of an undeleted wall.
\end{enumerate}
\end{enumerate}

Note that $T$ is still injective since in \cref{st-4-control-1-conn-faces}
we always checked that $h$ was not in the image of $T$ before assigning $T(g) = h$ (except for when $f \in A_1$, but simply because it is unnecessary to check as noted there). To control the size of $B$, we now show that within \cref{st-4-control-1-conn-faces}, every $h$ that was added in a pair to $B$ or added to the image of $T$ is unique. Indeed, every pairing of $g$ with $h$ was via a vertical shift. Thus, if there is overlap it must be that the starting faces $g_1$ and $g_2$ have the same projection. Following the notation of the steps above, suppose $g_1$ was connected to $f_1$, and $g_2$ to $f_2$. There are corresponding faces $f_1', f_2'$ in $\cI$ such that $f_1' = \theta_1 f_1$ and $f_2' = \theta_2 f_2$ for some vertical shifts $\theta_1, \theta_2$. Suppose $f_1, f_2$ are both in $A_1 \cup A_3 \cup A_4$. By how $T$ was defined there, we had $h_i = \theta_i g_i$. So, the only way we can pair the same $h$ to both $g_1, g_2$ is if $\theta_1$ and $\theta_2$ are different shifts, which implies that there must be a deleted wall of $\cI\setminus \cS_x^\cI$ separating $f_1'$ and $f_2'$. But by definition of the sets $A_1 \cup A_3 \cup A_4$ (and the comment above regarding $f \in A_4$), this is impossible.

On the other hand, if both $f_1, f_2 \in A_2$, then since they are both ceiling faces, by \cref{lem: wall-ceiling-geometry} (ii), it is only possible for $\rho(g_1) = \rho(g_2)$ when $f_1$ and $f_2$ are 1-connected and $g_1, g_2$ are attached to the common edge $f_1 \cap f_2$. But in this case, whichever face of $f_1, f_2$ was visited second will not do anything with $g_1, g_2$. Since we only used the property that $f_1, f_2$ are ceiling faces, the same logic applies if $f_1 \in A_1 \cup A_3$ and $f_2 \in A_2$.

Finally, suppose $f_1 \in A_4$, $f_2 \in A_2$. By \cref{lem: wall-ceiling-geometry} (ii), $g_2$ must be a vertical face. But this forces $\rho(f_1)$ to be 1-connected to $\rho(f_2)$, which cannot happen by the comment above regarding $f \in A_4$.

Thus, every pair $(g, h)$ added to $B$ must be such that either $h$ was already in the image of $T$ after \cref{st-2-control-1-conn-faces} or \cref{st-3-control-1-conn-faces}, or $h$ was part of a hair in $\cS_x^\cI$. However, at each point in \cref{st-4-control-1-conn-faces}, $h$ was always constructed as some face that is 1-connected to $\cI \setminus \cS_x^\cI$.
By (A2) of \cref{alg:Iso}, if there is a wall or interior ceiling of $\cI \setminus \cS_x^\cI$ that is distance 1 away from an increment $\sX_i$, then $i \leq j^*$. Combined, if $h$ was added to $B$ as part of a pair $(g, h)$ in \cref{st-4-control-1-conn-faces}, then $h$ is either part of or 1-connected to an increment with index $i \leq j^*$. Thus it suffices to show that $|\sF(\bigcup_{i \leq j^*} \sX_i)|\leq C\fm(\cI;\cJ)$ for some constant $C$. But by combining \cref{eq:incr-excess-area} with the upper bound on $j^*$ in \cref{clm:m-lower-bnd-in-column}, we have
\begin{equation*}
    \Big|\sF\Big(\bigcup_{i \leq j^*} \sX_i\Big)\Big| \leq \sum_{i \leq j^*} 5\fm(\sX_i) + 8j^* \leq C\fm(\cI;\cJ)\,.
\end{equation*}
(The constant above may depend on $L$, but that is not a problem.)
\end{proof}

\begin{lemma}\label{lem:control-interface-clusters}
Suppose that we have two interfaces $\cI \notin \Iso_{x, L, h}$ and $\cJ = \Phi_{\Iso}(\cI)$. Then, we have $\kappa_\cI - \kappa_\cJ \leq C\fm(I;J)$ for some constant $C$.
\end{lemma}
\begin{proof}
The proof of the exponential tails on groups of walls in \cite[Lem.~15]{GielisGrimmett02} already controls the difference in the number of open clusters resulting from deleting walls, and so we have the bound 
\begin{equation*}
\kappa_{\cI \setminus \cS_x^\cI} - \kappa_{\cJ \setminus \cP_x^\cJ} \leq \sum_{z \in \bD} 28\fm(\tilde{W}_z) \leq C\fm(I;J)\,,
\end{equation*}
where $\bD$ are the indices of all deleted walls in \cref{alg:Iso}.

Now, let $\kappa_{> j^*}$ be the number of open clusters which are separated from $\partial \Lambda_n$ by the portion of $\cS_x^\cI$ consisting of increments starting from index $j^* + 1$. Then, 
\begin{equation*}\kappa_\cJ - \kappa_{\cJ \setminus \cP_x^\cJ} = \kappa_{> j^*}\,.
\end{equation*}
On the other hand, if $\kappa_{\leq j^*}$ is defined analogously, then 
\begin{equation*}
\kappa_\cI - \kappa_{\cI \setminus \cS_x^\cI} \leq \kappa_{> j^*} + \kappa_{\leq j^*} + 1
\end{equation*}
(where the extra plus one is because it is possible for the joining together of the two parts of the spine to create an extra open cluster). Thus, it suffices to bound $\kappa_{\leq j^*}$ in terms of the excess area of the increments. However, the addition of a single face can add at most one cluster, and we can bound the number of faces in $\bigcup_{i = 1}^{j^*} \sX_i$ using \cref{eq:incr-excess-area} by
\begin{equation*}
\sum_{i = 1}^{j^*} |\sF(\sX_i)| - 4(j^* - 1) \leq \sum_{i = 1}^{j^*} 5\fm(\sX_i) + 4 \leq 5\fm(I;J) + 4 \leq 9\fm(I;J)\,.\qedhere
\end{equation*}
\end{proof}

\begin{proposition}
There exists $C > 0$ such that for all $\beta > \beta_0$, all $L$ large, and every $\cI$, $\Phi_\Iso(\cI) = \cJ$,
\begin{equation*}
\frac{\bar{\mu}_n(\cI)}{\bar{\mu}_n(\cJ)} \leq e^{-(\beta-CL^3) \fm(\cI;\cJ)}\,.
\end{equation*}
\end{proposition}
\begin{proof}
    An appropriate bound on the first two terms in the cluster expansion follow from the above two lemmas. See the proof of \cite[Proposition 4.10]{GL_DMP} for how to control the remaining $g$-terms.
\end{proof}

\begin{proposition}
There exists $C>0$ such that for all $L$ large, $M \geq 1$, $0 \leq h' \leq h$, and $\cJ \in E_x^{h'}$
\begin{equation*}
|\{\cI \in \Phi_\Iso^{-1}(\cJ): \fm(\cI;\cJ) = M\}| \leq C^{L^3M}\,.
\end{equation*}
\end{proposition}
\begin{proof}
    See the proof of \cite[Proposition 4.11]{GL_DMP}.
\end{proof}

\begin{proof}[\textbf{\emph{Proof of \cref{thm:iso-map}}}]
 For any $\cI \notin \Iso_{x, L, h}$, one has $\fm(\cI; \Phi_\Iso(\cI)) \geq 1$; thus, it suffices to prove the stronger statement that for some $C$ and any $r \geq 1$,
\begin{align*}
\mu_{n}^{\mp} \big(\fm(\cI; \Phi_{\Iso}(\cI)) \geq r\mid E_x^{h'}\big) \leq C\exp\big[-  (\beta- CL^3) r)\big]\,
\end{align*} 
  and take $L = L_\beta = \beta^{1/4}$, say.  
For every $r\ge 1$, 
\begin{align*}
    \bar{\mu}_n(\fm(\cI; \Phi_{\Iso}(\cI)) \geq r,  E_x^{h'}) &\leq
    \sum_{M\geq r}\,\, \sum_{\substack{\cI\in E_x^{h'}\\ \fm(\cI;\Phi_{\Iso}(\cI)) = M}} \bar{\mu}_n (\cI) \\
    & \leq \sum_{M\geq r} \sum_{\substack{\cI \in E_x^{h'}\\ \fm(\cI;\Phi_\Iso(\cI)) = M}} e^{ - (\beta- CL^3)M} \bar{\mu}_n(\Phi_{\Iso}(\cI))  \\
    & = \sum_{M\geq r}\,\, \sum_{\cJ\in \Phi_{\Iso}(E_x^{h'})}  \bar{\mu}_n(\cJ) \,\, \sum_{\substack{\cI\in \Phi_{\Iso}^{-1}(\cJ) \\ \fm(\cI; \Phi_{\Iso}(\cI))=M}} e^{-(\beta - CL^3 )M} \\
    & \leq \sum_{M\geq r} C^{L^3M} e^{- (\beta- C L^3)M} \mu^\mp_n(E_x^{h'})\,,
\end{align*}
where in the last line we use that $\Phi_{\Iso}(E_x^{h'}) \subseteq (E_x^{h'})$.
Dividing through by $\bar{\mu}_n(E_x^{h'})$ then yields the desired conditional bound.  \end{proof}

Next we prove that we have control over the size of increments at a given height by another map argument. We note that following the procedure in \cite[Proposition 4.1]{GL_max}  would work, but utilizing the cone separation properties of $\Iso_{x, L, h}$ greatly simplifies the proof.

\begin{definition}
Fix any $L$ and integer height $0 \leq h_0 < h$. Let $\cP_x$ be a pillar with height at least $h$. Suppose that the first increment in $\cP_x$ to have height $> h_0$ has index $j_0$. Then, we will say that $\cI \in \Incr_{x, L, h_0}$ if its pillar satisfies
\[ \fm(\sX_j) \leq \begin{cases}
0 & \text{if } 0 \leq j - j_0 \leq L\\
j - j_0 & \text{if } j - j_0 > L
\end{cases}\,.\]
\end{definition}
(Note that $j_0$ is defined so that the first increment which is guaranteed to be trivial has its two vertices at heights $h_0 - 1/2$ and $h_0 + 1/2$.)

\begin{theorem}\label{thm:incr-map}
For $\beta > \beta_0$ and $L$ sufficiently large, there exists constants $L'_\beta$ and $\epsilon_\beta$ (going to $\infty$ and 0 respectively as $\beta \to \infty$) such that for every $0 \leq L' \leq L'_\beta$ and all $h_0 \leq h' \leq h$,
\begin{equation*}
    \bar{\mu_n}(\cI \in \Incr_{x, L', h_0}| \hgt(\cP_x) \geq h', \Iso_{x, L, h}) \geq 1 - \epsilon_\beta
\end{equation*}
\end{theorem}

\begin{remark}\label{rem:incr-map}
    We can also define the map $\Incr_{x, L, j_0}$ by specifying directly the increment we want to trivialize, instead of specifying a height that we want to ensure a trivial increment to be at. We will still have
    \begin{equation*}
    \bar{\mu_n}(\cI \in \Incr_{x, L', j_0}| \hgt(\cP_x) \geq h', \Iso_{x, L, h}) \geq 1 - \epsilon_\beta
\end{equation*}
in the same setting as above, and the proof will follow in the same way.
\end{remark}

 We first show that the map $\Phi_\Incr$ is well-defined on $\Iso_{x, L, h}$ and yields an interface in $\Incr_{x, L', h_0}$. Note that we really need the starting interface to be in $\Iso_{x, L, h}$, otherwise the new pillar generated in Step~3 of \cref{alg:Incr} might intersect with existing walls of the interface. 

\begin{lemma}\label{lem:well-defined-incr}
For every $L, L'$ large, $h_0 \leq h' \leq h$, and $\cI \in E_{h'}^x \cap \Iso_{x, L, h}$, the image $\cJ:= \Phi_\Incr(\cI)$ is a well-defined interface in $E_{h'}^x \cap \Iso_{x, L, h} \cap \Incr_{x, L', h_0}$.
\end{lemma}

\begin{proof}
    This follows by \cref{prop:cone-seperation}. Since the only change to the interface is in the pillar, it suffices to show that the new pillar generated in Step~3 does not intersect the rest of the interface $\cI$. By \cref{prop:cone-seperation}, the rest of $\cI$ is a subset of $\bF_{-} \cup \bF_{\curlyvee}\cup \bF_{\ext}$. On the other hand, Step~2 ensures that the pillar being generated satisfies \cref{it:iso-spine} in the definition of an isolated pillar, and thus is a subset of $\bF_{\triangledown} \cup \bF_{\parallel}$. The disjointness then follows by the same proposition.
\end{proof}

\begin{figure}
\begin{algorithm2e}[H]
\DontPrintSemicolon
If $\cI \in \Incr_{x, L', h_0}$, then set $\Phi_\Incr(\cI) = \cI$. Otherwise, proceed as follows. Let $j_0$ be the index of the first increment of $\cS_x$ that reaches a height $> h_0$.\\
\nl
Let $(\sX_i)_{i\geq 1}$ be the increments of $\cS_{x}$.
\;

\nl
\For{$j=j_0$ \KwTo $\sT+1$}{     
     \If(\hfill){
\quad$ \fm(\sX_j) \geq \begin{cases} 0 & \mbox{if }0\le j-j_0\le L' \\ j-j_0-1 & \mbox{if }j-j_0> L'\end{cases}$ \quad\mbox{}}{Let $\fs \gets j$.} 
}
 Let $j^* \gets \fs$.\;

\BlankLine

\nl Let
\[ \cS^* \gets 
\big(X_1,\ldots,X_{j_0 -1},\underbrace{X_\trivincr,\ldots,X_\trivincr}_{\hgt(v_{j^*+1})-\hgt(v_{j_0})},\sX_{j^*+1},\ldots,\sX_{\sT},\sX_{>\sT}\big)\,.
 \]
\nl Obtain $\Phi_{\Incr}(\cI)$ by replacing the spine $\cS_x$ with $\cS^*$.\;

\caption{The map $\Phi_{\Incr}= \Phi_{\Incr}(x,L',h_0)$}
\label{alg:Incr}
\end{algorithm2e}
\end{figure}

We can split up any interface $\cI \in \Iso_{x, L, h}$ as follows. Let $j^*$ be as defined from \cref{alg:Incr}.
{\renewcommand{\arraystretch}{1.3}
\[
\begin{tabular}{m{0.05\textwidth}m{0.3\textwidth}m{0.5\textwidth}}
\toprule
\midrule
 $\bX_B^\cI$ & $\bigcup_{j \geq j^*+1} \sF(\sX_j)$ & Increments above $v_{j^*+1}$  \\ 
 $\bX_{A}^{\cI}$ & $\bigcup_{j_0\leq j \leq j^*}\sF(\sX_j)$ & Increments between $v_{j_0}$ and $v_{j^*+1}$ \\
 $\bX_C$ & $\bigcup_{j \leq j_0-1}\sF(\sX_j)$ & Increments below $v_{j_0}$ \\
 $\bB$ & $\cI\setminus(\cS_x^\cI)$
 & The remaining set of faces in $\cI$ \\
 \bottomrule
 \end{tabular}
\]
}
Define $\Phi_\Incr$ as in \cref{alg:Incr}. We can split up the faces of $\cJ = \Phi_\Incr(\cI)$ as follows:
{\renewcommand{\arraystretch}{1.3}
\[
\begin{tabular}{m{0.05\textwidth}m{0.3\textwidth}m{0.5\textwidth}}
\toprule
\midrule
 $\bX_B^\cJ$ & & Horizontally shifted copy of $\bX_B^\cI$\\ 
 $\bX_{A}^{\cJ}$ & & Trivial increments at the same height as $\bX_A^\cI$ \\
 $\bX_C$ & & Same set of faces as in $\cI$ \\
 $\bB$ & & Same set of faces as in $\cI$ \\
 \bottomrule
 \end{tabular}
\]
}
\begin{claim}\label{clm:bound-j*}
Let $\cJ = \Phi_\Incr(\cI)$ for $\cI \in E_x^{h'} \cap \Iso_{x, L, h}$. Then, there exists a constant $C > 0$ such that
\begin{equation}
|\bX_A^\cI \cup \bX_A^\cJ| \leq CL'\fm(\cI;\cJ)\,.
\end{equation}
\end{claim}
\begin{proof}
It suffices to bound $|\bX_A^\cI|$ since clearly $|\bX_A^\cJ| \leq |\bX_A^\cI|$. The number of faces of $\bX_A^\cI$ is \begin{equation*}
|\bX_A^\cI| = \sum_{j_0}^{j^*} |\sF(\sX_j)| - 4(j^* - j_0) \leq \sum_{j_0}^{j^*} 5\fm(\sX_j) + 4(j^* - j_0) +8\,.
\end{equation*}
Thus, it suffices to bound $j^* - j_0$, and by \cref{alg:Incr}, either $j^* - j_0 \leq L'$, or $j^* - j_0 \leq \fm(\sX_{j^*}) + 1$ 
\end{proof}

\begin{proposition}\label{prop:incr-energy}
There exists $C > 0$ such that for all $\beta > \beta_0$, and every $\cI \in E_x^{h'} \cap \Iso_{x, L, h}$, $\Phi_\Incr(\cI) = \cJ$,
\begin{equation}
 \frac{\bar{\mu_n}(\cI)}{\bar{\mu}_n(\cJ)} \leq e^{-(\beta-CL') \fm(\cI;\cJ)}\,.
\end{equation}
\end{proposition}
\begin{proof}
Using the cluster expansion, we have
\begin{equation}
\frac{\bar{\mu_n}(\cI)}{\bar{\mu}_n(\cJ)} = (1-e^{-\beta})^{|\partial \cI| - |\partial \cJ|}e^{-\beta \fm(\cI;\cJ)}q^{\kappa_\cI - \kappa_\cJ}\exp(\sum_{f\in \cI} \g(f, \cI) - \sum_{f \in \cJ} \g(f, \cJ))\,.
\end{equation}
To account for the faces in $\partial \cJ$ and $\partial \cI$, we follow the proof of \cref{lem:control-1-connected-faces} and define an injective map $T$ on a subset of $\partial \cJ$ to $\partial \cI$ and show that the number of faces we do not define $T$ on is bounded by $C\fm(\cI;\cJ)$ for some $C$. Faces which are 1-connected to $B \cup \bX_C$ can be mapped to themselves, and faces 1-connected to $\bX_B^\cJ$ can be mapped to their shifted copy in $\bX_B^\cI$. The remaining faces 1-connected to $\bX_A^\cJ$ can be handled by following the procedure in \cref{st-3-control-1-conn-faces} of \cref{lem:control-1-connected-faces}, noting there that the bound on the number of faces where $T$ was not defined was actually a constant times the sum of the excess areas of the increments being trivialized, which in this case is precisely $C\fm(\cI;\cJ)$, and so $|\partial \cJ|-|\partial \cI| \leq C\fm(\cI;\cJ)$.

A bound on $\kappa_\cI - \kappa_\cJ$ also follows as in \cref{lem:control-interface-clusters}. The difference in the number of open clusters between the two interfaces is bounded by the number of open clusters in $\bX_A^\cI + 2$ (where the +2 comes from potentially creating an extra open cluster when joining to $B \cup \bX_C$ below and/or to $\bX_B^\cI$ above). However, the addition of a single face can add at most one cluster, whence \cref{clm:bound-j*} gives us the bound $\kappa_\cI - \kappa_\cJ \leq CL'\fm(\cI;\cJ)$. 

Finally, we can decompose the sum of $g$-terms as
\begin{equation}
\sum_{f \in \bX_A^\cI} |\g(f, \cI)| + \sum_{f \in \bX_A^\cJ} |\g(f, \cJ)| + \sum_{f \in \bX_B^\cI} |\g(f, \cI) - \g(\theta f, \cJ)| + \sum_{f \in B \cup \bX_C} |\g(f, \cI) - \g(f, \cJ)|\,.
\end{equation}
The first two sums are bounded by $CL'\fm(\cI;\cJ)$ by \cref{clm:bound-j*} (for a different constant than in claim, but a constant nonetheless). 

For the latter two sums, we separate the analysis into cases according to which face $g \in \cI \cup \cJ$ attains the distance $r(f, \cI; \theta f, \cJ)$:
\begin{enumerate}[(i)]
    \item If $g \in \bX_A^\cI \cup \bX_A^\cJ$, then by summability of exponential tails and \cref{clm:bound-j*}, we have
\begin{equation*}
\sum_{f \in \sF(\Z^3)} \sum_{g \in \bX_A^\cI \cup \bX_A^\cJ} e^{-cd(f, g)} \leq C|\bX_A^\cI \cup \bX_A^\cJ| \leq CL'\fm(\cI;\cJ)\,,
\end{equation*}
which covers both sums.

    \item If $g \in \bX_B^\cI \cup \bX_B^\cJ$, we only need to check for the sum over $f \in B \cup \bX_C$, since every face in $\bX_B^\cJ$ is the same horizontal shift of a face in $\bX_B^\cI$. For $f \in B$, the sum is bounded by \cref{eq:pillar-in-cone-2}, since both $\cI$ and $\cJ$ are in $\Iso_{x, L, h}$. For $f \in \bX_C$, we have using summability of exponential tails, \cref{eq:incr-excess-area}, and \cref{alg:Incr},
\begin{align*}
\sum_{g \in \bX_B^\cI \cup \bX_B^\cJ}\;\sum_{f \in \bX_C} e^{-cd(f, g)} &\leq \sum_{j > j^*}\;\sum_{\substack{f \in \sF(\Z^3)\\ \hgt(f)\leq \hgt(v_{j_0})}} \max_{g \in \sX_j} |\sF(\sX_j)|e^{-cd(f, g)}\\
&\leq \sum_{j > j^*} |\sF(\sX_j)| e^{-c(j - j_0)}\\
&\leq \sum_{j > j^*}C(j - j_0)e^{-c(j - j_0)} \leq C\,.
\end{align*}

    \item If $g \in B \cup \bX_C$, we only need to consider the sum over $f \in \bX_B^\cI$, but then this is the same as case (ii) above with the roles of $f$ and $g$ reversed. \qedhere
\end{enumerate}
\end{proof}

\begin{proposition}\label{prop:incr-entropy}
There exists $C>0$ such that for all $M \geq 1$, $L', h', h$ as in the setting of \cref{thm:incr-map},  and $\cJ \in E_x^{h'} \cap \Iso_{x, L, h}$,
\begin{equation}
|\{\cI \in \Phi_\Incr^{-1}(\cJ): \fm(\cI;\cJ) = M\}| \leq C^{L'M}\,.
\end{equation}
\end{proposition}
\begin{proof}
We follow the proof of \cite[Lemma 7.9]{GL_DMP}, with the witness being the faces of $\bX_A^\cI$ together with the height of $v_{j_0}$. Indeed, suppose we are given such a witness with an interface $\cJ$. Then, to reconstruct $\cI$, we first take $\cJ$ and delete the portion of the pillar $\cP_x^\cJ$ with height $> \hgt(v_{j_0})$, and append $\bX_A^\cI$ to the pillar in such a way that the bottom four faces of $\bX_A^\cI$ around $v_{j_0}$ match the four faces around the cut-point of $\cP_x^\cJ$ that has height $\hgt(v_{j_0})$. Then, we append $\bX_B^\cJ$ (the portion of $\cP_x^\cJ$ with height $ \geq \hgt(v_{j^*+1}$), which can be read off from $\bX_A^\cI$ and $\hgt(v_{j_0})$) to the top of $\bX_A^\cI$, joining again at the respective cut-points.

Now, we already know that for any fixed $M$, by \cref{clm:bound-j*} $\bX_A^\cI$ is a 1-connected face set of size $\leq CL'M$. So, by \cref{lem:num-of-0-connected-sets} the number of possible face sets for $\bX_A^\cI$ is bounded by $s^{CL'M}$. Furthermore, we know that $\hgt(v_{j_0}) \in [h_0 - M - 1/2, h_0 - 1/2]$ since the excess area $\fm(\sX_{j_0})$ is at least $\hgt(v_{j_0 + 1}) - \hgt(v_{j_0}) - 1$, and so this leaves $M+1$ possible choices for what $\hgt(v_{j_0})$ can be. Thus, the number of possible witnesses is bounded by $(M+1)s^{CL'M}$.
\end{proof}

\begin{proof}[\textbf{\emph{Proof of \cref{thm:incr-map}}}] For any $\cI \notin \Incr_{x, L', h_0}$, $\fm(\cI; \Phi_\Incr(\cI)) \geq 1$, so it suffices to prove the stronger statement that for some $C$ and any $r \geq 1$,
\begin{equation}\label{eq:incr-precise-bound}
\bar{\mu_n}(\fm(\cI;\Phi_\Incr(\cI)) \geq r| \hgt(\cP_x) \geq h', \Iso_{x, L, h}) \leq C\exp{-(\beta - CL')r}
\end{equation}
and then take $L' = L'_\beta = \beta^{3/4}$ and $r = 1$. Indeed,
\begin{align*}
\bar{\mu_n}(\fm(\cI;\Phi_\Incr(\cI)) \geq r, \hgt(\cP_x) \geq h', \Iso_{x, L, h}) &= \sum_{M \geq r}\,\, \sum_{\substack{\cI \in E_x^{h'}\cap \Iso_{x, L, h}, \\\fm(\cI;\Phi_\Incr(\cI)) = M }}\bar{\mu}_n(\cI) \\
&\leq \sum_{M \geq r}\,\, \sum_{\cJ \in \Phi_\Incr(E_x^{h'} \cap \Iso_{x, L, h})}\,\,
\sum_{\cI \in \Phi_\Incr^{-1}, \fm(\cI;\cJ) = M} e^{-(\beta - CL')M} \bar{\mu}_n(\cJ)\\
&\leq \sum_{M \geq r}C^{L'M}e^{-(\beta - CL')M}\bar{\mu}_n(\Phi_\Incr(E_x^{h'} \cap \Iso_{x, L, h}))\\
&\leq Ce^{-(\beta - CL' - L'\log C)r}\bar{\mu}_n( E_x^{h'} \cap \Iso_{x, L, h})\,.
\end{align*}
Hence, dividing by $\bar{\mu}_n(E_x^{h'} \cap \Iso_{x, L, h})$ proves the claim. The above inequalities follow from \cref{prop:incr-energy}, \cref{prop:incr-entropy}, and the fact that $\Phi_\Incr(E_x^{h'} \cap \Iso_{x, L, h}) \subseteq E_x^{h'} \cap \Iso_{x, L, h}$.
\end{proof}

\begin{remark}\label{rem:incr-map-variations}
Note that the proof above still works if we condition on any subset $A \subseteq \Iso_{x, L, h} \cap E_x^{h'}$ that satisfies the property $\Phi_\Incr(A) \subseteq A$. In particular, this allows us to apply the map multiple times to ensure trivial increments at multiple locations. 
%We can also replace $\Iso_{x, L, h}$ by $\Iso_{x, L, h}^\mathrm{o}$, since the only property we used of the isolated pillar space was in the application of \cref{lem:cone-separation} which holds on both sets. (Alternatively, this follows by the definition of conditional expectation and \cref{thm:iso-map}).
\end{remark}

\section{Large deviation rate for random-cluster interfaces}\label{sec:ld-RC}

In this section, we come to the first large deviation result, which concerns the height of the $\Top$ interface $\Itop$ at a particular location. Throughout this section, we will focus on three heights $h_1,h_2$ and $h=h_1+h_2$, with the goal of proving the following proposition.

\begin{proposition}\label{prop:RC-rate}
For all $\beta > \beta_0$, every sequence of $n, x$ dependent on $h$  with $1 \ll h \ll n$ and $d(x, \partial \Lambda_n) \gg h$, and every $h = h_1 + h_2$,
\begin{equation}\label{eq:submultiplicativity-Ehx}
    \bar{\mu}_n(E_h^x) \leq (1+\epsilon_\beta)(e^\beta + q-1)^3 \bar{\mu}_n(E_{h_1}^x)\bar{\mu}_n(E_{h_2}^x)\,,
\end{equation}
 and consequently,
\begin{equation}\label{eq:alpha-rate}
    \lim_{h \to \infty} - \frac{1}{h}\log \bar{\mu}_n(\hgt(\cP_x) \geq h) = \alpha
\end{equation}
for some constant $\alpha$.
\end{proposition}

We first want to introduce a proxy event $A_h^x$ that is comparable to $E_h^x$ but is not defined with respect to an interface.

\begin{definition}\label{def:A_h^x}
Define $A_h^x$ to be the event that a certain set of faces are in $\clfaces$. Specifically, let $C$ be any finite connected set of vertices with the following conditions:
\begin{enumerate}
    \item \label{it:Ahx-cut-point-at-0} $C$ contains $x$, and this is the only vertex of $C$ with height 1/2;

    \item \label{it:Ahx-vertices-between-slabs} the vertices of $C$ have heights in $[1/2, h-1/2]$;

    \item \label{it:Ahx-simply-connected} $C$ is co-connected.
\end{enumerate}
Now, let $F(C)$ be the set of faces that form the side and top boundary of $C$. That is, if $u \in C$ and $v \notin C$ such that $u$ is adjacent to $v$, then we add the face $f_{[u, v]}$ to $F(C)$, except we do not add the face $f_{[x, x - \ez]}$ at the bottom. $A_h^x$ is the event that there is some such $C$ such that $F(C)\subseteq \clfaces$. 
\end{definition}

A crucial property of $A_h^x$ is that it is decreasing. Also important is the geometrical fact that any such bounding set of faces $F(C)$ is 1-connected (see \cite[Prop.~5]{GielisGrimmett02},\cite[Thm.~7.3]{Grimmett_RC}, noting that in our case because $C$ is connected and co-connected, the splitting set there is precisely the set of faces that separate $C$ from $C^c$, and removing the face $f_{[x, x - \ez]}$ to get $F(C)$ keeps $F(C)$ 1-connected).

Since we are including the faces bounding the top side of $C$ in the definition of $A_h^x$, the faces $F(C)$ form a shell that looks like a pillar in $E_h^x$ whose vertex set is capped at height $h$. This leads to the following definition, which will also appear throughout the rest of the paper:
\begin{definition}
For every $h \geq 1$, let $\widetilde E_h^x \subseteq (E_h^x \setminus E_{h+1}^x)$ be the set of pillars of height $h$ such that there are no faces of $\cP_x$ with height $\geq h$ except those forming the top boundary of vertices of $\cP_x$. 
\end{definition}

We state here the following fact that a $1 - \epsilon_\beta$ fraction of pillars in $E_h^x$ are actually in $\widetilde E_h^x$, but we defer the proof until \cref{lem:nice-space-likely} where we prove the stronger statement required there.
\begin{corollary}\label{cor:widehat-Ehx}
    For every $\beta > \beta_0$ and $h \geq 1$, there exists a constant $\epsilon_\beta$ such that
    \begin{equation*}
        \bar{\mu}_n(\widetilde E_h^x \mid E_h^x) \geq 1- \epsilon_\beta\,.
    \end{equation*}
\end{corollary}

The following proposition states that $\bar{\mu}_n(A_h^x)$ is comparable to $\bar{\mu}_n(E_h^x)$ (up to multiplicative constants depending on $\beta$). 
\begin{proposition}\label{prop:compare-A_h^x-E_h^x} In the setting of \cref{prop:RC-rate}, there exists a constant $\epsilon_\beta$ such that
\begin{equation*}
    \frac{q}{e^\beta+q-1}(1 - \epsilon_\beta)\bar{\mu}_n(A_h^x) \leq \bar{\mu}_n(E_h^x) \leq (1+\epsilon_\beta)\bar{\mu}_n(A_h^x)\,.
\end{equation*}
\end{proposition}

\begin{proof}
Beginning with the upper bound, we have (say, for $L = L_\beta$), 
\begin{equation*}
    \bar{\mu}_n(\widetilde{E}_h^x,\, \Iso_{x, L, h}) \leq \bar{\mu}_n(A_h^x, E_h^x)\,.
\end{equation*}
Indeed, if we have a pillar $\cP_x \in \widetilde E_h^x$, we can take $C$ in the definition of $A_h^x$ to be the set of vertices in the pillar. Recall that the vertices of $\cP_x$ form a co-connected set by \cref{obs:pillar-vert-simp-conn}, so this satisfies \cref{it:Ahx-simply-connected}, and the definition of $\widetilde E_h^x$ implies the height requirement of \cref{it:Ahx-vertices-between-slabs}. Then, the $\Iso_{L, h, x}$ event implies \cref{it:Ahx-cut-point-at-0} above. Furthermore, each face in $F(C)$ must be in $\clfaces$ because it separates a vertex in $\Atop^c$ from a vertex in $\Atop$. Thus, by \cref{thm:iso-map,cor:widehat-Ehx}, we have
\begin{equation*}
\bar{\mu}_n(\widetilde{E}_h^x,\, \Iso_{L, h, x}) \geq \bar{\mu}_n(E_h^x)(1 - \epsilon_\beta)\,.
\end{equation*}
Combining the above gives the following stronger statement which implies the upper bound
\begin{equation}\label{eq:stronger-comparison}
    \bar{\mu}_n(A_h^x \mid E_h^x) \geq 1 - \epsilon_\beta\,.
\end{equation}

For the lower bound, as a technical step, we want to first close the edge $[x, x - \ez]$ (this will be needed for an application of the Domain Markov property). By \cref{obs:close-edge}, we can close this edge at a cost of $\frac{e^\beta+q-1}{q}$, noting that closing this edge always creates a new open cluster in separating $x$ from $x - \ez$. We will call $\tilde A_h^x$ the event $A_h^x \cap \{f_{[x, x - \ez]} \in \clfaces\}$, so that we have
\begin{equation*}
\bar{\mu}_n(A_h^x) \leq \frac{e^\beta+q-1}{q}\bar{\mu}_n(\tilde A_h^x)\,.
\end{equation*}
We can split the event $\tilde  A_h^x$ based off whether or not the pillar at $x$ has height $\geq 0$ or $< 0$. We first show that $\tilde A_h^x \cap E_0^x \subseteq E_h^x$. Indeed, the event $E_1^x$ implies that $x$ is in $\Atop^c$, and $\tilde A_h^x \cap \{\hgt(\cP_x) = 0\}$ is empty since the presence of the faces $F(C)$ together with the face below $x$ make it impossible for $x$ to have a wired path to the upper half boundary, which is a contradiction (see \cref{rem:height-0-pillar}). Once we have established that $x \in \Atop^c$, then all of $C$ must also be in $\Atop^c$ since it is part of the same connected component of $\Vtop^c$ as $x$. Thus, the vertices of the pillar $\cP_x$ must contain all the vertices of $C$, which notably includes at least one vertex at height $h - 1/2$, so that the pillar has height at least $h$.

Thus, it suffices to show that
\begin{equation}\label{eq:Ahx-and-int-dig-below}
    \bar{\mu}_n(\tilde A_h^x, (E_0^x)^c) \leq \epsilon_\beta\bar{\mu}_n(E_h^x)\,.
\end{equation}
Now, for a given $\Top$ interface $I_{\Top}$, consider the set of vertices $v$ such that there exists $w$ with $f_{[v, w]} \in I_{\Top}$. Of these, let $V_1$ be the ones in $\Atop$ and $V_2$ be the ones in $\Atop^c$. With the notation $\Itop = I_\Top$ meaning the $\Top$ interface of the configuration $\omega$ is equal to the set of faces $I_\Top$, we claim that we can write 

\begin{flalign}
\bar{\mu}_n(\tilde A_h^x,\, (E_0^x)^c) & =  \sum_{I_{\Top} \in (E_0^x)^c} \bar{\mu}_n(\tilde A_h^x\mid\Itop = I_{\Top})\bar{\mu}_n(\Itop = I_{\Top})\nonumber \\
&= \sum_{I_{\Top} \in (E_0^x)^c} \mu_n(\tilde A_h^x\mid V_1(I_{\Top}) \subseteq \Vtop(\omega),\, V_2(I_{\Top}) \subseteq \Vtop^c(\omega))\bar{\mu}_n(\Itop = I_{\Top})\nonumber \\
&= \sum_{I_{\Top} \in (E_0^x)^c} \mu_n(\tilde A_h^x\mid V_1(I_{\Top}) \subseteq \Vtop(\omega),\, I_{\Top} \subseteq \clfaces)\bar{\mu}_n(\Itop = I_{\Top})\label{eq:V1-V2-describe-Itop}\,.
\end{flalign}

To justify the second line above, we need to prove that for any configuration $\omega$, the $\Top$ interface being a specified $I_{\Top}$ is equivalent to the event  $\{V_1(I_{\Top}) \subseteq \Vtop(\omega),\, V_2(I_{\Top}) \subseteq \Vtop^c(\omega)\}$. The forward implication is true as we already showed in \cref{rem:properties-of-Itop} that for every face of $I_{\Top}$, one of the adjacent vertices is in $\Vtop(\omega)$ and the other is not. For the reverse implication, the same remark showed that we can let $\Atop(I_{\Top})$ be the augmented $\Top$ component corresponding to $I_{\Top}$, and it suffices to show that $\Atop(\omega) = \Atop(I_{\Top})$. If $v \in \Atop(I_{\Top})$, then every path from $v$ to $\partial \Lambda_n^-$ must pass through a face of $I_{\Top}$, and hence must include a vertex of $V_1$. Since $V_1$ is part of $\Vtop(\omega)$, then $v$ cannot be in the infinite component of $\Vtop^c(\omega)$, so $v \in \Atop(\omega)$. This shows $\Atop(I_{\Top}) \subseteq \Atop(\omega)$. We note (for later use) that in the proof of this direction, we only used the fact that $V_1(I_\Top) \subseteq \Vtop(\omega)$. For the converse, if $v \in \Atop^c(I_{\Top})$, then every path from $v$ to $\partial \Lambda_n^+$ must pass through a face of $I_{\Top}$, and thus must include a vertex of $V_2$. Thus, $v$ cannot be in $\Vtop(\omega)$, and we have a partial converse $\Atop^c(I_{\Top}) \subseteq \Vtop^c(\omega)$. We need to rule out the possibility of $v$ being in a finite component of $\Vtop^c(\omega)$, say $A$. But such an $A$ by definition must be surrounded by vertices of $\Vtop(\omega)$, which by the partial converse, are in $\Atop(I_{\Top})$. Thus, by assumption we have $v \in \Atop^c(I_{\Top})$, yet $v$ is surrounded by vertices of $\Atop(I_{\Top})$, which contradicts the fact that $\Atop^c(I_{\Top})$ is connected (see \cref{rem:properties-of-Itop}).

Furthermore, the third line holds because the event  $\{V_1(I_{\Top}) \subseteq \Vtop(\omega),\, V_2(I_{\Top}) \subseteq \Vtop^c(\omega)\}$ is equal to the event $\{V_1(I_{\Top}) \subseteq \Vtop(\omega),\, I_{\Top} \subseteq \clfaces\}$. Indeed, conditional on $\{V_1(I_{\Top}) \subseteq \Vtop(\omega)\}$, the event $\{I_{\Top} \subseteq \clfaces\}$ is sufficient to show $\{V_2(I_{\Top}) \subseteq \Vtop^c(\omega)\}$ because every path from $V_2(I_{\Top})$ to $\partial \Lambda_n^+$ must cross a face of $I_{\Top}$, and it is necessary because otherwise there would be an open edge between some $u \in V_2(I_{\Top})$ and $v \in V_1(I_{\Top})$, which would imply that $u \in \Vtop(\omega)$.

Next we will argue that by the Domain Markov property, for every $I_\Top \in (E_0^x)^c$, we have 
\begin{equation}\label{eq:DMP-in-compare-A-E}
\mu_n(\tilde A_h^x\mid V_1(I_{\Top}) \subseteq \Vtop(\omega),\, I_{\Top} \subseteq \clfaces) = \mu_n(\tilde A_h^x\mid V_1(I_{\Top}) \subseteq \Vtop(\omega))\,.
\end{equation}
To begin, observe that for any $v \in \Atop(I_{\Top})$, every path from $v$ to $\Atop^c(I_{\Top})$ must pass through a vertex of $V_1(I_\Top)$, so that $V_1(I_\Top) \cup \partial \Lambda_n^+$ forms a vertex boundary of $\Atop(I_{\Top})$. Furthermore, conditioning on $V_1(I_{\Top}) \subseteq \Vtop(\omega)$ guarantees that the vertices $V_1 \cup \partial \Lambda_n^+$ are all part of the same open cluster. So, if $(\Atop(I_\Top), E)$ is the induced subgraph of $\Lambda_n$ on $\Atop(I_{\Top})$, it remains to show that conditional on $V_1(I_{\Top}) \subseteq \Vtop(\omega)$, the event $\tilde A_h^x$ only depends on $\omega_e$ for $e \in E$. Recall that $\tilde A_h^x$ is the event that there exists some finite $\Lambda_n$-connected set of vertices $C$ fulfilling the conditions of \cref{def:A_h^x}, such that its bounding faces $\tilde F(C)$ (including $f_{[x, x - \ez]}$ now) are all in $\clfaces$. We will argue that for \emph{any} finite $\Lambda_n$-connected set of vertices~$C$ containing $x$ such that its bounding faces $\tilde F(C) \subseteq \clfaces$, we have
\[\{e:\; f_e \in \tilde F(C)\} \subseteq E\,.\]
We first argue that $C$ must be a subset of $\Atop(I_\Top)$. Indeed, if $\tilde F(C) \subseteq \clfaces$, then $C$ cannot contain any vertices of $\Vtop(\omega)$, and in particular $C \cap V_1(I_\Top) = \emptyset$. But since $C$ is a $\Lambda_n$-connected and $V_1(I_\Top)$ is a vertex boundary for $\Atop(I_\Top)$, then $C$ must lie entirely in either $\Atop(I_\Top)$ or $\Atop^c(I_\Top)$. As $C$ contains $x$ (which must be in $\Atop(I_\Top)$ since $I_\Top \in (E_0^x)^c$), then $C \subseteq \Atop(I_\Top)$. Now 
suppose for contradiction that there is some face $f= f_{[u, v]} \in \tilde F(C)$ where $u \in \Atop^c(I_{\Top})$. Then, the fact that $C \subseteq \Atop(I_\Top)$ implies that not only $v \in C$, but also $f_{[u, v]} \in I_\Top$. Combined, this implies that $v \in V_1(I_\Top)$, and hence on the event $\{V_1(I_\Top) \subseteq \Atop(\omega)\}$, we have $v \in \Atop(\omega)$. But, this is impossible when $\tilde F(C) \subseteq \clfaces$, a contradiction. This concludes the proof of \cref{eq:DMP-in-compare-A-E}, which we can now plug into \cref{eq:V1-V2-describe-Itop}.

Finally, since $\tilde A_h^x$ is a decreasing event, we can use FKG followed by the rigidity of the $\Top$ interface to conclude that
\begin{flalign*}
\sum_{I_\Top \in (E_0^x)^c} \mu_n(\tilde A_h^x\mid V_1(I) \subseteq \Vtop(\omega))\bar{\mu}_n(\Itop = I_{\Top}) &\leq \sum_{I \in (E_0^x)^c} \mu_n(\tilde A_h^x)\bar{\mu}_n(\Itop = I_{\Top}) \\
&=\mu_n(\tilde 
A_h^x)\bar{\mu}_n((E_0^x)^c) \\
& \leq \epsilon_\beta \mu_n(\tilde A_h^x)\,.
\end{flalign*}

Thus, combining the above, we get
\begin{equation}\label{eq:Ahx-and-interface-dig-below-sharper}
    \bar{\mu}_n(\tilde A_h^x,\, (E_0^x)^c) \leq \epsilon_\beta \mu_n(\tilde A_h^x).
\end{equation}

Finally, a short computation using FKG gets us that
\begin{equation*}
\bar{\mu}_n(E_h^x) \geq \frac{\mu_n(E_0^x, \tilde A_h^x, \sep_n)}{\mu_n(\sep_n)} \geq \frac{\mu_n(E_0^x, \sep_n)}{\mu_n(\sep_n)}\mu_n(\tilde A_h^x) \geq (1-\epsilon_\beta)\mu_n(\tilde A_h^x),
\end{equation*}
which together with \cref{eq:Ahx-and-interface-dig-below-sharper} concludes the proof of \cref{eq:Ahx-and-int-dig-below}, and hence the proposition.
\end{proof}

\begin{definition} Let $\sH$ be the 1-connected set of faces of $\cL_{>0} \cap \clfaces$ that contains the faces on the four sides of $x$. Let $\sH_1$ be the restriction of $\sH$ to faces in $\cL_{> 0} \cap \cL_{\leq h_1}$.
\end{definition}

If we are on the event $A_h^x$ for any $h$, note that since $F(C)$ is 1-connected, then $\sH$ must include all the faces of $F(C)$. We next define an event $\Gamma_{h_1}^x$, to be thought of as a subset of the configurations where $A_{h_1+1}^x$ is achieved, except possibly up to the final face at height $h_1+1$, yet via a sufficiently ``nice'' pillar (with cut-points at height $\frac12$ and $h_1\pm\frac12$) making it easier to implement a submultiplicativity argument on the event $A_{h_1 + h_2}^x$.

\begin{definition}\label{def:Gamma-event}
Let $ \Gamma_{h_1}^x$ be the subset of configurations where
\begin{enumerate}[(1)]
    \item \label{it:gamma-cut-point-x} $\sH$ has a ``cut-point" at $x$, in that $\sH \cap \cL_{1/2}$ consists of only the four faces surrounding the sides of~$x$.
    \item \label{it:gamma-cut-points-y} $\sH$ has ``cut-points" at heights $h_1 + 1/2$ and $h_1 - 1/2$ for some $y$ and $y - \ez$ resp., in the sense of \cref{it:gamma-cut-point-x}. Furthermore, we ask that $\sH \cap \cL_{h_1}$ has no faces, except possibly the horizontal face $f_{[y, y - \ez]}$.
    \item \label{it:gamma-connect-to-UHB-LHB} For each of the four vertices $z_i$ adjacent to $y$ at height $h_1 + 1/2$, and each of the four $w_i$ which are adjacent to $x$ at height $1/2$, we require that $z_i, w_i \in \Vtop$. We also require that $x \in \Vbot$.
    \item \label{it:gamma-xy-close} $d(x, y - h_1\ez) \leq d(x, \partial \Lambda_n)/2$.
\end{enumerate}
\end{definition}

\begin{remark}\label{rem:properties-of-Gamma} Suppose  that our configuration $\omega $ satisfies $A_h^x \cap \Gamma_{h_1}^x\cap\sep_n$ for some $h > h_1$. We claim that the corresponding interface $\cI$ satisfies $\sH \subseteq \cI$, and furthermore, there is no $\omega' \notin A_h^x \cap \Gamma_{h_1}^x\cap\sep_n$ that could have the same interface $\cI$.
To see this, we first argue that the requirement $x \in \Vbot$ in \cref{it:gamma-connect-to-UHB-LHB} implies that for any set of vertices $C$ satisfying the definition of $A_h^x$, we have $C \subseteq \Atop^c$, and in particular, $y \in \Atop^c$. Indeed, any open path from $v \in C$ to $\partial \Lambda_n^+$ must pass through $x$ because of the faces $F(C) \subseteq \clfaces$, so that $C \subseteq \Vtop^c$. Since $C$ is connected, all the vertices of $C$ are in the same infinite component of $\Vtop^c$ as $x$, and so $C \subseteq \Atop^c$. The property that $\sH \subseteq \cI$ then readily follows: indeed,  the fact that $z_i \in \Vtop$ implies that $f_{[y, z_i]} \in \Itop$ since these faces are separating $y \in \Atop^c$ from $z_i \in \Atop$. This together with the fact that $\sH$ is a 1-connected component of faces in $\clfaces$ implies that $\sH \subseteq \cI$. To rule out the existence of $\omega'\notin A_h^x \cap\Gamma_{h_1}^x\cap\sep_n$ with the same interface $\cI$, argue as follows. First, $\sep_n$ is trivially satisfied by $\omega'$. Second, the event $A_h^x$ was satisfied via a subset of the faces $\sH$, all of which are in $\cI$, and hence is also satisfied by $\omega'$. Third, to confirm the event $\Gamma_{h_1}^x$, we note that \cref{it:gamma-cut-point-x,it:gamma-cut-points-y,it:gamma-xy-close} are satisfied via the same $\sH\subseteq\cI$, and it remains to check that $\cI$ determines \cref{it:gamma-connect-to-UHB-LHB}. As shown in \cref{rem:properties-of-Itop}, $\Itop$ determines $\Atop$, and so $\cI$ will already guarantee that $z_i, w_i \in \Atop$ and $x \in \Atop^c$. So, it suffices to show that $\cI$ will also determine whether $x, z_i, w_i$ are in finite components or not. This is true because if $z_i$ is part of a finite component $A$, then the set of faces $F$ which separate $A$ from the infinite component of $\Lambda_n \setminus A$ is 1-connected (see \cite[Prop.~5]{GielisGrimmett02},\cite[Thm.~7.3]{Grimmett_RC} for a proof). Since $\Atop$ and $\Atop^c$ are both connected (also shown in \cref{rem:properties-of-Itop}), this set $F$ includes $f_{[y, z_i]}$, whence $F\subseteq \cI$. The same argument applies for $x, w_i$.
\end{remark}

\begin{remark}\label{rem:shifted-A}
Let $\theta_{h_1}A_{h_2}^x$ be the event $\{\theta_{h_1}\omega: \omega \in A_{h_2}^x\}$, where $\theta_{h_1}\omega$ is the configuration that results from shifting all the edges of $\omega$ up by $h_1$. Then, by the way we defined $\Gamma_{h_1}^x$, we have that $\Gamma_{h_1}^x \cap A_h^x$ implies the event $\theta_{h_1}A_{h_2}^{y - h_1\ez}$.
\end{remark}

We are now ready to begin the proof of the submultiplicativity statement in \cref{eq:submultiplicativity-Ehx}. We already showed in \cref{prop:compare-A_h^x-E_h^x} that we can move from $E_h^x$ to $A_h^x$ by paying a cost of $(1+\epsilon_\beta)$, and we next show how a slight modification of the proof there allows us to further move onto the nicer space $\Gamma_{h_1}^x$:

\begin{lemma}\label{lem:move-to-gamma}
For all $\beta > \beta_0$, every sequence of $n, x$ dependent on $h$  with $1 \ll h \ll n$ and $d(x, \partial \Lambda_n) \gg h$, and every $h = h_1 + h_2$, we have
\begin{equation*}
\bar{\mu}_n(E_h^x) \leq 
(1+\epsilon_\beta)\bar{\mu}_n(A_h^x, \Gamma_{h_1}^x) \,.
\end{equation*}
\end{lemma}

\begin{proof}
Because of the prior map arguments (see \cref{thm:iso-map,thm:incr-map,cor:widehat-Ehx}), it suffices to show that $\widetilde{E}_h^x \cap \Iso_{x, L, h} \cap \Incr_{x, 1, h_1}$ implies $\Gamma_{h_1}^x \cap A_h^x$ (for say, $L = L_\beta$). We have already proved the implication of $A_h^x$ in the first part of \cref{prop:compare-A_h^x-E_h^x}, so we need to check that we have all the items of $\Gamma_{h_1}^x$.

On $\Iso_{x, L, h}$, the four faces surrounding $x$ at height $1/2$ are a part of $\cP_x$. As $\sH$ is 1-connected, this implies that $\sH \subseteq \cI$. But in \cref{cor:pillar-doesnt-touch-other-walls}, we proved that on $\Iso_{x, L, h}$, the pillar $\cP_x$ is only connected to the rest of the interface via faces at height 0. Since $\sH \subseteq \cL_{>0}$, this implies that $\sH \subseteq \cP_x$. Thus, we have \cref{it:gamma-cut-point-x} of $\Gamma_{h_1}^x$ because of the cut-point in the pillar at $x$. If we are additionally on $\Incr_{x, 1, h_1}$, then we have \cref{it:gamma-cut-points-y} because the pillar is just a trivial increment there. To show \cref{it:gamma-connect-to-UHB-LHB}, note first that $z_i, w_i \in \Atop$, as if those vertices were in $\Atop^c$, then they would also be part of the pillar which would violate the cut-point condition imposed by $\Iso_{x, L, h}$ and $\Incr_{x, 1, h_1}$. Then, the fact that $x, y \in \Atop^c$ implies that the faces $f_{[y, z_i]}$ and $f_{[x, w_i]}$ are in $\Itop$, whence we conclude that $z_i, w_i \in \Vtop$ as at least one of the vertices adjacent to a face of $\Itop$ is in $\Vtop$ by \cref{rem:properties-of-Itop}. We already proved that $x \in \Vbot$ in \cref{clm:x-in-Vbot}. Finally, we have \cref{it:gamma-xy-close} by the fact that the pillar lies in a cone (see \cref{prop:cone-seperation}) and the assumption that $d(x, \partial \Lambda_n) \gg h$.
\end{proof}

Before we continue, we record here some definitions and geometrical statements from \cite{Grimmett_RC}, which will be useful in justifying the Domain Markov argument used in proving \cref{lem:submultiplicativity-in-Gamma}.

\begin{definition}
Let $H$ be any 1-connected set of faces, and $G$ a component of the lattice with the edges corresponding to $H$ removed (so $G$ is a subgraph of the lattice). Define $\Delta_{\scV, H}G$ to be the set of all vertices $v \in G$ such that there exists another vertex $w$ with $f_{[v, w]} \in \overline{H}$. Define $\Delta_{\scE, H}G$ to be the set of edges $e \in G$ such that $f(e) \in \overline{H} \setminus H$.
\end{definition}

\begin{definition}
For a set $T$ in $\R^d$, let $\out(T)$ denote the union of the unbounded connected components of $\R^d \setminus T$. When $H$ in the above definition is a finite 1-connected set of faces, then there is a unique component $G$ which lies in $\out(H)$. As a shorthand in notation, we write $\Delta_\scV H$ and $\Delta_\scE H$ when using this choice of $G$.
\end{definition}

\begin{proposition}[{\cite[Thm.~7.6, special case]{Grimmett_RC}}]\label{prop:out-graph-connected}  Let H be a finite 1-connected set of faces, corresponding to an edge set $D$. Let $G = (V, E)$ be the subgraph of $(\Z^3, \E^3 \setminus D)$ comprising of all vertices and edges in $\out(H)$. Then, the graph $(\Delta_{\scV} H, \Delta_{\scE} H)$ is connected.
\end{proposition}

We also prove a useful lemma regarding height shifts:
\begin{lemma}\label{lem:height-shift}
Let $A$ be an event measurable with respect to the configuration $\omega$ restricted to edges of $\Lambda_n$ whose height is in $[-L_0,L_0]$, for some $L_0>0$ which may depend on $n$.
For any $k$, we have 
\begin{equation*}
    \mu_n(\theta_k A) \leq q^2\mu_n(A)\,.
\end{equation*}
\end{lemma}
\begin{proof}
Let $\Lambda_n^{[a,b]}=\Lambda_n \cap \{\Z^2 \times [a,b]\}$ denote the finite cylinder confined between the heights $a<0$ and~$b>0$. Let~$\mu_n^{[a,b]}$ be the FK measure with Dobrushin boundary conditions (still about height $0$, i.e., the boundary configuration $\eta$ has $\eta_e = 0$ if $e = [x, y]$ for some $x = (x_1, x_2, \frac12)$ and $y = (y_1, y_2, \frac12)$, and $\eta_e = 1$ otherwise). Let $\Omega_n^{[a,b]}$ be the set of configurations for the FK model on $\Lambda_n^{[a,b]}$, and for any event $A \subseteq \Omega_n^{[a, b]}$, let $\theta_k A \subseteq \Omega_n^{[a+k,b+k]}$ be the event $\{\theta_k \omega\,:\; \omega \in A\}$, where $\theta_k \omega$ is the configuration obtained by shifting every edge of $\omega$ up by $k$. Noting $\mu_n^{[-L,L]}\to\mu_n$ as $L\to\infty$, we will compare $\mu_n^{[-L+k,L+k]}(\theta_k A)$ to $\mu_n^{[-L,L]}(A)$ for $L > L_0$.

Recall that the weight of a configuration $\omega$ is given by  $p^{\#\{e\in E\,:\;\omega_e = 1\}}(1-p)^{\#\{e\in E\,:\;\omega_e = 0\}}q^{\kappa(\omega)}$. Comparing the weight of $\omega$ and $\theta_k \omega$, since they have the same number of open/closed edges, a change in weight can only come from a change in the number of open clusters via interactions with the boundary (since the open clusters that do not touch the boundary are preserved by the height shift). However, every vertex that is connected to the boundary  $\partial \Lambda_n^{[-L,L]}$ via open edges will still be connected to $\partial \Lambda_n^{[-L+k, L+k]}$ after the height shift. Hence, because of the boundary conditions, the only possible variable in the number of open clusters is whether the two wired boundary components above and below height zero are joined via open edges of $\omega$ (or $\theta_k \omega$). So, the number of clusters can change by at most 1. Thus, if $Z_n^{[a,b]}$ denotes the partition function of $\mu_n^{[a,b]}$ then
\begin{equation*}
    Z_n^{[-L, L]} \leq q Z_n^{[-L+k, L+k]}\,,
\end{equation*}
and similarly the weight of $\theta_k \omega$ can increase by at most $q$, which together give
\begin{equation*}
    \mu_n^{[-L+k, L+k]}(\theta_h A) \leq q^2 \mu_n^{[-L, L]}(A)\,.
\end{equation*}
The proof is concluded by taking $L \to \infty$, yielding this inequality for $\mu_n$.
\end{proof}

With \cref{lem:move-to-gamma} and the above results in hand, we next prove the following inequality, which is arguably the most delicate part of this paper.

\begin{lemma}\label{lem:submultiplicativity-in-Gamma}
For all $\beta > \beta_0$, every sequence of $n, x$ dependent on $h$  with $1 \ll h \ll n$ and $d(x, \partial \Lambda_n) \gg h$, and every $h = h_1 + h_2$, we have
\begin{equation*}
\bar{\mu}_n(A_h^x,\, \Gamma_{h_1}^x) \leq (1+\epsilon_\beta) q(e^\beta + q-1)^2\bar{\mu}_n(A_{h_1}^x)\bar{\mu}_n(E_{h_2}^x)\,.
\end{equation*}
\end{lemma}

\begin{remark}
    The goal is to analyze the increasing and decreasing information gained by climbing up to height $h_1$ (i.e. the event $A_{h_1}^x$) with respect to climbing from height $h_1$ to $h_1 + h_2$. We recall here the proof idea of \cite[Proposition 5.1]{GL_tightness}, which is the Ising analog of our claim here. The idea in that paper was that upon revealing the plus component connecting $x$ to $\cL_{h_1}$, there is revealed a minus boundary all along the sides of the plus component so that by Domain Markov, it is equivalent to revealing just the minus boundary and the plus spins at the top and bottom. However, the $\Gamma_{h_1}^x$ event ensures that there will only be one plus spin at the top and another at the bottom, so that these spins can be disregarded at a constant cost. Then, the conditioning on the minus spins can be removed by FKG.

    We would like to follow this proof, but some difficulties stand in the way. The primary issue is that our ``minus spins'' are vertices in $\Vtop$, yet whether or not a vertex is in $\Vtop$ is not something that can be determined locally, so revealing a set of vertices is not suitable for a Domain Markov proof. Instead, we reveal the dual faces that fulfill the event $A_{h_1}^x$, along with components of faces in $\clfaces \cap \cL_{>0} \cap \cL_{\leq h_1}$ which are 1-connected to them (namely, $\sH_1$). By maximality, this reveals a side boundary of open edges. We would like to also use Domain Markov to forget the closed edges revealed and only remember the boundary of open edges, so that we can use FKG. However, to utilize the FKG property of the random-cluster measure, we need to move off our conditioned space $\sep_n$. This requires us to additionally reveal not only the faces described above, but also the entire interface. However, we can not reveal the faces fulfilling $\theta_{h_1}A_{h_2}$, which on $\Gamma_{h_1}^x$ are a part of the interface, so this step needs to be treated more delicately. Furthermore, since the object we are revealing is not a component of vertices but of dual faces, the geometry is more complicated and one needs to be more careful when applying the Domain Markov step. 
    
    Finally, we note that the fact that $A_h^x$ is a decreasing event is critical for this proof to work because of the usage of FKG. This is the reason that we are starting with the $\Top$ interface, as opposed to the analogously defined $\Bot$ interface. Roughly speaking, for the $\Top$ interface to rise up requires the existence of faces forming a shell of $\Atop^c$ vertices, while for the $\Bot$ interface to rise up requires the existence of an open path of vertices to penetrate upwards. The former as we have seen can be compared to a decreasing event, while the latter is very much an increasing event. 
\end{remark}

\begin{proof}
We first sum over all possible sets of faces that can make up $\sH_1$ on the event $\Gamma_{h_1}^x$. Let $y$ be as in \cref{def:Gamma-event}, i.e., $y$ is the unique vertex at height $h_1 + 1/2$ that has sides bounded by faces of $\sH$. We can write
\begin{equation}
\bar{\mu}_n(A_h^x, \Gamma_{h_1}^x)= \frac{1}{\mu_n(\sep_n)}\sum_{H_{1}} \mu_n(\sH_1=H_1 ,\, A_h^x,\, \Gamma_{h_1}^x,\, \sep_n)\,.
\end{equation}
To sum over interfaces, we define
\begin{align*} \sep_n^1(H_{1}) &= \left\{ I= \cI(\omega)\mbox{ for some }\omega\in A_h^x \cap \sep_n \cap \Gamma_{h_1}^x \mbox{ and } \sH_1=H_1 \right\}\,.
\end{align*}
We can then write 
\begin{equation}\label{eq:separate-by-interface}
    \frac{1}{\mu_n(\sep_n)}\sum_{H_{1}} \mu_n(\sH_1=H_1 ,\, A_h^x,\, \Gamma_{h_1}^x,\, \sep_n) = \frac{1}{\mu_n(\sep_n)}\sum_{H_{1}}
    \sum_{I \in \sep_n^1(H_{1})}\mu_n(\cI = I)\,,
\end{equation}
where we really have an equality because we proved (in \cref{rem:properties-of-Gamma}) that no $\omega'\notin A_h^x \cap \Gamma_{h_1}^x \cap \sep_n$ can lead to an interface $\cI\in\sep_n^1(H_1)$.

For every $\omega \in \sep_n^1(H_1)$, closing the edge $[y, y - \ez]$ always creates an additional open cluster because of the cut-point condition in \cref{it:gamma-cut-points-y} of $\Gamma_{h_1}^x$. Moreover, the resulting configuration is always still in $\sep_n^1(H_1)$, as the only non-trivial thing to check is \cref{it:gamma-connect-to-UHB-LHB} of $\Gamma_{h_1}^x$, and this property is unaffected by closing the edge $[y, y - \ez]$ because we proved in \cref{rem:properties-of-Gamma} that both $y, y - \ez$ are in $\Atop^c$ for this choice of $\omega$. Thus, we can force the face below $y$ to be in $\clfaces$ at a cost of $\frac{e^\beta+q-1}{q}$ by \cref{obs:close-edge}. So, defining
\[ \hat{\sep}_n^1(H_1) =\left\{ I\in\sep_n^1(H_1)\,:\; f_{[y,y-\ez]}\in I\right\}\,,\]
we get that
\[\mu_n(\sep_n^1(H_1)) \leq \frac{e^\beta+q-1}{q} \mu_n(\hat\sep_n^1(H_1))\,.\]

We want to reveal only the portion of the interface below the face $f_{[y, y-\ez]}$, so for every interface $\cI \in \hat\sep_n^1$, we define its truncation $\cI'$ as the set of faces that are in $\cI$ minus the faces of $\sH \setminus \sH_1$. The purpose of adding the face $f_{[y, y - \ez]}$ to the definition of $\hat\sep_n^1$ is threefold: It guarantees that $\cI'$ is still an interface so that we are still in $\sep_n$, it acts as a ``top boundary'' so that together with the faces $\sH_1$, we are in $A_{h_1}^x$, and it brings us into a situation where we can apply Domain Markov property. (Note the importance of \cref{it:gamma-cut-points-y} in $\Gamma_{h_1}^x$ for the first point --- it is a priori possible that the face set $\sH \setminus \sH_1$ comes down and reconnects to the interface at several locations, so that deleting these faces creates an arbitrary number of gaps in the interface. The event $\Gamma_{h_1}^x$ makes this impossible, and ensures that the only place where the faces of $\sH \setminus \sH_1$ connects to the rest of $\cI$ is at the four faces to the sides of $y - \ez$ at height $h_1 - 1/2$. Thus, adding just a single face $f_{[y, y - \ez]}$ ensures that $\cI'$ is still an interface.) 

Now define $\partial^\dagger I'$ by deleting from $\partial I'$ the 4 faces that are 1-connected to the face $f_{[y, y - \ez]}$ (out of the 12 such faces) and have height $> h_1$. We would like to have $\partial I'$ capture all the faces that we know are not present in $\clfaces$ by the maximality of $I$; however, the four faces adjacent to $y$ are exceptional, in that we truncated $I'$ in the slab $\cL_{h_1+1/2}$ by choice. (In fact, on $\Gamma_{h_1}^x$ we know that those four faces actually are in $\clfaces$, so they definitely cannot be in $\partial^\dagger I'$.) By grouping the terms in the above sum \cref{eq:separate-by-interface} according to the truncated interface $I'$, and recalling that $A_h^x$ implies $\theta_{h_1}A_{h_2}^{y - h_1}$, we have an upper bound of
\begin{equation}\label{eq:separate-by-truncated-interface}
     \bar{\mu}_n(A_h^x, \Gamma_{h_1}^x) \leq \frac{e^\beta+q-1}{q}\frac{1}{\mu_n(\sep_n)}\sum_{H_{1}}\sum_{I': I \in \hat \sep_n^1(H_{1})} \mu_n(I' \subseteq \clfaces,\, \partial^\dagger I' \subseteq \opfaces ,\, \theta_{h_1}A_{h_2}^{y - h_1\ez})\,.
\end{equation}
(One might note that in moving from \cref{eq:separate-by-interface} to \cref{eq:separate-by-truncated-interface}, we are enlarging the set of interfaces we are summing over since it is possible for an interface $J$ that violates $\Gamma_{h_1}^x$ to still have truncation $I'$. This is not a problem because from now on we will only use the information from $\Gamma_{h_1}^x$ that is measurable with respect to the event $I' \subseteq \clfaces, \partial^\dagger I' \subseteq \opfaces$ and the fact that $I'$ came from a truncation of some $I \in \hat \sep_n^1(H_1)$, and we are only claiming an upper bound.)

Writing the latter probability as
\[ \mu_n(I' \subseteq \clfaces,\, \partial^\dagger I' \subseteq \opfaces ,\, \theta_{h_1}A_{h_2}^{y - h_1\ez})= \mu_n\left(\theta_{h_1}A_{h_2}^{y-h_1\ez} \mid \cS_{I'}\right) \mu_n\left(\cS_{I'}\right)
 \]
 for
 \begin{equation}
 \cS_{I'}:=\left\{\omega:\; I'\subset \clfaces\,,\, \partial^\dagger I'\subseteq \opfaces\right\}\,,
 \end{equation}
the next claim will establish that the events $\cS_{I'}$ are disjoint:

\begin{claim}\label{clm:S-I'-disjoint}
The events $\cS_{I'}$ are mutually disjoint across all $I' \in \hat\sep_n^1(H_1)$ and all possible sets of faces that can make up $H_1$.
\end{claim}
\begin{proof}
Consider two face sets $H_1$ and $\tilde H_1$ (possibly the same) such that each one makes up $\sH_1$ for some configuration on the event $\Gamma_{h_1}^x$. Suppose $I \in \hat\sep_n^1(H_1)$ and $J \in \hat\sep_n^1(\tilde H_1)$, with truncations $I', J'$ respectively. We need to show that, if $I'\neq J'$, then the events $\cS_{I'}$ and $\cS_{J'}$ are disjoint, i.e., that \[(I' \subseteq \clfaces) \cap (\partial^\dagger I' \subseteq \opfaces) \cap (J' \subseteq \clfaces) \cap (\partial^\dagger J' \subseteq \opfaces)=\emptyset\,.\] It suffices to exhibit a face in $I' \cap \partial^\dagger J'$ or $J' \cap \partial^\dagger I'$. 

Let us first define $H_1(I')$ by taking the 1-connected set of faces which are in $I'$ and have height $(0, h_1]$ that contains the four faces to the sides of $x$). By \cref{it:gamma-cut-points-y} of $\Gamma_{h_1}^x$, there can only be four faces of $H_1(I')$ which have height $h_1 - 1/2$, and they are all adjacent to a single vertex which we can call $y(I') - \ez$. The same applies to $H_1(J')$, leading to an analogously defined $y(J') - \ez$. 

\begin{enumerate}[label=\textbf{Case~\arabic*}:, ref=\arabic*, wide=0pt, itemsep=1ex]
\item \label[case]{case-1:S-I'-disjoint} %\noindent{\bf{Case 1:}
$y(I') - \ez = y(J') - \ez$. Since $I'\neq J'$, without loss of generality we may take $f \in I' \setminus J'$. As we know that $I' \cap J'\neq\emptyset$ (because they both must contain the four faces to the sides of $x$), we may take $g \in I' \cap J'$. Since both $I'$ and $J'$ are 1-connected and their intersection is nonempty, then $I' \cup J'$ is also 1-connected. Let $P = (f = f_1, \ldots, f_k = g)$ be a 1-connected path of faces in $I' \cup J'$. Let $f_{j+1}$ be the first face in $P$ that is in $J'$. Then, $f_j \in I' \cap \partial J'$. But, since $y(I') - \ez = y(J') - \ez$ by assumption, then $\partial I' \setminus \partial^\dagger I' = \partial J' \setminus \partial^\dagger J'$ (both are equal to the four faces surrounding $y(I') = y(J')$). So, $I' \cap \partial J' \setminus \partial^\dagger J' = \emptyset$, and $f_j \in I' \cap \partial^\dagger J'$.

\item \label[case]{case-2:S-I'-disjoint}  %\noindent{\bf{Case 2:} 
$y(I') - \ez \neq y(J') - \ez$. Here $H_1(I')$ can only have the four faces surrounding $y(I') - \ez$ at height $h_1 - 1/2$, and similarly for $H_1(J')$. Thus, we can let $f \in H_1(I') \setminus H_1(J')$. We have $H_1(I') \cap H_1(J') \neq \emptyset$ since both sets must contain the four faces to the sides of $x$. Let $g \in H_1(I') \cap H_1(J')$.  Since both $H_1(I')$ and $H_1(J')$ are 1-connected and their intersection is nonempty, then $H_1(I') \cup H_1(J')$ is also 1-connected. Let $P = (f = f_1, \ldots, f_k = g)$ be a 1-connected path of faces in $H_1(I') \cup H_1(J')$. Let $f_{j+1}$ be the first face in $P$ that is in $H_1(J')$. Then, $f_j \in H_1(I') \cap \partial H_1(J')$. We additionally know that $f_j \in \partial J'$ since if $f_j \in J'$, this would violate the maximality of $H_1(J')$ (because $f \in H_1(I')$ implies that $\hgt(f_j) \leq h_1$). Moreover, we have that $f_j \notin \partial J' \setminus \partial^\dagger J'$ because the faces of $\partial J' \setminus \partial^\dagger J'$ have height $h_1 + 1/2$. Thus, $f_j \in H_1(I') \cap \partial^\dagger J' \subseteq I' \cap \partial^\dagger J'$.
\end{enumerate}
This concludes the proof.
\end{proof}

Since every $\cS_{I'}$ for $I'\in\hat\sep_n^1(H_1)$ further implies $A_{h_1}^x$ and $\sep_n$, it follows from the above claim that
\[ \sum_{H_1}\sum_{I':I\in\hat\sep_n^1(H_1)} \mu_n(\cS_{I'})\leq \mu_n(\sep_n,\, A_{h_1}^x)\,,\]
and consequently (together with \cref{eq:separate-by-truncated-interface}):
\begin{equation}\label{eq:mu(Ah-Gamma)-up-to-max-Ah2}
\bar{\mu}_n(A_h^x, \Gamma_{h_1}^x) \leq 
\frac{e^\beta+q-1}{q}\bar{\mu}_n(A_{h_1}^x)\max_{H_1}\max_{I':I\in\hat\sep_n^1(H_1)}\mu_n\left(\theta_{h_1}A_{h_2}^{y-h_1\ez} \mid \cS_{I'}\right)\,.
\end{equation}
Hence, to conclude the proof it will suffice to show that for any admissible $H_1$ and $I'$ such that $ I \in \hat\sep_n^1(H_1)$, we have 
$\mu_n(\theta_{h_1}A_{h_2}^{y - h_1\ez} \mid \cS_{I'})) \leq C(\beta,q) \bar{\mu}_n(E_{h_2}^x)$; namely, we prove this for $C(\beta,q)=(1+\epsilon_\beta)q(e^\beta+q-1)$.

Our definition of $H_1$ and $I'$ was tailored to infer the following result.
\begin{lemma}\label{lem:Ah2-DMP}
For every admissible $H_1$ and $I'\in\hat\sep_n^1(H_1)$ we have
\begin{equation}\label{eq:Ah2-DMP}
\mu_n(\theta_{h_1}A_{h_2}^{y - h_1\ez}\mid I' \subseteq \clfaces,\,\partial^\dagger I' \subseteq \opfaces) = \mu_n(\theta_{h_1}A_{h_2}^{y - h_1\ez}\mid f_{[y, y - \ez]} \in \clfaces\,, \partial^\dagger I' \subseteq \opfaces)\,.
\end{equation}
\end{lemma}
This is a subtle point in the argument --- while Domain Markov applications are often straightforward in Ising and Potts models, here we are conditioning on a certain set of open edges in $\Z^3$ (the ones dual to $\partial^\dagger I'$), and wish to infer that they form a cut that separates every vertex lying ``above'' $I'$ from those ``below''~it.
More precisely, we would like to construct a set of edges separating a subdomain $G$ from $G^c$, so that the number of connected components in $G$ is unaffected by the edge configuration within $G^c$. The delicate definition of $I'$ was designed to have the edges dual to $\partial^\dagger I'$ serve that purpose, along with \cref{prop:out-graph-connected}. In what follows, we now condition on the event $\{ \partial^\dagger I' \subseteq \opfaces ,\,  f_{[y, y-\ez]} \in \clfaces\}$ for some $I'$ which was a truncation of an interface $I \in \hat\sep_n^1$, and we build such a set of separating edges.

We know by \cref{prop:out-graph-connected} that the subgraph $K=(\Delta_{\scV}I', \Delta_{\scE}I')$ is connected. (Note that this subgraph includes some vertices and edges that are not in $\Lambda_n$.) Now let $B_\scV$ be the vertices of $\Delta_{\scV}I' \cap \Lambda_n$ with a $\Lambda_n$-path to $\partial \Lambda_n^+$ that do not cross a face of $I'$, and let $B_\scE$ be the edges of the induced subgraph of $K$ on $B_\scV$.
\begin{claim}\label{clm:Bv-connected}
Let $I'\in\sep_n$ be any interface (not necessarily the truncation of $I\in\hat\sep_n^1$), and let $B_\scV$ as defined above. Then the induced subgraph of $K=(\Delta_\scV I',\Delta_\scE I')$ on $B_\scV$ is connected.
\end{claim}
\begin{proof}
Let $a, b$ be any two vertices in $B_\scV$, and let $P$ be a path connecting them in $K$. If the path uses only vertices of $B_\scV$, then there is nothing to prove. Otherwise, let $c, d$ be the first and last vertices of $P$, respectively, that are in $\Delta_{\scV} I' \setminus B_\scV$. Let $c^-$ be the vertex that comes right before $c$ in the path $P$, and $d^+$ the vertex that comes right after $d$, so that $c^-$ and $d^+$ are both in $B_\scV$. Consider the edge $e=[c^-,c]$; since $c^-\in B_\scV$, there is a $\Lambda_n$-path $P^-$ from it to $\partial \Lambda_n^+$ that does not cross any face of $I'$. We argue that this implies that $c\notin \Lambda_n$: indeed, if $c\in\Lambda_n$, then the fact that 
$c\notin B_{\scV}$ would imply that $f_e\in I'$ (otherwise the path $e\cup P^-$ would qualify $c$ to be included in $B_{\scV}$), and yet $e\in P \subseteq \Delta_{\scE} I'$ by construction, so in particular $f_e\in \partial I'$ (by definition of $\Delta_{\scE}I'$), which is disjoint to $I'$. By the same argument, $d\notin\Lambda_n$. Thus, $c^-, d^+ \in \partial \Lambda_n$. But $I'$ separates $\partial \Lambda_n^-$ from $\partial \Lambda_n^+$, so the fact that $c^-, d^+ \in B_{\scV}$ implies that $c^-, d^+ \in \partial \Lambda_n^+$. 

Now we furthermore prove that $\hgt(c^-)=\hgt(d^+)=\frac12$. Since $c \in \Delta_\scV I'$, $c$ is incident to some edge $e$ such that $f_e \in \overline{I'}$. $f_e$ must be 1-connected to some face $f_{e'} \in I'$, say that $e' = [u, v]$. In general, there are three possible ways that $c$ can be positioned with respect to $u, v$, pictured in \cref{fig:submultiplicativity-c-u-v}. 

\begin{figure}
    \centering

    \begin{tikzpicture}

    \begin{scope}[scale=1,shift={(0,0)}]
    \filldraw [fill=blue!20, draw=black,thick] (0,0)--(.6,0.2)--(.6,1.2)--(0,1.)--cycle;
    \filldraw [fill=blue!20, draw=black,thick] (0,1)--(.6,1.2)--(.6,2.2)--(0,2.)--cycle;
    \node[circle,scale=0.4,fill=gray,label={[label distance=2pt]left:{\tiny$u$}}] (u) at (-0.15,1.6) {};
    \node[circle,scale=0.4,fill=gray,label={[label distance=2pt]right:{\tiny$v$}}] (v) at (.75,1.6) {};
    \node[circle,scale=0.4,fill=gray,label={[label distance=2pt]right:{\tiny$c$}}] (c) at (.75,.6) {};    
    \draw (.3,1.6) node
    {\tiny$\diamond$};
    \draw [black, dashed] (u)--(v);
    \end{scope}
    %%%%%

    \begin{scope}[scale=1,shift={(3,.5)}]
    \filldraw [fill=blue!20, draw=black,thick] (0,0)--(.6,0.2)--(.6,1.2)--(0,1.)--cycle;
    \filldraw [fill=blue!20, draw=black,thick] (0.,1)--(.9,1.)--(1.5,1.2)--(0.6,1.2)--(0.,1);
    \node[circle,scale=0.4,fill=gray,label={[label distance=2pt]right:{\tiny$v$}}] (v) at (.75,1.6) {};
    \node[circle,scale=0.4,fill=gray,label={[label distance=2pt]right:{\tiny$c=u$}}] (c) at (.75,.6) {};    
    \draw (.75,1.1) node
    {\tiny$\diamond$};
    \draw [black, dashed] (v)--(c);
    \end{scope}

    %%%%%

    \begin{scope}[scale=1,shift={(7,0.5)}]
    \filldraw [fill=blue!20, draw=black,thick] (0,0)--(.6,0.2)--(.6,1.2)--(0,1.)--(0,0);
    \filldraw [fill=blue!20, draw=black,thick] (.6,1.2)--(-0.3,1.2)--(-0.9,1)--(0.,1.)--(.6,1.2);
    \node[circle,scale=0.4,fill=gray,label={[label distance=2pt]left:{\tiny$v$}}] (v) at (-0.15,1.6) {};
    \node[circle,scale=0.4,fill=gray,label={[label distance=2pt]left:{\tiny$u$}}] (u) at (-.15,.6) {};
    \node[circle,scale=0.4,fill=gray,label={[label distance=2pt]right:{\tiny$c$}}] (c) at (.75,.6) {};    
    \draw (-.15,1.1) node
    {\tiny$\diamond$};
    \draw (.3,0.6) node
    {\tiny$\diamond$};    
    \draw [black, dashed] (v)--(u)--(c);
    \end{scope}
    
    \end{tikzpicture}
    \caption{The three possible positions that $u, v$ can have with respect to $c$.}
    \label{fig:submultiplicativity-c-u-v}
\end{figure}

Regardless of which case we are in, the (Euclidean) distance between $u, v$ and $c$ is at most $\sqrt{2}$, and $c$ is $\Z^3$-adjacent to at least one of $u$ or $v$. However, the distance between $c$ and any vertex of $\Lambda_n \setminus \partial \Lambda_n$ is at least 2, which means that both $u, v \in \partial \Lambda_n$. The important observation is that the Dobrushin boundary conditions imply that the faces of $I'$ dual to an edge between two vertices of $\partial \Lambda_n$ are precisely the set of horizontal faces separating some $w \in \partial \Lambda_n^+$ from $w - \ez \in \partial \Lambda_n^-$, where $\hgt(w) = 1/2$. In our case, $[u,v]=[w,w-\ez]$, and as there is only one vertex adjacent to such a $w$ (or to $w - \ez$) that is also in $\Lambda_n^c$, and it has the same height as $w$ (or as $w - \ez$), we can conclude that $\hgt(c) = 1/2$ or $-1/2$. But conversely, there is only vertex adjacent to $c$ that is also in $\Lambda_n$, and it has the same height as $c$, so that $\hgt(c^-) = \hgt(c)$. But $c \in \partial \Lambda_n^+$, so it must be that $\hgt(c^-) = 1/2$, and the same argument implies that $\hgt(d^+) = 1/2$.

In fact, we claim that we can moreover infer that every vertex at height $1/2$ in $\partial \Lambda_n^+$ is in $B_\scV$, and that the edge between every two such adjacent vertices is in $B_\scE$. Indeed, all of $\partial \Lambda_n^+$ is in $\out(I')$, so that for any $u \in \partial \Lambda_n^+$ with $\hgt(u) = 1/2$, the fact that $f_{[u, u - \ez]} \in I'$ implies that $u \in B_\scV$. Moreover, if $u$ is adjacent to another vertex $w \in \partial \Lambda_n^+$ with $\hgt(w) = 1/2$, then the face $f_{[u, w]}$ is 1-connected to the face $f_{[u, u - \ez]}$. So, $f_{[u, w]} \in \overline{I'}$, but as observed above, the Dobrushin boundary conditions imply that $f_{[u, w]} \notin I'$, so $f_{[u, w]} \in \partial I'$ and $[u, w] \in B_\scE$. Now, the vertices of $\partial \Lambda_n^+$ with height $1/2$ are just the four sides of a square and are notably connected, so that $c^-$ and $d^+$ can be connected by a path $Q$ that only uses edges of $B_\scE$ by travelling along the sides of this height $1/2$ square. Thus, we can replace the portion of the path $P$ from $c^-$ to $d^+$ by the path $Q$, and we have thus exhibited a path from $a$ to $b$ using only edges of $B_\scE$, which proves that the induced subgraph of $K$ on $B_\scV$ is  connected.
\end{proof} 

We will now address the subgraph $G$ of $\Lambda_n$ induced on the set of vertices $V$ that are not disconnected from  $\partial \Lambda_n^+$ by $I'$ (to be thought of as the vertices that lie ``above'' $I'$). Note that $\out(I')$ does not (necessarily) contain all of $\Z^3$ because $I'$ is not a truncation of the $\Top$ interface $I_{\mathsf{top}}$, but a truncation of the decorated interface $I$, and can thus enclose some vertices. In fact, the property in $\Gamma_{h_1}^x$ that the side neighbors of $y$ are in $\Vtop$ is needed to guarantee that the subgraph $G \subseteq \out(I')$ is the right graph to be looking at for the event $\theta_{h_1}A_{h_2}^{y - h_1\ez}$, since otherwise it is possible that $I'$ encapsulates $y$ in a big bubble, and the next claim will establish that we are not in this case. For ease of reference, denote the four adjacent vertices to $y$ that have height $h_1 + 1/2$ as $z_1, z_2, z_3, z_4$.

\begin{claim}\label{clm:Ah2-E-measurable}
Let $I'$ be the truncation of some interface $I\in\hat\sep_n^1$. Let $G=(V,E)$ be the induced subgraph of $\Lambda_n$ on the vertices that are connected to $\partial \Lambda_n^+$ in $\Lambda_n\setminus \{e':f_{e'}\in I'\}$.
Then conditional on $\partial^\dagger I' \subseteq \opfaces$, the event $\theta_{h_1} A_{h_2}^{y-h_1\ez} $ is measurable w.r.t.\ $\{\omega_e : e\in E\}$.
\end{claim}
\begin{proof}
Recall from \cref{def:A_h^x} the event $\theta_{h_1}A_{h_2}^{y - h_1\ez}$ concerns the existence of a certain 1-connected set of faces $F \subseteq \clfaces \cap \cL_{>h_1}$ that includes $\{f_{[y, z_i]}\}_{i = 1}^4$.
We will argue that, for \emph{any} $1$-connected subset $F$ of $\clfaces\cap\cL_{>h_1}$ that includes $\{f_{[y,z_i]}\}_{i=1}^4$, the edges $\{e \,:\; f_e\in F\}$ must all belong to $E$. First, we show that  
\begin{equation}\label{eq:side-edges-y-zi-in-E}
 \{[y, z_i]\}_{i = 1}^4 \subseteq E\,,   
\end{equation} 
or equivalently that $y$ and each $z_i$ are in $V$. For any $I \in \hat\sep_n^1(H_{1})$, \cref{it:gamma-connect-to-UHB-LHB} of $\Gamma_{h_1}^x$ ensures that $I$ does not separate any of the $z_i$ from $\partial \Lambda_n^+$, and $I' \subseteq I$. Thus, $\{z_i\}_{i = 1}^4 \subseteq V$. Furthermore, since $\{f_{[y, z_i]}\}_{i = 1}^4 \cap I' = \emptyset$, then $y$ is also in $V$. (In fact, since $f_{[y, y - \ez]} \in I'$, we additionally have that $y, z_i \in B_\scV$.) Second, we show that 
\begin{equation}\label{eq:f-1-connected-to-f_y_z-not-in-I'}
    \left\{ f\,:\;  \hgt(f) > h_1 \mbox{ and $f$ is 1-connected to  $\bigcup_{i = 1}^4 f_{[y, z_i]}$}\right\} \cap I'  = \emptyset\,.
\end{equation} 
Indeed, we know that for any $I \in \hat\sep_n^1$, by \cref{it:gamma-cut-points-y} of $\Gamma_{h_1}^x$, we have $f_{[y,z_i]}\in I\setminus I'$ for each $i=1,\ldots,4$. Thus, any faces whose height exceeds $ h_1$ and are 1-connected to one of the $f_{[y, z_i]}$ would have been cut out in the truncation of $I$, and therefore cannot be in $I'$. (The faces at height exactly $h_1$ are also not in $I'$ because \cref{it:gamma-cut-points-y} of $\Gamma_{h_1}^x$ directly excludes them, but we will not use this fact.)
Now, consider the faces $F$. Since $F \subseteq \clfaces $, on the event $\partial^\dagger I' \subseteq \opfaces$ we have
\begin{equation}\label{eq:F-intersect-partial-I'}
F \cap \partial I' \subseteq \partial I' \setminus \partial^\dagger I' =  \{f_{[y, z_i]}\}_{i = 1}^4.
\end{equation}
We claim that by definition of $F$ and \cref{eq:f-1-connected-to-f_y_z-not-in-I',eq:F-intersect-partial-I'} we can infer that
\begin{equation}\label{eq:F-intersect-I'-is-empty} F \cap I' = \emptyset\,;
\end{equation}
to see this, suppose there exists some $f \in F \cap I'$, and (recalling $F$ is $1$-connected) let $P=(f_i)_1^m$ be a 1-connected of faces in $F$ connecting $f_0=f$ to $f_m=f_{[y, z_1]}$. Let $j$ be the minimal index such that $f_j\notin I'$ (well-defined since $f_m\notin I'$). Then $f_j\in F\cap \partial I'$, hence $f_j=f_{[y,z_i]}$ for some $i$ by \cref{eq:F-intersect-partial-I'}, whence $f_{j-1}$ cannot exist by \cref{eq:f-1-connected-to-f_y_z-not-in-I'}, contradiction.

We are now ready to show that every edge $e$ with $f_e \in F$ must be in $E$. For any $f \in F$, there is a 1-connected path $P$ of faces in $F$ from $f$ to one of the $f_{[y, z_i]}$. If $f= f_e$ for some $e \notin E$, then let $g = g_{[u, v]}$ be the last face in the path $P$ such that $[u, v] \notin E$, so that $g$ is 1-connected to $g' = g'_{[u', v']}$ where $[u', v'] \in E$. W.l.o.g., let $u \notin V$. No matter how $g$ and $g'$ are connected to each other, $u$ is always $\Lambda_n$-adjacent to $u'$ (or $v'$), with the face $g'' = g''_{[u, u']}$ (or $= g''_{[u, v']})$ being either equal to or 1-connected to $g$. However, since $g''$ separates $u \notin V$ from $u' \in V$, then $g'' \in I' $. Hence, as $g$ and $g''$ are equal or $1$-connected, we have $g \in \overline{I'}$. But then the assumption that $g = g_{[u, v]}$ for $[u, v] \notin E$ contradicts the combination of \cref{eq:side-edges-y-zi-in-E,eq:F-intersect-partial-I',eq:F-intersect-I'-is-empty}.
This concludes the proof.
\end{proof}

The next claim will establish that $B_{\scV}\cup\partial \Lambda_n^+$ forms a vertex boundary for $G$, as well as identify its open clusters given the configuration in $(\omega \setminus E)\cup B_{\scE}$.

\begin{claim}\label{clm:G-bc}
Let $I'$ be the truncation of some interface $I\in\hat\sep_n^1$. Define $(B_\scV,B_\scE)$ and $G=(V,E)$ as above. The following hold:
\begin{enumerate}[(i)]
    \item \label{it:vertex-bc}
    The vertices $B_{\scV} \cup \partial \Lambda_n^+$ form a vertex boundary for $V$ (in that every $\Lambda_n$-path from $v\in V$ to $V^c$ must cross one of those vertices).

    \item \label{it:edge-bc}
The graph obtained from $(B_\scV,B_\scE)$ by deleting the vertex $y$ (and edges incident to it) is connected.
Consequently, on the event $\partial^\dagger I' \subseteq \opfaces$, the vertices $B_\scV \setminus \{y\}$ are all part of a single open cluster in~$\omega$.

    \item \label{it:y-bc}
    On the event $f_{[y, y - \ez]} \in \clfaces$, there cannot be a path of open edges in $E^c$ connecting $y$ to $\partial \Lambda_n^+ \cup B_\scV \setminus \{y\}$.
\end{enumerate}
\end{claim}
\begin{proof}

To prove \cref{it:vertex-bc}, recall that if $u\in V $, then necessarily $u\in\out(I')$ (as it is connected to $\partial \Lambda_n^+$ via a path not crossing a face of $I'$), whence we have that
\[B_\scV =  \left\{ u\in V\,:\; \exists v\mbox{ s.t.\ } f_{[u,v]}\in \overline{I'}\right\}\,.\]
We first claim that if $u \in V$ is $\Lambda_n$-adjacent to $v\in\Lambda_n\setminus V$, then necessarily $u\in B_\scV$. Indeed, we must have $f_{[u,v]}\in I'$ by definition of $V$; in particular, $f_{[u,v]}\in \overline {I'}$, and by the last display, $u\in B_\scV$. 
Second, note that $\partial \Lambda_n^+ \subseteq V$ and $\partial \Lambda_n^- \cap V = \emptyset$. Combined, we find that $B_\scV \cup \partial \Lambda_n^+$ forms a complete vertex boundary for $V$. 

Having established that $B_\scV\cup\partial \Lambda_n^+$ forms a vertex boundary for $G=(V,E)$, we proceed to \cref{it:edge-bc}.
Recall that $(B_\scV,B_\scE)$ is connected, as per \cref{clm:Bv-connected}, hence for this item we need only account for the effect of deleting $y$. 
A-priori, we only know that $f_e \in \partial I'$ for all $e \in B_\scE$, but would like to instead say that $f_e \in \partial^\dagger I'$ so that on the event $\{\partial^\dagger I' \subseteq \opfaces\}$, every such $e$ would be open. 
To this end, let $\tilde B_\scE$ be the outcome of removing from $B_\scE$ the four edges $[y, z_i]$ (the faces $f_{[y, z_i]}$ are precisely the four faces removed from $\partial I'$ to obtain $\partial^\dagger I'$). 

First, we claim that there are no other edges of $B_\scE$ incident to $y$, via the following two items:
\begin{enumerate}[(a)]
    \item $[y, y - \ez] \notin B_\scE$ since
    $f_{[y, y - \ez]} \in I'$;
    \item \label{it:y_y+e3-not-in-Be}
    $[y, y+\ez] \notin B_\scE$, as otherwise, having $f_{[y, y+\ez]} \in \partial I'$, there must be a face $g \in I'$ that is 1-connected to $f_{[y, y+\ez]}$ with $\hgt(g) > h_1$. By the truncation, this face $g$ cannot be some $f_{[y,z_i]}$ (it can only be part of $I'$ via another pillar $\cP_{x'}$ for $x'\neq x$) yet it must be 1-connected to $f_{[y, z_i]}$ for some $i$. But, by \cref{eq:f-1-connected-to-f_y_z-not-in-I'}, this is impossible.
\end{enumerate}
Thus we have shown that the only adjacent vertices of $y$ in $(B_\scV, B_\scE)$ are its four side neighbors $z_i$, and as a consequence, the graph $(B_\scV \setminus \{y\}, \tilde B_\scE)$ is equal to the subgraph of $(B_\scV, B_\scE)$ induced on $B_\scV \setminus \{y\}$. So, to show that $(B_\scV \setminus \{y\}, \tilde B_\scE)$ is connected, it suffices to exhibit a path in $\tilde B_{\scE}$ between $z_1 = y+\ex$ and $z_2=y+\ey$ (whence by symmetry there will be such paths between any two of the $z_i$'s). These are connected in $\Lambda_n$ by the path 
\begin{equation*}
    P = \Big(y+\ex ,y+\ex-\ez, y+\ex+\ey-\ez, y+\ey-\ez, y+\ey\Big)\,.
\end{equation*}
Now, \cref{it:gamma-cut-points-y} of the definition of $\Gamma_{h_1}^x$ (and the fact that $I'$ was a truncation of some $I \in \hat\sep_n^1(H_1)$) readily implies that for any edge $e\in P$, the face $f_e$ is in $\partial I'$. 
%%implies that the faces of $\sH_1$ at heights $[h_1 - 1/2, h_1]$ are exactly the four faces $f_{[y, z_i]}$, and each $f_e$ is a face with height $[h_1 - 1/2, h_1]$ that is 1-connected (but not equal) to one of the faces $f_{[y, z_i]}$. 
Since $z_1 \in B_\scV$, this implies every vertex in the path $P$ is also in $B_\scV$. Thus, $P$ uses only edges in $\tilde B_\scE$, as required, and altogether $(B_\scV \setminus \{y\}, \tilde B_\scE)$ is connected.

It remains to prove \cref{it:y-bc}. By \cref{eq:f-1-connected-to-f_y_z-not-in-I'}, we have $f_{[y, y+\ez]}\notin I'$, hence (recall $y\in V$) also $y+\ez \in V$. Since the edge $[y, y - \ez]$ is the only edge of the form $[y,y']$ with $y'\notin V$, on the event $f_{[y, y - \ez]} \in \clfaces$ we see $y$ can never have an open path to $V\setminus \{y\}$ (and in particular to $\partial \Lambda_n^+ \cup B_\scV \setminus \{y\}$) using only edges of $E^c$. 
\end{proof}

We are now in a position to prove the Domain Markov-type identity in \cref{lem:Ah2-DMP}.
\begin{proof}[\textbf{\emph{Proof of \cref{lem:Ah2-DMP}}}]
By \cref{clm:G-bc}, 
if $I'$ is the truncation of some $I\in \hat\sep_n^1(H_1)$, and $\eta$ is any configuration of $E^c$ satisfying that $f_{[y, y - \ez]} \in \clfaces[\eta]$ and $\partial^\dagger I' \subseteq \opfaces[\eta]$, then the law of $\omega\restriction_E$ for $\omega\sim\mu_n(\cdot \mid f_{[y, y - \ez]} \in \clfaces,\, \partial^\dagger I' \subseteq \opfaces,\, \omega\restriction_{E^c}=\eta)$ is that of the random-cluster model on $G=(V,E)$
 with the vertex boundary $B_{\scV}\cup\partial \Lambda_n^+$ and boundary conditions 
 that are wired on $\partial \Lambda_n^+\cup  B_{\scV}\setminus\{y\}$ and free on $y$ (using Domain Markov to disregard the configuration $\eta$).
 Note the boundary condition is fully prescribed by the closed edge $[y,y-\ez]$ and open edges dual to $\partial^\dagger I'$. 
 Recalling \cref{clm:Ah2-E-measurable}, the event $\theta_{h_1} A_{h_2}^{y-h_1\ez}$ is measurable w.r.t.\ the configuration $\omega\restriction_E$. Combined, we arrive at \cref{eq:Ah2-DMP}.
\end{proof}

Next we look at the right-hand of \cref{eq:Ah2-DMP} and compute the cost of  conditioning on the face $f_{[y, y - \ez]}\in\clfaces$. Let ${\bf e} = [y, y - \ez]$. 
For every configuration $\omega \in \theta_{h_1}A_{h_2}^{y - h_1\ez} \cap \{\partial^\dagger I' \subseteq \opfaces\}$ (likewise for $\omega \in \{\partial^\dagger I' \subseteq \opfaces\}$), both $\omega^{{\bf e}, 0}$ and $\omega^{{\bf e}, 1}$ are still in the event (as ${\bf e}\notin \partial^\dagger I'$ nor is it in $E$), where $\omega^{e, 0}$ (resp., $\omega^{e, 1}$) denotes the version of $\omega$ with the edge $e$ closed (resp., $e$ open). Now, $\mu_n(\omega^{{\bf e}, 1})/\mu_n(\omega^{{\bf e}, 0})$ is either $\frac{p}{q(1-p)}$ or $\frac{p}{1-p}$. Summing over $\omega$, we get
\begin{align}  
\mu_n(\theta_{h_1}A_{h_2}^{y - h_1\ez}\mid f_{[y, y - \ez]} \in \clfaces\,, \partial^\dagger I' \subseteq \opfaces)
%=\frac{\mu_n(\theta_{h_1}A_{h_2}^{y - h_1\ez}, \partial^\dagger I' \subseteq \opfaces, \omega_{\bf e} = 0)}{\mu_n(\partial^\dagger I' \subseteq \opfaces, \omega_{\bf e} = 0)} 
&\leq \frac{\mu_n(\theta_{h_1}A_{h_2}^{y - h_1\ez}, \partial^\dagger I' \subseteq \opfaces)}{\mu_n(\partial^\dagger I' \subseteq \opfaces)}\frac{1 + \frac{p}{1-p}}{1+\frac{p}{q(1-p)}}\nonumber \\ 
&\leq q\mu_n(\theta_{h_1}A_{h_2}^{y - h_1\ez}\mid \partial^\dagger I' \subseteq \opfaces)\label{eq:remove-conditioning-on-edge}
\end{align}
At this point when may apply FKG to get 
\begin{equation*}
    \mu_n(\theta_{h_1}A_{h_2}^{y - h_1\ez}\mid \partial^\dagger I' \subseteq \opfaces) \leq \mu_n(\theta_{h_1}A_{h_2}^{y - h_1\ez})
\end{equation*}
since the event $\theta_{h_1}A_{h_2}^{y - h_1\ez}$ is decreasing, while the event $\partial^\dagger I' \subseteq \opfaces$ is increasing. By \cref{lem:height-shift}, we can pay a factor of $q^2$ to move to the non-shifted event $A_{h_2}^{y - h_1\ez}$:
\[\mu_n(\theta_{h_1}A_{h_2}^{y - h_1\ez}) \leq q^2\mu_n(A_{h_2}^{y - h_1\ez})\]
By FKG again (now using that $\sep_n$ is decreasing), we have
\begin{equation*}
    \mu_n(A_{h_2}^{y - h_1\ez}) \leq \bar{\mu}_n(A_{h_2}^{y - h_1\ez})
\end{equation*}
Finally, to move from $A_{h_2}^{y - h_1\ez}$ to $A_{h_2}^x$, we utilize (a special case of) \cref{cor:decorrelation-1}, a decorrelation result on pillars, that implies that for some constant $C$ and all $x, y$ such that $d(x, \partial \Lambda_n) \wedge d(y, \partial \Lambda_n) \geq r$, we have
\[\bar{\mu}_n(E_h^y) \leq \bar{\mu}_n(E_h^x) + Ce^{-r/C}\,.\]
(We defer the proof of said estimate to the appendix, along with the analogous results for the Potts model.) By putting together the assumptions $d(x, y - h_1\ez) \leq d(x, \partial \Lambda_n)/2$ and $d(x, \partial \Lambda_n) \gg h$ with the bounds on $\bar{\mu}_n(E_{h_2}^x)$ from \cref{prop:bound-RC-rate}, we get that
\[\bar{\mu}_n(E_{h_2}^{y-h_1\ez}) \leq (1+o(1))\bar{\mu}_n(E_{h_2}^x)\]
(where the $o(1)$ is as $h \to \infty$).
We can then apply \cref{prop:compare-A_h^x-E_h^x} to get 
\begin{flalign}
    \bar{\mu}_n(A_{h_2}^{y - h_1\ez}) &\leq (1+\epsilon_\beta)\frac{e^\beta+q-1}{q}\bar{\mu}_n(E_{h_2}^{y-h_1\ez})\nonumber \\
    &\leq (1+\epsilon_\beta)\frac{e^\beta+q-1}{q}(1+o(1))\bar{\mu}_n(E_{h_2}^x)\label{eq:Ahx-move-locations}
\end{flalign}

Combining \cref{lem:Ah2-DMP} with the inequalities between 
\cref{eq:remove-conditioning-on-edge} to \cref{eq:Ahx-move-locations}, we have that for some $\epsilon_\beta$,
\begin{equation}
\max_{H_1}\max_{I':I\in\hat\sep_n^1(H_1)}\mu_n\left(\theta_{h_1}A_{h_2}^{y-h_1\ez} \mid \cS_{I'}\right) \leq (1+\epsilon_\beta)q^2(e^\beta+q-1)\bar{\mu}_n(E_{h_2}^x)
\end{equation}
which together with \cref{eq:mu(Ah-Gamma)-up-to-max-Ah2} concludes the proof of \cref{lem:submultiplicativity-in-Gamma}.
\end{proof}

\begin{proof}[{\textbf{\emph{Proof of \cref{prop:RC-rate}}}}]
Combining \cref{prop:compare-A_h^x-E_h^x,lem:move-to-gamma,lem:submultiplicativity-in-Gamma} immediately implies the submultiplicativity statement of \cref{eq:submultiplicativity-Ehx}. By using the decorrelation estimates in \cref{cor:decorrelation-1}, we can generalize to the case where $x, n$ on the right hand side can depend on $h_1$ and $h_2$, as long as we still have $1 \ll h_i \ll n_{h_i}$ and $d(x_{h_i}, \partial \Lambda_{n_{h_i}}) \gg h_i$:
\begin{equation*}
    \bar{\mu}_n(E_h^x) \leq (1+\epsilon_\beta+o_{h_1}(1)+o_{h_2}(1))(e^\beta + q-1)^3\bar{\mu}_{n_{h_1}}(E_{h_1}^{x_{h_1}})\bar{\mu}_{n_{h_2}}(E_{h_2}^{x_{h_2}})\,.
\end{equation*}
Fekete's Lemma now gives the existence of the limit $\alpha$ in \cref{eq:alpha-rate}, and the bounds of \cref{prop:bound-RC-rate} immediately gives the corresponding bound $4(\beta - C) \leq \alpha \leq 4\beta$.
\end{proof}

\section{Large deviation rate for Potts interfaces}\label{sec:ld-potts-bot}

The pillar $\cP_x$ was used to locally measure the height of the $\Top$ interface at a location $x$. There, we needed a more complicated definition of the pillar including the hairs attached to it so that we could apply various map arguments to prove properties of a typical pillar. For the $\Blue$ and $\Red$ Potts interfaces and $\Bot$ random-cluster interface, rather than consider an analogous pillar on its own, we will study the event that a path of a particular component of vertices reaches height $h$, conditional on $\cP_x$ reaching height at least $h$.

\begin{definition}
    Let $\cA^{\noRed}_{x, h}$ be the event that there is a $\Ared^c$-path from $x$ to $h$ using only vertices that are part of $\cP_x$. We also analogously define $\cA^{\Blue}_{x, h}$ and $\cA^{\Bot}_{x, h}$ as paths of vertices in $\Ablue$ and $\Abot$ respectively. More generally, we use the notation $\cA^{\noRed}_{v_i, v_{i+1}}$ to mean that there is a $\Ared^c$-path of vertices from $v_i$ to $v_{i+1}$ (endpoints included) that uses only vertices of $\cP_x$ in the slab $\cL_{[\hgt(v_i), \hgt(v_{i+1})]}$, and analogously for $\cA^{\Blue}_{v_i, v_{i+1}}, \cA^{\Bot}_{v_i, v_{i+1}}$.
\end{definition}

\begin{remark}
We note that one may attempt to define a pillar in $\cI_\Red$ analogously to how it was defined w.r.t.\ $\Itop$. That is, for a vertex $x$ at height 1/2, the non-$\Red$ pillar at $x$ would be the connected set of vertices in $\Ared^c$ which have height $\geq 1/2$. However, by the ordering of the interfaces (i.e., the fact that $\Ared^c \subseteq \Atop^c$), the non-$\Red$ pillar always lies entirely within $\cP_x$. Hence, the event that the height of the non-$\Red$ pillar at $x$ reaches height $h$ is exactly the same as the event $\cA^{\noRed}_{x, h}$. However, we will not refer to such a non-$\Red$ pillar and instead refer to events of the form $\cA^{\noRed}_{x, h}$ because the latter is more easily broken up into parts --- one can view the event $\cA^{\noRed}_{x, h}$ as an intersection of events of the form $\cA^{\noRed}_{v_i, v_{i+1}}$, and this reflects the proof ideas of this section.
\end{remark}

The goal of this section is to prove the large deviation rates for the pillars of the $\Blue$ and $\Red$ Potts interfaces and the $\Bot$ interface of the random-cluster model.
\begin{proposition}\label{prop:RC-Potts-diff-rate}
For every $\beta > \beta_0$ and integer $q\geq 2$ there exist $\delta,\delta' \geq 0$ such that, for every sequence of $n, x$ dependent on $h$  with $1 \ll h \ll n$ and $d(x, \partial \Lambda_n) \gg h$, 
\begin{align}
\label{eq:delta-beta}
    \lim_{n \to \infty} - \frac{1}{h}\log \phi_n(\cA^{\noRed}_{x, h} \mid \hgt(\cP_x)\geq h) &= \delta\,,\\
\label{eq:delta-beta-prime}
    \lim_{n \to \infty} - \frac{1}{h}\log \phi_n(\cA^{\Blue}_{x, h} \mid \hgt(\cP_x)\geq h) &= \delta'\,.
\end{align}
Moreover, for every $\beta>\beta_0$ and real $q\geq 1$ there exists $\delta''\geq 0$ such that, for every sequence $n,x$ as above,
\begin{align}\label{eq:delta-beta-double-prime}
    \lim_{n \to \infty} - \frac{1}{h}\log \bar\mu_n(\cA^{\Bot}_{x, h} \mid \hgt(\cP_x)\geq h) &= \delta''\,.
\end{align}
\end{proposition}
Combining this with \cref{prop:RC-rate}, we derive the following rates:
\begin{align}
\label{eq:non-red-rate}
\gamma &:= \lim_{n \to \infty} - \frac{1}{h}\log \phi_n(\cA^{\noRed}_{x, h}) = \alpha + \delta\,,\\
\label{eq:blue-rate}
 \gamma' &:= \lim_{n \to \infty} - \frac{1}{h}\log \phi_n(\cA^{\Blue}_{x, h}) = \alpha + \delta' \\ 
\label{eq:bot-rate}
  \alpha'&:= \lim_{n \to \infty} - \frac{1}{h}\log \bar\mu_n(\cA^{\Bot}_{x, h}) = \alpha + \delta''\,.
\end{align}

Once we establish the above rates, we will also provide bounds on their differences. In particular, we show that all the rates are different from each other, whence using the symmetry that the upward deviations of $\Itop$ are the same as the downward deviations of $\Ibot$, we conclude that each interface has an asymmetry between its maximum and its minimum.
\begin{proposition}
    \label{prop:bound-rates}
There exists a sequence $\epsilon_\beta$ going to $0$ as $\beta\to\infty$ such that, for every fixed $\beta>\beta_0$, the rates $\delta,\delta',\delta''$ from \cref{prop:RC-Potts-diff-rate} 
satisfy
\begin{align}
\label{prop:bound-potts-rate-1}
    \delta &= (1\pm \epsilon_\beta) e^{-\beta}\,,\\
\label{prop:bound-potts-rate-2}
    \delta' &= (1\pm \epsilon_\beta) (q-1)e^{-\beta}\,,\\
\label{prop:bound-bot-rate}
    \delta'' &= (1\pm \epsilon_\beta) qe^{-\beta}
    \end{align}
where $a=(1\pm\epsilon)b$ is notation for $a\in[(1-\epsilon)b,(1+\epsilon)b]$.
\end{proposition}

Proving the above propositions would conclude the proof of \cref{prop:comparison-of-rates}, as we already showed the bound on $\alpha$ at the end of \cref{sec:ld-RC}.

\begin{remark} To prove the existence of the rates in \cref{prop:RC-Potts-diff-rate}, the sub-additivity claim we are after is essentially that a non-$\Red$ path climbing to height $h_1 + h_2$ is comparable to climbing to height $h_1$, and then independently climbing up to height $h_2$. Since inside a given pillar the coloring of different clusters is independent to begin with, this is seemingly obvious. However, we are aiming for sub-additivity conditional on the event $E_h^x$ and not on a fixed pillar, so to make this rigorous we need to show that the joint law of the part of a pillar in $E_h^x$ below height $h_1$ and the part above it is comparable to the law of a pillar in $E_{h_1}^x$ and an independently sampled pillar in $E_{h_2}^x$. This is true only if we add some restrictions to control the interactions between the two halves of the pillar, and the interactions between the pillar and the rest of the interface. So, in \cref{lem:move-to-nice-space} we prove that we can move onto this space of nicer pillars, and in \cref{lem:3-to-3-map} we prove the claim on the law of the pillars by utilizing a 3 to 3 swapping map similar to the swapping maps in \cite{GL_max}. Along the way, we also need to be cautious that we are actually asking for a path of vertices in $\Ared^c$, not just non-$\Red$ vertices, and we also need to work on the joint space of configurations $(\omega, \sigma)$.
\end{remark}

\subsection{Establishing the Potts rates}

The bulk of this section is devoted to proving the following submultiplicativity statement:
\begin{proposition}\label{prop:submultiplicativity-Potts}
    For every $\beta > \beta_0$, there exists a constant $\epsilon_\beta$ such that for every $h = h_1 + h_2$, and every sequence $x, n$ dependent on $h$ such that $d(x, \partial \Lambda_n) \gg h$,
    \begin{equation}\label{eq:submultiplicativity-noRed}
\phi_n(\cA^{\noRed}_{x, h_1+h_2}\mid E_{h_1+h_2}^x) \leq (1+\epsilon_\beta)\phi_n(\cA^{\noRed}_{x, h_1}\mid E_{h_1}^x)\phi_n(\cA^{\noRed}_{x, h_2}\mid E_{h_2}^x)\,.
\end{equation}
The same statement holds if we replace $\noRed$ by $\Blue$.
\end{proposition}

As mentioned in the remark above, we will use the following nicer spaces of pillars, which are subsets of spaces of isolated pillars with some additional restrictions. Suppose that we fix $h_1 + h_2 = h$, and choose $0 \leq L \leq L_\beta$, where $L_\beta \uparrow \infty$ is as in \cref{thm:iso-map}. 

\begin{definition}[The subset $\Omega_{h}$ of isolated pillar interfaces] \label{def:Omega-h}
Let $x\in \cL_{1/2}$, and define $\Omega_{h}$ to be the set of interfaces in $\Iso_{x, L, h}$ satisfying the following properties  (\cref{it:omega-cone-base,it:omega-face-bound-weak,it:omega-walls} are precisely the criteria for $\Iso_{x,L,h}$; we repeat the statement of these conditions here for an easy comparison with the next definition.)
\begin{enumerate}[(1)]
\item \label{it:omega-cone-base}$\fm(\sX_t) \leq \begin{cases} 0 & \text{if } t \leq L^3\\
t & \text{if } t > L^3
\end{cases}$
\item \label{it:omega-L-separate-top}There is a stretch of trivial increments from height $h -1/2- L^3$ to $h - 1/2$
\item \label{it:omega-face-bound-weak}$|\sF(\cS_x)| \leq 10h$
\item \label{it:omega-pillar-cap}$\cP_x \in \widetilde E_{h}^x$
\item \label{it:omega-walls}For the walls of $\cI \setminus \cP_x$, we have\\
$\fm(W_y) \leq \begin{cases}
0 & \text{if } d(y, x) \leq L\\
\log(d(y, x)) & \text{if } L < d(y, x) < L^3h
\end{cases}$\\
and $f_{[x, x - \ez]} \notin \cI$.
\end{enumerate}
\end{definition}
\begin{definition}[The subset $\Omega_{h_1,h_2}$ of isolated pillar interfaces]
Let $x\in \cL_{1/2}$, and define $\Omega_{h_1, h_2}$ to be the set of interfaces $\Omega_{h_1+h_2}$ from \cref{def:Omega-h} such that the following additional properties are satisfied:
\begin{enumerate}[(1), start=6]
\item \label{it:omega-L-seperate-middle} There is a stretch of trivial increments from height $0 \vee h_1 -1/2- L^3$ to $h_1+1/2$
\item \label{it:omega-cone-middle}
Let $j_0$ be the index of the increment with bottom cut-point at height $h_1 + 1/2$. Then,\\
$\fm(\sX_t) \leq \begin{cases} 0 & \text{if } t - j_0 \leq L^3\\
t- j_0 & \text{if } t - j_0 > L^3
\end{cases}$
\item \label{it:omega-face-bound-strong} $|\sF(\cS_x) \cap \cL_{\leq h_1}| \leq 10h_1$, and $|\sF(\cS_x) \cap \cL_{\geq h_1}| \leq 10h_2$.
\end{enumerate}
\end{definition}

\begin{remark}\label{rem:cut-paste-pillars}
For simplicity, we can also say that a pillar $\cP_x \in \Omega$ if it satisfies the pillar properties of the space, i.e., there exists $\cI \in \Omega$ with pillar $\cP_x$. These two spaces of pillars are defined such that we can write $\Omega_{h_1, h_2} = \Omega_{h_1} \times \Omega_{h_2}$ in the following sense: Suppose at the vertex $x$, we take a pillar $P^T \in \Omega_{h_2}$ and attach it to the top increment of a pillar $P_B \in \Omega_{h_1}$. This location of attachment is well-defined because of the cut-point conditions imposed in \cref{it:omega-cone-base,it:omega-L-separate-top} of $\Omega_h$. By \cref{it:omega-pillar-cap}, there is an extra face separating the top vertex of $P^T$ and the bottom vertex of $P_B$; remove it. Then, the resulting combined pillar satisfies the pillar properties of $\Omega_{h_1, h_2}$. We denote this combined pillar by $P_B \times P^T$. Conversely, we can decompose any $P \in \Omega_{h_1, h_2}$ into $P = P_B \times P^T$ by cutting the pillar at height $h_1$ and then adding the face where we cut to be the `top cap' of $P_B$ (this face is never in $P$ because of \cref{it:omega-L-seperate-middle} of $\Omega_{h_1, h_2}$), and we will have $P_B \in \Omega_{h_1}$, $P^T \in \Omega_{h_2}$.
\end{remark}

\begin{lemma}\label{lem:nice-space-likely}
For any $\beta > \beta_0$, and any $x$ such that $d(x, \partial \Lambda_n) \gg h_1 + h_2$, there exists a constant $\epsilon_\beta$ such that for any $h_1, h_2$, \begin{equation}
\bar{\mu}_n(\Omega_{h_1, h_2}\mid E_{h_1+h_2}^x) \geq 1 - \epsilon_\beta\,.
\end{equation}
As $\Omega_{h_1, h_2} \subseteq \Omega_{h_1+h_2}$, then we consequently also have $\bar{\mu}_n(\Omega_{h_1}\mid E_{h_1}^x) \;\wedge\; \bar{\mu}_n(\Omega_{h_2} \mid E_{h_2}^x) \geq 1 - \epsilon_\beta$.
\end{lemma}
\begin{proof}
To lower bound $\bar{\mu}_n(\Omega_{h_1, h_2}\mid E_{h_1+h_2}^x)$, note that if we begin with any interface in $E_{h_1+h_2}^x$, we can guarantee all the properties except \cref{it:omega-pillar-cap,it:omega-face-bound-strong} of $\Omega_{h_1, h_2}$ by applying $\Phi_\Iso$ and $\Phi_\Incr$. Call the image of the composition of these maps $\tilde{\Omega}_{h_1, h_2}$. We are allowed to apply $\Phi_\Incr$ a constant number of times by \cref{rem:incr-map-variations}, and each will cost a factor of $1 - \epsilon_\beta$. (We need to apply the increment map $\Phi_\Incr$ three times, at heights $h_1 - L^3$, $h_1+1$, and $h_1 + h_2 - L^3$ with $L' = L^3$, to get \cref{it:omega-cone-middle,,it:omega-L-seperate-middle,it:omega-L-separate-top}). Thus, \cref{thm:iso-map,thm:incr-map} proves that
\begin{equation}\label{eq:tilde-Omega}
    \bar{\mu}_n(\tilde{\Omega}_{h_1, h_2} \mid E_{h_1 + h_2}^x).
\end{equation}
Now for any $\cI \in \tilde{\Omega}_{h_1, h_2}$, we claim we can use another map argument to additionally ensure we have \cref{it:omega-pillar-cap,it:omega-face-bound-strong}. Let $\cP_x = \cP_x^\cI$ be the pillar at $x$ in $\cI$. Let $\sX_{T}$ be the trivial increment with height $h_1 +h_2$ (so that its two vertices have heights $h_1 + h_2 - 1/2$ and $h_1 + h_2 - 1 - 1/2$). We consider three cases:

\begin{enumerate}[{Case} A.]
\item If $|\sF(\cS_x) \cap \cL_{\leq h_1}| > 10h_1$, then let $\sX_{j}$ be the trivial increment whose bottom cut-point is at $h_1 - 1/2$ (which exists per \cref{it:omega-L-seperate-middle}), and define \[\cP_x^\cJ = \big(\underbrace{X_\trivincr,\ldots,X_\trivincr}_{h_1},\sX_{j+1},\ldots,\sX_{T}\big)\,.\]

\item Otherwise, if $|\sF(\cS_x) \cap \cL_{\geq h_1}| > 10h_2$, then let $\sX_{j}$ be as above and define \[\cP_x^\cJ = \big(\sX_1,\ldots,\sX_{j-1},\underbrace{X_\trivincr,\ldots,X_\trivincr}_{h_2}\big)\,.\]
Note that case A and B cannot occur simultaneously since $|\sF(\cS_x)| \leq 10h$ by $\Phi_\Iso$.

\item If neither of the above cases hold, then define \[\cP_x^\cJ = \big(\sX_1,\ldots,\sX_{T}\big)\,.\]
\end{enumerate}

Let $\Phi$ be a map that takes $\cI$ and gives the interface $\cJ$ which replaces $\cP_x^\cI$ with $\cP_x^\cJ$. We will prove the required energy and entropy bounds assuming we are in Case A, as the proof for Cases B and C are essentially the same. In Case B we just have $\bX_B^\cI$ and $\bX_A^\cI$ defined below switch roles (in fact it is even simpler because there is no shift of increments), and in Case C we just note that all the computations below would still hold if we did not change any of the increments in $\bX_A^\cI$. We begin by proving the energy bound:
\begin{equation*}
\bar{\mu}_n(\fm(\cI;\Phi(\cI)) \geq r \mid \tilde{\Omega}_{h_1, h_2}) \leq Ce^{-(\beta - C)r}
\end{equation*} 
We can split up any interface $\cI \in \Iso_{x, L, h}$ as follows:
{\renewcommand{\arraystretch}{1.3}
\[
\begin{tabular}{m{0.05\textwidth}m{0.2\textwidth}m{0.5\textwidth}}
\toprule
\midrule
 $\bR$ & $\cS_x \cap \cL_{\geq h_1 + h_2}$ & ``Remainder" increments above height $h_1 + h_2$  \\ 
 $\bX_{B}^{\cI}$ & $\bigcup_{j+1\leq j \leq T-1}\sF(\sX_j)$ & Increments between $v_{j+1}$ and $v_{T}$ \\
 $\bX_{A}^{\cI}$ & $\cS_x \cap \cL_{\leq h_1}$ & Increments below height $h_1$ (these will be trivialized) \\
 $\bB$ & $\cI\setminus(\cS_x^\cI)$
 & The remaining set of faces in $\cI$ \\
 \bottomrule
 \end{tabular}
\]
}
Similarly, we can divide the interface $\cJ$:
{\renewcommand{\arraystretch}{1.3}
\[
\begin{tabular}{m{0.05\textwidth}m{0.2\textwidth}m{0.5\textwidth}}
\toprule
\midrule
 $\bX_{B}^{\cJ}$ & & Horizontally shifted copy of $\bX_{B}^\cI$ \\
 $\bX_{A}^{\cJ}$ & & Trivial increments between heights 0 to $h_1$ \\
 $\bB$ & & Same set of faces as in $\cI$ \\
 \bottomrule
 \end{tabular}
\]
}
The trivial increments in $\bX_{A}^\cJ$ have $4h_1$ faces, while the number of faces in $\bX_A^\cI$ is at least $10h_1$ by assumption. So, the excess area of the map is 
\begin{equation}\label{eq:m-lower-bnd-omega-map}
    \fm(\cI;\cJ) = |\bR| + |\bX_A^\cI| - |\bX_A^\cJ| \geq (|\bR| + \frac{3}{5}|\bX_A^\cI|) \vee (|\bR| + \frac{3}{2}|\bX_A^\cJ|)
\end{equation}
Using the cluster expansion, we have
\begin{equation*}
\frac{\bar{\mu_n}(\cI)}{\bar{\mu}_n(\cJ)} = (1-e^{-\beta})^{|\partial \cI| - |\partial \cJ|}e^{-\beta \fm(\cI;\cJ)}q^{\kappa_\cI - \kappa_\cJ}\exp(\sum_{f\in \cI} \g(f, \cI) - \sum_{f \in \cJ} \g(f, \cJ))
\end{equation*}
As in the proof of \cref{lem:control-1-connected-faces}, we can define an injective map $T$ on a subset of $\partial \cJ$ to $\partial \cI$ and show that the number of faces we do not define $T$ on is bounded by $C\fm(\cI;\cJ)$ for some $C$. Faces which are 1-connected to $B$ can be mapped to themselves, and faces 1-connected to $X_B^\cJ$ can be mapped to their shifted copy in $X_B^\cI$ (the cone separation property ensures there is no problem here). The remaining faces which are 1-connected $X_A^\cJ$ can be handled by following the procedure in \cref{st-3-control-1-conn-faces} of \cref{lem:control-1-connected-faces}, or more simply in this case we can just bound the number of such faces by $C_0|\bX_A^\cJ| \leq \frac{2}{3}C_0\fm(\cI;\cJ)$, where $C_0$ is the number of faces that can be 1-connected to a particular face, and so we do not need to define $T$ on these faces.

We also have 
\begin{equation*}
\kappa_\cI - \kappa_\cJ \leq |\bR| +|\bX_A^\cI| \leq \frac53\fm(\cI;\cJ)
\end{equation*}
since adding a face can only create at most one more open cluster. 

Finally, we bound the influence of the $g$-terms. We can write the absolute value of their sum as
\begin{equation*}
    \sum_{f \in \bR \cup \bX_A^\cI} |\g(f, \cI)| + \sum_{f \in \bX_A^\cJ} |\g(f, \cJ)| + \sum_{f \in \bB} |\g(f, \cI) - \g(f, \cJ)| + \sum_{f \in \bX_B^\cI} |\g(f, \cI) - \g(\theta f, \cJ)|
\end{equation*}
where $\theta$ is the horizontal shift that moves $\bX_B^\cI$ to $\bX_B^\cJ$.

We can bound the first and second terms by $KC\fm(\cI;\cJ)$ by the bound in \cref{eq:m-lower-bnd-omega-map}.

For the third term, we note that since both pillars have the same stretch of $L^3$ trivial increments at the bottom, we have by \cref{eq:pillar-in-cone-2},
\begin{equation*}
    \sum_{f \in \bB} |\g(f, \cI) - \g(f, \cJ)| \leq \sum_{f \in \bB}\sum_{g \in (\bX_A^\cI \cup \bX_A^\cJ) \cap \cL_{\geq L^3}} Ke^{-cd(f, g)} \leq KCe^{-cL}
\end{equation*}
Finally, for the fourth term, when the $r$-distance in the cluster expansion is attained by a face in $\bR \cup \bX_A^\cI \cup \bX_A^\cJ$, we can use \cref{eq:m-lower-bnd-omega-map}, and when it is attained by a face in $\bB$, we can use \cref{eq:pillar-in-cone-2}. That is, we have
\begin{align*}
    \sum_{f \in \bX_B^\cI} |\g(f, \cI) - \g(\theta f, \cJ)| &\leq \sum_{f \in \bX_B^\cI}\sum_{g \in \bR \cup \bX_A^\cI \cup \bX_A^\cJ}  Ke^{-cd(f, g)} + \sum_{f \in \bX_B^\cI}\sum_{g \in \bB}  Ke^{-cd(f, g)}\\
    &\leq \sum_{f \in \sF(\Z^3)}\sum_{g \in \bR \cup \bX_A^\cI \cup \bX_A^\cJ}  Ke^{-cd(f, g)} + KCe^{-cL}\\
    &\leq KC\fm(\cI;\cJ) + KCe^{-cL}
\end{align*}

Thus, we have proved the energy bound
\begin{equation*}
\frac{\bar{\mu_n}(\cI)}{\bar{\mu}_n(\cJ)} \leq e^{-(\beta - C)\fm(\cI;\cJ)}
\end{equation*}
For the entropy bound, we simply note that given any $\cJ \in \Phi(\tilde{\Omega}_{h_1, h_2})$, we can recover $\cI$ if we are given the 1-connected set $\bR$ which has size $\leq \fm(\cI;\cJ)$, and the 1-connected set $\bX_A^\cI$ which has size $\leq \frac{5}{3}\fm(\cI;\cJ)$. Indeed, we can take $\cJ \setminus \cP_x^\cJ$, attach $\bX_A^\cI$ at $x$, then append the portion of $\cP_x^\cJ$ with height larger than $\hgt(\bX_A^\cI)$, and finally attach $\bR$ at the top cut-point. Thus, by \cref{lem:num-of-0-connected-sets}, we have
\begin{equation*}
|\{\cI \in \Phi^{-1}(\cJ): \fm(\cI;\cJ) = M\}| \leq s^{\frac83M}
\end{equation*}
Thus, we have for any $r \geq 1$,
\begin{align*}
\bar{\mu}_n(\fm(\cI;\Phi(\cI)) \geq r, \tilde{\Omega}_{h_1, h_2}) &= \sum_{M \geq r}\,\, \sum_{\substack{\cI \in \tilde{\Omega}_{h_1, h_2},\\ \fm(\cI;\Phi(\cI)) = M}} \bar{\mu}_n(\cI)\\
&\leq \sum_{M \geq r}\,\, \sum_{\cJ \in \Phi(\tilde{\Omega}_{h_1, h_2})}\,\, \sum_{\substack{\cI \in \Phi^{-1}(\cJ),\\ \fm(\cI;\cJ) = M}} e^{-(\beta - C)M}\bar{\mu}_n(\cJ)\\
&\leq \sum_{M \geq r} s^{\frac83M} e^{-(\beta - C)M}\bar{\mu}_n(\Phi(\tilde{\Omega}_{h_1, h_2}))\\
&\leq Ce^{-(\beta - C - \frac83\log s)r}\bar{\mu}_n(\tilde{\Omega}_{h_1, h_2})
\end{align*}
Then, dividing by $\bar{\mu}_n(\tilde{\Omega}_{h_1, h_2})$ yields
\begin{equation*}
\bar{\mu}_n(\fm(\cI;\Phi(\cI)) \geq r \mid \tilde{\Omega}_{h_1, h_2}) \leq Ce^{-(\beta - C)r}
\end{equation*} 
Taking $r = 1$ above and combining with \cref{eq:tilde-Omega} concludes the proof of the lower bound for $\bar{\mu}_n(\Omega_{h_1, h_2}\mid E_{h_1+h_2}^x)$.
\end{proof}

Now, we have shown that a typical pillar in $E_{h_1 + h_2}^x$ will also be in $\Omega_{h_1, h_2}$. However, we need to show that in the joint space of configurations $(\omega, \sigma)$, the event $\cA^{\noRed}_{x, h}$ also occurs primarily on pillars in $\Omega_{h_1, h_2}$. For this, it will be useful to show that the event $\cA^{\noRed}_{x, h}$ can naturally be broken up increment by increment. However, in general we can only determine if a vertex is in $\Ared$ or $\Ared^c$ by looking at the entire configuration $\sigma$. Hence, we need to establish a Domain Markov type result in the joint space showing that once we reach a cut-point $v_i \in \cP_x$, the influence of the coloring outside of $\cP_x$ on a vertex inside $\cP_x$ is only through $v_i$. We begin with the following lemma, where in what follows, we refer to vertices interior to an increment shell $X^{\mathrm{o}}$ as all the vertices of $\cP_x$ that are part of said increment.

\begin{lemma}\label{lem:locally-determine-Ared}
    Fix an increment shell $X_\star^\mathrm{o}$ rooted at a vertex $v_\star\in \Lambda_n$, and let $G_\star = (V_\star, E_\star)$ be the induced subgraph of $\Lambda_n$ on the vertices that are interior to $X_\star^\mathrm{o}$. Let $i\geq 1$. Conditional on the event $\sX_i^\mathrm{o} = X_\star^\mathrm{o}$ in the pillar $\cP_x$ (i.e., $v_i(\cP_x)=v_\star$) and the event $v_\star \in \Ared^c$ (resp., $v_\star \in \Ablue$), the random set $V_\star \cap \Ared^c$ (resp., $V_\star \cap \Ablue$) depends only on $\sigma\restriction_{V_\star}$. 
\end{lemma}
\begin{proof}
    Let $V'$ denote the vertices inside the pillar shell which have height $\geq \hgt(v_i)$. We will prove the case where $X_i^\mathrm{o}$ ends in a cut-point (the case where $X_i^\mathrm{o}$ is the remainder increment is simpler as then $V_\star = V'$). Let $W$ (resp., $W'$) be the set of vertices in $\Z^3 \setminus V'$ which are $\Lambda_n$-adjacent to $V_*$ (resp., $V'$). We know that $W' \subseteq \Vtop$, and hence $W \subseteq W' \subseteq \Vred$.  Let $U \subseteq V_\star$ be the subset of vertices $u$ such that there is a $\Lambda_n$-path $(u = u_1, \ldots, u_k)$ of vertices such that $u_k \in W$, and $u_l$ are $\Red$ vertices in $V$ for $l < k$. Then, $U \subseteq \Vred$. Let $\widehat U$ be the union of $U$ with the vertices in $V_\star$ which are in a finite component of $\Z^3 \setminus (U \cup W')$. Then, $\widehat U \subseteq \Ared$. 
    
    We now argue that $V_* \setminus \widehat U \subseteq \Ared^c$. Observe that every vertex of $v \in \Vred \cap V_\star$ must have a $\Lambda_n$-path of $\Red$ vertices in $V_\star$ connecting to $W'$. Because of the cut-point at $v_{i+1}$, there must actually be a $\Lambda_n$-path of $\Red$ vertices in $V_\star$ connecting $v$ to $W$, whence $\Vred \cap V_\star \subseteq U$. Furthermore, by definition of $\widehat U$, we know that for every $w \in V_\star \setminus \widehat U$, there is a $\Lambda_n$-path connecting $w$ to $v_i$ that does not include any vertices of $U$. Combined, $w$ is in the same component of $\Vred^c$ as $v_i$, whence $w \in \Ared^c$. In other words, we have shown that $V_\star \cap \Ared = \widehat U$. The set $U$ clearly only depends on $\sigma\restriction_{V_\star}$. Although the definition of $\widehat U$ further involves the set $W'$, the specific shape of $W' \setminus W$ does not affect which vertices of $V_\star$ are in $\widehat U$, and the set $W$ is fixed by $X_i^\mathrm{o}$. Hence, the set $\widehat U$ only depends on $\sigma\restriction_{V_\star}$.
    
    The $\Blue$ case is similar. First observe that if $v_i \in \Ablue$, we must actually have $v_i \in \Vblue$ since being a cut-point, the side neighbors of $v_i$ are in $\Ablue^c$. Let $U \subseteq V_\star$ be the subset of vertices $u$ such that there is a $\Lambda_n$-path $(u = u_1, \ldots, u_k = v_i)$ of vertices in $V_\star$ such that $u_l$ are $\Blue$ for $l < k$. Since as defined above, $W \subseteq W' \subseteq \Vred$, then we have $\Vblue \cap V_\star = U$. Let $\widehat U$ be the union of $U$ with the vertices in $V_\star$ which are in a finite component of $\Z^3 \setminus U$. Since $\Ablue$ is co-connected, then $\widehat U \subseteq \Ablue$. Finally, for every $w \in V_\star \setminus \widehat U$, there is a $\Lambda_n$-path of vertices in $V_\star$ connecting $w$ to $W'$ that does not include any vertices of $U$. Because of the cut-point at $v_{i+1}$, there must actually be a $\Lambda_n$-path of vertices in $V_\star$ connecting $w$ to $W$ (still not including vertices of $U$). Since $\Vblue \cap V_\star = U$, this then implies that $w$ is in the same component of $\Vblue^c$ as $W$, whence $w \in \Ablue^c$. Thus, $V_\star \cap \Ablue = \widehat U$.
\end{proof}

\begin{lemma}\label{lem:DMP-pillar-shell}
Fix an increment shell $X_\star^\mathrm{o}$ rooted at a vertex $v_\star\in \Lambda_n$, and let $G_\star = (V_\star, E_\star)$ be the induced subgraph of $\Lambda_n$ on the vertices that are interior to $X_\star^\mathrm{o}$. 
Let $i\geq 1$,
condition on the event $\sX_i^{\rm o}=X_\star^{\rm o}$ in the pillar $\cP_x$, and let $\cW_\omega $ be the set of vertices in $\Lambda_n$ excluding all vertices in $\cP_x^\mathrm{o}$ with height $> \hgt(v_\star)$, noting that on the event $\{\sX_i^\mathrm{o} = X_\star^\mathrm{o}\}$, the set $\cW_\omega$ is measurable w.r.t.\ $\omega\restriction_{E_\star^c}$. Let $\cF$ be the $\upsigma$-field generated by $\omega\restriction_{E_\star^c}$ along with $\sigma\restriction_{\cW_\omega}$. Then the law  $\phi_n\left((\omega,\sigma)\restriction_{G_\star} \in \cdot \mid 
\sX_i^{\mathrm{o}} = X_\star^{\mathrm o}\,, \cF\right)$ is that of the coupled FK--Potts model on $G_\star$ with boundary conditions that are free except at $v_\star$, whose color is specified by $\cF$.
\end{lemma}

\begin{proof}
    As above, we will assume that $X_i^\mathrm{o}$ is not the remainder increment, as that case is the same except there is no $v_{i+1}$ to worry about. Note first that the event $\sX_i^\mathrm{o} = X_\star^\mathrm{o}$ does not impose any conditions on $\omega\restriction_{E_\star}$. Indeed, it follows by the definition of the pillar shell that for every $\omega \in \{\sX_i^\mathrm{o} = X_i^\mathrm{o}\}$ and $\eta\restriction_{E_\star^c} = \omega\restriction_{E_\star^c}$, we still have $\eta \in \{\sX_i^\mathrm{o} = X_i^\mathrm{o}\}$.
    Now, fix any boundary condition $(\bar{\omega}, \bar\sigma) \in \{\sX_i^\mathrm{o} = X_\star^\mathrm{o}\}$. Let $\partial V_\star \subseteq V_\star$ be the subset of vertices which are $\Lambda_n$-adjacent to $V_\star^c$. Observe that for any vertex $v \in \partial V_\star$, every edge $e \in E_\star^c \setminus \{[v_i, v_i - \ez], [v_{i+1}, v_{i+1} + \ez]\}$ incident to $v$ is such that $f_e \in X_i^\mathrm{o}$, and hence $\bar{\omega}_e = 0$. Thus, by the Domain Markov property of the coupled FK--Potts model, the law of $(\omega, \sigma)\restriction_{G_\star}$ under $\phi_n(\cdot \mid \omega\restriction_{E_\star^c} = \bar{\omega}\restriction_{E_\star^c},\, \sigma\restriction_{V_\star^c} = \bar{\sigma}\restriction_{V_\star^c})$ is an FK--Potts model on $G_\star$ with free boundary conditions except $\sigma_{v_i} = \bar\sigma_{v_i}$ if $\bar{\omega}_{[v_i, v_i - \ez]} = 1$ and $\sigma_{v_{i+1}} = \bar\sigma_{v_{i+1}}$ if $\bar{\omega}_{[v_{i+1}, v_{i+1} + \ez]} = 1$.
    Now, any path from $v_{i+1}$ to $\cW_{\bar{\omega}}$ using edges of $E_\star^c$ must cross a face of $\cP_x$ and hence include a closed edge, so $v_{i+1}$ is not in the same component of $\bar{\omega}\restriction_{E_\star^c}$ as any vertices of $\cW_{\bar{\omega}}$.
    Hence, if we condition on $\{\omega\restriction_{E_\star^c} = \bar{\omega}\restriction_{E_\star^c}, \sigma\restriction_{\cW_{\bar{\omega}}} = \bar\sigma\restriction_{\cW_{\bar{\omega}}}\}$, we are in the above situation except we always fix  $\sigma_{v_i} = \bar\sigma_{v_i}$ as $v_i \in \cW_{\bar{\omega}}$, and the boundary condition on $\sigma_{v_{i+1}}$ integrates out via symmetry to being a uniform distribution over colors, which is the same as having no boundary condition.    
\end{proof}

\begin{corollary}\label{cor:DMP-for-Ared}
    In the notation of \cref{lem:DMP-pillar-shell}, let $\cX_\star$ be any event that is measurable w.r.t.\ $(\omega, \sigma)\restriction_{G_\star}$, and let $\cY$ be any event which, conditionally on $\{\sX_i^\mathrm{o} = X_\star^\mathrm{o}\}$, is $\cF$-measurable.
    Then, letting $\nu_\star$ be the coupled FK--Potts model on $G_\star$ with free boundary conditions, we have the following for any event $\cA$:
    \begin{enumerate}
        \item 
        If $\cA$ is measurable w.r.t.\ the random set $V_\star \cap \Ared^c$
        and $\{\sX_i^\mathrm{o} = X_\star^\mathrm{o},\, v_\star\in\Ared^c,\, \cX_\star,\, \cY\} \neq \emptyset$
        then
    \[
        \phi_n(\cA \mid \sX_i^\mathrm{o} = X_\star^\mathrm{o},\,v_\star \in \Ared^c,\, \cX_\star,\, \cY) = \nu_\star(\cA \mid \cX_\star, \sigma_{v_\star} \neq \Red)\,.
    \]
    \item If $\cA$ is measurable w.r.t.\ the random set $V_\star \cap \Ablue$
        and $\{\sX_i^\mathrm{o} = X_\star^\mathrm{o},\, v_\star\in\Ablue,\, \cX_\star,\, \cY\} \neq \emptyset$
        then
        \[\phi_n(\cA \mid \sX_i^\mathrm{o} = X_\star^\mathrm{o},\,v_\star \in \Ablue,\, \cX_\star,\, \cY) = \nu_\star(\cA \mid \cX_\star, \sigma_{v_\star} = \Blue)\,.
    \]
    \end{enumerate}
\end{corollary}
\begin{proof}
    Consider the $\Ared^c$ case (the $\Ablue$ case follows similarly). By \cref{lem:locally-determine-Ared}, the event $\cA$ can be expressed as an event on  $\sigma\restriction_{V_\star}$, so the expression $\nu_\star(\cA)$ is well defined. Note that conditionally on $\{\sX_i^\mathrm{o} = X_\star^\mathrm{o}\}$, the event $\{v_i \in \Ared^c\}$ is also $\cF$-measurable (the vertices surrounding the pillar shell are always in $\Vtop \subseteq \Vred$, so $v_i \in \Ared^c$ iff there is a path of $\Ared^c$ vertices in $\cW_\omega$ from $v_i$ to $\partial \Lambda_n^-$). Thus, it follows from \cref{lem:DMP-pillar-shell} that the law of $(\omega, \sigma)\restriction_{G_\star}$ under the measure $\phi_n(\cdot \mid \sX_i^\mathrm{o} = X_\star^\mathrm{o},\,v_\star \in \Ared^c,\, \cY)$ is the coupled FK--Potts model on $G$ with free boundary conditions except at $v_\star$, whose color is as specified by $\cY \cap  \{v_\star \in \Ared^c\}$.  Since $v_\star=v_i$ is a cut-point, then $v_\star \in \Ared^c$ implies that $\sigma_{v_\star} \neq \Red$, so the boundary condition on $\sigma_{v_\star}$ is some distribution over the non-$\Red$ colors (arising from $\cY \cap \{v_\star \in \Ared^c\}$). However, it is clear via the proof of \cref{lem:locally-determine-Ared} that the actual non-$\Red$ color of $\sigma_{v_\star}$ does not affect the set $V_\star \cap \Ared^c$, so for the conditional probability of $\cA$, we can equivalently condition on $v_\star = \noRed$. In the $\Blue$ case, $v_\star \in \Ablue$ implies that $\sigma_{v_\star} = \Blue$.
\end{proof}

\begin{remark}
    While \cref{cor:DMP-for-Ared} asks for $\cY$ to be measurable w.r.t.\ the edges $\omega\restriction_{E^c_\star}$ and vertex colors $\sigma\restriction_{\cW_\omega}$, our application of this corollary will be for $\cY$ that is measurable w.r.t.\ a smaller subset of edges: those in the interface $E_\star^c \cap \{e : f_e\in\cI\}$ along with those in ${E(U)^c}$ for $U = \{u\in \cP_x\,:\; \hgt(u)> \hgt(v_\star)\}$.    
\end{remark}
\begin{example}\label{ex:application-of-Ared-DMP}
    Oftentimes, we will want to establish an equality of the form 
    \begin{equation}\label{eq:DMP-for-Ared-example}\phi_n(\cA^{\noRed}_{v_i,v_{i+1}}\mid \cI = I,\, \cA^{\noRed}_{x, v_i}) = \phi_n(\cA^{\noRed}_{v_i,v_{i+1}}\mid \sX_i = X_i, v_i \in \Ared^c) 
    \end{equation}
    Observe that fixing $\cI = I$ can be split up as fixing the increment shell $\sX_i^\mathrm{o}$, fixing the hairs inside $\sX_i^\mathrm{o}$, and then fixing the rest of $\cI$. Then, in the notation of the above corollary, we can take $\cX_\star$ to be the event that fixes the hairs inside $\sX_i^\mathrm{o}$, and $\cY$ to be the event that fixes $\cI \setminus \sX_i$, intersected with the event $\cA^{\noRed}_{x, v_i}$. The above corollary then implies that the left hand side of \cref{eq:DMP-for-Ared-example} is equal to $\nu_\star(\cA \mid \cX_\star, \sigma_{v_\star} \neq \Red)$ for some event $\cA$ defined in terms of $\sigma\restriction_V$. A similar argument shows the same for the right hand side, where we additionally note that $X_i$ does not have to be a rooted increment because the measure $\nu_\star(\cdot \mid \cX_\star, \sigma_{v_\star} \neq \Red)$ no longer depends on the location of the graph $G_\star = (V_\star, E_\star)$ inside $\Lambda_n$ (nor the index $i$ of the increment). 
\end{example}

With this Domain Markov type result in hand, we can establish a $\Phi$-monotonicity property for our events of interest. 
\begin{lemma}\label{lem:phi-Phi-monotonicity}
Let $\Phi$ be any map on interfaces sending $E_h^x$ into $\Iso_{x, L, h}$ such that the action of $\Phi$ on $\cP_x$ is to shift increments or replace them by a stack of trivial increments, and to replace the base by a stack of trivial increments with equal height. (In particular, we can take $\Phi$ to be the composition of the sequence of maps used in \cref{lem:nice-space-likely} to move from $E_h^x$ to $\Omega_{h_1, h_2}$.) Then, for any $I, J$ such that $J = \Phi(I)$, we have 
\begin{equation*}
    \phi_n(\cA^{\noRed}_{x, h} \mid I) \leq \phi_n(\cA^{\noRed}_{x, h} \mid J)\,.
\end{equation*}
Moreover, the statement above holds if we replace $\cA^{\noRed}_{x, h}$ by $\cA^{\Blue}_{x, h}$.
\end{lemma}
\begin{proof}
Let $T$ be the index of the increment in $P_x^I$ that first reaches height $h$. Let $X_i$ be the $i$-th increment of the pillar $P_x^I$. By definition, we can always write 
    \begin{equation*}
        \phi_n(\cA^{\noRed}_{x, h} \mid I) = \phi_n(\cA^{\noRed}_{x,v_1} \mid I)\phi_n(\cA^{\noRed}_{v_T,h} \mid I, \cA^{\noRed}_{x, v_T})\prod_{i = 1}^{T-1} \phi_n(\cA^{\noRed}_{v_i,v_{i+1}}\mid I,\, \cA^{\noRed}_{x, v_i})\,. 
    \end{equation*}
Then, by \cref{cor:DMP-for-Ared}, we can write
\begin{equation}\label{eq:nored-pillar-for-I}
\phi_n(\cA^{\noRed}_{x, h}\mid I) = \phi_n(\cA^{\noRed}_{x,v_1} \mid I)\phi_n(\cA^{\noRed}_{v_T,h} \mid X_T, v_T \in \Ared^c)\prod_{i = 1}^{T-1} \phi_n(\cA^{\noRed}_{v_i,v_{i+1}}\mid X_i, v_i \in \Ared^c)\,.
\end{equation} 

To write an analogous equation for $\phi_n(\cA^{\noRed}_{x, h} \mid J)$, let $Y_i$ correspond to either the shifted copy of $X_i$ in $P_x^J$, or the stack of trivial increments in $P_x^J$ from $\hgt(v_i)$ to $\hgt(v_{i+1})$. Let $Y_0$ be the stack of trivial increments from height $1/2$ to $\hgt(v_1)$. Finally, let $w_i$ correspond to the cut-point in $\cP_x^J$ at height $\hgt(v_i)$, with $w_0 = x$. Then, applying \cref{cor:DMP-for-Ared} for $J$, we can write
\begin{equation}\label{eq:nored-pillar-for-J}
\phi_n(\cA^{\noRed}_{x, h} \mid J) = \phi_n(x \in \Ared^c\mid J)\phi_n(\cA^{\noRed}_{w_T,h} \mid Y_T, w_T \in \Ared^c)\prod_{i = 0}^{T-1}\phi_n(\cA^{\noRed}_{w_i,w_{i+1}}\mid Y_i, w_i \in \Ared^c)
\end{equation}
Now comparing the above two equations, we see that if $Y_i$ is a shifted copy of $X_i$, then their corresponding terms are equal (see \cref{ex:application-of-Ared-DMP} regarding the shift invariance). Otherwise, we can upper bound the remaining terms in \cref{eq:nored-pillar-for-I} by 1. To see that the remaining terms in \cref{eq:nored-pillar-for-J} are all equal to 1, observe that in a stack of trivial increments, all the vertices inside are guaranteed to be in the same open cluster (and hence have the same color under the coupling). Moreover, we argued in \cref{clm:x-in-Vbot} that on $\Iso_{x, L, h}$, we deterministically have $x \in \Vbot$ (and hence $x \in \Ared^c$). Since $J \in \Phi(E_h^x) \subseteq \Iso_{x, L, h}$, then in the above equation, $\phi_n(x \in \Ared^c\mid J) = 1$. 
\end{proof}

The next lemma shows how the previous monotonicity result can be used to establish the comparison of our events under the two measures $\phi_n(\cdot \mid E_h^x)$ and $\phi_n(\cdot \mid \Omega_{h_1, h_2})$. The lemma may be of independent interest, and is stated in a more general setting.
\begin{lemma}\label{lem:move-to-nice-space}
Let $\Phi$ be any map on interfaces sending $E_h^x$ into itself such that for any $J \in \Phi(E_h^x)$, we have $\bar{\mu}_n(\Phi^{-1}(J)) \leq (1+\epsilon_\beta)\bar{\mu}_n(J)$. Let $\cA$ be any event (possibly in the joint space of configurations $(\omega, \sigma)$) such that 
\begin{enumerate}
    \item $\cA \subseteq E_h^x$

    \item\label{it:phi-Phi-monotonicity} For any $I, J$ such that $J = \Phi(I)$, we have $\phi_n(\cA \mid I) \leq \phi_n(\cA \mid J)$
\end{enumerate}
Then, for any space $\Omega$ such that $\Phi(E_h^x) \subseteq \Omega \subseteq E_h^x$, there exists a constant $\epsilon_\beta$ such that 
\begin{equation}
\left|\frac{\phi_n(\cA \mid \Omega)}{\phi_n(\cA \mid E_h^x)} - 1\right| \leq \epsilon_\beta
\end{equation}
\end{lemma}
\begin{proof} 
The conditions on $\Phi$ easily imply that $\bar{\mu}_n(\Phi(E_h^x) \mid E_h^x) \geq 1 - \epsilon_\beta$, and hence $\bar{\mu}_n(\Omega \mid E_h^x) \geq 1 - \epsilon_\beta$. Together with the condition that $\cA \subseteq E_h^x$, we compute that
\[
\phi_n(\cA \mid \Omega) = \frac{\phi_n(\cA,\, \Omega)}{\phi_n(\Omega)} \leq (1+\epsilon_\beta)\frac{\phi_n(\cA , E_h^x)}{\phi_n(E_h^x)}\leq (1+\epsilon_\beta)\phi_n(\cA \mid E_h^x)\,.
\]
 By a similar computation, we see that in order to prove
\[
\phi_n(\cA \mid E_h^x) \leq (1+\epsilon_\beta) \phi_n(\cA \mid \Omega)\,,
\]
it suffices to show that
\begin{equation*}
\phi_n\left(\Omega\mid \cA\right) \geq 1 - \epsilon_\beta\,.
\end{equation*}
Since $\cA \subseteq E_h^x$, we can first write
\begin{align*}
    \phi_n(\cA) &= \sum_{I \in E_h^x} \phi_n(\cA \mid I)\bar{\mu}_n(I)\\
    &= \sum_{J \in \Phi(E_h^x)}\;\sum_{I \in \Phi^{-1}(J)}\phi_n(\cA \mid I)\bar{\mu}_n(I)\,.
\end{align*}
Using \cref{it:phi-Phi-monotonicity} followed by the bound $\bar{\mu}_n(\Phi^{-1}(J)) \leq (1+\epsilon_\beta)\bar{\mu}_n(J)$, we have
\begin{align*}
    \sum_{J \in \Phi(E_h^x)}\;\sum_{I \in \Phi^{-1}(J)}\phi_n(\cA \mid I)\bar{\mu}_n(I)
    &\leq (1+\epsilon_\beta)\sum_{J \in \Phi(E_h^x)} \phi_n(\cA \mid J)\bar{\mu}_n(J) \\
    &\leq (1+\epsilon_\beta)\sum_{J \in \Omega} \phi_n(\cA \mid J)\bar{\mu}_n(J) \\
    &= (1+\epsilon_\beta)\phi_n(\cA \mid \Omega)\,.\qedhere
\end{align*}
\end{proof}

\begin{remark}\label{rem:Phi-condition}
    Note that if $\Phi$ is the composition of the sequence of maps used in \cref{lem:nice-space-likely} to move from $E_h^x$ to $\Omega_{h_1, h_2}$, then $\Phi$ satisfies the conditions of the above lemma. Indeed, each map $\Psi$ in the composition satisfies the energy bound that if $\fm(I; \Psi(I)) = k$, then $\bar{\mu}_n(I) \leq e^{-(\beta - C)k}\bar{\mu}_n(\Psi(I))$ for some constant $C$, as well as the entropy bound that the number of preimages $I \in \Psi^{-1}(J)$ such that $\fm(I;J) = k$ is bounded by $s^k$ for some constant $s$. Together, this implies that $\bar{\mu}_n(\Psi^{-1}(J)) \leq (1 + \epsilon_\beta)\bar{\mu}_n(J)$ for $\epsilon_\beta = \tilde{C}e^{-\beta}$, and clearly the same bound holds when taking a composition of such maps for a different $\epsilon_\beta$. 
\end{remark}

\begin{lemma}\label{lem:3-to-3-map}
In the setting of \cref{prop:RC-Potts-diff-rate}, there exists $\epsilon_\beta$ such that for any pillar $P = P_B \times P^T \in  \Omega_{h_1, h_2}$,
\begin{equation}
\bar{\mu}_n(\cP_x = P_B\times P^T \mid \Omega_{h_1, h_2}) \leq (1 + \epsilon_\beta)\bar{\mu}_n(\cP_x = P_B \mid \Omega_{h_1})\bar{\mu}_n(\cP_x = P^T \mid \Omega_{h_2})
\end{equation}
\end{lemma}
\begin{proof}
For any interface $\cI$, we can denote it in terms of the pillar at $x$ and the rest of the interface, $\cI = (\cP_x, \cI \setminus \cP_x)$. Note that in general, by the definition of the truncated interface $\cI \setminus \cP_x$ (in \cref{def:truncated-interface}), there are possibly some extra faces added to fill in the gaps created by removing the pillar $\cP_x$, and it is a priori ambiguous from the pair $(\cP_x, \cI \setminus \cP_x)$ which of these faces were originally in $\cI$ and which needed to be added in. However, for interfaces in $\Iso_{x, L, h}$ (and hence for all the interfaces considered here), there is no ambiguity as the cut-point criteria at $x$ implies that the only face that might need to be added in is $f_{[x, x - \ez]}$, yet this face is also required to be missing from $\cI$ as part of the definition of $\Iso_{x, L, h}$. Now, recalling the notation in \cref{rem:cut-paste-pillars}, suppose we have three interfaces, $(P_B \times P^T, A) \in \Omega_{h_1, h_2}, (Q_B, A') \in \Omega_{h_1}, (Q^T, A'') \in \Omega_{h_2}$. For more concise notation, we write $\bar{\mu}_n(P) = \bar{\mu}_n(\cP_x = P)$ and $ \bar{\mu}_n(I) = \bar{\mu}_n(\cI = I)$. We have the following inequality
\begin{align*}
&\bar{\mu}_n(P_B\times P^T \mid \Omega_{h_1, h_2}) - \bar{\mu}_n(P_B \mid  \Omega_{h_1})\bar{\mu}_n(P^T \mid \Omega_{h_2})\\
= \sum_{\substack{A, A', A''\\ Q_B, Q^T}} &\bar{\mu}_n((P_B \times P^T, A) \mid\Omega_{h_1, h_2})\bar{\mu}_n((Q_B, A')\mid \Omega_{h_1})\bar{\mu}_n((Q^T, A'')\mid \Omega_{h_2})\\
- &\bar{\mu}_n((Q_B \times Q^T, A) \mid \Omega_{h_1, h_2})\bar{\mu}_n((P_B, A')\mid  \Omega_{h_1})\bar{\mu}_n((P^T, A'')\mid \Omega_{h_2})
\end{align*}
Here, the sum is over all possible truncated interfaces $A, A', A''$ that satisfy the respective wall requirements, and over all possible pillars $Q_B, Q^T$ that satisfy the pillar requirements of $\Omega_{h_1}, \Omega_{h_2}$ respectively. We can factor out the term being subtracted and cancel out the conditional events so that the above is bounded by 
\begin{align}
\sum_{\substack{A, A', A''\\ Q_B, Q^T}} &\bar{\mu}_n((Q_B \times Q^T, A) \mid \Omega_{h_1, h_2})\bar{\mu}_n((P_B, A')\mid  \Omega_{h_1})\bar{\mu}_n((P^T, A'')\mid \Omega_{h_2})\nonumber \\
\label{eq:3-to-3-ratio}
&\cdot\left| \frac{\bar{\mu}_n((P_B \times P^T, A))\bar{\mu}_n((Q_B, A'))\bar{\mu}_n((Q^T, A'')}{\bar{\mu}_n((Q_B \times Q^T, A))\bar{\mu}_n((P_B, A'))\bar{\mu}_n((P^T, A''))}-1\right|
\end{align}
If we are able to bound the absolute value term in \cref{eq:3-to-3-ratio} by $\epsilon_\beta$, then we would be done since the rest of the sum is equal to $\bar{\mu}_n(P_B \mid \Omega_{h_1})\bar{\mu}_n(P^T \mid \Omega_{h_2})$.

To bound \cref{eq:3-to-3-ratio}, we plug in the cluster expansion expressions from \cref{eq:CE} for each term in the fraction above. There are 6 interfaces that we need to refer to; in numerator from left to right, let them be denoted $I_P^P, I'_Q, I''^Q$, and in the denominator let them be denoted $I_Q^Q, I'_P, I''^P$, as drawn in \cref{fig:3-to-3}.

\begin{figure}
\vspace{-1.25cm}
    \begin{tikzpicture}[font=\small]
   \node (fig1) at (0,.9) {
    	\includegraphics[width=0.3\textwidth]{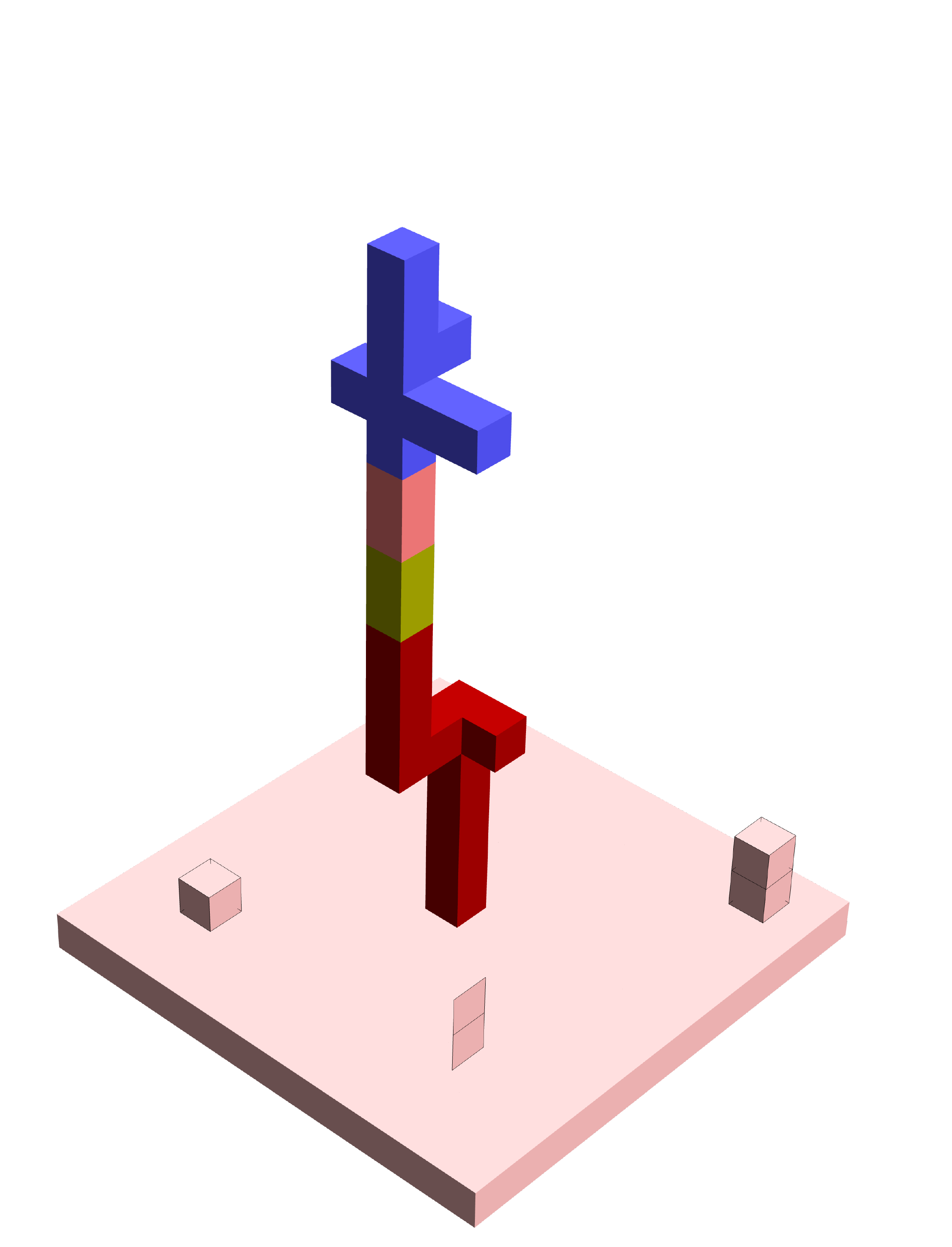}};
    \node[font=\large] at (-1.6,-1.6) {$I_P^P$};
    \node[color=blue!75!black] at (-1,2.5) {$P^T$};
    \node[color=red!75!black] at (-1,0.4) {$P^B$};
    \node[color=yellow!50!black] at (.4,1) {$Z_1$};
    \node[color=pink] at (.4,1.5) {$Z_2$};
   \node (fig2) at (5,0) {
   	\includegraphics[width=0.3\textwidth]{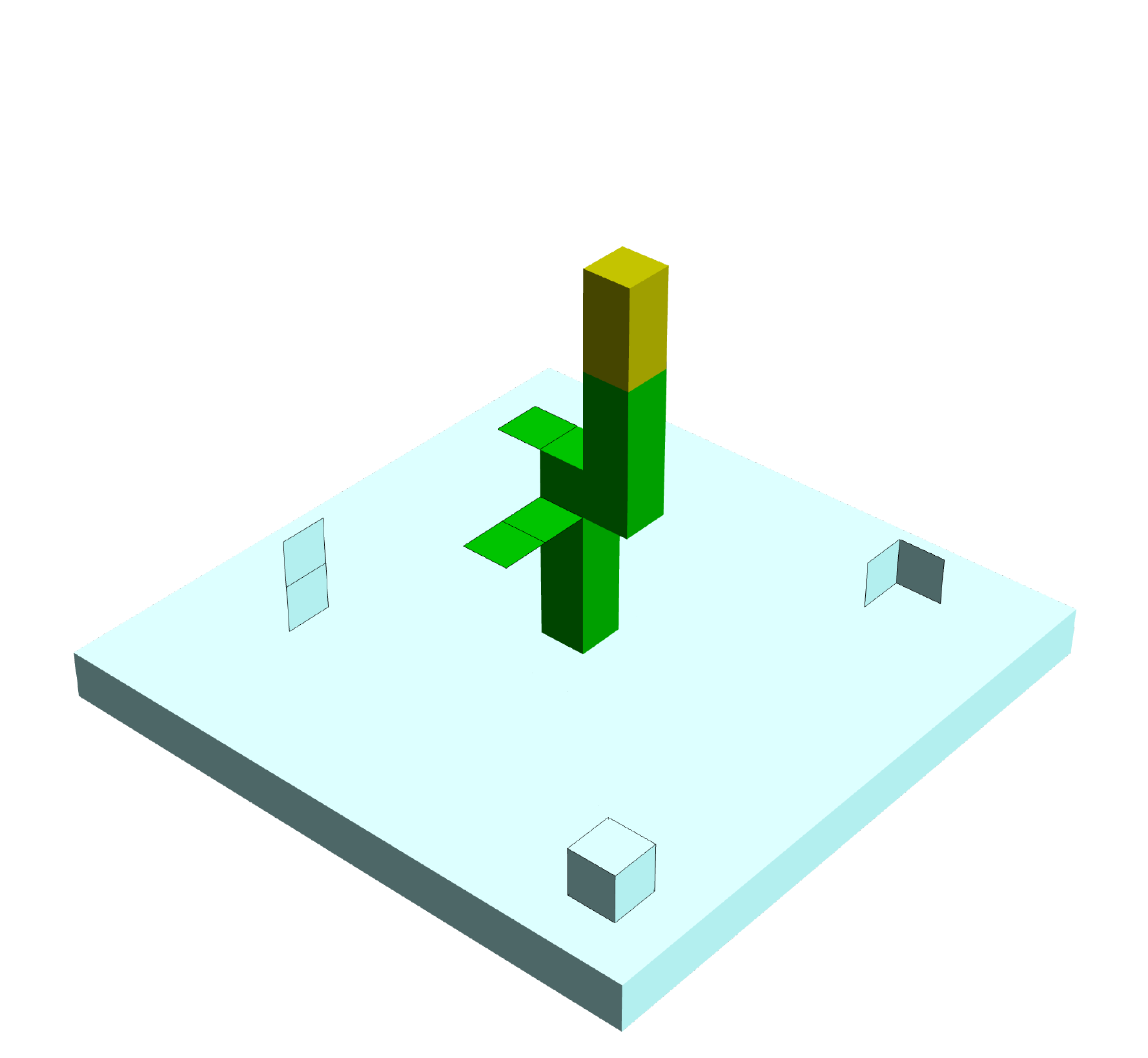}};
    \node[font=\large] at (3.7,-1.6) {$I'_Q$};
   \node[color=green!50!black] at (4.3,0.3) {$Q^B$};
    \node[color=yellow!50!black] at (5.8,1) {$Z_1$};
    \node (fig3) at (10,.2) {
   	\includegraphics[width=0.3\textwidth]{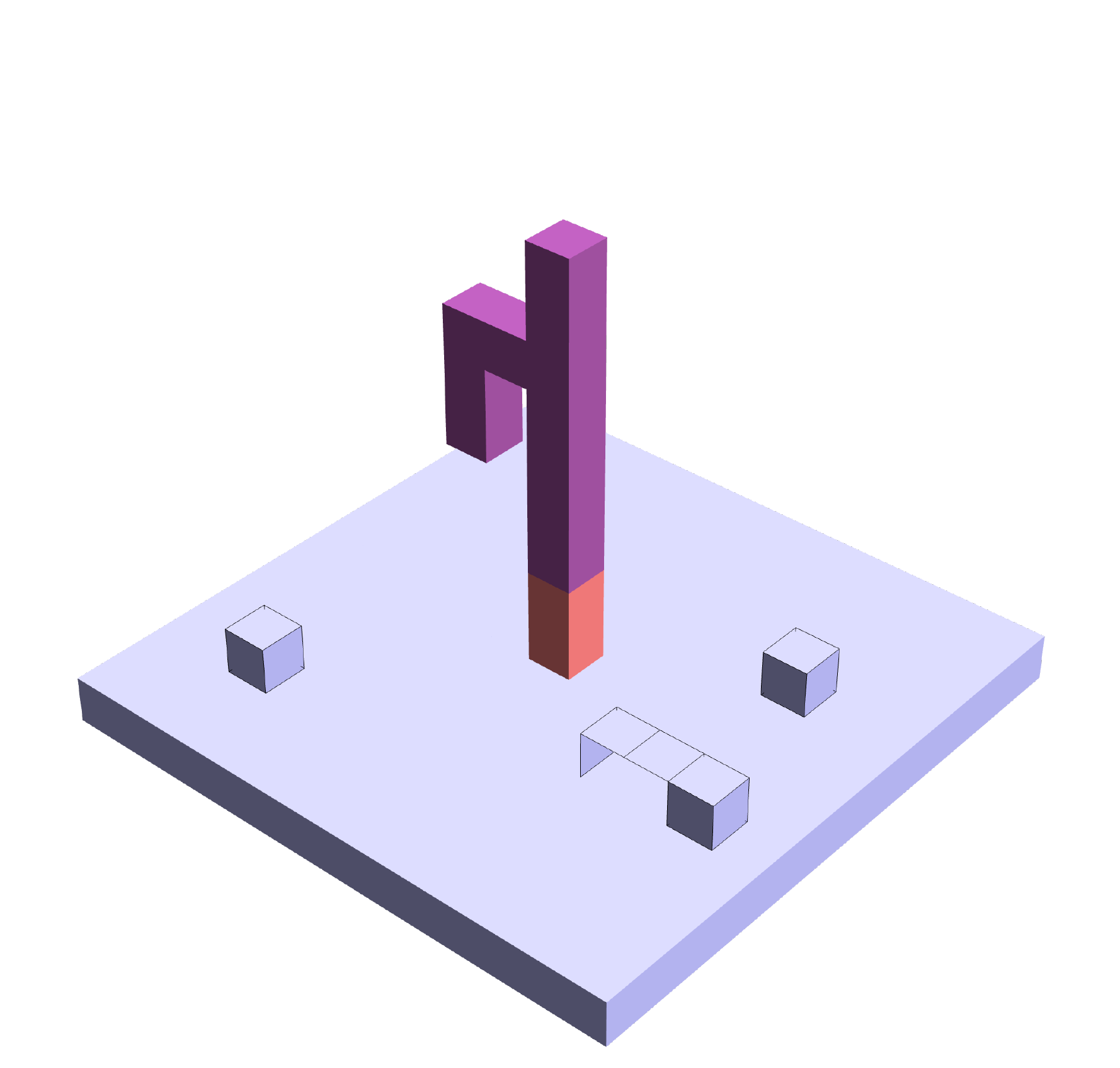}};
   \node[font=\large] at (8.7,-1.6) {$I''_Q$};
   \node[color=purple!50!black] at (9.,1) {$Q^T$};
    \node[color=pink!75!black] at (10.55,-0.15) {$Z_2$};
    \begin{scope}[shift={(0,-6)}]
    \node (fig1b) at (0,0.9) {
    	\includegraphics[width=0.3\textwidth]{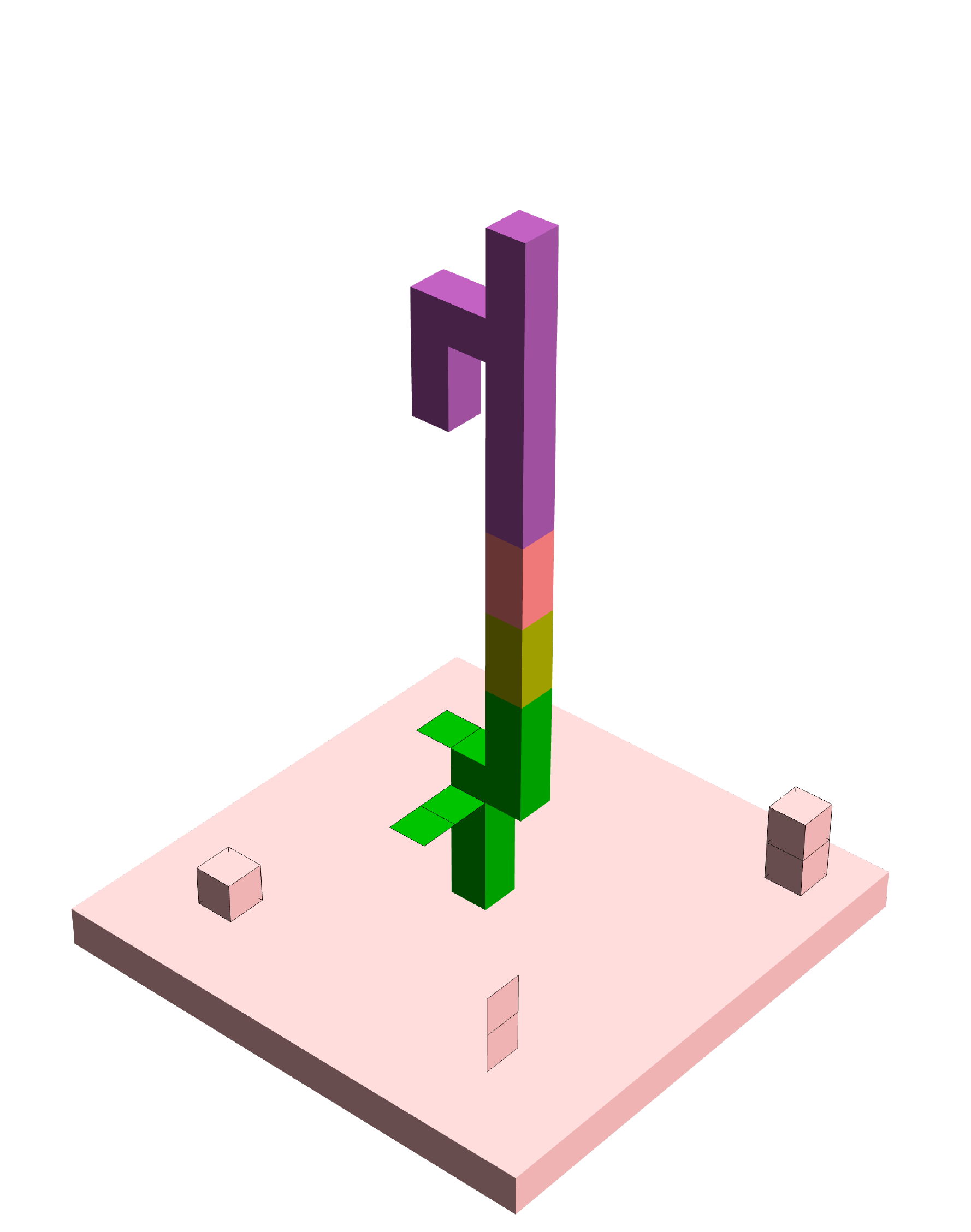}};
   \node[font=\large] at (-1.6,-1.6) {$I_Q^Q$};  
   \node[color=green!50!black] at (-0.7,0.1) {$Q^B$};
    \node[color=yellow!50!black] at (.8,.65) {$Z_1$};
    \node[color=pink] at (.8,1.15) {$Z_2$};
    \node[color=purple!50!black] at (-0.7,2) {$Q^T$};
   \node (fig2b) at (5,0.4) {
   	\includegraphics[width=0.3\textwidth]
    {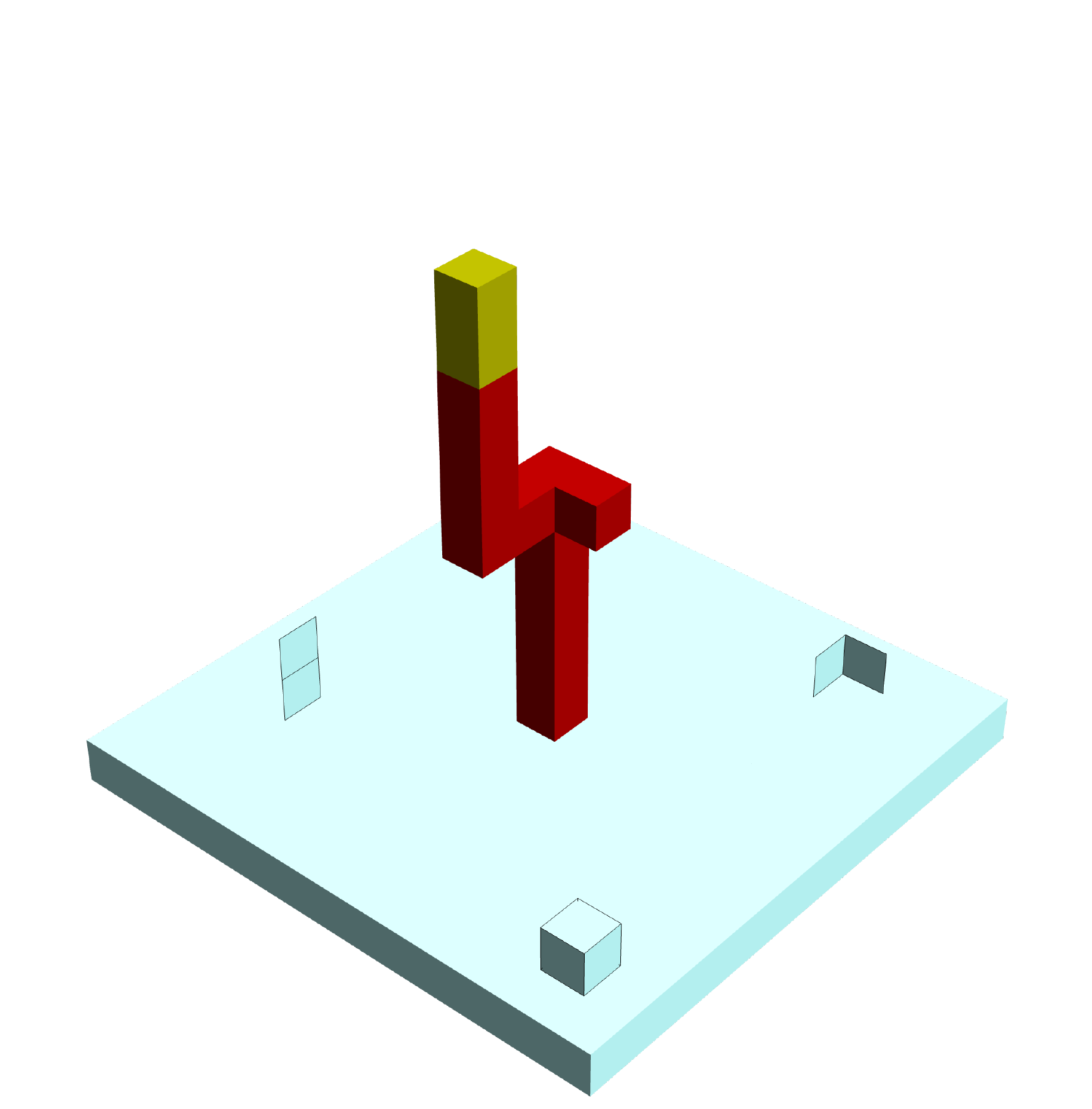}};
    \node[font=\large] at (3.7,-1.6) {$I'_P$};
    \node[color=red!75!black] at (4.1,0.4) {$P^B$};
    \node[color=yellow!50!black] at (5.4,1.45) {$Z_1$};
   \node (fig3b) at (10,0.2) {
   	\includegraphics[width=0.3\textwidth]{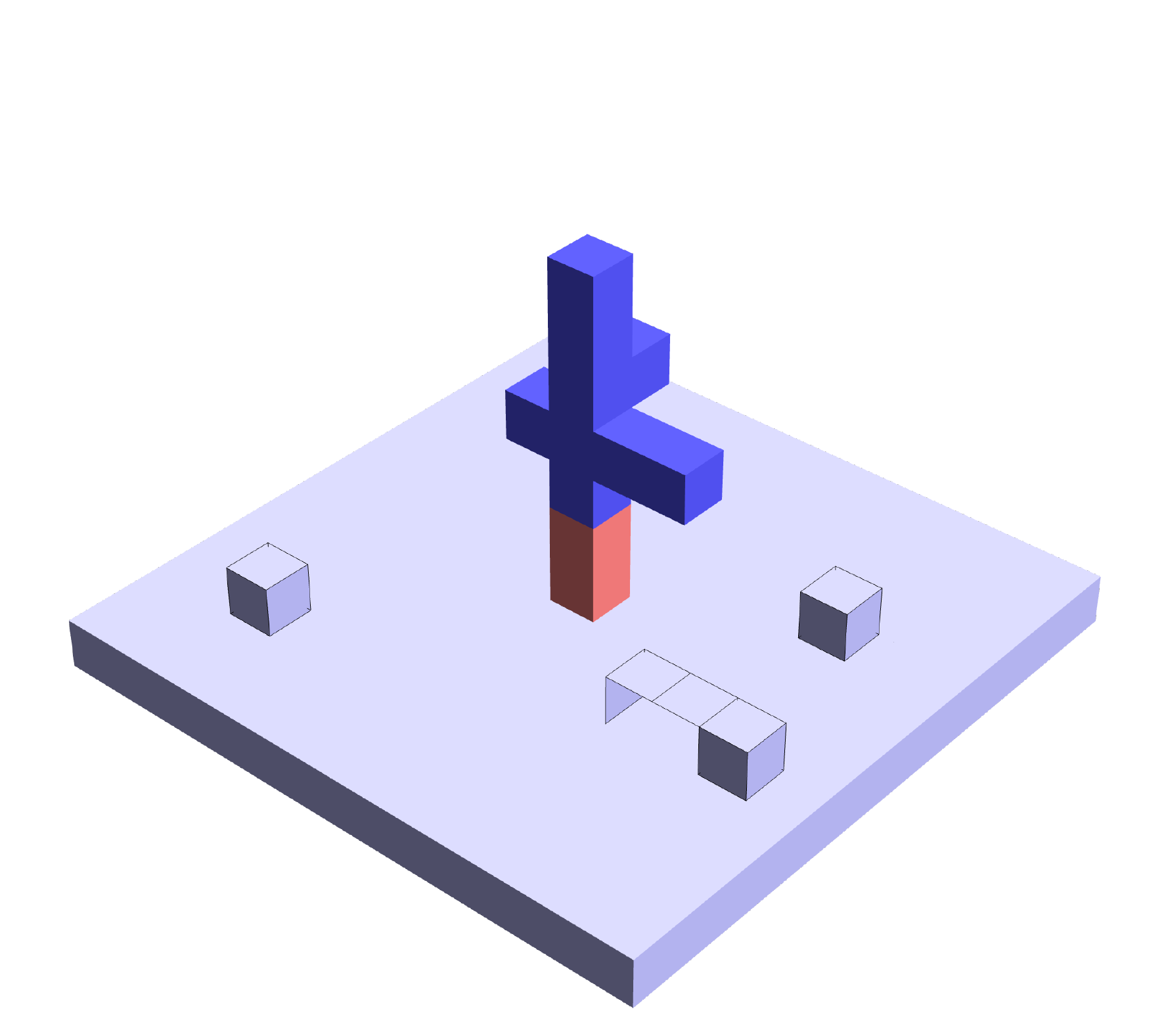}};
    \node[font=\large] at (8.7,-1.6) {$I''_P$};
    \node[color=blue!75!black] at (9.25,1) {$P^T$};
    \node[color=pink!75!black] at (10.55,-.1) {$Z_2$};
    \end{scope}
    \end{tikzpicture}
    \vspace{-0.25cm}
    \caption{The 3-to-3 map sends the top three interfaces to the bottom three. The figure is color-coded according to which faces are paired together in the cluster expansion computation. (See how the terms in \cref{eq:3-to-3-g-terms} are separated into the terms in \cref{eq:3-to-3-I_P^P-sums,eq:3-to-3-I_Q^Q-sums,eq:3-to-3-A'-A''-sums}.)}
    \label{fig:3-to-3} 
\end{figure}

Note that the two sets of interfaces have the same number of total faces, open clusters, and contributions to the term $|\partial \cI|$ in the cluster expansion. Indeed, the relationship between the interfaces is a cut and paste operation on the pillars, and furthermore \cref{prop:cone-seperation} applies for all of these interfaces, ensuring that there is no interaction between the pillars and the surrounding walls that could potentially affect one of the terms above in the cluster expansions. Thus, it remains to control the $g$-terms, 
\begin{equation}\label{eq:3-to-3-g-terms}
\exp\bigg[\sum_{f \in I_P^P}\g(f, I_P^P) + \sum_{f \in I'_Q} \g(f, I'_Q) + \sum_{f \in I''^Q} \g(f, I''^Q) - \sum_{f \in I_Q^Q}\g(f, I_Q^Q) - \sum_{f \in I'_P} \g(f, I'_P) - \sum_{f \in I''^P} \g(f, I''^P)\bigg]\,.
\end{equation}
As in \cref{fig:3-to-3}, let the top $L^3/2$ increments of $P_B, Q_B$ be referred to as $Z_1$, and the bottom $L^3/2$ increments of $P^T, Q^T$ be referred to as $Z_2$. Note that the top $L^3$ increments of $P_B$ and $Q_B$ are trivial, so there is a $L^3/2$ buffer distance between $Z_1$ and the first non-trivial increment of $P_B, Q_B$, and likewise for $Z_2$ with $P^T, Q^T$. We split up the terms of the sum that involve the interface $I_P^P$ as follows:
\begin{align}\label{eq:3-to-3-I_P^P-sums}
&\sum_{f \in P_B \setminus Z_1}|\g(f, I_P^P) - \g(f, I'_P)| + \sum_{f \in P^T \setminus Z_2} |\g(f, I_P^P) - \g(\theta f, I''^P)| + \sum_{f \in A} |\g(f, I_P^P) - \g(f, I_Q^Q)|\nonumber \\
&+ \sum_{f \in Z_1} |\g(f, I_P^P) - \g(\theta f, I_Q^Q)| + \sum_{f \in Z_2} |\g(f, I_P^P) - \g(\theta f, I_Q^Q)|\nonumber \\
&= \Xi_1 + \Xi_2 + \Xi_3 + \Xi_4 + \Xi_5\,.
\end{align}
Here, the sums are all over faces of the interface $I_P^P$, and $\theta f$ is the shifted copy of $f$ in the corresponding interface. Although each $\theta$ is a different shift depending on the target interface, none of the computations that follow depend on the particular shift so we will not distinguish between them and call them all $\theta$.

Begin with $\Xi_1$. Using the bounds in \cref{eq:g-bound-2-faces} and \cref{eq:pillar-in-cone-1}, the part of the sum where $r(f, I_P^P; f, I'_P)$ is attained by a face in $A$ or $A'$ is bounded by $\bar{C}e^{-\bar{c}L}$. Otherwise, if $r$ is attained by a face in $P^T$, suppose that the first increment of $P^T$ has index $j_0$. Then, using condition (3) of $\Omega$ to control the size of the increments,
\begin{align*}
\sum_{f \in P_B \setminus Z_1}\;\sum_{g \in P^T} Ke^{-cd(f, g)} &\leq \sum_{j \geq 0} \;\sum_{g \in X_{j_0+j}}\;\sum_{\substack{f \in \sF(\Z^3),\\ \hgt(f) \leq \hgt(P_B\setminus Z_1)}} Ke^{-cd(f,g)}\\
&\leq \sum_{j \geq 0}|\sF(X_{j_0+j})|Ke^{-c(L^3/2 + j)}\\
&\leq \sum_{j \geq 0}\tilde{K}je^{-c(L^3/2 + j)} \leq \tilde{K}e^{-cL^3/2}
\end{align*}
The second sum $\Xi_2$ is bounded similarly. Again, the terms where $r$ is attained by a face in $A$ or $A''$ is bounded by $\bar{C}e^{-\bar{c}L/2}$ using \cref{eq:pillar-in-cone-2}. Otherwise, when $r$ is attained by a face in $P_B$,  we have a similar computation as above:
\begin{align*}
\sum_{f \in P_B}\;\sum_{g \in P^T \setminus Z_2} Ke^{-cd(f, g)} &\leq \sum_{j \geq L^3/2} \;\sum_{g \in X_{j_0 + j}} \;\sum_{\substack{f \in \sF(\Z^3),\\ \hgt(f) \leq \hgt(P_B)}} Ke^{-cd(f,g)}\\
&\leq\sum_{j \geq L^3/2}|\sF(X_{j_0+j})|Ke^{-cj}\\
&\leq \sum_{j \geq L^3/2}\tilde{K}je^{-cj} \leq \tilde{K}e^{-cL^3/2}
\end{align*}
The third sum $\Xi_3$ is immediately bounded by $\bar{C}e^{-\bar{c}L}$ using \cref{eq:pillar-in-cone-2}.

Finally, the fourth and fifth sums $\Xi_4, \Xi_5$ are both bounded by $2L^3Ke^{-cL^3/2}$ since there are $2L^3$ faces, and the buffer of $L^3/2$ increments above and below ensures that the distance $r$ to a face where the interfaces differ is at least $L^3/2$.

Now for the remaining terms in \cref{eq:3-to-3-g-terms}, 
the remaining faces in $I_Q^Q$ are captured in the sums
\begin{equation}\label{eq:3-to-3-I_Q^Q-sums}
\sum_{f \in Q_B \setminus Z_1}|\g(f, I_Q^Q) - \g(f, I'_Q)| + \sum_{f \in Q^T \setminus Z_2} |\g(f, I_Q^Q) - \g(\theta f, I''^Q)|.
\end{equation}
These sums can be bounded above by $Ke^{-cL^3/2}$ for some constants $c, K$ in the same way as $\Xi_1$ and $\Xi_2$ above. Furthermore, the sums
\begin{equation}\label{eq:3-to-3-A'-A''-sums}
\sum_{f \in A'} |\g(f, I'_Q) - \g(f,I'_P)| + \sum_{f \in A''} |\g(f, I''^Q) - \g(f, I''^P)|
\end{equation}
are bounded by $\bar{C}e^{-\bar{c}L}$ using \cref{eq:pillar-in-cone-2}. It remains to take care of the copies of $Z_1, Z_2$ in the interfaces $I'_Q, I'_P, I''^Q, I''^P$. We have
\begin{equation*}
\sum_{f \in Z_1} |\g(f, I'_Q) - \g(\theta f, I'_P)| \leq \bar{C}e^{-\bar{c}L} \wedge 2L^3Ke^{-cL^3/2}
\end{equation*}
since interactions with walls of $A'$ are handled by \cref{eq:pillar-in-cone-2} and interactions with $P_B$ and $Q_B$ are handled similarly to $\Xi_4$ and $\Xi_5$ above. We also have 
\begin{equation*}
\sum_{f \in Z_2} |\g(f, I''^Q) - \g(\theta f, I''^P)| \leq \bar{C}e^{-\bar{c}L} \wedge 2L^3Ke^{-cL^3/2}
\end{equation*}
by the same reasoning, except we need to use \cref{eq:pillar-in-cone-1} this time instead. 

Thus, putting everything together and recalling that we could take $L = L_\beta \uparrow \infty$ as $\beta\uparrow\infty$, we get that
\begin{equation*}
\frac{\bar{\mu}_n((P_B \times P^T, A))\bar{\mu}_n((Q_B, A'))\bar{\mu}_n((Q^T, A'')}{\bar{\mu}_n((Q_B \times Q^T, A))\bar{\mu}_n((P_B, A'))\bar{\mu}_n((P^T, A''))} \in [e^{-Ce^{-cL_\beta}}, e^{Ce^{-cL_\beta}}]
\end{equation*}
for some different constants $C, c >0$.
\end{proof}

We are now ready to prove the submultiplicativity statment of \cref{prop:submultiplicativity-Potts}:
\begin{proof}[\textbf{\emph{Proof of \cref{prop:submultiplicativity-Potts}}}]
We will write the proof in the notation of the $\noRed$ case, noting that the previous lemmas (and hence this proof) apply to the $\Blue$ case as well. Let $\Phi$ be defined as the composition of the sequence of maps used in \cref{lem:nice-space-likely} to move from $E_h^x$ to $\Omega_{h_1, h_2}$. By applying \cref{lem:move-to-nice-space} for this choice of $\Phi$, $\Omega = \Omega_{h_1, h_2}$, and $\cA = \cA^{\noRed}_{x, h}$, we have
\begin{equation*}
    \phi_n(\cA^{\noRed}_{x, h_1+h_2}\mid E_{h_1+h_2}^x) \leq (1+\epsilon_\beta)\phi_n(\cA^{\noRed}_{x, h_1+h_2}\mid \Omega_{h_1, h_2})\,.
\end{equation*}
We can always decompose the space $\Omega_{h_1, h_2}$ according to the pillar $\cP_x$ to write
\begin{equation}\label{eq:split-Omega-by-P}
    \phi_n(\cA^{\noRed}_{x, h_1+h_2}\mid \Omega_{h_1, h_2}) = \sum_{P \in \Omega_{h_1, h_2}}\phi_n(\cA^{\noRed}_{x, h_1+h_2} \mid \cP_x = P,\, \Omega_{h_1, h_2})\bar{\mu}_n(\cP_x = P \mid \Omega_{h_1, h_2})
\end{equation}
As argued in \cref{clm:x-in-Vbot}, we know that on the event $\Omega_{h_1, h_2} \subseteq \Iso_{x, L, h}$, we have $x \in \Vbot$ and hence $x \in \Ared^c$. Since $x$ is a cut-point of $P$, we can apply \cref{cor:DMP-for-Ared} (with the convention that $v_0 = v_1 = x$) to get that
\begin{align}\label{eq:split-Nred-by-increments}
    \phi_n(\cA^{\noRed}_{x, h_1+h_2} \mid \cP_x = P,\, \Omega_{h_1, h_2}) = &\prod_{i=1}^{T-1}\phi_n(\cA^{\noRed}_{v_i, v_{i+1}} \mid \mbox{$\bigcap \{\sX_j = X_j\}$},\, \Omega_{h_1, h_2},\, \cA^{\noRed}_{x, v_i})\nonumber\\
    =&\prod_{i=1}^{T-1}\phi_n(\cA^{\noRed}_{v_i, v_{i+1}} \mid \{\sX_i = X_i\},\, v_i \in \Ared^c)
\end{align}
where $X_T$ is the last increment of $P$. (Recall that in $\Omega_{h_1, h_2}$, the pillar is capped at height $h_1 + h_2$ and the last increment is trivial). Now recall by \cref{rem:cut-paste-pillars} that we can always write $P = P_B \times P^T$ and change the sum over $P \in \Omega_{h_1, h_2}$ into a double sum over $P_B \in \Omega_{h_1}$ and $P^T \in \Omega_{h_2}$. Let $y$ be the cut-point of $P_B \times P^T$ with height $h_1 + 1/2$, and let $i^*$ be index of the trivial increment with vertices $y, y - \ez$ (so that $y = v_{i^*+1}$). First, note that since $X_{i^*}$ is a trivial increment, then $y$ and $y - \ez$ are in the same open cluster, and hence
\[\phi_n(\cA^{\noRed}_{v_{i^*}, v_{i^*+1}} \mid \{\sX_{i^*} = X_{i^*}\},\, v_{i^*} \in \Ared^c) = 1\]
Next, observe that the event $\cP_x = P_B$ is equal to the event that $\cP_x$ has increments $(X_1,\, \dots,\, X_{i^*-1})$, while $\cP_x = P^T$ is equal to the event that $\cP_x$ has increments $(X_{i^*+1},\, \ldots,\, X_{T})$. Thus, by applying \cref{cor:DMP-for-Ared} again (and noting \cref{ex:application-of-Ared-DMP} following it with regards to the shift from being rooted at $y$ to being rooted at $x$), we have that the product in \cref{eq:split-Nred-by-increments} above is equal to
\begin{equation}
    \phi_n(\cA^{\noRed}_{x, h_1} \mid \cP_x = P_B,\, \Omega_{h_1})\phi(\cA^{\noRed}_{x, h_2} \mid \cP_x = P^T,\, \Omega_{h_2})\,.
\end{equation}

Combining the above three equations with \cref{lem:3-to-3-map}, we have
\begin{align}
    \phi_n(\cA^{\noRed}_{x, h_1+h_2}&\mid \Omega_{h_1, h_2})\nonumber\\
    &\leq (1+\epsilon_\beta)\sum_{P_B \in \Omega_{h_1}}\sum_{P^T \in \Omega_{h_2}} \phi_n(\cA^{\noRed}_{x, h_1} \mid P_B,\, \Omega_{h_1})\phi(\cA^{\noRed}_{x, h_2} \mid P^T,\, \Omega_{h_2})\bar{\mu}_n(P_B\mid \Omega_{h_1})\bar{\mu}_n(P^T\mid \Omega_{h_2})\nonumber\\
    &=(1+\epsilon_\beta)\phi_n(\cA^{\noRed}_{x, h_1}\mid \Omega_{h_1})\phi_n(\cA^{\noRed}_{x, h_2}\mid \Omega_{h_2})\,.
\end{align}
Finally, we can conclude by applying \cref{lem:move-to-nice-space} again for $\Omega = \Omega_{h_1}$ and $\Omega =\Omega_{h_2}$. 
\end{proof}
Thus, we have proved the submultiplicativity statement \cref{prop:submultiplicativity-Potts}. By using the decorrelation estimates of \cref{cor:Potts-decorrelation-1}, we can generalize to the case where $x, n$ on the right hand side can depend on $h_1$ and $h_2$, as long as we still have $1 \ll h_i \ll n_{h_i}$ and $d(x_{h_i}, \partial \Lambda_{n_{h_i}}) \gg h_i$:
\begin{equation*}
    \phi_n(\cA^{\noRed}_{x_h, h} \mid E_h^x) \leq (1+\epsilon_\beta+o_{h_1}(1)+o_{h_2}(1))\phi_{n_{h_1}}(\cA^{\noRed}_{x_{h_1}, h_1} \mid E_{h_1}^{x_{h_1}})\phi_{n_{h_2}}(\cA^{\noRed}_{x_{h_2}, h_2} \mid E_{h_2}^{x_{h_2}})
\end{equation*}
The analogous statement for $\Blue$ also holds in the same way. Now we can apply Fekete's Lemma to prove the existence of the first two limits in \cref{prop:RC-Potts-diff-rate}.

\subsection{Establishing the rate for the bottom interface}

We will now prove the large deviation rate for the event $\cA^{\Bot}_{x, h}$ as in \cref{eq:bot-rate}. This case is substantially easier because we do not need to work on the joint space of configurations $(\omega, \sigma)$. Moreover, defining $x \xleftrightarrow{\omega} h$ to be the event that there is a path of open edges connecting $x$ to height $h$ via vertices of $\cP_x$, we have the following observation:
\begin{observation}\label{obs:simpler-cABot}
    On the event $\Iso_{x, L, h}$, the events $\cA^{\Bot}_{x, h}$ and $x \xleftrightarrow{\omega} h$ are equal. Indeed, on $\Iso_{x, L, h}$ we know that $x \in \Vbot$, whence it immediately follows that $x \xleftrightarrow{\omega} h \subseteq \cA^{\Bot}_{x, h}$. For the other direction, note that the vertices (with height $>0$) surrounding those of $\cP_x$ are all in $\Vtop$. Together with the assumption that $x$ is a cut-point and in $\Vbot$, this implies that every vertex in $\cP_x$ which is in $\Abot \setminus \Vbot$ must be part of a finite component which is surrounded by vertices of $\cP_x$ in $\Vbot$. But all the vertices of $\cP_x$ which are in $\Vbot$ have an open path of edges connecting to $x$ inside $\cP_x$, and so $\cA^{\Bot}_{x, h} \subseteq  \{x \xleftrightarrow{\omega} h\}$.
\end{observation}

With this in mind, we prove the following analog of \cref{lem:phi-Phi-monotonicity,lem:move-to-nice-space}.
\begin{lemma}\label{lem:move-to-nice-space-bot}
    Let $\Phi$ be any map on interfaces sending $E_h^x$ into $\Iso_{x, L, h}$ such that the action of $\Phi$ on $\cP_x$ is to shift increments or replace them by a stack of trivial increments, and to replace the base by a stack of trivial increments with equal height. Suppose moreover that  $\bar{\mu}_n(\Phi^{-1}(J)) \leq (1+\epsilon_\beta)\bar{\mu}_n(J)$ holds for any $J \in \Phi(E_h^x)$.
    Then, for any space $\Omega$ such that $\Phi(E_h^x) \subseteq \Omega \subseteq E_h^x$, there exists a constant $\epsilon_\beta$ such that 
\begin{equation}
\left|\frac{\bar{\mu}_n(\cA^{\Bot}_{x, h} \mid \Omega)}{\bar{\mu}_n(\cA^{\Bot}_{x, h} \mid E_h^x)} - 1\right| \leq \epsilon_\beta\,.
\end{equation}
\end{lemma}
\begin{proof}
    By the same computation as in the proof of \cref{lem:move-to-nice-space}, the facts  $\cA^{\Bot}_{x, h} \subseteq E_h^x$ and $\bar{\mu}_n(\Omega \mid E_h^x) \geq 1- \epsilon_\beta$ reduce the proof to showing that
    \[\bar{\mu}_n(\Omega\mid\cA^{\Bot}_{x, h}) \geq 1 - \epsilon_\beta\,.\]
    Using the bound $\bar{\mu}_n(\Phi^{-1}(J)) \leq (1+\epsilon_\beta)\bar{\mu}_n(J)$, we can write
    \begin{align*}
        \bar{\mu}_n(\cA^{\Bot}_{x, h}) &= \sum_{I \in \cA^{\Bot}_{x, h}} \bar{\mu}_n(I)\\
        &\leq \sum_{J \in \Phi(\cA^{\Bot}_{x, h})}\,\sum_{I \in \Phi^{-1}(J)}\bar{\mu}_n(I)\\
        &\leq \sum_{J \in \Phi(\cA^{\Bot}_{x, h})} \bar{\mu}_n(J)(1+\epsilon_\beta)\,.
    \end{align*}
    We conclude by arguing that the conditions on $\Phi$ ensure that $\Phi(\cA^{\Bot}_{x, h}) \subseteq \Omega \cap \cA^{\Bot}_{x, h}$. Indeed, if $I \in \cA^{\Bot}_{x, h}$ and $J = \Phi(I)$, then $J$ has a path of open edges in $P_x^J$ connecting $x$ up to $\hgt(v_1)$, where $v_1$ is the first cut-point of $P_x^I$ (since $J$ is just a stack of trivial increments there). More generally, any stack of trivial increments in $P_x^J$ also has an open path connecting the bottom and top cut-points of the stack. Furthermore, for every increment $X_i \in P_x^I$, \cref{obs:simpler-cABot} shows that there must be a path of open edges connecting $v_i$ to $v_{i+1}$, and hence the same must be true regarding the shifted copy of $X_i$ in $P_x^J$. Hence, there must be an open path in $J$ connecting $x$ to height $h$ inside $P_x^J$, which implies $\cA^{\Bot}_{x, h}$ by \cref{obs:simpler-cABot} and the assumption that $\Phi(\cA^{\Bot}_{x, h}) \subseteq \Iso_{x, L, h}$.
\end{proof}

Equipped with the 3-to-3 map of \cref{lem:3-to-3-map}, we can prove the submultiplicativity result for $\cA^{\Bot}_{x, h}$ directly.
\begin{proposition}\label{prop:submultiplicativity-bot}
For every $\beta > \beta_0$, there exists a constant $\epsilon_\beta$ such that for every $h = h_1 + h_2$, and every sequence $x, n$ dependent on $h$ such that $d(x, \partial \Lambda_n) \gg h$,
    \begin{equation}
        \phi_n(\cA^{\Bot}_{x, h_1+h_2}\mid E_{h_1+h_2}^x) \leq (1+\epsilon_\beta)\phi_n(\cA^{\Bot}_{x, h_1}\mid E_{h_1}^x)\phi_n(\cA^{\Bot}_{x, h_2}\mid E_{h_2}^x)\,.
    \end{equation}
\end{proposition}
\begin{proof}
    By \cref{lem:move-to-nice-space-bot} above, it suffices to prove instead
    \[\phi_n(\cA^{\Bot}_{x, h_1+h_2}\mid \Omega_{h_1, h_2}) \leq (1+\epsilon_\beta)\phi_n(\cA^{\Bot}_{x, h_1}\mid \Omega_{h_1})\phi_n(\cA^{\Bot}_{x, h_2}\mid \Omega_{h_2})\,.\]
    But, \cref{obs:simpler-cABot} readily implies that if $P = P_B \times P^T$ for a pillar $P \in \Omega_{h_1, h_2} \cap \cA^{\Bot}_{x, h_1 + h_2}$, then $P_B \in \Omega_{h_1} \cap \cA^{\Bot}_{x, h_1}$ and $P^T \in \Omega_{h_2} \cap \cA^{\Bot}_{x, h_2}$. Thus, we compute using \cref{lem:3-to-3-map} that
    \begin{align*}
    \phi_n(\cA^{\Bot}_{x, h_1+h_2}\mid \Omega_{h_1, h_2}) &= \sum_{P_B\times P^T \in \Omega_{h_1, h_2} \cap \cA^{\Bot}_{x, h_1+h_2}} \bar{\mu}_n(P_B \times P^T \mid \Omega_{h_1, h_2})\\
    &\leq (1 + \epsilon_\beta)\sum_{P_B\times P^T \in \Omega_{h_1, h_2} \cap \cA^{\Bot}_{x, h_1+h_2}} \bar{\mu}_n(P_B \mid \Omega_{h_1})\bar{\mu}_n( P^T \mid \Omega_{h_2})\\
    &\leq(1+\epsilon_\beta)\sum_{P_B \in \Omega_{h_1} \cap \cA^{\Bot}_{x, h_1}} \sum_{P^T \in \Omega_{h_2} \cap \cA^{\Bot}_{x, h_2}}\bar{\mu}_n(P_B \mid \Omega_{h_1})\bar{\mu}_n( P^T \mid \Omega_{h_2})\\
    &=(1 + \epsilon_\beta)\bar{\mu}_n(\cA^{\Bot}_{x, h_1}\mid \Omega_{h_1})\bar{\mu}_n(\cA^{\Bot}_{x, h_2}\mid \Omega_{h_2})\,.     
    \qedhere
\end{align*}
\end{proof}

As done before, by using the decorrelation estimates of \cref{cor:decorrelation-1}, we can generalize to the case where $x, n$ on the right hand side can depend on $h_1$ and $h_2$, as long as we still have $1 \ll h_i \ll n_{h_i}$ and $d(x_{h_i}, \partial \Lambda_{n_{h_i}}) \gg h_i$:
\begin{equation*}
    \phi_n(\cA^{\Bot}_{x_h, h} \mid E_h^x) \leq (1+\epsilon_\beta+o_{h_1}(1)+o_{h_2}(1))\phi_{n_{h_1}}(\cA^{\Bot}_{x_{h_1}, h_1} \mid E_{h_1}^{x_{h_1}})\phi_{n_{h_2}}(\cA^{\Bot}_{x_{h_2}, h_2} \mid E_{h_2}^{x_{h_2}})\,.
\end{equation*}
Fekete's Lemma then implies the existence of the last rate in \cref{prop:RC-Potts-diff-rate}.

\subsection{Estimating the rates}
To conclude this section, we want to prove that the above rates are distinct, and provide some better bounds on their differences. Call an increment $\sX_i^\mathrm{o}$ a simple block if it consists of just two vertices $v_i, v_{i+1}$ where $v_{i+1} = v_i + \ez$. For some constant $C^*$ sufficiently large (to be determined below), let $\sG$ be the good event that the pillar shell $\cP_x^\mathrm{o}$ has less than $4h + 1+ \frac{C^*}{2\beta}h$ faces.

\begin{lemma}\label{lem:many-simple-blocks}
    There exists constants $C^*,C, c > 0$ such that for $\beta, L$ sufficiently large, for all $h \geq 1$,
    \begin{equation}\label{eq:many-simple-blocks}
        \bar{\mu}_n(\sG \mid E_h^x, \Iso_{x, L, h}^\mathrm{o}) \geq 1 - Ce^{-ch}\,.
    \end{equation}
    Furthermore, any any pillar in $E_h^x \cap \sG$ has at least $h(1 - \frac{C^*}{\beta})$ simple blocks below height $h$.
\end{lemma}
\begin{proof} 
Suppose we have an interface $\cI$ from $E_h^x \cap \Iso_{x, L, h}^\mathrm{o}$.  We first prove that if there are fewer than $h(1 - \frac{C^*}{\beta})$ simple blocks used to reach height $h$, then $|\sF(\cP_x^\mathrm{o})| \geq 4h+1+\frac{C^*}{2\beta}h$. Indeed, suppose we expose the increments one by one. When we expose an increment $\sX_i^\mathrm{o}$ which is not a simple block, the height increases by $\hgt(v_{i+1}) - \hgt(v_i)$, and the number of faces added to $\cP_x^\mathrm{o}$ must be at least $4(\hgt(v_{i+1}) - \hgt(v_i)) + (\hgt(v_{i+1}) - \hgt(v_i) - 1)$ since each height in between the vertices is not a cut-height. That is, the number of faces added in addition to four times the height increase is at least half the height increase (for increments which are not simple blocks, the height increase is at least two). When we expose an increment that is a simple block, we increase the height by one, and we add at least four faces to $\cP_x^\mathrm{o}$. But, the latter can only happen at most $h(1 - \frac{C^*}{\beta})$ times, and so the remaining height of $\frac{C^*}{\beta}h$ is made up by increments which are not simple blocks. Thus, the number of faces in the pillar shell is at least $4h + 1 + \frac{C^*}{2\beta}h$ (where the plus one is just because there must be at least horizontal face that forms a ``cap" of the pillar at the top). 

Now, we can define the map $\Phi$ as follows: If $\cI \in \sG \cap \Iso_{x, L, h}^\mathrm{o} \cap E_h^x$, then $\Phi$ is the identity map. Otherwise, let $\Phi(\cI) = \cJ$ be the interface that replaces $\cP_x^\cI$ with a stack of trivial increments of height $h$. Let $\cI \in \sG^c \cap \Iso_{x, L, h}^\mathrm{o} \cap E_h^x$, so that $\fm(\cI;\cJ) \geq \frac{C^*}{2\beta}h$. For any such $\cI$, using the cluster expansion we have
\begin{equation*}
\frac{\bar{\mu_n}(\cI)}{\bar{\mu}_n(\cJ)} = (1-e^{-\beta})^{|\partial \cI| - |\partial \cJ|}e^{-\beta \fm(\cI;\cJ)}q^{\kappa_\cI - \kappa_\cJ}\exp(\sum_{f\in \cI} \g(f, \cI) - \sum_{f \in \cJ} \g(f, \cJ))\,.
\end{equation*}
To control the term $(1-e^{-\beta})^{|\partial \cI| - |\partial \cJ|}$, note that $|\partial \cJ| \leq 4C_0 h \leq \frac{8\beta C_0}{C^*} \fm(\cI;\cJ)$, where $C_0$ is the number of faces that can be 1-connected to a particular face. Thus, we have
\begin{equation*}
    (1-e^{-\beta})^{|\partial \cI| - |\partial \cJ|} \approx e^{-e^{-\beta}(|\partial \cI| - |\partial \cJ|)} \leq e^{e^{-\beta}\frac{8\beta C_0}{C^*} \fm(\cI;\cJ)} \leq e^{\frac{8C_0}{C^*}\fm(\cI;\cJ)}
\end{equation*}
for sufficiently large $\beta$. (The $\approx$ can be seen to be an equality up to a factor of $(1-\epsilon_\beta)$ in the exponent, which has no affect on the final inequality. See for instance the computation in \cref{eq: prob-non-red-in-trivial}.)

To control the difference in open clusters, we will be slightly more careful than before. In $\cI$, we can first expose the vertical faces that bound the sides of the vertices of the pillar $\cP_x^\cI$. Since we are only exposing vertical faces which notably are not 1-connected to any faces of $\cI \setminus \cP_x^\cI$ except at height 0 (by \cref{cor:pillar-doesnt-touch-other-walls}), there are not yet any new open clusters created. Since the pillar has height $\geq h$, we must have already exposed at least $4h$ faces. Now, there are at most $\fm(\cI;\cJ) + 1$ faces left to expose in the pillar (since we are on the event $\sG$), and each one can create at most one open cluster, so that 
\begin{equation*}
\kappa_\cI - \kappa_\cJ \leq \fm(\cI;\cJ) + 1\,.
\end{equation*}
Finally, we bound the $g$-terms. We can write the absolute value of the sum of the terms as
\begin{equation*}
    \sum_{f \in \cI \setminus \cP_x} |\g(f, \cI) - \g(f, \cJ)| + \sum_{f \in \cP_x^\cI} |\g(f, \cI)| + \sum_{f \in \cP_x^\cJ} |\g(f, \cJ)|\,.
\end{equation*}
We can bound the first term using \cref{eq:pillar-in-cone-2}:
\begin{equation*}
    \sum_{f \in \cI \setminus \cP_x} |\g(f, \cI) - \g(f, \cJ)| \leq \sum_{f \in \cI \setminus \cP_x}\;\sum_{g \in (\cP_x^\cI \cup \cP_x^\cJ) \cap \cL_{\geq L^3}} Ke^{-cd(f, g)} \leq Ke^{-cL}\,.
\end{equation*}
The second and third terms can be bounded by the number of faces:
\begin{equation*}
    \sum_{f \in \cP_x^\cI} |\g(f, \cI)| + \sum_{f \in \cP_x^\cJ} |\g(f, \cJ)| \leq K(8h + 2+ \fm(\cI;\cJ)) \leq K(\frac{16\beta}{C^*} + 1)\fm(\cI;\cJ) + 2K
\end{equation*}
Thus, we have the energy bound:
\begin{equation*}
    \frac{\bar{\mu}_n(\cI)}{\bar{\mu}_n(\cJ)} \leq Ce^{-(\beta - \frac{K16\beta}{C^*} -K - \frac{8C_0}{C^*} - \log q)\fm(\cI;\cJ)}\,.
\end{equation*}
For the entropy bound, we can recover $\cI$ from $\cJ$ if we are given the faces of $\cP_x^\cI$, since both $\cI$ and $\cJ$ are in $\Iso_{x, L, h}$. There are $4h + 1 + \fm(\cI;\cJ) \leq (\frac{8\beta}{C^*} + 1)\fm(\cI;\cJ) + 1$ faces in $\cP_x$.
Thus, by \cref{lem:num-of-0-connected-sets}, we have for some $s > 0$,
\begin{equation*}
|\{\cI \in \Phi^{-1}(\cJ): \fm(\cI;\cJ) = M\}| \leq s^{(\frac{8\beta}{C^*} + 1)M}\,.
\end{equation*}
Thus, we have 
\begin{align*}
\bar{\mu}_n(\sG^c, E_h^x, \Iso_{x, L, h}^\mathrm{o}) &\leq \sum_{M \geq \frac{C^*}{2\beta}h}\,\, \sum_{\substack{\cI \in \sG^c \cap E_h^x \cap \Iso_{x, L, h}^\mathrm{o},\\ \fm(\cI;\Phi(\cI)) = M}} \bar{\mu}_n(\cI)\\
&\leq \sum_{M \geq \frac{C^*}{2\beta}h} \,\, \sum_{\cJ \in E_h^x \cap \Iso_{x, L, h}^\mathrm{o}}\,\, \sum_{\substack{\cI \in \Phi^{-1}(\cJ),\\ \fm(\cI;\cJ) = M}} Ce^{-(\beta - \frac{K16\beta}{C^*} -K - \frac{8C_0}{C^*} - \log q)\fm(\cI;\cJ)}\bar{\mu}_n(\cJ)\\
&\leq \sum_{M \geq \frac{C^*}{2\beta}h} Ce^{-(\beta - \frac{K16\beta}{C^*} -K - \frac{8C_0}{C^*} - \log q-(\frac{8\beta}{C^*}+1)\log s)\fm(\cI;\cJ)}\bar{\mu}_n(E_h^x, \Iso_{x, L, h}^\mathrm{o})\\
&\leq \tilde Ce^{-(\frac{C^*}{2} - 8K-4\log s)h + \frac{C^*}{2\beta}(\log q+\log s)h}\bar{\mu}_n(E_h^x, \Iso_{x, L, h}^\mathrm{o})\,.
\end{align*}
The lemma follows by dividing by $\bar{\mu}_n(E_h^x, \Iso_{x, L, h}^\mathrm{o})$ and taking $C^*/2$ strictly larger than $8K + 4\log s$ and then taking $\beta$ sufficiently large.
\end{proof}

\begin{remark}
    The above lemma says that a typical pillar reaching height $h$ will have $h(1 - \epsilon_\beta)$ simple blocks, which for our purposes is all the precision that is needed. We note one can get a sharper bound of having at least $h(1 - Ce^{-c\beta})$ simple blocks via the following proof strategy: We can reveal the increments $\sX_i^\mathrm{o}$ one by one, and each increment will increase the number of faces revealed in the pillar shell by at least 4. By \cref{eq:incr-precise-bound} (and noting by \cref{rem:incr-map-variations} that we can really apply this bound one increment at a time), the number of additional faces revealed for each increment is stochastically dominated by $\operatorname{Geom}(p^*) - 1$ with $p^* = 1 -e^{-(\beta - C)}$ for some constant $C$. There are at most $h$ increments needed for the pillar to reach height $h$, so the total number of faces in the pillar shell with height $\leq h$ is stochastically dominated by $\operatorname{NegBin}(h, p^*)+3h$. We can then use known large deviation results concerning the Binomial distribution to bound the probability that the number of faces in the pillar shell exceeds $4h + e^{-c\beta}h$ for some constant $c$, and argue as in \cref{lem:many-simple-blocks} to show how this implies the lower bound on the number of simple blocks. Although the map argument presented above gives a weaker result, it allows us to use the machinery of \cref{lem:move-to-nice-space} in what follows.
\end{remark}

We are now in a position to obtain lower and upper bounds on the rates $\delta,\delta',\delta''$ that are sharp up to a factor of $1+\epsilon_\beta$.

\begin{proof}[\emph{\textbf{Proof of \cref{prop:bound-rates}}}]
It will turn out that the probabilities in question are on the scale of $e^{-O(e^{-\beta}h)}$, so we can throw the event $\sG^c$ as an additive error and it will not affect the large deviation rates. Now, on the event $\sG \cap \Iso_{x, L, h}^\mathrm{o}\cap E_h^x$, suppose we reveal the pillar shell $\cP_x^\mathrm{o} = P_x^\mathrm{o}$. Let $\fB$ be the indices of the first $(1 - \frac{C^*}{\beta})h$ increments intersecting with $\cL_{\leq h}$ which are simple blocks. Now, suppose we have a simple block increment consisting of vertices $v$ and $v + \ez$, where we know that $v \in \Ared^c$. Then, since $v + \ez$ is a cut-point and is thus surrounded by vertices of $\Vred$, the event $v + \ez \in \Ared^c$ is the same as $v+\ez$ just being non-$\Red$. Thus, $v+\ez \in \Ared^c$ can occur either if the edge $[v, v + \ez]$ is open, or if $[v, v+\ez]$ is closed and $v + \ez$ is colored non-$\Red$. By \cref{cor:DMP-for-Ared}, this conditional probability can be
computed as if on a coupled FK--Potts model on two vertices with boundary condition $\sigma_v = \noRed$. This probability is
\begin{align}\label{eq: prob-non-red-in-trivial}
\frac{p}{p+(1-p)q} + \frac{(1-p)q}{p+(1-p)q}\frac{q-1}{q} = 1 - e^{-\beta}(1-\epsilon_\beta) &= e^{-e^{-\beta}(1-\epsilon_\beta)}+O(e^{-2\beta})\nonumber \\
&= e^{-e^{-\beta}(1-\epsilon_\beta)+\log(1+O(e^{-2\beta - \epsilon_\beta}))}\nonumber \\
&=e^{-e^{-\beta}(1-\tilde\epsilon_\beta)}\,.
\end{align}

On the event $\Iso_{x, L, h}^\mathrm{o}$, we know $x$ is a cut-point and $x \in \Vbot$. So, via a computation similar to \cref{eq:split-Nred-by-increments}, we can use \cref{cor:DMP-for-Ared} to write
\begin{align}\label{eq:upper-bnd-diff-Potts-RC}
\phi_n(\cA^{\noRed}_{x, h} \mid P_x^\mathrm{o}, \sG, \Iso_{x, L, h}^\mathrm{o}, E_h^x)
& = \prod_{i \in \fB} \phi_n(\cA^{\noRed}_{v_i,v_{i+1}}|X_i^\mathrm{o},\, v_i \in \Ared^c)\prod_{\substack{i \notin \fB\\ X_i^\mathrm{o} \cap \cL_{\leq h} \neq \emptyset}} \phi_n(\cA^{\noRed}_{v_i,v_{i+1}}|X_i^\mathrm{o},\, v_i \in \Ared^c)\nonumber \\
&= e^{-e^{-\beta}(1-\tilde\epsilon_\beta)(1 - \frac{C^*}{\beta})h}
\prod_{\substack{i \notin \fB\\ X_i^\mathrm{o} \cap \cL_{\leq h} \neq \emptyset}} \phi_n(\cA^{\noRed}_{v_i,v_{i+1}}|X_i^\mathrm{o},\, v_i \in \Ared^c)\nonumber \\
&\leq e^{-e^{-\beta}(1-\tilde\epsilon_\beta)(1 - \frac{C^*}{\beta})h}\,.
\end{align}
We can also get a lower bound by considering the probability that for each increment $X_i^\mathrm{o}, i \notin \fB$ that intersects $\cL_{\leq h}$, we have a path of open edges connecting $v_i$ to $v_{i+1}$. Let $Q_i$ be a minimal $\Lambda_n$-path from $v_i$ to $v_{i+1}$ using vertices of $X_i^\mathrm{o}$. We argue that we can control the length $|Q_i|$ by the number of faces in $X_i^\mathrm{o}$. Let $V$ be the set of vertices in $X_i^\mathrm{o}$ (including $v_i$ and $v_{i+1}$). Let $H$ be the set of faces of $X_i^\mathrm{o}$ plus the faces $f_{[v_i, v_i - \ez]}$ and $f_{[v_{i+1}, v_{i+1} + \ez]}$. Then, $V$ is precisely the set of vertices in the component of $\R^3 \setminus H$ containing $v_i$. Note that by definition, none of the faces of $X_i^\mathrm{o}$ (and hence of $H$) separate two vertices of $V$. Then, defining $\Delta_{\scV, H}V$ as the subset 
\[\Delta_{\scV, H}V =  \left\{ u\in V\,:\; \exists v\mbox{ s.t.\ } f_{[u,v]}\in \overline{H}\right\}\,,\]
we know that $\Delta_{\scV, H}V$ is $\Lambda_n$-connected (see \cite[Prop.~6]{GielisGrimmett02},\cite[Thm.~7.5]{Grimmett_RC}) and contains $v_i, v_{i+1}$, so that the length of $Q_i$ is at at most $|\Delta_{\scV, H}V|$. We have a crude upper bound  $|\Delta_{\scV, H}V| \leq 10|H|$. Now, to avoid overcounting faces of $P_x^\mathrm{o}$, let us attribute to each increment $X_i^\mathrm{o}$ all of its faces except the four faces adjacent to $v_i$ at height $\hgt(v_i)$. Then, at least $4h(1-\frac{C^*}{\beta})$ faces are attributed to increments with indices $i \in \fB$. Since $P_x^\mathrm{o}$ has at most $4h + 1 + \frac{C^*}{2\beta}h$ faces (recall we are on $\sG$), this leaves at most $\frac{9C^*}{2\beta}$ faces to be attributed to increments with indices $i \notin \fB$. The number of faces in $H$ is six more than the number of faces attributed to $X_i^\mathrm{o}$, and each $X_i^\mathrm{o}$ gets attributed at least four faces. Hence, for another constant $\tilde C^*$ (namely, $\tilde C^* = 10\cdot\frac{5}{2}C^*$), we have 
\[\sum_{i \notin \fB} |Q_i| \leq \frac{9\tilde C^*}{2\beta}.\]
In other words, we can guarantee that $v_i$ is in the same open cluster as $v_{i+1}$ for each $i \notin \fB$ if we force a specific set of $\frac{9\tilde C^*}{2\beta}$ edges to be open. The probability of an edge $e = [u, v]$ being open is at least the conditional probability that $e$ is open given that $u, v$ are not in the same open cluster in $\omega\restriction_{\Lambda_n \setminus \{e\}}$. We compute this to be
\begin{equation}\label{eq:prob-bot-in-trivial}
    \frac{p}{p+(1-p)q} = 1 - qe^{-\beta}(1-\epsilon_\beta) = e^{-qe^{-\beta}(1-\tilde\epsilon_\beta)}, 
\end{equation}
where the second equality is computed similarly to \cref{eq: prob-non-red-in-trivial}. Thus, combined we have
\begin{align}\label{eq:lower-bnd-diff-Potts-RC}
\phi_n(\cA^{\noRed}_{x, h} \mid P_x^\mathrm{o}, \sG, \Iso_{x, L, h}^\mathrm{o}, E_h^x)
& = \prod_{i \in \fB} \phi_n(\cA^{\noRed}_{v_i,v_{i+1}}|X_i^\mathrm{o},\, v_i \in \Ared^c)\prod_{\substack{i \notin \fB\\ X_i^\mathrm{o} \cap \cL_{\leq h} \neq \emptyset}} \phi_n(\cA^{\noRed}_{v_i,v_{i+1}}|X_i^\mathrm{o},\, v_i \in \Ared^c)\nonumber \\
&\geq e^{-e^{-\beta}(1-\tilde\epsilon_\beta)(1 - \frac{C^*}{\beta})h}e^{-qe^{-\beta}(1-\tilde\epsilon_\beta)\frac{9\tilde C^*}{2\beta}h}\nonumber \\
&= e^{-e^{-\beta}(1 - \tilde\epsilon_\beta)(1 - \frac{C^*}{\beta}+q\frac{9\tilde C^*}{2\beta})h}\,.
\end{align}
Now, the bounds in \cref{eq:upper-bnd-diff-Potts-RC,eq:lower-bnd-diff-Potts-RC} are uniform over $P_x^\mathrm{o}$, so the same bounds apply for $\phi_n(\cA^{\noRed}_{x, h} \mid E_h^x, \sG, \Iso_{x, L, h}^\mathrm{o})$, which has the same large deviation rate as $\phi_n(\cA^{\noRed}_{x, h} \mid E_h^x)$ by \cref{lem:phi-Phi-monotonicity,lem:move-to-nice-space} applied to the composition of $\Phi_\Iso$ with the map used in \cref{lem:many-simple-blocks}. Thus, we have established \cref{prop:bound-potts-rate-1}.

Similar to before, we work out the following probability that the top vertex of a simple block increment is in $\Ablue$ given that the bottom one is in $\Ablue$:
\begin{equation*}
\frac{p}{p+(1-p)q} + \frac{(1-p)q}{p+(1-p)q}\frac{1}{q} = 1 - (q-1)e^{-\beta}(1-\epsilon_\beta) = e^{-(q-1)e^{-\beta}(1-\tilde\epsilon_\beta)}\,,
\end{equation*}
whence the same argument as above implies \cref{prop:bound-potts-rate-2}.

Finally, we would like to use an analog of \cref{cor:DMP-for-Ared} to once again break up the event $\cA^{\Bot}_{x, h}$ increment by increment so that we have analogs of \cref{eq:upper-bnd-diff-Potts-RC,eq:lower-bnd-diff-Potts-RC}. Then, the proof of \cref{prop:bound-bot-rate} would conclude as above via the computation of the probability of having an open edge between two vertices of a simple block (which was already computed in \cref{eq:prob-bot-in-trivial}). The statement in \cref{cor:DMP-for-Ared} is a Domain Markov statement in the joint space of configurations, which is stronger than the statement we need for just the random-cluster model and could easily be adapted to handle the case of $\cA = \cA^{\Bot}_{x, h}$. The one minor issue is that the joint measure $\phi_n$ used there is only defined for integer valued $q$, and we want the result for all real $q \geq 1$. So, we adapt the proof of \cref{lem:DMP-pillar-shell} to apply in the context of the random-cluster model for the more general set of $q$.

\begin{lemma}
Fix a rooted increment shell $X_\star^\mathrm{o}$ and let $G_\star = (V_\star, E_\star)$ be the induced subgraph of $\Lambda_n$ on the vertices of $X_\star^\mathrm{o}$. Then, conditional on the event $\sX_i^\mathrm{o} = X_\star^\mathrm{o}$, the law of $\omega\restriction_{E_\star}$ is that of a random-cluster model on $G_\star$ with free boundary conditions. 
\end{lemma}
\begin{proof}
    As shown in the proof of \cref{lem:DMP-pillar-shell}, the event $\sX_i^\mathrm{o} = X_\star^\mathrm{o}$ does not impose any conditions on $\omega\restriction_{E_\star}$. 
    Now, let $\partial V_\star \subseteq V_\star$ be the subset of vertices which are $\Lambda_n$-adjacent to $V_\star^c$. By the Domain Markov property, it suffices to show that on the event $\sX_i^\mathrm{o} = X_\star^\mathrm{o}$, there is no path of open edges in $E_\star^c$ that connects two vertices of $\partial V_\star$. For any vertex $v \in \partial V_\star \setminus \{v_i, v_{i+1}\}$, every edge $e \in E_\star^c$ incident to $v$ is such that $f_e \in X_\star^\mathrm{o}$, and hence $\omega_e = 0$. Moreover, regardless of what $\cP_x^\mathrm{o}$ is, any path $Q$ connecting $v_i$ to $v_{i+1}$ using only edges of $E_\star^c$ must  include an edge $e$ such that $f_e \in \cP_x^\mathrm{o}$, whence $Q$ must include a closed edge. 
\end{proof}
Together with \cref{obs:simpler-cABot}, this enables us to write
\begin{align*}\phi_n(\cA^{\Bot}_{x, h} \mid P_x^\mathrm{o}, \sG, \Iso_{x, L, h}^\mathrm{o}, E_h^x)
&=\phi_n(x \xleftrightarrow{\omega} h \mid P_x^\mathrm{o}, \sG, \Iso_{x, L, h}^\mathrm{o}, E_h^x)\\
&= \prod_{i \in \fB} \phi_n(v_i \xleftrightarrow{\omega} v_{i+1}|X_i^\mathrm{o})\prod_{\substack{i \notin \fB\\ X_i^\mathrm{o} \cap \cL_{\leq h} \neq \emptyset}} \phi_n(v_i \xleftrightarrow{\omega} v_{i+1}|X_i^\mathrm{o})\,,
\end{align*}
and the proof of \cref{prop:bound-bot-rate} follows via a similar computation as done in \cref{eq:lower-bnd-diff-Potts-RC,eq:upper-bnd-diff-Potts-RC}.
\end{proof}

\section{Maximum of the random-cluster and Potts interfaces}
\label{sec:max}
This section uses a modified second moment argument to establish the tightness of the minima/maxima of the Potts and random-cluster interfaces from the large deviation rates established in \cref{sec:ld-RC,sec:ld-potts-bot}, as was done for the Ising interface in the proof of \cite[Proposition 6.1]{GL_tightness}. We prove that the maximum of the four interfaces we have defined are tight around a specific constant which we also identify. Since the proofs for the different interfaces are largely the same, we will focus on proving the result for the $\Top$ interface of the random-cluster model and note along the way what modifications are needed for the other interfaces.

Even though we proved the large deviation rates for the events $\cA^{\noRed}_{x, h}, \cA^{\Blue}_{x, h},  \cA^{\Bot}_{x, h}$ in the previous section, we still want estimates on the probability of these events for small $h$ as the goal is to establish tightness. Thus, we begin by noting that the upper bound in \cref{prop:exp-tail-pillar-height} has an  immediate corollary resulting from the fact that the $\Top$ interface lies above all the other interfaces.

\begin{corollary}\label{cor:Potts-exp-tail-pillar-height}
For the same $\beta_0$ and constant $C > 0$ as in \cref{prop:exp-tail-pillar-height}, for every $\beta \geq \beta_0$, for all $x$, and for all $h \geq 1$, 
\begin{equation*}
\phi_n(\cA^{\noRed}_{x, h}) \leq \exp[-4(\beta - C)h]\,,
\end{equation*}
and similarly for $\cA^{\Blue}_{x, h}, \cA^{\Bot}_{x, h}$.
\end{corollary}

Similar to before, we can also prove a rough lower bound on these exponential tails:
\begin{proposition}\label{prop:bound-Potts-rate}
For the same $\beta_0$ and constant $C > 0$ as above, for every $\beta \geq \beta_0$, for all $x$, and for all $h \geq 1$, 
\begin{equation*}
 -4\beta + \log\frac{p}{p+(1-p)q} \leq \frac 1h \log\phi_n(\cA^{\noRed}_{x, h}) \leq -4\beta+C\,,
\end{equation*}
and similarly for $\cA^{\Blue}_{x, h}, \cA^{\Bot}_{x, h}$.
\end{proposition}
\begin{proof}
The same proof as \cref{prop:bound-RC-rate} holds here with the following minor adjustment. Recall that in the proof for the lower bound there, we showed that the probability of having an interface $\cI$ with a ceiling face at $f_{[x, x - \ez]}$, and then appending the faces surrounding a column of $h$ vertices above $x$ is $\geq (1 - \epsilon_\beta)e^{-\beta(4h+1)}$. On this event, we can force open $h$ edges to connect all the vertices in the column to each other and to $x - \ez$, which was in the same open cluster as $\partial \Lambda_n^-$ to begin with (as we started with $f_{[x, x - \ez]}$ being a ceiling face). This guarantees the event $\cA^{\Bot}_{x, h}$ (which implies $\cA^{\noRed}_{x, h}, \cA^{\Blue}_{x, h}$), and the cost of forcing these $h$ edges to be open is $(\frac{p}{p+(1-p)q})^h$ (we have a weight of $(1-p)q$ for closed edges because each closed edge in the column always creates a new open cluster).
\end{proof}

Towards defining our desired tightness results, first note that \cref{cor:decorrelation-1} shows the existence of the following limit for any $h \geq 1$:
\[\alpha_h := \lim_{n \to \infty} -\log \bar{\mu}_n(E_h^o),\]
where $o = (1/2, 1/2, 1/2)$. Taking the limit $n \to \infty$ in \cref{eq:submultiplicativity-Ehx}, we have
\begin{equation}\label{eq:alpha-superadditivity}
\alpha_{h_1+h_2} \geq  \alpha_{h_1}+\alpha_{h_2}-3\beta - \epsilon_\beta\,.
\end{equation}
By Fekete's Lemma, we know that the limit $\frac{1}{h}\alpha_h$ exists, and it is moreover equal to $\alpha$ since \cref{prop:RC-rate} holds for any $n = n_h$ such that $d(x, \partial \Lambda_n) \gg h$.

Analogously, by \cref{cor:Potts-decorrelation-1} we can define  
\[\alpha_h^\noRed := \lim_{n \to \infty} -\log \phi_n(\cA^{\noRed}_{o, h}),\]
and similarly for $\Blue$ and $\Bot$. Combining \cref{eq:submultiplicativity-Ehx} with the submultiplicativity propositions we proved for the other interfaces (\cref{prop:submultiplicativity-Potts,prop:submultiplicativity-bot}) proves \cref{eq:alpha-superadditivity} for $\alpha_h^\noRed, \alpha_h^\Blue,\alpha_h^\Bot$.

Now we want to compare $\alpha_h$ and $\alpha_{h+1}$. Because of the increment map and \cref{thm:incr-map}, it suffices to consider (at a $(1 + \epsilon_\beta)$ multiplicative cost) just the subset of pillars in $E_h^o$ with a cut-height at $h - 1/2$. Let $w$ be the vertex in $\cP_x$ with height $h - 1/2$ , and let $y = w + \ez$. Every configuration with the edge $[y, w]$ open is already in $E_{h+1}^o$. For the remaining configurations with $[y, w]$ closed, we can first force the five edges $[y, y\pm \ex], [y, y\pm \ey], [y, y + \ez]$ to be closed at a cost of $e^{5\beta}/q$ (see the computation done in \cref{obs:close-edge}, noting that closing these edges creates a new open cluster $\{y\}$). For any resulting configuration $\omega$, the edge ${\bf e} = [y, y - \ez]$ is closed, but we can recover a factor of $qe^{-\beta}/(1 - e^{-\beta})$ by considering versions of $\omega$ with ${\bf e}$ open. That is, $\omega^{{\bf e}, 1} \in E_{h+1}^o$ and $\mu_n(\omega)/\mu_n(\omega^{{\bf e}, 1}) = qe^{-\beta}/(1 - e^{-\beta})$. Combined, we have
\begin{equation}\label{eq:increase-alpha-by-one}
    \alpha_{h+1} \leq \alpha_h + 4\beta + \epsilon_\beta,
\end{equation}
and by induction we have for any $l \geq 1$,
\begin{equation}\label{eq:increase-alpha-by-l}
    \alpha_{h+l} \leq \alpha_h + (4\beta + \epsilon_\beta)l.
\end{equation}
Note that the above computation ended with configurations in $E_{h+1}^o$, with the edge $[w, y]$ open. Hence, the same computation proves the analog for $\alpha_h^\Bot$. Similarly, the analogs for $\alpha_h^{\noRed}, \alpha_h^\Blue$ can be shown by bounding the cost of changing the color of a single spin in the Potts model by $e^{4\beta(1+\epsilon_\beta)}$ (see for instance the computation in \cite[Proposition 2.29]{GL_max}).

By \cref{prop:bound-RC-rate}, there exists a constant $C>0$ such that for all $h$,
\begin{equation}\label{eq:bound-alpha-h}
    4(\beta - C)h \leq \alpha_h \leq 4\beta h\,.
\end{equation}
Finally, Fekete's Lemma additionally tells us that $\alpha = \sup \frac{\alpha_h - 3\beta - \epsilon_\beta}{h}$ as long as we have \cref{eq:alpha-superadditivity}.
(For the Potts interfaces $\Blue$ and $\Red$ and the FK $\mathsf{bottom}$ interface, the upper bound needs to be adjusted to $(4\beta -\log\frac{p}{p+(1-p)q})h$ using \cref{prop:bound-Potts-rate}, but this will never matter in the computations below as we will only use this bound to show that $\alpha_h = O(h)$.)

Now, define 
\begin{equation}\label{eq:def-mn}
m_n^* = \inf\{h\,:\; \alpha_h > 2\log n-\beta/2\}\,,
\end{equation}
and analogously for the other interfaces. We can now state the main proposition of this section. 

\begin{proposition}\label{prop:tight-max-mn*}
Consider the maximum $M_n$ of $\cI_\Top$ for $q\geq 1$ fixed. Setting $m_n^*$ as in \cref{eq:def-mn}, there exist $\beta_0$ and $C_0$ such that for all $\beta > \beta_0$ and sufficiently large $n$,
\begin{equation}
    \label{eq:Mn-2-values}
    \bar\mu_n(M_n\notin \{m_n^*-1,m_n^*\}) \leq C_0\, e^{-\beta/2}\,.
\end{equation}
Moreover, for every $2 \leq l \leq \sqrt{\log n}$, 
\begin{align*}
    \bar{\mu}_n(M_n \geq m_n^* + l) &\leq C_0\, \exp\big(-\alpha_{l-1}+\tfrac{5\beta}2\big)\,,\\
    \bar{\mu}_n(M_n < m_n^* - l) &\leq C_0\, \exp\big(-\alpha_{l-1}+\tfrac{5\beta}2\big)\,.
\end{align*}
In fact, the right tail can be extended to all $1 \leq l \leq \frac1\beta \log n$. Furthermore, for $l > \frac1\beta \log n$, we have the tail
\begin{equation*}
    \bar{\mu}_n(M_n \geq m_n^* + l) \leq e^{-(2\beta - C_0)l}
\end{equation*}
The same statements for the maxima of the $\Red$, $\Blue$, and $\Bot$ interfaces also hold for $m_n^*$ and $\alpha_l$ defined by their respective interfaces, where $q \geq 2$ in the Potts setting.
\end{proposition}

We will get the right tail using a union bound, and the left tail by using a second moment computation. For this, we need a few preliminary results. Let $\cL_{1/2, n}$ denote the set of vertices with height $1/2$ in $\Lambda_n$. Let $\cL_{1/2, n}^\mathrm{o}$ be the subset of $\cL_{1/2, n}$ with distance larger than $\log^2 n$ from $\partial \Lambda_n$. 

\begin{definition}\label{def:Ghx} Define the event $G_h^x$ to be the event $E_h^x$ with the following additional requirements:
\begin{enumerate}
    \item\label{it:Ghx-cut-point-x} The vertex $x$ is a cut-point of $\cP_x$
    \item\label{it:Ghx-capped-pillar} $\cP_x \in \widetilde{E}_h^x$ in the context of the maximum of $\cI_{\Top}$; for the other interfaces, further require: 
    \begin{itemize}
        \item $\cA^{\noRed}_{x,h}$ in the context of the maximum of $\cI_{\Red}$;
        \item $ \cA^{\Blue}_{x,h}$ in the context of the maximum of $\cI_{\Blue}$;
        \item $ \cA^{\Bot}_{x,h}$ in the context of the maximum of $\cI_{\Bot}$.  
    \end{itemize}
    \item\label{it:Ghx-faces-of-I-near-pillar} The faces of $\cI$ which are 1-connected to $\cP_x$ (and not in $\cP_x$) are the four faces 1-connected to the face $f_{[x, x - \ez]}$ with height 0, and possibly the face $f_{[x, x - \ez]}$ itself.
\end{enumerate} Define also the random variable $Z_h$ by
\begin{equation*}
    Z_h = \sum_{x \in \cL_{1/2, n}^\mathrm{o}} \one_{\{G_h^x\}}\,.
\end{equation*}
\end{definition}

Note that in the case of $\Itop$, $G_h^x$ is implied by $\Iso_{x, L, h} \cap \widetilde E_h^x$, and thus for $x \in \cL_{1/2, n}^\mathrm{o}$, $h \ll \log^2 n \leq d(x, \partial \Lambda_n)$, we have
\begin{equation}\label{eq:G_h^x-likely}
    \bar{\mu}_n(G_h^x) \geq (1-\epsilon_\beta)\bar{\mu}_n(E_h^x)\,.
\end{equation}
For the cases of the other interfaces, we have
\begin{equation}
    \bar{\mu}_n(G_h^x) \geq (1-\epsilon_\beta)\bar{\mu}_n(\cA^{\star}_{x, h})
\end{equation}
by additionally applying \cref{lem:move-to-nice-space,lem:move-to-nice-space-bot} with $\Omega = G_h^x$, where $\star$ can be any of $\noRed, \Blue, \Bot$.

To get a lower bound for $\E[Z_h]$, we begin by noting that by \cref{cor:decorrelation-1} and taking $m \to \infty$, we have that for $1 \leq h \ll \log^2 n$, $x \in \cL_{1/2, n}^\mathrm{o}$,
\begin{equation}\label{eq:estimate-Ehx}
\bar{\mu}_n(E_h^x) = e^{-\alpha_h} + O(e^{-(\log^2 n)/C}) = (1+o(1))e^{-\alpha_h}
\end{equation}
since $\alpha_h = O(h)$. For the other interfaces, we can similarly apply the appropriate decorrelation result to the events $\cA^{\noRed}_{x, h}, \cA^{\Blue}_{x, h}, \cA^{\Bot}_{x, h}$ to get that for $\star$ being $\noRed, \Blue, \Bot$,
\begin{equation}\label{eq:estimate-cA-events}
\bar{\mu}_n(\cA^{\star}_{x, h}) = e^{-\alpha_h} + O(e^{-(\log^2 n)/C}) = (1+o(1))e^{-\alpha_h}\,.
\end{equation}
Note also that by using \cref{eq:increase-alpha-by-one}, we have
\begin{equation}\label{eq:estimate-alpha_mn*}
    2\log n -\tfrac{\beta}{2}< \alpha_{m_n^*} \leq 2\log n + \tfrac{7\beta}{2} + \epsilon_\beta\,.
\end{equation}
Now in preparation for the proof of the left tail, take $h = m_n^* - l$ for any $l \leq \sqrt{\log n}$. One can check (via \cref{eq:estimate-alpha_mn*} and the fact that $\lim_{h \to \infty} \alpha_h/h = \alpha$) that $m_n^* = (\frac {2}{\alpha} + o(1))\log n$, and so we have $h  \ll \log^2 n$ as needed for the results above. For $l = 1$, we simply have by definition of $m_n^*$ that
\begin{equation*}
    \alpha_{m_n^* - 1} \leq 2\log n - \tfrac{\beta}{2}\,.
\end{equation*} For $l \geq 2$, by \cref{eq:alpha-superadditivity,eq:estimate-alpha_mn*}, we have 
\begin{equation*}
\alpha_{m_n^* -l} \leq \alpha_{m_n^*-1} - \alpha_{l-1} + 3\beta + \epsilon_\beta \leq 2\log n - \alpha_{l-1} + \tfrac{5\beta}{2} + \epsilon_\beta\,.
\end{equation*}

Plugging this estimate into  \cref{eq:estimate-Ehx} and also noting that $|\cL_{1/2, n}^\mathrm{o}| = (1 - o(1))n^2$, we have for sufficiently large $n$ and $l \geq 2$, 
\begin{equation}\label{eq:bound-expectation-Z_h}
    \E[Z_{m_n^* - l}] = \sum_{x \in \cL_{1/2, n}^\mathrm{o}} \bar{\mu}_n(G_{m_n^* - l}^x) \geq (1-\epsilon_\beta)e^{-\frac{5}{2}\beta}e^{\alpha_{l-1}}\,,
\end{equation}
and for $l = 1$,
\begin{equation}\label{eq:bound-expectation-Z_h-l=1}
    \E[Z_{m_n^* - 1}] = \sum_{x \in \cL_{1/2, n}^\mathrm{o}} \bar{\mu}_n(G_{m_n^* - 1}^x) \geq (1-\epsilon_\beta)e^{\frac{\beta}{2}}\,.
\end{equation}

We also have the following estimate concerning pillars in $G_h^x, G_h^y$ for $x, y$ close to each other.
\begin{claim}\label{clm:G-split-product}
    For all $\beta > \beta_0$, there exists a constant $\epsilon_\beta$ such that for all $h \ll \log^2 n$, $(x, y) \in \cL_{1/2, n}^\mathrm{o}$ with $d(x, y) \leq \log^2 n$ and $n$ sufficiently large,
    \begin{equation*}
        \bar{\mu}_n(G_h^x,\, G_h^y) \leq (1+\epsilon_\beta)\frac{(e^\beta + q-1)^2}{q}\bar{\mu}_n(E_h^x)\bar{\mu}_n(E_h^y)
    \end{equation*}
    where the definition of $G_h^x$ can be taken with respect to any of the four interfaces.
\end{claim}
\begin{proof} First note that by set inclusion, it suffices to prove the case of $\Itop$. The proof is similar to that of \cref{lem:submultiplicativity-in-Gamma}. The idea is to reveal the interface $\cI \setminus \cP_y$ and use the Domain Markov property to show that the information revealed is essentially all increasing information (with the exception of a single closed edge). Then, using $A_h^y$ as a proxy for $G_h^y$, we can use FKG to remove the conditional information generated by revealing $\cI \setminus \cP_y$. Since the justification of the Domain Markov step is quite lengthy, yet almost the exact same as the one provided in the proof of \cref{lem:submultiplicativity-in-Gamma}, we defer the proof of this claim to \cref{clm:Appendix-G-split-product}.
\end{proof}

For $x, y$ far away from each other, we still have the decorrelation statement that for some $C>0$,
\begin{equation}\label{eq:decorrelation-G-far}
    |\bar{\mu}_n(G_h^x,\, G_h^y) - \bar{\mu}_n(G_h^x)\bar{\mu}_n(G_h^y)| \leq Ce^{-d(x, y)/C}\,.
\end{equation}
For justification, see \cref{sec:decorrelation}, noting that because the conditions of $G_h^x$ have been chosen so they are determined entirely by the pillar $\cP_x$ and the walls it is a part of, this decorrelation statement follows immediately from \cref{prop:pillar-marginal-of-wall,prop:decorrelation-2,clm:nested-walls-determine-pillars-RC-Potts}.

\begin{proof}[\textbf{\emph{Proof of \cref{prop:tight-max-mn*}}}]
Set $h= m_n^* + l$. For the right tail, take any $1 \leq l \leq \frac{1}{\beta}\log n$. We have
\begin{align*}
    \bar{\mu}_n(M_n \geq m_n^* + l) &\leq \sum_{x \in \cL_{1/2, n} \setminus \cL_{1/2, n}^\mathrm{o}}\bar{\mu}_n(E_h^x) + \sum_{x \in \cL_{1/2, n}^\mathrm{o}}\bar{\mu}_n(E_h^x) \\ 
    &\leq |\cL_{1/2, n} \setminus \cL_{1/2, n}^\mathrm{o}|e^{-4(\beta - C)h} + |\cL_{1/2, n}^\mathrm{o}|(1 + o(1))e^{-\alpha_h}\,,
\end{align*}
using \cref{prop:exp-tail-pillar-height} (or \cref{cor:Potts-exp-tail-pillar-height} for other interfaces) for the first sum and \cref{eq:estimate-Ehx} for the second. For the maximum with respect to $\cI_\Red$, $E_h^x$ needs to be replaced by $\cA^{\noRed}_{x, h}$, and similarly for $\cI_\Blue$ and $\cI_\Bot$. Recalling that $\alpha \leq 4\beta$ and $\alpha = \sup_h \frac{\alpha_h - 3\beta - \epsilon_\beta}{h}$, we have
\begin{equation*}
    \alpha_h -4(\beta - C)h \leq \alpha h + 3\beta + \epsilon_\beta - 4(\beta - C)h \leq Ch + 3\beta + \epsilon_\beta\,.
\end{equation*}
(For the other interfaces, the upper bound on the large deviation rate is $4\beta + \epsilon_\beta$, but the above statement still holds with a different $C$.) 
Thus, we have
\begin{equation*}
    \frac{|\cL_{1/2, n} \setminus \cL_{1/2, n}^\mathrm{o}|e^{-4(\beta - C)h}}{|\cL_{1/2, n}^\mathrm{o}|e^{-\alpha_h}} \leq \frac{4n\log^2 n}{(1+o(1))n^2}e^{3\beta + \epsilon_\beta}e^{C(\frac{2}{\alpha} + o(1) + \frac{1}{\beta})\log n} = o(1)
\end{equation*}
as long as $\beta$ is large enough so that $C(\frac{2}{\alpha} + o(1) + \frac{1}{\beta}) < 1$. Plugging back into the first inequality and using \cref{eq:alpha-superadditivity} followed by \cref{eq:estimate-alpha_mn*} to estimate $e^{-\alpha_h}$, we have
\begin{align}\label{eq:right-tail}
    \bar{\mu}_n(M_n \geq m_n^* + l) &\leq (1+o(1))|\cL_{1/2, n}^\mathrm{o}|e^{-\alpha_h} \leq (1+o(1))|\cL_{1/2, n}^\mathrm{o}|e^{-\alpha_{m_n^*} - \alpha_l + 3\beta + \epsilon_\beta}\nonumber \\
    &\leq (1+o(1))n^2e^{-2\log n + \frac{7}{2}\beta + \epsilon_\beta - \alpha_l} \leq (1+\epsilon_\beta) e^{\frac{7}{2}\beta-\alpha_l}\,,
\end{align}
which proves the right tail. (Note that because $\alpha_l \geq \alpha_{l-1} + \alpha_1 - 3\beta - \epsilon_\beta \geq \alpha_{l-1} + \beta - C$, the right tail can be rewritten as $Ce^{\frac{5\beta}{2} - \alpha_{l-1}}$ as in the statement of the proposition.)

To extend the right tail for all $l$ at a sub-optimal rate, we can use (for all four interfaces) the bound of $\bar{\mu}_n(E_h^x) \leq e^{-4(\beta - C)h}$ and the observation that for $l > \frac{1}{\beta}\log n$, we have $n^2e^{-2\beta l} < 1$. Thus, for $l > \frac{1}{\beta}\log n$,
\begin{equation*}
    \bar{\mu}_n(M_n \geq m_n^* + l) \leq \sum_{x \in \cL_{1/2, n}} \bar{\mu}_n(E_{m_n^*+l}^x) \leq \sum_{x \in \cL_{1/2, n}} \bar{\mu}_n(E_{l}^x) \leq n^2e^{-4(\beta - C)l} \leq e^{-2(\beta - C)l}\,.
\end{equation*}

To prove the left tail, let $h = m_n^* - l$ for $l \leq \sqrt{\log n}$. We compute:
\begin{align*}
    \E[Z_h^2] = \sum_{x, y \in \cL_{1/2, n}^\mathrm{o}} \bar{\mu}_n(G_h^x,\, G_h^y) &\leq \sum_{x \in \cL_{1/2, n}^\mathrm{o}} \bar{\mu}_n(G_h^x)\\
    &+ \sum_{x \in \cL_{1/2, n}^\mathrm{o}}\,\, \sum_{y \in \cL_{1/2, n}^\mathrm{o} \cap B(x, \log^2 n): y \neq x} \bar{\mu}_n(G_h^x,\, G_h^y)\\
    &+\sum_{x \in \cL_{1/2, n}^\mathrm{o}}\,\, \sum_{y \in \cL_{1/2, n}^\mathrm{o} \setminus B(x, \log^2 n)} \bar{\mu}_n(G_h^x)\bar{\mu}_n(G_h^y) + |\bar{\mu}_n(G_h^x)\bar{\mu}_n(G_h^y) - \bar{\mu}_n(G_h^x,\, G_h^y)|\\
    &=: \Xi_1 + \Xi_2 + \Xi_3\,.
\end{align*}
By definition, $\Xi_1 = \E[Z_h]$. 

By \cref{clm:G-split-product}, we have 
\begin{equation}\label{eq:Xi2-UB}
    \Xi_2 \leq 4n^2\log^4 n (1+\epsilon_\beta)\frac{(e^\beta + q-1)^2}{q}\sup_{x, y\in \cL_{1/2, n}^\mathrm{o}} \bar{\mu}_n(E_h^x)\bar{\mu}_n(E_h^y) \leq n^{2 + o(1)}\sup_{x \in \cL_{1/2,n}^\mathrm{o}} \bar{\mu}_n(E_h^x)^2\,.
\end{equation}
By \cref{eq:estimate-Ehx}, we have $\bar{\mu}_n(E_h^x) = (1+o(1))e^{-\alpha_h}$. But by \cref{eq:estimate-alpha_mn*,eq:increase-alpha-by-l} and the fact that $l \leq \sqrt{\log n}$, we get that $e^{-\alpha_h} = n^{-2 + o(1)}$. Combined with \cref{eq:Xi2-UB}, we have that 
\[\Xi_2 \leq n^{-2+o(1)} = o(1)\,.\]
Finally, for $\Xi_3$ we have by expanding the square that
\begin{equation*}
    \sum_{x \in \cL_{1/2, n}^\mathrm{o}}\,\, \sum_{y \in \cL_{1/2, n}^\mathrm{o} \setminus B(x, \log^2 n)} \bar{\mu}_n(G_h^x)\bar{\mu}_n(G_h^y) \leq \E[Z_h]^2,
\end{equation*}
and by \cref{eq:decorrelation-G-far} we have
\begin{equation*}
    \sum_{x \in \cL_{1/2, n}^\mathrm{o}}\,\, \sum_{y \in \cL_{1/2, n}^\mathrm{o} \setminus B(x, \log^2 n)} |\bar{\mu}_n(G_h^x)\bar{\mu}_n(G_h^y) - \bar{\mu}_n(G_h^x,\, G_h^y)| \leq C|\cL_{1/2, n}^\mathrm{o}|^2e^{-\log^2 n/C} = o(1).
\end{equation*}
Thus, by Paley--Zygmund, we have
\begin{equation*}
    \bar{\mu}_n(Z_h > 0) \geq \frac{\E[Z_h]^2}{\E[Z_h^2]} \geq \frac{\E[Z_h]^2}{\E[Z_h]^2+\E[Z_h] + o(1)}\,,
\end{equation*}
or equivalently,
\begin{equation*}
    \bar{\mu}_n(M_n < h) \leq \bar{\mu}_n(Z_h = 0) \leq \frac{1 + o(1)}{\E[Z_h] + 1 + o(1)}\,.
\end{equation*}
For $h = m_n^* - 1$, the lower bound on the expectation computed in \cref{eq:bound-expectation-Z_h-l=1} gives us
\begin{equation}\label{eq:left-tail-l=1}
    \bar{\mu}_n(M_n - m_n^* < -1) \leq (1+\epsilon_\beta)e^{\frac{-\beta}{2}}\,,
\end{equation}
whereas for $h = m_n^* - l$ for $l \geq 2$, we have by \cref{eq:bound-expectation-Z_h} that
\begin{equation}\label{eq:left-tail}
    \bar{\mu}_n(M_n - m_n^* < -l) \leq (1+\epsilon_\beta)e^{\frac{5\beta}{2} - \alpha_{l-1}}\,.
\end{equation}
This concludes the proof of the right and left tails, and combining \cref{eq:right-tail,eq:left-tail-l=1} immediately proves the claim in \cref{eq:Mn-2-values} that $M_n$ is, with probability $(1 - C_0e^{-\frac{\beta}{2}})$, either $m_n^* - 1$ or $m_n^*$.
\end{proof}

\begin{corollary}\label{cor:E[M_n]-close-to-mn*}
    There exists $\beta_0$ such that for all $\beta > \beta_0$, for sufficiently large $n$, 
    \begin{equation*}
        m_n^* - 1-\epsilon_\beta \leq \E[M_n] \leq m_n^* + \epsilon_\beta\,,
    \end{equation*}
    and this holds for $M_n, m_n^*$ defined with respect to any of the four interfaces in random-cluster/Potts. 
\end{corollary}
\begin{proof}
    By the right tails of \cref{prop:tight-max-mn*}, we can write
\begin{align*}
    \E[(M_n - m_n^*)_+] \leq \sum_{l = 1}^\infty \bar{\mu}_n(M_n - m_n^* \geq l) \leq C_0e^{-\beta/2} + \sum_{l = 2}^{\log n}C_0e^{- \alpha_{l-1} + \tfrac{5\beta}{2} } + \sum_{l > \log n} e^{-2(\beta - C)l}\,.
\end{align*}
By the estimate of $e^{-\alpha_l} \leq e^{-4(\beta - C)l}$ in \cref{eq:bound-alpha-h}, we have that
\[\E[(M_n - m_n^*)_+] \leq \epsilon_\beta.\]
Similarly using the left tail, we have
\begin{align*}
    \E[(m_n^* - 1 - M_n)_+] &= \E[(m_n^* - 1 - M_n)_+\one_{\{M_n \leq m_n* - \sqrt{\log n}\}}] + \E[(m_n^* - 1 - M_n)_+\one_{\{M_n > m_n^* - \sqrt{\log n}\}}]\\
    &\leq m_n^* \bar{\mu}_n(M_n \leq m_n^* - \sqrt{\log n}) + \sum_{l = 1}^{\sqrt{\log n}} \bar{\mu}_n(m_n^* - 1 - M_n \geq l)\\
    &\leq O(\log n) e^{-O(\sqrt{\log n})} + C_0e^{-\beta/2} + \sum_{l = 2}^{\sqrt{\log n}} C_0e^{- \alpha_{l-1} + \tfrac{5\beta}{2}} \leq \epsilon_\beta\,.
\end{align*}
Now define $p_n = \bar{\mu}_n(M_n < m_n^*)$, so that
\begin{align*}
    &\E[M_n\one_{\{M_n \geq m_n^*\}}] = m_n^*(1-p_n) + \E[(M_n - m_n^*)_+]\\
    & \E[M_n\one_{\{M_n \leq m_n^* -1\}}] = (m_n^* -1)p_n - \E[(m_n^* - 1 - M_n)_+]\,.
\end{align*}
Adding these together and applying the bounds computed above, we have
\begin{equation*}
    m_n^* - p_n - \epsilon_\beta \leq \E[M_n] \leq m_n^* - p_n + \epsilon_\beta\, ,
\end{equation*}
whence the proof concludes by using the trivial bound $0 \leq p_n \leq 1$.
\end{proof}

Thus, the results of this section show that the maxima of the four interfaces in random-cluster/Potts are tight around their means, and their means are equal to $(\frac{2}{\alpha} + o(1))\log n$ where $\alpha$ should be replaced with the appropriate large deviation rate for the respective interface. By observing that the minimum of the $\Top$ interface has the same law as the maximum of the $\Bot$ interface, and the minimum of the $\Blue$ interface has the same law as the maximum of the $\Red$ interface, we conclude the proofs of \cref{thm:potts,thm:rc}. 

\appendix 
\section{Decorrelation estimates}\label{sec:decorrelation}
\begin{proposition}\label{prop:pillar-marginal-of-wall}
Let $\fW_x$ be the collection of walls nesting $x$. With probability $1 - Ce^{-(\beta-C)r}$ for some constant $C>0$, the walls in $\fW_x$ are indexed by vertices distanced at most $r$ from $x$.
\end{proposition}
\begin{proof}
If there is a wall $W$ nesting $x$ such that $W$ is not indexed by any vertices within distance $r$ from $x$, then the excess area of $W$ must be at least $r$. The proposition then follows immediately from the bound on the excess area of a group of nested walls in \cref{eq:tail-on-nested-wall-group}.
\end{proof}

\begin{claim}\label{clm:nested-walls-determine-pillars-RC-Potts}
    The entire pillar $\cP_x$ (and hence the event $E_h^x$), as well the event $\cA^{\Bot}_{x, h}$, is determined by $\fW_x$, the collection of walls nesting $x$. The collection $\fW_x$ moreover determines the conditional probabilities of the events $\cA^{\noRed}_{x, h}, \cA^{\Blue}_{x, h}$.
\end{claim}
\begin{proof}
    Let $F$ be a finite (maximal) 1-connected component of faces in $\clfaces$, that is moreover disjoint from $\cI$. Let $V$ denote the set of vertices separated by $F$ from $\partial \Lambda_n$. By maximality of $F$, all the edges of $\Delta_\scE F$ are open, and by \cref{prop:out-graph-connected}, the graph $(\Delta_\scV F, \Delta_\scE F)$ is connected. Hence, all the vertices of $\Delta_\scV F$ are part of the same open cluster. By the definition of $\Atop$, the vertices of $V$ are in $\Atop$ if and only if the vertices of $\Delta_\scV F$ are, and similarly for $\Abot$. Thus, the face set $F$ plays no role in determining whether or not the vertices of $V$ are in $\Atop$, and similarly for $\Abot$. In particular, both the pillar $\cP_x$ and the event $\cA^{\Bot}_{x, h}$ are unaffected by such components as $F$, and are thus determined entirely by the collection of walls $\fW_x$. Now recall that by the Edwards--Sokal coupling, we can sample the Potts model by first revealing the edge configuration, and then coloring open clusters independently at random. Again, for $F$ and $V$ as above, the random color(s) assigned to $V$ do not affect whether or not the vertices of $V$ are in $\Ared$, and similarly for $\Ablue$. Hence, fixing the collection of walls $\fW_x$ also determines the conditional probabilities of the events $\cA^{\noRed}_{x, h}, \cA^{\Blue}_{x, h}$. 
\end{proof}

The proofs of the next two propositions (\cref{prop:decorrelation-1,,prop:decorrelation-2}) follow from what is already known in the literature. Indeed, in \cite[Propositions 2.1, 2.3]{BLP79b}, it is shown how the decorrelation statements in Ising follow from the machinery developed by Dobrushin in \cite[Lemmas 1, 2]{Dobrushin72} once certain bounds have been proved relating to groups of walls in the interface (see \cite[Eqs.~(2.2)--(2.7)]{BLP79b}). However, Dobrushin's machinery is general and not restricted to the Ising model, and hence the proof in \cite{BLP79b} holds in the Random cluster setting as long as we can prove the analogous bounds. In fact, the only remaining bound not already proved in \cite{GielisGrimmett02} is the following: Take any admissible group of walls $(F_x)_{x \in \cL_{1/2}}$. Recall that $\sE$ denotes an empty wall. Let $\sZ_x^n(F_x \mid (F_y)_{y \neq x})$ denote 
\begin{equation*}
    \sZ_x^n(F_x \mid (F_y)_{y \neq x}) := \frac{\bar{\mu}_n(F_x, (F_y)_{y \neq x})}{\bar{\mu}_n(\sE_x, (F_y)_{y\neq x})}\,.
\end{equation*}
Then, for some constants $C, c>0$ and all $\beta > \beta_0$, all $x, z$,
\begin{equation}\label{eq:decorrelation-CE-bound-1}
    \left|\log \frac{\sZ_x^n(F_x \mid (F_y)_{y \neq x})}{\sZ_x^n(F_x \mid (F_y)_{y \notin (x, z)}, \sE_z)}\right| \leq Ce^{-c\beta|x-z|}
\end{equation}
if 
\begin{equation*}
    |x-z| \geq 10(\fm(F_x) + \fm(F_z))\,.
\end{equation*}
Furthermore, denoting $W_n = \Lambda_n \cap \cL_{1/2}$, we have for any $m \geq n$, 
\begin{equation}\label{eq:decorrelation-CE-bound-2}
    \left|\log \frac{\sZ_x^m(F_x \mid (F_y)_{y \in W_n \setminus x}, (\sE_z)_{z \in W_m \setminus W_n})}{\sZ_x^n(F_x \mid (F_y)_{y \in W_n \setminus x})}\right| \leq Ce^{-c\beta(\min_{y \in W_m \setminus W_n} |x-y|)}
\end{equation}
if
\begin{equation*}
    \min_{y \in W_m \setminus W_n} |x-y| \geq 10\fm(F_x)\,.
\end{equation*}
The proof of these two bounds uses cluster expansion, and is done in the Ising case in \cite{Dobrushin73}. The same proof applies here verbatim as long as we can additionally control the terms $(1-e^{-\beta})^{|\partial \cI| - |\partial \cJ|}q^{\kappa_\cI - \kappa_\cJ}$ in the cluster expansion when comparing interfaces. However, it is clear that looking at the ratios in \cref{eq:decorrelation-CE-bound-1,eq:decorrelation-CE-bound-2}, these terms will all cancel out to be equal to 1. Hence, we have

\begin{proposition}\label{prop:decorrelation-1}
For every $\beta > \beta_0$, there is a constant $C > 0$ such that for every $n \leq m$, $r > 0$, and sequence $x = x_n$,
\begin{equation*}
    \norm{\bar{\mu}_n((\sF_s)_{|s-x| < r} \in \cdot) - \bar{\mu}_m((\sF_s)_{|s-x| < r} \in \cdot)}_{\tv} \leq C\exp(-(d(x, \partial \Lambda_n) - r)/C)\,.
\end{equation*}
\end{proposition}

\begin{proposition}\label{prop:decorrelation-2}
For every $\beta > \beta_0$, there is a constant $C > 0$ such that for every $n$, $r > 0$, and sequences $x = x_n$ and $y = y_n$,
\begin{equation*}
    \norm{\bar{\mu_n}((\sF_s)_{|s-x| < r} \in \cdot, (\sF_t)_{|t-y|<r} \in \cdot) - \bar{\mu}_n((\sF_x)_{|s-x| < r} \in \cdot)\bar{\mu}_n((\sF_y)_{|t-y|<r} \in \cdot)}_{\tv} \leq Ce^{-(|x-y| - 2r)/C}\,.
\end{equation*}
\end{proposition}

We now apply these decorrelation estimates to our events of interest, which we phrase as the following corollaries:
\begin{corollary}\label{cor:decorrelation-1}
For every $\beta > \beta_0$, there is a constant $C > 0$ such that for every $n \leq m$, and sequences $x = x_n$ $y = y_n$ such that $d(x, \partial \Lambda_n) \wedge d(y, \partial \Lambda_m) \geq r$,
\begin{equation*}
    |\bar{\mu}_n(E_h^x) - \bar{\mu}_m(E_h^y)| \leq C\exp[-r/C]\,.
\end{equation*}
Moreover, the same statement holds with the events $\cA^{\Bot}_{x, h}, \cA^{\Bot}_{y, h}$ instead.
\end{corollary}
\begin{proof}
    We follow the proof of \cite[Corollary 6.4]{GL_max}. For $N$ large, we can write
    \begin{equation*}
        |\bar{\mu}_n(E_h^x) - \bar{\mu}_m(E_h^y)| \leq |\bar{\mu}_n(E_h^x) - \bar{\mu}_N(E_h^x)| + |\bar{\mu}_N(E_h^x) - \bar{\mu}_N(E_h^y)| +|\bar{\mu}_m(E_h^y) - \bar{\mu}_N(E_h^y)|\,.
    \end{equation*}
    By \cref{prop:pillar-marginal-of-wall,clm:nested-walls-determine-pillars-RC-Potts}, we have after paying an additive error of $Ce^{-(\beta - C)r}$ that the first and third terms are bounded by $C\exp[-r/C]$ by \cref{prop:decorrelation-1}. For $\beta$ large, this additive error is of smaller order than our bound. The second term vanishes as $N \to \infty$ by translation invariance in the $xy$-directions of the infinite volume measure.
\end{proof}

\begin{corollary}\label{cor:decorrelation-2}
For every $\beta > \beta_0$, there is a constant $C > 0$ such that for every $n$, and sequences $x = x_n$ and $y = y_n$ such that $d(x, y) \geq r$, we have
\begin{equation*}
    |\bar{\mu}_n(E_h^x, E_h^y) - \bar{\mu}_n(E_h^x)\bar{\mu}_n(E_h^y)| \leq C\exp[-r/C]\,.
\end{equation*}
Moreover, the same statement holds with the events $\cA^{\Bot}_{x, h}, \cA^{\Bot}_{y, h}$ instead.
\end{corollary}
\begin{proof}
    This is immediate by combining \cref{prop:pillar-marginal-of-wall,clm:nested-walls-determine-pillars-RC-Potts} with \cref{prop:decorrelation-2}.
\end{proof}

For the Potts model, the results only make sense for $q \geq 2$, but otherwise the proofs are exactly the same.

\begin{corollary}\label{cor:Potts-decorrelation-1}
For every $\beta > \beta_0$, $q \geq 2$, there is a constant $C > 0$ such that for every $n \leq m$, and sequences $x = x_n$ $y = y_n$ such that $d(x, \partial \Lambda_n) \wedge d(y, \partial \Lambda_m) \geq r$,
\begin{equation*}
|\phi_n(\cA^{\noRed}_{x, h}) - \phi_m(\cA^{\noRed}_{y, h})| \leq C\exp[-r/C]\,.
\end{equation*}
Moreover, the same statement holds with the events $\cA^{\Blue}_{x, h}, \cA^{\Blue}_{y, h}$ instead.
\end{corollary}

\begin{corollary}\label{cor:Potts-decorrelation-2}
For every $\beta > \beta_0$, there is a constant $C > 0$ such that for every $n$, and sequences $x = x_n$ and $y = y_n$ such that $d(x, y) \geq r$, we have
\begin{equation*}
|\phi_n(\cA^{\noRed}_{x, h}, \cA^{\noRed}_{y, h}) - \phi_n(\cA^{\noRed}_{x, h})\phi_n(\cA^{\noRed}_{y, h})| \leq C\exp[-r/C]\,.
\end{equation*}
Moreover, the same statement holds with the events $\cA^{\Blue}_{x, h}, \cA^{\Blue}_{y, h}$ instead.
\end{corollary}

Finally, we provide the missing proof of \cref{clm:G-split-product}, which is restated here for convenience. Recall the definition of $G_h^x$ in \cref{def:Ghx}.

\begin{claim}\label{clm:Appendix-G-split-product}
    For all $\beta > \beta_0$, there exists a constant $\epsilon_\beta$ such that for all $h \ll \log^2 n$, $(x, y) \in \cL_{1/2, n}^\mathrm{o}$ with $d(x, y) \leq \log^2 n$ and $n$ sufficiently large,
    \begin{equation*}
        \bar{\mu}_n(G_h^x,\, G_h^y) \leq (1+\epsilon_\beta)\frac{(e^\beta + q-1)^2}{q}\bar{\mu}_n(E_h^x)\bar{\mu}_n(E_h^y)\,,
    \end{equation*}
    where the definition of $G_h^x$ can be taken with respect to any of the four interfaces.
\end{claim}

\begin{proof}
As noted before, we can assume that we are working with $G_h^x$ defined with respect to $\Itop$. Begin by defining the sets
\begin{align*} \sep_n^1 &= \left\{ I= \cI(\omega)\mbox{ for some }\omega\in G_h^x \cap G_h^y \right\}\,,
\end{align*}
and
\begin{align*} \hat{\sep}_n^1 =\left\{ I\in\sep_n^1\,:\; f_{[y,y-\ez]}\in I\right\}\,.
\end{align*}

We can force the face below $y$ to be in $\clfaces$ at a cost of $\frac{e^\beta+q-1}{q}$ by \cref{obs:close-edge}, noting that closing this edge always creates an additional open cluster because the event $G_h^y$ ensures that $y$ is a cut-point and thus cannot have a path of open edges to $y - \ez$ without using the edge $[y, y - \ez]$. Furthermore, the event $G_h^x$ only concerns properties of the interface. Thus, we have
\begin{equation*}
    \bar{\mu}_n(G_h^x,\, G_h^y) = \frac{\mu_n(\sep_n^1)}{\mu_n(\sep_n)} \leq \frac{e^\beta+q-1}{q}\frac{\mu_n(\hat\sep_n^1)}{\mu_n(\sep_n)} = \frac{e^\beta+q-1}{q}\frac{1}{\mu_n(\sep_n)}\sum_{I \in \hat\sep_n^1} \mu_n(\cI = I)\,.
\end{equation*}
Now, we want to group the interfaces according to the truncation $\cI \setminus \cP_y$. Recall that this truncated interface is obtained by removing from $\cI$ the faces of $\cP_y$ and adding in the faces which are directly below vertices of $\cP_y$ which have height $1/2$ (see \cref{def:truncated-interface}). As this is not equal to the face set ``$\cI$ set-minus $\cP_y$", we will write $\cI' = \cI \setminus \cP_y$ to avoid confusion and also highlight the parallel to the proof of \cref{lem:submultiplicativity-in-Gamma}. With this notation, the above sum is equal to
\begin{equation*}
    \frac{e^\beta+q-1}{q}\frac{1}{\mu_n(\sep_n)}\sum_{I':\, I \in \hat\sep_n^1} \mu_n(\cI' = I',\, G_h^y)\,.
\end{equation*}
Now recall that we showed in the beginning of the proof of \cref{prop:compare-A_h^x-E_h^x} that $\widetilde E_h^x \cap \Iso_{x, L, h} \subseteq A_h^y$. The only property of $\Iso_{x, L, h}$ used in that proof was that $x$ is a cut-point, and hence $G_h^y \subseteq A_h^y$ as well, so the above is easily upper bounded by
\begin{equation*}
    \frac{e^\beta+q-1}{q}\frac{1}{\mu_n(\sep_n)}\sum_{I':\, I \in \hat\sep_n^1} \mu_n(\cI' = I',\, A_h^y)\,.
\end{equation*}
It is important that we move from the event $G_h^y$ to $A_h^y$ because the latter is defined independently from the interface, and is also a decreasing event. Now, define $\partial^\dagger I'$ by deleting from $\partial I'$ the 4 faces that are 1-connected to the face $f_{[y, y - \ez]}$ (out of the 12 such faces) and have height $> 0$. On the event $\cI' = I'$, we know that $\partial^\dagger I' \subseteq \opfaces$ by maximality of $\cI$. Ordinarily, we would not know that $I' \subseteq \clfaces$ because $I'$ can include faces that are not in $I$. However, the event $G_h^y$ ensures that the only possible extra face in $I'$ is $f_{[y, y - \ez]}$, and $\hat\sep_n^1$ was defined so that this face is always in $I$. Hence, combining the above gets us
\begin{equation}\label{eq:Ghx-separate-by-truncated-interface}
    \bar{\mu}_n(G_h^x,\, G_h^y) \leq \frac{e^\beta+q-1}{q}\frac{1}{\mu_n(\sep_n)}\sum_{I':\, I \in \hat\sep_n^1} \mu_n(I' \subseteq \clfaces,\, \partial^\dagger I' \subseteq \opfaces,\, A_h^y)\,.
\end{equation}
Writing the latter probability as
\[ \mu_n(I' \subseteq \clfaces,\, \partial^\dagger I' \subseteq \opfaces ,\, A_h^y)= \mu_n\left(A_h^y \mid \cS_{I'}\right) \mu_n\left(\cS_{I'}\right)
 \]
 for
 \begin{equation}
 \cS_{I'}:=\left\{ I'\subseteq \clfaces\,,\, \partial^\dagger I'\subseteq \opfaces\right\}\,,
 \end{equation}
 we note that the events $\cS_{I'}$ are disjoint by applying verbatim \cref{case-1:S-I'-disjoint}
 from the proof of \cref{clm:S-I'-disjoint}. Since every $\cS_{I'}$ for $I'\in\hat\sep_n^1$ further implies $G_h^x$ and $\sep_n$, it follows from the above claim that
\[ \sum_{I':I\in\hat\sep_n^1} \mu_n(\cS_{I'})\leq \mu_n(G_h^x,\, \sep_n)\,,\]
and consequently (together with \cref{eq:Ghx-separate-by-truncated-interface} and the fact that $G_h^x \subseteq E_h^x$):
\begin{equation}\label{eq:Ghx-mu(Ah-Gamma)-up-to-max-Ah2}
\bar{\mu}_n(G_h^x,\, G_h^y)\leq 
\frac{e^\beta+q-1}{q}\bar{\mu}_n(E_h^x)\max_{I':I\in\hat\sep_n^1}\mu_n\left(A_h^y \mid \cS_{I'}\right)\,.
\end{equation}
Hence, to conclude the proof it will suffice to show that for $I'$ such that $I \in \hat\sep_n^1$, we have 
$\mu_n(A_h^y \mid \cS_{I'})) \leq C(\beta,q) \bar{\mu}_n(E_h^y)$; namely, we prove this for $C(\beta,q)=(1+\epsilon_\beta)(e^\beta+q-1)$.

As before, most of the labor is showing that
\begin{equation}\label{eq:Ghx-DMP-step}
    \mu_n(A_h^y \mid I' \subseteq \clfaces,\, \partial^\dagger I' \subseteq \opfaces) = \mu_n(A_h^y \mid f_{[y, y - \ez]} \in \clfaces,\, \partial^\dagger I' \subseteq \opfaces)\,.
\end{equation}
By \cref{prop:out-graph-connected}, the subgraph $K=(\Delta_{\scV}I', \Delta_{\scE}I')$ is connected. Let $B_\scV$ be the vertices of $\Delta_{\scV}I' \cap \Lambda_n$ with a $\Lambda_n$-path to $\partial \Lambda_n^+$ that do not cross a face of $I'$, and let $B_\scE$ be the edges of the induced subgraph of $K$ on $B_\scV$. Then, \cref{clm:Bv-connected} implies that the graph $(B_\scV, B_\scE)$ is connected, as $I'$ is an interface. 

Now let $G$ be the subgraph of $\Lambda_n$ induced on the set of vertices $V$ that are not disconnected from $\partial \Lambda_n^+$ by $I'$. Let $E$ be the edge set of $G$. The next claim says that $G$ is the right graph to be looking at, and is the analog of \cref{clm:Ah2-E-measurable}. The proof is nearly identical, except we need to use properties of $G_h^y$ instead of properties of the event $\Gamma_{h_1}^x$ defined there. We include the full proof for completion. For ease of reference, denote the four adjacent vertices to $y$ that have height $1/2$ as $z_1, z_2, z_3, z_4$.
\begin{claim}\label{clm:Ghx-Ah2-E-measurable}
    For any interface $I \in \hat\sep_n^1$, let $G = (V, E)$ be defined as above (w.r.t.\ $I'$). Then, conditional on $\partial^\dagger I' \subseteq \opfaces$, the event $A_h^y$ is measurable w.r.t.\ $\{\omega_e :\; e \in E\}$.
\end{claim}
\begin{proof}
As in the proof of \cref{clm:Ah2-E-measurable}, 
by the definition of $A_h^y$ it suffices to show that for any $1$-connected subset $F$ of $\clfaces\cap\cL_{>0}$ that includes $\{f_{[y,z_i]}\}_{i=1}^4$, the edges $\{e \,:\; f_e\in F\}$ must all belong to $E$. First, we show   
\begin{equation}\label{eq:Ghx-side-edges-y-zi-in-E}
 \{[y, z_i]\}_{i = 1}^4 \subseteq E\,,   
\end{equation} 
or equivalently that $y$ and each $z_i$ are in $V$. For any $I \in \hat\sep_n^1$, the requirement that $P_y$ has a cut-point at $y$ ensures that $I$ does not separate any of the $z_i$ from $\partial \Lambda_n^+$, and $I' \subseteq I$. Thus, $\{z_i\}_{i = 1}^4 \subseteq V$. Furthermore, since $\{f_{[y, z_i]}\}_{i = 1}^4 \cap I' = \emptyset$, then $y$ is also in $V$. (In fact, since $f_{[y, y - \ez]} \in I'$, we additionally have that $y, z_i \in B_\scV$.) Second, we show that 
\begin{equation}\label{eq:Ghx-f-1-connected-to-f_y_z-not-in-I'}
    \left\{ f\,:\;  \hgt(f) > 0 \mbox{ and $f$ is 1-connected to  $\bigcup_{i = 1}^4 f_{[y, z_i]}$}\right\} \cap I'  = \emptyset\,.
\end{equation} 
Indeed, we know that for any $I \in \hat\sep_n^1$, by the fact that $y$ is a cut-point of $P_y$, we have $f_{[y,z_i]}\in P_y$ for each $i=1,\ldots,4$. Thus, any faces whose height exceeds $0$ and are 1-connected to one of the $f_{[y, z_i]}$ would have been included in $P_y$ as on the event $G_h^y$, the only faces of $I$ that are in $\partial P_y$ are at height 0. 

Now, consider the faces $F$. Since $F \subseteq \clfaces$, on the event $\partial^\dagger I' \subseteq \opfaces$ we have
\begin{equation}\label{eq:Ghx-F-intersect-partial-I'}
F \cap \partial I' \subseteq \partial I' \setminus \partial^\dagger I'=  \{f_{[y, z_i]}\}_{i = 1}^4.
\end{equation}
We claim that by definition of $F$ and \cref{eq:Ghx-f-1-connected-to-f_y_z-not-in-I',eq:Ghx-F-intersect-partial-I'} we can infer that
\begin{equation}\label{eq:Ghx-F-intersect-I'-is-empty} F \cap I' = \emptyset\,;
\end{equation}
to see this, suppose there exists some $f \in F \cap I'$, and let $P=(f_i)_1^m$ be a 1-connected of faces in $F$ connecting $f_0=f$ to $f_m=f_{[y, z_1]}$. Let $j$ be the minimal index such that $f_j\notin I'$ (well-defined since $f_m\notin I'$). Then $f_j\in F\cap \partial I'$, hence $f_j=f_{[y,z_i]}$ for some $i$ by \cref{eq:Ghx-F-intersect-partial-I'}, whence $f_{j-1}$ cannot exist by \cref{eq:Ghx-f-1-connected-to-f_y_z-not-in-I'} and the fact that $F \subseteq \cL_{>0}$, contradiction.

We are now ready to show that every edge $e$ with $f_e \in F$ must be in $E$. For any $f \in F$, there is a 1-connected path $P$ of faces in $F$ from $f$ to one of the $f_{[y, z_i]}$. If $f= f_e$ for some $e \notin E$, then let $g = g_{[u, v]}$ be the last face in the path $P$ such that $[u, v] \notin E$, so that $g$ is 1-connected to $g' = g'_{[u', v']}$ where $[u', v'] \in E$. W.l.o.g., let $u \notin V$. No matter how $g$ and $g'$ are connected to each other, $u$ is always $\Lambda_n$-adjacent to $u'$ (or $v'$), with the face $g'' = g''_{[u, u']}$ (or $= g''_{[u, v']})$ being either equal to or 1-connected to $g$. However, since $g''$ separates $u \notin V$ from $u' \in V$, then $g'' \in  I'$. Hence, as $g$ and $g''$ are equal or $1$-connected, we have $g \in \overline{I'}$. But then the assumption that $g = g_{[u, v]}$ for $[u, v] \notin E$ contradicts the combination of \cref{eq:Ghx-side-edges-y-zi-in-E,eq:Ghx-F-intersect-partial-I',eq:Ghx-F-intersect-I'-is-empty}.
This concludes the proof.
\end{proof}
\begin{claim}\label{clm:Ghx-G-bc}
For any interface $I \in G_h^x \cap G_h^y$, let $(B_\scV, B_\scE)$ and $G = (V, E)$ be defined as above (w.r.t.\ $I'$). The following hold:
\begin{enumerate}[(i)]
    \item \label{it:Ghx-vertex-bc}
    The vertices $B_{\scV} \cup \partial \Lambda_n^+$ form a vertex boundary for $V$ (in that every $\Lambda_n$-path from $v\in V$ to $V^c$ must cross one of those vertices).

    \item \label{it:Ghx-edge-bc}
    The graph obtained from $(B_\scV,B_\scE)$ by deleting the vertex $y$ (and edges incident to it) is connected. Consequently, on the event $\partial^\dagger I' \subseteq \opfaces$, the vertices $B_\scV \setminus \{y\}$ are all part of a single open cluster in~$\omega$.

    \item \label{it:Ghx-y-bc}
    On the event $f_{[y, y - \ez]} \in \clfaces$, there cannot be a path of open edges in $E^c$ connecting $y$ to $\partial \Lambda_n^+ \cup B_\scV \setminus \{y\}$.
\end{enumerate}
\end{claim}
\begin{proof}
The proof of \cref{it:Ghx-vertex-bc,it:Ghx-y-bc} follows verbatim from the proof in \cref{clm:G-bc}.

For \cref{it:Ghx-edge-bc}, let $\tilde B_\scE$ be the outcome of removing from $B_\scE$ the four edges $[y, z_i]$. First, we claim that there are no other edges of $B_\scE$ incident to $y$, via the following two items:
\begin{enumerate}[(a)]
    \item $[y, y - \ez] \notin B_\scE$ since
    $f_{[y, y - \ez]} \in I'$;
    \item 
    $[y, y+\ez] \notin B_\scE$, as otherwise, having $f_{[y, y+\ez]} \in \partial I'$, there must be a face $g \in I'$ that is 1-connected to $f_{[y, y+\ez]}$ with $\hgt(g) > 0$. The face $g$ must be 1-connected to $f_{[y, z_i]}$ for some $i$, but by \cref{eq:Ghx-f-1-connected-to-f_y_z-not-in-I'}, this is impossible.
\end{enumerate}
Thus, the graph $(B_\scV \setminus \{y\}, \tilde B_\scE)$ is equal to the subgraph of $(B_\scV, B_\scE)$ induced on $B_\scV \setminus \{y\}$. So, to show that $(B_\scV \setminus \{y\}, \tilde B_\scE)$ is connected, it suffices to exhibit a path in $\tilde B_{\scE}$ between $z_1 = y+\ex$ and $z_2=y+\ey$ (whence by symmetry there will be such paths between any two of the $z_i$'s). These are connected in $\Lambda_n$ by the path 
\begin{equation*}
    P = \big(y+\ex ,y+\ex + \ey, y+\ey\big)\,.
\end{equation*}
Now, by \cref{it:Ghx-faces-of-I-near-pillar} of the definition of $G_h^x$ (and the fact that $I \in \hat\sep_n^1$), we know that $I'$ contains the faces directly below $z_i$ for each $i=1,\ldots,4$. Furthermore, we know by \cref{eq:Ghx-f-1-connected-to-f_y_z-not-in-I'} that the faces $f_{[y + \ex, y + \ex + \ey]}$ and $f_{[y+\ex+\ey, y+\ey]}$ are not in $I'$. Combined, the two aforementioned faces are in $\partial I'$. Since $z_1 \in B_\scV$, this implies every vertex in the path $P$ is also in $B_\scV$. Thus, the path $P$ uses only edges in $\tilde B_\scE$ as required, and altogether $(B_\scV \setminus \{y\}, \tilde B_\scE)$ is connected.
\end{proof}
Combining \cref{clm:Ghx-Ah2-E-measurable,clm:Ghx-G-bc} with the Domain Markov property, we have \cref{eq:Ghx-DMP-step}. Next, the same computation as in \cref{eq:remove-conditioning-on-edge} shows that we can remove the conditioning on the event $f_{[y, y - \ez]} \in \clfaces$ by paying a factor of $q$. We can thereafter remove the conditioning on $\partial^\dagger I' \subseteq \opfaces$ by FKG, getting that
\begin{equation*}
    \mu_n(A_h^y \mid f_{[y, y - \ez]} \in \clfaces,\, \partial^\dagger I' \subseteq \opfaces) \leq q\mu_n(A_h^y)\,.
\end{equation*}
Since $A_h^y$ is a decreasing event, we have by FKG again that
$\mu_n(A_h^y) \leq \bar{\mu}_n(A_h^y)$.
Using \cref{prop:compare-A_h^x-E_h^x} (which, we recall, compares $A_h^y$ to $E_h^y$), we have
\begin{equation*}
    \bar{\mu}_n(A_h^y) \leq (1+\epsilon_\beta)\frac{e^\beta + q-1}{q} \bar{\mu}_n(E_h^y)\,.
\end{equation*}
Thus, combining the above with \cref{eq:Ghx-DMP-step}, we have
\begin{equation*}
\max_{I':I\in\hat\sep_n^1}\mu_n\left(A_h^y \mid \cS_{I'}\right) \leq (1+\epsilon_\beta)(e^\beta + q-1)\bar{\mu}_n(E_h^y)\,,
\end{equation*}
which together with \cref{eq:Ghx-mu(Ah-Gamma)-up-to-max-Ah2} concludes the proof.
\end{proof}

\subsection*{Acknowledgements}
We thank an anonymous referee for many useful comments. This research was supported by NSF grants DMS-1812095 and DMS-2054833.

% \subsection*{Declarations}
% The authors have no competing interests to declare that are relevant to the content of this article.

\bibliographystyle{abbrv}
\bibliography{Potts3D_max}

\end{document}